\documentclass[11pt]{amsart}

\input{preamble.sty}

\usepackage[utf8]{inputenc} 
\usepackage[english]{babel}
\usepackage{csquotes} 
\usepackage{lmodern}
\usepackage[foot]{amsaddr}

\usepackage{hyperref}
\usepackage[capitalize]{cleveref} 

\usepackage[
    backend=biber,
    style=alphabetic,
    maxnames=100   
]{biblatex}
\newtheorem{theorem}{Theorem}[section]
\newtheorem{proposition}[theorem]{Proposition}
\newtheorem{lemma}[theorem]{Lemma}
\newtheorem{corollary}[theorem]{Corollary}
\newtheorem{maintheorem}{Theorem}
\newtheorem*{AssumptionCon}{Assumption ConFib}
\newcounter{AssumptionCon}

\Crefname{AssumptionCon}{Assumption}{Assumptions}

\newtheorem*{AssumptionDil}{Assumption Dil}
\newcounter{AssumptionDil}

\Crefname{AssumptionDil}{Assumption}{Assumptions}

\newtheorem*{AssumptionStraight}{Assumption Straight}
\newcounter{AssumptionStraight}

\Crefname{AssumptionStraight}{Assumption}{Assumptions}

\theoremstyle{definition}
\newtheorem{definition}[theorem]{Definition}

\theoremstyle{remark}
\newtheorem{remark}[theorem]{Remark}
\newtheorem{example}[theorem]{Example}

\newtheoremstyle{assumption}{}{}{\itshape}{}{\bfseries}{.}{.5em}{#1 #3}
\theoremstyle{assumption}

\Crefname{assumption}{Assumption}{Assumptions}

\newcounter{diagram}  % Create counter for diagrams

\crefname{diagram}{diagram}{diagrams}
\crefname{diagram}{Diagram}{Diagrams}
\creflabelformat{diagram}{(#1) #2 #3}

\newcommand{\llp}{(\!(}
\newcommand{\rrp}{)\!)}
\newcommand{\llb}{\llbracket}
\newcommand{\rrb}{\rrbracket}

\title{The Moduli Stack of Breuil--Kisin Modules with Descent Data for Reductive Groups}

\author{Eivind Otto Hjelle}
\email{eohjelle@gmail.com}

\begin{document}

\begin{abstract}
    We introduce and study the moduli stack $\mathcal{Y}$ of Breuil--Kisin modules with $\widehat{G}$-structure and descent data, or Breuil--Kisin $(\Gamma,\widehat{G})$-torsors for short. Specifically, for a dominant cocharacter $\mu$, we define the moduli stack $\mathcal{Y}^{\leq \mu}$ of Breuil--Kisin $(\Gamma,\widehat{G})$-torsors with Hodge--Tate weights bounded by $\mu$. We prove that $\mathcal{Y}^{\leq \mu}$ is a $p$-adic formal algebraic stack, and show that it is smoothly equivalent to (the $p$-adic completion of) a twisted Schubert variety $\operatorname{Gr}^{\leq \mu}_{\mathcal{G}}$ in the sense of Pappas--Zhu \cite{pappas-zhu}, generalizing results of \cite{pappasPhiModulesCoefficient2009}, \cite{caraiani-levin}, \cite{local-models}, \cite{leeEmertonGeeStacks2023}, and others. This is a reformatted and lightly edited version of the author's PhD thesis, submitted to Northwestern University in August 2024.
\end{abstract}

\maketitle

\tableofcontents

\section{Introduction}

Let us first introduce some notation that will be fixed throughout this introduction. 
Let \(L \supseteq \mathbb{Q}_{p}\) be a tamely ramified Galois extension with Galois group \(\Gamma = \operatorname{Gal}(L / \mathbb{Q}_{p})\), and let \(G\) be a
connected reductive group over \(\mathbb{Q}_{p}\) which splits over \(L_{0}\), where \(L_{0}\) is the maximal unramified extension of \(\mathbb{Q}_{p}\) inside \(L\).
Let \(\Gamma_{\mathbb{Q}_{p}} = \operatorname{Gal}(\overline{\mathbb{Q}}_{p} / \mathbb{Q}_{p})\), \(\Gamma_{L} = \operatorname{Gal}(\overline{\mathbb{Q}}_{p} / L)\),  and let \(E \supseteq \mathbb{Q}_{p}\) be a (large) finite extension with ring of integers \(\mathcal{O}\), uniformizer \(\varpi \in \mathcal{O}\) and residue field \(\mathbb{F}\).
In this introduction as well as some other sections we will assume that \(E \supseteq L \). 
We form the Langlands dual group \({^L G} = \widehat{G} \rtimes \Gamma\),\footnote{For some questions in the Langlands program it is more natural to use the \(C\)-group of \cite{buzzard-gee}, but the \(C\)-group is the \(L\)-group of another \(G\), so for our purpose the \(L\)-group will suffice.}
where \(\Gamma\) acts on \(\widehat{G}\) via pinned automorphisms, encoding the non-splitness of \(G\).

\subsection{Context: The categorical \(p\)-adic Langlands program}
In recent years there has been a surge of interest in the \(p\)-adic and mod \(p\) Langlands program, which can be thought of as a \(p\)-adic deformation of the Langlands program.
Our work can be motivated by questions in this field, so we begin by giving some of the surrounding context. 
When explaining this background, we focus on the groups \(\GL_{n}\), \(n \in \mathbb{Z}_{\geq 1}\). 

\subsubsection{Langlands reciprocity conjectures}
The global Langlands reciprocity conjecture asserts that there is a (Langlands) correspondence
\[\begin{tikzcd}
    {\left\lbrace \begin{array}{c} \text{``certain'' representations}\\ \operatorname{Gal}\left( \overline{\mathbb{Q}} / \mathbb{Q} \right) \to \operatorname{GL}_{n}(\overline{\mathbb{Q}}_\ell) \end{array} \right\rbrace}
    & {\left\lbrace \begin{array}{c} {\text{``certain'' (algebraic automorphic)}} \\ {\text{representations of }}\operatorname{GL}_{n}(\mathbb{A}_{\mathbb{Q}}) \end{array} \right\rbrace} , 
	\arrow[dashed, tail reversed, from=1-1, to=1-2]
\end{tikzcd}\] 
where \(\ell\) is a prime, and \(\mathbb{A}_{\mathbb{Q}} = \mathbb{R} \times \prod_{p \text{ prime}}' \mathbb{Q}_{p}\) is the adele ring of \(\mathbb{Q}\).
We call the left hand side the ``Galois side''  and the right hand side the ``automorphic side''.
The correspondence is subject to a number of compatibilities that we will not spell out here, but that may be illustrated by the following examples:
\begin{itemize}
    \item \(n=1\). The correspondence is global class field theory. 
    \item \(n=2\). The correspondence subsumes the \emph{modularity theorem} (formerly known as the Taniyama--Shimura conjecture), namely the fact that all elliptic curves over \(\mathbb{Q}\) are modular. This was proved by \cite{wilesModularEllipticCurves1995} \cite{taylorRingTheoreticPropertiesCertain1995} \cite{breuilModularityEllipticCurves2001}, and led to the proof of Fermat's Last Theorem.
\end{itemize}

Another form of the global reciprocity conjecture states that algebraic automorphic representations correspond to \emph{motives}, which are ``linear algebra models'' of algebraic varieties carrying all cohomological information. 
In practice, explicit constructions of Galois representations attached to automorphic representations are found in the \(\ell\)-adic cohomology of Shimura varieties. 
From this point of view, the \(\ell\)-adic Galois representations on the Galois side are the \(\ell\)-adic realizations of the corresponding motives, or (parts of) the \(\ell\)-adic cohomology of certain varieties over \(\mathbb{Q}\).

What is then the role of the prime \(\ell\)?
The automorphic side is completely independent of \(\ell\), so in some sense the choice of prime \(\ell\) seems immaterial. 
If we think of the \(\ell\)-adic representations on the Galois side as \(\ell\)-adic realizations of motives corresponding to automorphic forms, it is clear that such a Galois representation is merely a choice of representative for an entire ``compatible family'' of \(\ell\)-adic Galois representations as \(\ell\) ranges over all primes (as made precise in \cite{taylorGaloisRepresentations2004}, for example).
On the other hand, the specific Galois representation one obtains certainly depends on the prime \(\ell\). 

Suppose now that we fix a prime \(\ell\),
and a Galois representation 
\[\rho : \operatorname{Gal}\left( \overline{\mathbb{Q}} / \mathbb{Q} \right) \to \GL_{n}(\overline{\mathbb{Q}}_\ell).\]
One way to understand a Galois representation such as \(\rho\) is to consider the restrictions \(\rho |_{\Gamma_{\mathbb{Q}_{p}}} : \Gamma_{\mathbb{Q}_{p}} = \operatorname{Gal}\left( \overline{\mathbb{Q}}_{p} / \mathbb{Q}_{p} \right) \to \GL_{n}(\overline{\mathbb{Q}}_\ell)\) for various primes \(p\), 
where we view \(\Gamma_{\mathbb{Q}_{p}} \leq \operatorname{Gal}(\overline{\mathbb{Q}} / \mathbb{Q})\) as the subgroup fixing a choice of prime in \(\overline{\mathbb{Q}}_{p}\) above \(p\) (the decomposition group at \(p\)).
At this point we must distinguish between the vastly different cases \(\ell \neq p\) and \(\ell = p\); it turns out that \(\rho|_{\Gamma_{\mathbb{Q}_{p}}}\) remembers much more about \(\rho\) if \(\ell = p\) than if \(\ell \neq p\).

\begin{example}
    Suppose that \(\rho\) arises as the étale cohomology of a proper smooth variety \(X\) over \(\mathbb{Q}\).
    If \(\ell \neq p\), and \(X\) has good reduction at \(p\), then \(\rho|_{\Gamma_{\mathbb{Q}_{p}}}\) is unramified (and the converse is true if \(X\) is an abelian variety, by Néron--Ogg--Shafarevich). 
    On the other hand, if \(\ell = p\), then \(\rho|_{\Gamma_{\mathbb{Q}_{p}}}\) will rarely be unramified, and the machinery of \(p\)-adic Hodge theory allows us to extract more intricate information about \(X\). 
    For example, the fact that \(X\) is smooth and proper implies that \(\rho|_{\Gamma_{\mathbb{Q}_{p}}}\) is de Rham, and remembers the Hodge filtration. 
    If \(X\) has good reduction at \(p\), then the representation \(\rho|_{\Gamma_{\mathbb{Q}_{p}}}\) is moreover \emph{crystalline}. 
\end{example}
\begin{remark}
    From a group theory perspective, the study of representations \(\Gamma_{\mathbb{Q}_{p}} \to \GL_{n}(\overline{\mathbb{Q}}_\ell)\) 
    is more complex when \(\ell = p\) than when \(\ell \neq p\), since the Galois group \(\Gamma_{\mathbb{Q}_{p}}\) contains a large pro-\(p\) subgroup (the wild inertia), 
    whereas \(\GL_{n}(\overline{\mathbb{Q}}_\ell)\) is locally pro-\(\ell\). Hence there are much fewer such representations when \(\ell \neq p\) than when \(\ell = p\). 
\end{remark}

In the context of the global Langlands reciprocity conjectures, what happens at the prime \(p = l\) detects most of the complexity of the global situation. 
This is well illustrated by:
\begin{itemize} 
    \item The Fontaine-Mazur conjecture, which is closely related the the global Langlands reciprocity conjecture. Namely, representations \(\rho\) as above are conjectured to arise as the \(\ell\)-adic realization of a motive (or more concretely, as a subquotient of the \(\ell\)-adic cohomology of a smooth proper variety over \(\mathbb{Q}\)), 
        if \(\rho|_{\Gamma_{\mathbb{Q}_{p}}}\) is unramified for all but finitely many primes \(p \neq \ell\), and \(\rho|_{\Gamma_{\mathbb{Q}_{p}}}\) is de Rham for \(p = \ell\).
    \item Automorphy lifting theorems, which to date remain the most successful approach by which to establish Langlands reciprocity. Roughly speaking, the idea is to study Galois representations in families encoding congruences between automorphic forms. In the proofs of these theorems, one studies local Galois deformation rings prime by prime, and Galois deformation rings at the prime \(p=\ell\) with \(p\)-adic Hodge theory conditions imposed play a crucial role.
\end{itemize}

Let us now say a few words about how local Langlands reciprocity fits into the global picture. 
Suppose \(\rho\) (as above) corresponds to an automorphic representation \(\pi\) via Langlands reciprocity.
By Flath's decomposition theorem we can write \(\pi = \bigotimes'_{p \text{ prime}} \pi_{p}\) as a restricted tensor product, where \(\pi_{p}\) is a smooth representation of \(\GL_{n}(\mathbb{Q}_{p})\).
The rough idea behind the local Langlands correspondence is that for each prime \(p\) there should be an intrinsic correspondence relating \(\rho |_{\Gamma_{\mathbb{Q}_{p}}}\) (or more precisely an associated Weil--Deligne representation) to \(\pi_{p}\). 
For \(\ell \neq p\) this is realized by the local Langlands correspondence for \(\GL_{n}(\mathbb{Q}_{p})\), which is a theorem due to \cite{harrisGeometryCohomologySimple2001}, with local-global compatibility due to \cite{taylorCompatibilityLocalGlobal2006}.
There is now a mismatch at the prime \(\ell = p\), because \(\pi_{p}\) only remembers a coarse shadow of \(\rho|_{\Gamma_{\mathbb{Q}_{p}}}\); specifically, \(\pi_{p}\) only remembers the Weil--Deligne representation attached to \(\rho|_{\Gamma_{\mathbb{Q}_{p}}}\), and in particular it forgets the Hodge filtration. 
This mismatch is what leads to the \emph{\(p\)-adic Langlands program}. 

\subsubsection{The categorical \(p\)-adic Langlands program}
Motivated by the above considerations, one naturally seeks the relation between the full local Galois representation \(\rho |_{\Gamma_{\mathbb{Q}_{p}}}\) (at \(l=p\)) and a \(p\)-adic deformation of \(\pi_{p}\). 
This is very tricky because any naive picture breaks down, and only recently the mystery seems to unravel. 
In short, the story at \(l=p\) should be a categorical equivalence rather than a matching of objects. 
Specifically, the \emph{categorical \(p\)-adic Langlands reciprocity conjecture}, as formulated in \cite[Conjecture 6.1.14]{emertonIntroductionCategoricalPadic2023}, states that 
there is an \(\mathcal{O}\)-linear fully faithful functor 
\begin{equation}
    \label{eq:categorical-p-adic-langlands}
    \mathfrak{A} : D_{\text{f. p.}}^\flat \left( \text{sm.} \GL_{n}(\mathbb{Q}_{p}) \right) \hookrightarrow D_{\text{coh}}^\flat(\mathcal{X}_{n}) , 
\end{equation}
where the left hand side is a certain derived category of smooth representations of \(\GL_{n}(\mathbb{Q}_{p})\), and the right hand side is a derived category of coherent sheaves on 
a geometric object \(\mathcal{X}_{n}\) called the \emph{Emerton--Gee stack of rank \(n\) étale \((\varphi, \Gamma)\)-modules}\footnote{This \(\Gamma\) is not the same as what we denote by \(\Gamma\) elsewhere in this work.} \cite{emertonModuliStacksEtale2023}.
This functor \(\mathfrak{A}\) is subject to a number of compatibilities listed in \cite[Conjecture 6.1.14]{emertonIntroductionCategoricalPadic2023}, such as the condition that certain locally algebraic representations which play a pivotal role in automorphy lifting theorems map via \(\mathfrak{A}\) to sheaves which are naturally supported on substacks characterized by \(p\)-adic Hodge theory conditions. 
The point we want to emphasize is that in the categorical \(p\)-adic Langlands program, \emph{it is paramount 
to understand the geometry of \(\mathcal{X}_{n}\), as well as the geometry of certain substacks defined by \(p\)-adic Hodge theory conditions}.

\begin{remark}
    It is expected that \eqref{eq:categorical-p-adic-langlands} can be refined into an equivalence of categories, by taking the left hand side to be a category of \(p\)-adic sheaves on \(\operatorname{Bun}_{G}\), the moduli of \(G\)-torsors on the Fargues-Fontaine curve (where \(G = \GL_{n}\) for the present discussion).
    But as of the time of writing this document, no such conjecture has been formulated.
    A large part of the difficulty is to define the appropriate category of \(p\)-adic sheaves on \(\operatorname{Bun}_{G}\). 
    See \cite[Remark 1.4.6]{emertonIntroductionCategoricalPadic2023} for additional discussion. 
\end{remark}
\begin{remark}
    The formulation of Langlands reciprocity as an equivalence of categories is standard in the geometric Langlands program, and such formulations of the local Langlands correspondence (for \(\ell \neq p\)) has been given by \cite{farguesGeometrizationLocalLanglands2024} and \cite{zhuCoherentSheavesStack2021a}.
\end{remark}

\subsubsection{The study of crystalline Emerton--Gee stacks}
\label{sec:crystalline-eg-stacks}
Informally, the stack \(\mathcal{X}_{n}\) can be thought of as a ``moduli stack 
of \(n\)-dimensional \(p\)-adic Galois representations'', and for any finite \(\mathcal{O}\)-algebra \(A\), the \(A\)-points of \(\mathcal{X}_{n}\) (or more precisely \(\operatorname{Spf}A\)-points)
can be canonically identified with the groupoid of continuous representations \(\operatorname{Gal}(\overline{\mathbb{Q}}_{p} / \mathbb{Q}_{p}) \to \GL_{n}(A)\). 
As such, there are special loci of \(\mathcal{X}_{n}\) cut out by \(p\)-adic Hodge theory conditions (defined by taking flat closure). 
In particular, for an inertial type \(\tau\) and dominant cocharacter \(\nu\), \cite{emertonModuliStacksEtale2023} constructs a substack \(\mathcal{X}_{n}^{\text{crys}, \nu, \tau} \subset \mathcal{X}_{n}\) whose \(A\)-points for a finite \emph{flat} (hence free) \(\mathcal{O}\)-algebra \(A\) 
correspond to continuous representations \(\operatorname{Gal}(\overline{\mathbb{Q}}_{p} / \mathbb{Q}_{p}) \to \GL_{n}(A)\) which are potentially crystalline with descent data described by \(\tau\), and with Hodge--Tate weights \(\nu\). 
A key feature of the stacks \(\mathcal{X}_{n}\) is that for a given representation \(\overline{\rho} : \operatorname{Gal}(\overline{\mathbb{Q}}_{p} / \mathbb{Q}_{p}) \to \GL_{n}(\overline{\mathbb{F}}_{p})\), 
the universal deformation ring \(R_{\overline{\rho}}\) is a versal ring for \(\mathcal{X}_{n}\) at the corresponding \(\overline{\mathbb{F}}_{p}\)-valued point. 
Similarly, the potentially crystalline deformation ring \(R_{\overline{\rho}}^{\text{crys},\nu,\tau}\) is a versal ring for \(\mathcal{X}_{n}^{\text{crys},\nu,\tau}\) (where we allow the possibility that \(\operatorname{Spec} R_{\overline{\rho}}^{\text{crys},\nu,\tau} = \emptyset\)). 

It is difficult to analyze \(\mathcal{X}_{n}^{\text{crys},\nu,\tau}\) directly, because it is actually defined by taking the \emph{closure} of a locus defined by \(p\)-adic Hodge theory conditions, and as such has no inherent moduli interpretation (which is related to the fact that integral \(p\)-adic Hodge theory is difficult). 
A basic tool to study these stacks is the theory of Breuil--Kisin modules. 
This was used to great effect by Kisin in \cite{kisinModuliFiniteFlat2009}, who ``resolved'' a potentially crystalline deformation ring by a moduli space of Breuil--Kisin modules, and used this to establish strong modularity lifting theorems (for \(\operatorname{GL}_{2}\)-valued Galois representations). 
At the time of Kisin's work, Emerton--Gee stacks had not been introduced yet, so the central geometric object in play was the potentially crystalline deformation ring itself.
But a natural question to ask is ``how can Kisin's ideas be adapted to the study of crystalline Emerton--Gee stacks?''

One suggestion of how Kisin's template can be applied to the study of \(\mathcal{X}_{n}^{\text{crys},\nu,\tau}\) is provided by \cite{local-models}, which is the principal inspiration for our work.
The approach of \cite{local-models} can be roughly described as a program in three steps:
\begin{enumerate}
    \item Construct and study the moduli stack of Breuil--Kisin modules. One should impose conditions corresponding to having appropriate Hodge--Tate weights and inertial type. 
    \item The moduli stack of (1) is usually too large, so we study a substack corresponding to \(\mathcal{X}_{n}^{\text{crys},\nu,\tau}\). In \cite{local-models} this substack is described by a certain ``monodromy condition''. 
    \item Find a nice ``linear algebra model'' for the moduli stack of (1) as well as the corresponding model for \(\mathcal{X}_{n}^{\text{crys},\nu,\tau}\). Use this to deduce properties of the stacks \(\mathcal{X}_{n}^{\text{crys},\nu,\tau}\). 
\end{enumerate}
A key observation in \cite{local-models} is that the moduli stack of (1) is closely related to a twisted Schubert variety in the sense of Pappas--Zhu \cite{pappas-zhu} (the linear algebra model), and one can understand how \(\mathcal{X}^{\text{crys},\nu,\tau}\) is correspondingly related to a substack of the twisted Schubert variety. 
In loc. cit. this is used to establish results on generalized Serre weight conjectures \cite{GHS} \cite{linDeligneLusztigTypeCorrespondence2023} (which generalize congruences between automorphic forms modulo \(p\)), and the geometric Breuil-Mézard conjecture \cite[Chapter 8]{emertonModuliStacksEtale2023} \cite[Section 6.1.37]{emertonIntroductionCategoricalPadic2023} (which concerns how the reduced special fiber of \(\mathcal{X}^{\text{crys},\nu,\tau}\) decomposes into irreducible components). 

\subsection{What is in this work?}
The global Langlands reciprocity conjectures in full generality are formulated for general reductive groups, so it is natural to ask how much of the story above generalize to such groups.
More precisely, consider the connected reductive group \(G\) over \(\mathbb{Q}_{p}\) which splits over the finite Galois extension \(L \supset \mathbb{Q}_{p}\).
The dual group \(\widehat{G}\) is the split group determined by the dual root datum to that of \(G\), and it is a Chevalley group which we view as defined over \(\mathbb{Q}_{p}\). 
The Galois group \(\Gamma = \operatorname{Gal}(L / \mathbb{Q}_{p})\) acts on \(\widehat{G}\) via pinned automorphisms, encoding the non-splitness of \(G\), and we form the Langlands dual group \({^L G} = \widehat{G} \rtimes \Gamma\).
The relevant objects on the Galois side of the \(p\)-adic Langlands program are \(L\)-parameters \(\Gamma_{\mathbb{Q}_{p}} = \operatorname{Gal}(\overline{\mathbb{Q}}_{p} / \mathbb{Q}_{p}) \to {^L G}(\overline{\mathbb{Q}}_{p})\), 
meaning homomorphisms for which the composition with the projection to \(\Gamma = \operatorname{Gal}(L / \mathbb{Q}_{p})\) is the canonical quotient map \(\Gamma_{\mathbb{Q}_{p}} \to \Gamma\).
The Emerton--Gee stacks for such \(L\)-parameters have been constructed in \cite{EG-tame-1}, and the crystalline stacks \(\mathcal{X}_{^L G}^{\text{crys},\nu,\tau}\) in \cite{EG-tame-II} (under the assumption that \(G\) is tamely ramified). 
It is expected that we have an analogue of \cite[Conjecture 6.1.14]{emertonIntroductionCategoricalPadic2023} in this setting, in which \eqref{eq:categorical-p-adic-langlands} is modified by replacing \(\GL_{n}\) with \(G\) and \(\mathcal{X}_{n}\) with \(\mathcal{X}_{^L G}\).
As before, the geometry of the stacks \(\mathcal{X}_{^L G}\) and \(\mathcal{X}_{^L G}^{\text{crys},\nu,\tau}\) are important in the \(p\)-adic Langlands program (for general connected reductive groups \(G\)). 

In this work, we study Breuil--Kisin modules with descent data and \(\widehat{G}\)-structure (we call them Breuil--Kisin \((\Gamma,\widehat{G})\)-torsors), which play an analogous role in the present setting as Breuil--Kisin modules with descent data (and no \(\widehat{G}\)-structure) do in the story above.
A crucial assumption is that the group \(G\) over \(\mathbb{Q}_{p}\) is \emph{tamely ramified}, i. e. splits over the tamely ramified extension \(L \supset \mathbb{Q}_{p}\). 
Following the template of \cite{local-models}, as formulated in the three steps above, we carry out step (1) and study the moduli \(\mathcal{Y}\) of such objects. 
We also carry out step (3) for this moduli stack \(\mathcal{Y}\). 
We leave to future work the question of how \(\mathcal{X}_{^L G}^{\text{crys},\nu,\tau}\) is related to \(\mathcal{Y}\), and finding the corresponding ``linear algebra model'' for \(\mathcal{X}_{^L G}^{\text{crys},\nu,\tau}\). 

\begin{remark}
    Let us mention some earlier work in a similar direction to ours -- namely, that of generalizing elements of the story above from \(\GL_{n}\) to more general groups. 
    The PhD thesis of Levin \cite{levinValuedFlatDeformations2013} \cite{levinValuedCrystallineRepresentations2015} generalized Kisin's arguments from \cite{kisinModuliFiniteFlat2009} to Galois representations with \(\widehat{G}\)-structure, working at the level of deformation rings (i. e. the completion of the Emerton--Gee stack at a point \(\overline{\rho}\)). 
    This involves the study of Breuil--Kisin modules with \(\widehat{G}\)-structure (but without descent data), so it can be seen as a precursor to the present work.
    We also mention \cite{koziolSerreWeightConjectures2022} which employs Breuil--Kisin modules with descent data for an unramified unitary group defined over an unramified extension of \(\mathbb{Q}_{p}\) and proves a version of Serre's weight conjectures in this setting, also working at the level of deformation rings. 
    The PhD thesis of Lee \cite{leeEmertonGeeStacks2023} essentially carries out all 3 steps of the program for the group \(G = \operatorname{GSp}_{4}\) defined over an unramified extension of \(\mathbb{Q}_{p}\). 
    We emphasize that \cite{levinValuedFlatDeformations2013}, \cite{levinValuedCrystallineRepresentations2015}, and \cite{koziolSerreWeightConjectures2022} work with ``local'' objects, namely deformation rings (which can be thought of as \(p\)-adic completions of the relevant Emerton--Gee stack at a \(\overline{\mathbb{F}}_{p}\)-valued point), whereas \cite{leeEmertonGeeStacks2023} as well as this thesis work with families in which the residual \(L\)-parameter \(\overline{\rho}\) is allowed to vary. 
\end{remark}

\subsection{Crystalline \(L\)-parameters and Breuil--Kisin \((\Gamma,\widehat{G})\)-torsors}
\label{sec:L-parameters}
We will now be precise and introduce some of the principal objects of study, namely Breuil--Kisin modules with descent data and \(\widehat{G}\)-structure (or just Breuil--Kisin \((\Gamma,\widehat{G})\)-torsors).
Before giving the definition, we review the notion Breuil--Kisin modules in the case of \(G = \GL_{n}\) and how they arise in the classification of lattices in crystalline representations as a consequence of Kisin's work \cite{kisinCrystallineRepresentationsFcrystals2006}.

\subsubsection{Review of Breuil--Kisin modules for \(\operatorname{GL}_{n}\)}
Let \(V\) be an \(E\)-vector space, and let \(\rho : \Gamma_{\mathbb{Q}_{p}} \to \GL_{V}(E)\) be a crystalline representation. 
A lattice in this crystalline representation is a \(\Gamma_{\mathbb{Q}_{p}}\)-stable \(\mathcal{O}\)-submodule \(\mathcal{L} \subset V\) for which \(\mathcal{L}[1/p] = V\). 
To such a lattice we can associate two semilinear algebra objects: 
\begin{enumerate}
    \item An étale \(\varphi\)-module. That is, a module \(M = D_{\et}(\mathcal{L})\) over \(\mathcal{E}_{\mathcal{O}} = \mathcal{O} \llp v \rrp^{\wedge \varpi}\) together with a Frobenius endomorphism \(\phi : \varphi^*M \isoto M \). 
    \item An isocrystal over \(E\), i. e. a \(\varphi\)-module \(N = D_{\text{crys}}(V) = (V \otimes_{\mathbb{Q}_{p}} B_{\text{crys }})^{\Gamma_{\mathbb{Q}_{p}}}\) together with a Frobenius endomorphism \(\phi : \varphi^* N \isoto N\). In fact, what one gets is the more refined structure of a filtered isocrystal. 
\end{enumerate}
The Breuil--Kisin module \(\mathfrak{M}\) attached to \(\mathcal{L}\) interpolates between these two things. 
That is, \(\mathfrak{M}\) is a module over \(\mathfrak{S}_{\mathcal{O}} := \mathcal{O}\llb v \rrb\) with Frobenius structure which specializes to (1) after inverting \(v\) and \(p\)-adically completing, 
and which specializes to (2) after inverting \(p\) and reducing modulo \(v\).
Kisin has shown that this construction gives rise to a fully faithful functor from lattices in crystalline representations to Breuil--Kisin modules \cite[Theorem 1.2.1]{kisinIntegralModelsShimura2010}. 
\begin{remark}
    \label{rem:log-poles}
    Not all Breuil--Kisin modules are attached to lattices in crystalline representations, i. e. Kisin's fully faithful functor is not essentially surjective.
    This accounts for the discrepancy between the moduli stacks \(\mathcal{X}_{n}^{\text{crys},\nu,\tau}\) and corresponding moduli stacks of Breuil--Kisin modules, as mentioned in step (2) earlier. 
    However, the essential image of Kisin's functor can be characterized in terms of having a connection with logarithmic poles, which is the starting point for the characterization of \(\mathcal{X}_{n}^{\text{crys},\nu,\tau}\) as a substack of the moduli of Breuil--Kisin modules used in \cite{local-models}. 
    This is related to how the Hodge filtration can be recovered from \(\mathfrak{M}\). 
\end{remark}
\begin{remark}
    The Hodge--Tate weights of \(V\) can be recovered in terms of elementary divisors of a matrix representing \(\phi : \varphi^* \mathfrak{M} \dashrightarrow \mathfrak{M}\). 
\end{remark}
\begin{remark}
    Similar constructions to the above work for lattices in potentially crystalline representations, which to a lattice \(\mathcal{L} \subset V\)
    attaches a Breuil--Kisin module \(\mathfrak{M}\) over \(\mathfrak{S}_{L,\mathcal{O}} = (\mathcal{O} \otimes_{\mathbb{Z}_{p}}W(l))\llb u \rrb\) with descent data for the group \(\Gamma\).
\end{remark}

\subsubsection{Breuil--Kisin \((\Gamma,\widehat{G})\)-torsors}
The idea behind a Breuil--Kisin \((\Gamma,\widehat{G})\)-torsor is that it should be to an ``integral crystalline \(L\)-parameter'' what a Breuil--Kisin module is to a lattice in a crystalline \(L\)-parameter. 
In order to understand what structure this should yield, let us first specify what we mean by an integral crystalline \(L\)-parameter. 

We say that a continuous \(L\)-parameter \(\Gamma_{\mathbb{Q}_{p}} \to {^L G}(E)\) is \emph{crystalline} if the restriction \(\Gamma_{L} \to \widehat{G}(E)\) is crystalline in the sense 
that for any (or equivalently, one faithful) algebraic representation \(\widehat{G} \hookrightarrow \GL_{V}\) over \(E\), the resulting representation \(\Gamma_{L} \to \GL_{V}(E)\) is crystalline.
Note that this is a common generalization of Galois representations which become crystalline over \(L\), and of crystalline Galois representations with \(\widehat{G}\)-structure.
An \emph{integral} crystalline \(L\)-parameter is a continuous \(L\)-parameter \(\Gamma_{\mathbb{Q}_{p}} \to {^L G}(\mathcal{O})\) for which the composition \(\Gamma_{\mathbb{Q}_{p}} \to {^L G}(\mathcal{O}) \hookrightarrow {^L G}(E)\) is a crystalline. 

To deal with a general group \(G\) one uses Tannakian formalism, which is valid over a Dedekind base by \cite[Theorem 2.5.2]{levinValuedFlatDeformations2013}.
The functors \(D_{\et}\) and \(D_{\text{crys}}\) are both exact monoidal, and hence ``compatible'' with Tannakian formalism. 
On the other hand, we warn the reader that we can not simply apply Tannakian formalism directly to Kisin's functor which attached Breuil--Kisin modules to lattices in crystalline representations, because this functor is not exact. 
From an integral crystalline \(L\)-parameter we obtain an étale \(\varphi\)-module and an isocrystal with descent data and \(\widehat{G}\)-structure via Tannkian formalism. 
The notion of a \((\Gamma,\widehat{G})\)-torsor (see \cref{sec:equivariant-torsors}) precisely encodes this notion of ``descent data and \(\widehat{G}\)-structure'', where the \(\Gamma\) may have a group-theoretic action on \(\widehat{G}\) in addition to the semilinear action over the base.
The corresponding object at the level of Breuil--Kisin modules is what we call a \emph{Breuil--Kisin \((\Gamma,\widehat{G})\)-torsor}. 

\begin{definition}
    A \emph{Breuil--Kisin \((\Gamma,\widehat{G})\)-torsor} with coefficients in an \(\mathcal{O}\)-algebra \(R\) is a pair \((\mathcal{P},\phi)\), where 
    \(\mathcal{P}\) is a \((\Gamma,\widehat{G})\)-torsor over \(\mathfrak{S}_{L,R } = (R \otimes_{\mathbb{Z}_{p}} \mathcal{O}_{L_{0}})\llb u \rrb\) and \(\phi : \varphi^* \mathcal{P} \dashrightarrow \mathcal{P}\)
    is an isomorphism of \((\Gamma,\widehat{G})\)-torsors, where the dashed arrow indicates that \(\phi \) is only defined after inverting \(E(u )\) (the Eisenstein polynomial over \(L_{0}\) for a uniformizer in \(\mathcal{O}_{L}\)).
    We also refer to Breuil--Kisin \((\Gamma,\widehat{G})\)-torsors as Breuil--Kisin modules with \(\widehat{G}\)-structure and descent data (as in the title of this work). 
\end{definition}
\begin{remark}
    Let \(\mathbb{Q}_{p} \subset K \subset L_{0}\) be an intermediate unramified extension.
    When \(G = \operatorname{Res}^K_{\mathbb{Q}_{p}} \GL_{n}\), a Breuil--Kisin \((\Gamma,\widehat{G})\)-torsor is essentially the 
    same as a Breuil--Kisin module with descent data as in \cite[Section 5]{local-models}.
    When \(L = K \), \(H\) is a split connected reductive group over \(K\) and \(G = \operatorname{Res}^K_{\mathbb{Q}_{p}}H\), then a Breuil--Kisin \((\Gamma,\widehat{G})\)-torsor 
    is essentially the same as a Breuil--Kisin module with \(\widehat{H}\)-structure in the sense of \cite{levinValuedCrystallineRepresentations2015}.
    When \(G = \operatorname{Res}^K_{\mathbb{Q}_{p}} \operatorname{GSp}_{4}\), a Breuil--Kisin \((\Gamma,\widehat{G})\)-torsor is essentially the same as a symplectic
    Breuil--Kisin module with descent data as in \cite{leeEmertonGeeStacks2023}. 
\end{remark}
\begin{remark}
    In practice, we truncate moduli stacks of Breuil--Kisin \((\Gamma,\widehat{G})\)-torsors by height, which is a generalization of the condition that the elementary divisors of a matrix representing \(\phi\) are bounded. 
    We will return to this point in \cref{sec:intro-local-models}. 
\end{remark}

\subsection{The moduli of Breuil--Kisin \((\Gamma,\widehat{G})\)-torsors}
\label{sec:moduli-of-BK-modules}
As mentioned earlier, the motivation behind this paper is to study the moduli stack of Breuil--Kisin \((\Gamma,\widehat{G})\)-torsors, with a view towards applications in the study of Galois representations.
Let \(\mathcal{Y}\) denote the moduli stack of Breuil--Kisin \((\Gamma,\widehat{G})\)-torsors over \(\operatorname{Spf} \mathcal{O}\). 
This moduli space has a natural presentation as a stack, but of an unfamiliar kind. 
More precisely, we have by \cref{res:moduli-interpretation-as-BK-modules} an equivalence of stacks over \(\operatorname{Spf}\mathcal{O}\)
\begin{equation}
    \label{res:presentation-of-moduli-stack-of-BKL-parameters}
    \mathcal{Y} \cong \left[ L \mathcal{G}_{x}^{\wedge \varpi} /^{\varphi, c} L^+ \mathcal{G}_{x}^{\wedge \varpi} \right], 
\end{equation}
where the \(\mathcal{G}_{x}\) is a certain Bruhat--Tits group scheme over \(\mathbb{A}^1_{\mathcal{O}}\), \(L \mathcal{G}_{x}(R) := \mathcal{G}_{x}(R\llp v + p \rrp)\), \(L^+\mathcal{G}_{x}(R) := \mathcal{G}_{x}(R\llb v + p \rrb)\), the superscript \(\wedge \varpi\) means \(p\)-adic completion (which can be interpreted as base change along \(\operatorname{Spf}\mathcal{O} \to \operatorname{Spec}\mathcal{O}\)), \(c \in L \mathcal{G}_{x}(\mathcal{O})\) is a certain element, and 
the fppf quotient stack on the right is formed with respect to the right action given by the formula \(X \star A = A^{-1} X c \varphi(A)c^{-1}\).

\begin{remark}
    \label{rem:moduli-of-BKL-parameters}
    Whenever \(p\) is nilpotent in an \(\mathcal{O}\)-algebra \(R\), the equivalence of \eqref{res:presentation-of-moduli-stack-of-BKL-parameters} identifies \(\mathcal{Y}(R)\) with the groupoid of Breuil--Kisin \((\Gamma,\widehat{G})\)-torsors with \(R\)-coefficients.
    However, we warn the reader that \(\mathcal{Y} (\operatorname{Spf}\mathcal{O})\) will not be identified with the groupoid of Breuil--Kisin \((\Gamma,\widehat{G})\)-torsors with \(\mathcal{O}\)-coefficients (although it can be identified with families of Breuil--Kisin \((\Gamma,\widehat{G})\)-torsors with \(\mathcal{O}/\varpi^a\)-coefficients, where \(a\) ranges through \(\mathbb{Z}_{\geq 1}\)). 
    The reason for this failure has to do with the difference between \(\operatorname{Spf}\) and \(\operatorname{Spec}\), and is essentially the somewhat subtle reason that \(\mathfrak{S}_{\mathcal{O}}\left[ (v+p)^{-1} \right] = \mathfrak{S}_{\mathcal{O}} \left[ v^{-1} \right] \neq \mathfrak{S}_{\mathcal{O}} \left[ v^{-1} \right]^{\wedge \varpi} = \mathcal{E}_{\mathcal{O}}\).
    But once we pass to a certain bounded substack \(\mathcal{Y}^{\leq \mu}\) (defined in \cref{sec:intro-local-models} or in \cref{sec:twisted-affine-grassmannian}), then an \(\operatorname{Spf}\mathcal{O}\)-valued point will correspond to an actual Breuil--Kisin \((\Gamma,\widehat{G})\)-torsor with \(\mathcal{O}\)-coefficients. 
\end{remark}

We will say some words about what lies behind \eqref{res:presentation-of-moduli-stack-of-BKL-parameters}, as this will explain what the group \(\mathcal{G}_{x}\) is and where the twisted action comes from. 

Let \(R\) be an \(\mathcal{O}\)-algebra, and consider the map \(\pi : \operatorname{Spec}\mathfrak{S}_{L,R} \to \operatorname{Spec}\mathfrak{S}_{R}\) corresponding to the map \(\mathfrak{S}_{R} := R\llb v \rrb \subset (R \otimes_{\mathbb{Z}_{p}} W(l))\llb u \rrb =: \mathfrak{S}_{L,R}\)
which is given by \(x \mapsto x \otimes 1\) on coefficients and \(v \mapsto u^e\), where \(e = [L : L_{0}]\) is the ramification index of \(L\) (which is prime to \(p\) by the tameness assumption). 
Because \(\pi\) is ramified at the origin, 
there is a non-trivial invariant of \((\Gamma,\widehat{G})\)-torsors over \(\mathfrak{S}_{L,R}\) called a type, which is described by a class \([\tau] \in H^1 \left(\Gamma, \widehat{G}(\mathfrak{S}_{L,R})\right)\) (see \cref{sec:specialized-equivariant-torsors}). 
For \(R = \mathcal{O}\) we give a building-theoretic classification of this set of cohomology classes in \cref{res:building-classification-of-types}, with respect to which a class \([\tau]\) corresponds an equivalence class of points in the extended building \(\widetilde{\mathcal{B}}(\widehat{G},\mathbb{F}\llp v \rrp)\). 
Let us now make a choice of such a point \(x\) describing \([\tau]\) according to this classification. 
It then follows by a general principle explained in \cref{sec:equivariant-torsors} (and specialized to the present situation in \cref{sec:specialized-equivariant-torsors}) that \((\Gamma,\widehat{G})\)-torsors of type \(\tau\) correspond to 
\(\mathcal{G}_{x}\)-torsors over \(\mathfrak{S}_{R}\), where \(\mathcal{G}_{x}\) is defined as the ``invariant pushforward'' (or invariants of Weil restriction) of a group defined as a ``twist'' of \(\widehat{G}\) over \(\mathfrak{S}_{L,\mathcal{O}}\).
For the detailed construction of \(\mathcal{G}_{x}\), see \cref{sec:bruhat-tits-group-schemes}. 

If we are given a Breuil--Kisin \((\Gamma,\widehat{G})\)-torsor \((\widetilde{\mathcal{P}},\phi)\) over \(\mathfrak{S}_{L,R}\) of type \(\tau\), 
then the previous paragraph shows that the \((\Gamma,\widehat{G})\)-torsor \(\mathcal{P}\) ``descends'' to a \(\mathcal{G}_{x}\)-torsor \(\mathcal{P}\) over \(\mathfrak{S}_{R}\).
On the other hand, the Frobenius structure \(\phi \) does not descend to a Frobenius structure on \(\mathcal{P}\), but to a ``\(c\)-twisted'' Frobenius structure, 
where \(c \in \widehat{G}(\mathcal{O}\llp v \rrp)\) satisfies \(c \cdot \varphi(x) = x \) (and the fact that we have such a \(c\) is a consequence of \cref{sec:frobenius-invariant-types}).
By choosing a trivialization of \(\mathcal{G}_{x}\)-torsors \(\mathcal{E}^0 \isoto \mathcal{P}\) (which can be done étale locally on \(R\)), the Frobenius structure corresponds to an element \(X \in \mathcal{G}_{x} \left( \mathfrak{S}_{R}\left[ (v+p)^{-1} \right] \right)\), 
and modifying the trivialization by left translation by an element \(A \in L^+ \mathcal{G}_{x}(\mathcal{O})\) has the effect of changing \(X\) into \(A^{-1} X c \varphi(A)c^{-1}\).
Now note that if \(p^aR = 0\) for some \(a \in \mathbb{Z}_{\geq 1}\), then \(\mathfrak{S}_{R}\left[(v+p)^{-1}\right] = R \llp v + p \rrp\) (see \cref{res:inverting-different-variables} for details), 
and consequently \(X \in L \mathcal{G}_{x}(R)\). 
The presentation of \eqref{res:presentation-of-moduli-stack-of-BKL-parameters} then follows.

Our group scheme \(\mathcal{G}_{x}\) is closely related to the group scheme for \(\widehat{G}\) attached to \(x\) in the sense of Pappas--Zhu \cite[Theorem 4.1]{pappas-zhu}. 
The only discrepancy between our \(\mathcal{G}_{x}\) and the group scheme of Pappas--Zhu is that our \(\mathcal{G}_{x}\) can have disconnected fibers, as we will see in \cref{ex:PGL3-fibers}. 
In \cref{res:connected-fibers} we give a concrete criterion that ensures \(\mathcal{G}_{x}\) has connected fibers, which under the assumption that \(G\) is unramified simplifies to the statement that the derived group \(\widehat{G}_{\text{der}}\) of \(\widehat{G}\) is simply connected.
For the remainder of this introduction, and in most of the paper, we will assume for simplicity that \(\mathcal{G}_{x}\) has connected fibers, and hence coincides with the group scheme of Pappas--Zhu. 

\begin{example}
    \label{ex:parahoric-group-scheme-for-GL2}
    Let us give a simple example illustrating that the groups \(\mathcal{G}_{x}\) often have a genuine ``parahoric'' structure at the origin. 
    Suppose \(\widehat{G} = \GL_{2}\), and \(L = \mathbb{Q}_{p}((-p)^{\frac{1}{p-1}})\). 
    Let \(\varpi_{L} = (-p)^{\frac{1}{p-1}}\) denote a choice of uniformizer in \(\mathcal{O}_{L}\), the ring of integers of \(L\). 
    Consider a type \(\tau\) given by 
    \begin{align*}
        \tau : \Gamma = I & \to \GL_{2}(\mathcal{O}) \subset \GL_{2}(\mathfrak{S}_{L,\mathcal{O}}) \\ 
        \gamma & \mapsto \begin{pmatrix} 1 & 0 \\ 0 & \omega(\gamma) \end{pmatrix} ,
    \end{align*}
    where \(\omega(\gamma) = \frac{\gamma(\varpi_{L})}{\varpi_{L}}\). 
    Then we define \(\widetilde{\mathcal{G}}_{x} := u^{(0,1)} \widehat{G} u^{-(0,1)}\) as a group scheme over \(\operatorname{Spec} \mathcal{O}[u]\), where we use the convention \(u^{(a,b)} = \begin{pmatrix} u^a & 0 \\ 0 & u^b \end{pmatrix}\). 
    (The role of \(u^{(0,1)}\) here is to trivialize \(\tau\), in the sense that \(\tau(\gamma) = u^{-(0,1)}(^\gamma u^{(0,1)})\).)
    The group scheme \(\mathcal{G}_{x}\) is then defined as the invariant pushforward \(\pi_{*}^I \widetilde{\mathcal{G}}_{x}\). Note that the \(\mathcal{O}[v]\)-points of \(\mathcal{G}_{x}\) are then given by 
    \[
    \mathcal{G}_{x} (\mathcal{O}[v]) = \begin{pmatrix}
        \mathcal{O}[u] & u^{-1} \mathcal{O}[u] \\ u \mathcal{O}[u] & \mathcal{O}[u] \end{pmatrix}^I 
        = \begin{pmatrix}
            \mathcal{O}[v] & \mathcal{O}[v] \\ v \mathcal{O}[v] & \mathcal{O}[v]      \end{pmatrix} 
    \]
    (recall that \(u^{p-1} = v\) and \(\gamma\in I\) acts on \(u\) via multiplication by \(\omega(\gamma)\)). 
    So the group scheme \(\mathcal{G}_{x}\) is a group of ``upper triangular matrices modulo \(v\)''. 
    One way to formalize this is to define \(\mathcal{G}_{x}\) as the dilation of \(\GL_{2}\) in the Borel at the origin \(v = 0\) (and this is precisely the approach \cite{local-models} uses to define their analogue of the groups \(\mathcal{G}_{x}\)).
    In \cref{rem:pappas-zhu-dilation} we will give precise conditions for when the group schemes \(\mathcal{G}_{x}\) can be described as dilations like this. 
\end{example}

\subsubsection{Stratification by height}
\label{sec:intro-local-models}
The generic fiber of \(L \mathcal{G}_{x}\) and \(L^+ \mathcal{G}_{x}\) is \(L \widehat{G}\) and \(L^+ \widehat{G}\), respectively
(see \cref{sec:special-and-generic-fibers} for additional details). 
Let \(X_{*}(\widehat{T})^{\text{dom}} \subset X_{*}(\widehat{T})\) denote the set of dominant cocharacters. 
Cartan decomposition yields 
\[
L \widehat{G}(\overline{\mathbb{Q}}_{p}) = \bigsqcup_{\nu \in X_{*}(\widehat{T})^{\text{dom}}} L^+ \widehat{G}(\overline{\mathbb{Q}}_{p}) \cdot (v+p)^{\nu} \cdot L^+ \widehat{G}(\overline{\mathbb{Q}}_{p}) . 
\]
Given \( \mu \in X_{*}(\widehat{T})^{\text{dom}}\), we then define \(L^{\leq \mu }\mathcal{G}_{x}\) as the reduced Zariski closure of 
\[
\bigsqcup_{\nu \in X_{*}(\widehat{T})^{\text{dom}}, \nu \leq \mu} L^+ \widehat{G}(\overline{\mathbb{Q}}_{p}) \cdot (v+p)^{\nu} \cdot L^+ \widehat{G}(\overline{\mathbb{Q}}_{p})
\]
in \(L \mathcal{G}_{x}\).
This is stable under both left and right translations by \(L^+ \mathcal{G}_{x}\).
We then define the quotient stack 
\[
\mathcal{Y}^{\leq \mu} := \left[ L^{\leq \mu} \mathcal{G}_{x}^{\wedge \varpi} /^{c,\varphi} L^+ \mathcal{G}_{x}^{\wedge \varpi} \right]^{\wedge \varpi} \subset \mathcal{Y} .
\]
We call \(\mathcal{Y}^{\leq \mu}\) the moduli stack of Breuil--Kisin \((\Gamma,\widehat{G})\)-torsors of \emph{height \(\leq \mu\)}. 
This is consistent with the terminology of \cite{local-models}. 
We emphasize that this is a purely group theoretic definition of height, which does not rely on choosing an embedding into \(\GL_{n}\).

\subsection{The local model theorem for \(\mathcal{Y}^{\leq \mu}\)} 
The quotient of \(L^{\leq \mu} \mathcal{G}_{x}\) by the left translation action of \(L^+ \mathcal{G}_{x}\) is the Pappas--Zhu local model 
\[
\operatorname{Gr}_{\mathcal{G}_{x}}^{\leq \mu} := \left[ L^+ \mathcal{G}_{x} \backslash L^{\leq \mu} \mathcal{G}_{x} \right] 
\]
(as explained in \cref{sec:bounded-loop-group}).
Taking the quotient of \(L^{\leq \mu}\mathcal{G}_{x}^{\wedge \varpi}\) by the two different actions, we obtain a diagram 
\[
\begin{tikzcd}
    & L^{\leq \mu} \mathcal{G}_{x}^{\wedge \varpi} \arrow[dl,twoheadrightarrow] \arrow[dr,twoheadrightarrow] & \\
    \mathcal{Y}^{\leq \mu} & & \operatorname{Gr}^{\leq \mu, \wedge \varpi}_{\mathcal{G}_{x}} , 
\end{tikzcd}
\]
where both arrows are \(L^+\mathcal{G}_{x}\)-torsors.
The above ``local model diagram'' is a direct generalization to more general groups of the local model diagrams appearing in the main results of \cite{pappasPhiModulesCoefficient2009} and \cite{caraiani-levin}. 
The only caveat is that we only deal with at most \emph{tamely} ramified groups, and as such our results don't apply to wildly ramified groups such as the Weil restriction of \(\GL_{n}\) defined over a wildly ramified extension \(K\) of \(\mathbb{Q}_{p}\).

The defect of the above diagram is that the arrows are merely formally smooth, and not smooth, because \(L^+ \mathcal{G}_{x}\) is not of finite presentation. 
But we can indeed produce a local model diagram in which both arrows are \emph{smooth}, as the following theorem shows. 
\begin{maintheorem}[{\cref{res:smooth-equivalence}}]
    \label{res:cl-local-model}
    Let \(\mu : \mathbb{G}_{m} \to \widehat{T}\) be a dominant cocharacter.
    Then \(\mathcal{Y}^{\leq \mu}\) is a \(p\)-adic formal algebraic stack, and there exists \(p\)-adic formal scheme \(\mathcal{Z}^{\wedge \varpi}\) over \(\operatorname{Spf}\mathcal{O}\) and a diagram 
    \[
    \begin{tikzcd}
    & \mathcal{Z}^{\wedge \varpi}  \arrow[dl] \arrow[dr] & \\
    \mathcal{Y}^{\leq \mu} & & \operatorname{Gr}_{\mathcal{G}}^{\leq \mu,\wedge \varpi} 
    \end{tikzcd}
    \]
    in which both maps are smooth coverings of \(p\)-adic formal algebraic stacks over \(\operatorname{Spf}\mathcal{O}\). 
\end{maintheorem}
The crux of the proof of \cref{res:cl-local-model} is the observation
that, over some fixed \(\mathbb{Z} / p^a\), 
the orbits in \(L^{\leq \mu }\mathcal{G}_{x}\) for the \(c\)-twisted \(\varphi \)-conjugation and left translation coincide once we pass to a sufficiently deep congruence subgroup of \(L^+ \mathcal{G}_{x}\). 
We refer to this property as ``straightening'', meaning that the \(c\)-twisted \(\varphi\)-conjugation action can be ``straightened''.
A detailed analysis of precisely when this can be done is carried out in \cref{sec:contraction-and-straightening}, and our bounds are often optimal.
Moreover, all of these calculations are purely group-theoretic and don't rely on choosing an embedding into \(\operatorname{GL}_{n}\). 

As a consequence of the detailed analysis of when straightening can be performed, we can produce stronger results with additional hypotheses.
The following result, which is a direct generalization of  \cite[Theorem 5.3.3]{local-models}, is one such example. 
\begin{maintheorem}[{\cref{res:explicit-smooth-equivalence}}]
    \label{res:lm-local-model}
    Assume that \(G\) is unramified, and the derived group \(\widehat{G}_{\text{der}}\) of \(\widehat{G}\) is simply connected.
    Let \(\mu : \mathbb{G}_{m} \to \widehat{T}\) be a dominant cocharacter. 
    Assume that \(x\) is lowest alcove (according to \cref{def:lowest-alcove}, a very mild assumption) and \(d\)-generic (according to \cref{def:generic-type}), where \(d > h_{\mu} := \max_{a \in \Phi(\widehat{G},\widehat{T})} \langle a , \mu \rangle\).
    Then there exist Zariski open covers \(\lbrace U_{i}^{\wedge \varpi} \hookrightarrow \operatorname{Gr}^{\leq \mu, \wedge \varpi} \rbrace_{i \in \mathcal{I}}\)
    and \(\lbrace \mathcal{Y}_{i}^{\leq \mu} \hookrightarrow \mathcal{Y}^{\leq \mu} \rbrace_{i \in \mathcal{I}}\) such that for each \(i \in \mathcal{I}\) there is a diagram over \(\operatorname{Spf}\mathcal{O}\)
    \[
    \begin{tikzcd}
    & \widetilde{U}_{i}^{\wedge \varpi} \arrow[dl] \arrow[dr] & \\ 
    \mathcal{Y}_{i}^{\leq \mu} & & U_{i}^{\wedge \varpi}
    \end{tikzcd}
    \]
    where both maps are \(\widehat{T}\)-torsors.
    Moreover, we can take
    these Zariski open charts \(U_{i} \subset \operatorname{Gr}^{\leq \mu}\) to be the explicit covers labeled by the admissible set \(\mathcal{I} = \operatorname{Adm}(\mu)\) (see \cref{sec:admissible-set} for the definition) defined in \cref{sec:negative-loop-group-charts} in terms of the negative loop group \(L^{--}\mathcal{G}\) from \cref{sec:negative-loop-group}. 
\end{maintheorem}
\begin{remark}
    The assumption that \(\widehat{G}_{\text{der}}\) is simply connected is equivalent to the assertion that \(\pi_{1}(\widehat{G})\) is torsion-free by \cref{res:simply-connected-derived-group}. 
    This assumption is used to ensure that \(\mathcal{G}_{x}\) has connected fibers, but it is not always necessary. 
    In addition, it implies that \(p \nmid \pi_{1}(\widehat{G})\), which we use (citing \cite{pappas-zhu}) to identify the special fiber of \(\operatorname{Gr}_{\mathcal{G}}^{\leq \mu}\). 
\end{remark}
A crucial part of the proof of both \cref{res:cl-local-model} and \cref{res:lm-local-model} is the construction of good Zariski open charts for \(\operatorname{Gr}^{\leq \mu}_{\mathcal{G}_{x}}\), or more accurately for certain 
covers \(\operatorname{Gr}_{f}^{\leq \mu} := \left[ L^+ \mathcal{G}_{x,f} \backslash L^{\leq \mu} \mathcal{G}_{x} \right] \twoheadrightarrow \operatorname{Gr}^{\leq \mu}_{\mathcal{G}_{x}}\),
where \(L^+ \mathcal{G}_{x,f} \subset L^+ \mathcal{G}_{x}\) is a congruence subgroup if \(f = n \in \mathbb{Z}_{\geq 1}\) (and \(L^+ \mathcal{G}_{x,f}\) is defined more generally for a concave (i. e. subadditive) function \(f \geq 0\) on the root system \(\widehat{\Phi}(\widehat{G},\widehat{T})\) in the sense of \cite[Definition 7.3.1]{kaletha-prasad}). 
In the case of \cref{res:cl-local-model} it suffices to show that the \(L^+ \mathcal{G}_{x,f}\)-torsors \(L^{\leq \mu} \mathcal{G} \to \operatorname{Gr}_{\mathcal{G}_{x},f}^{\leq \mu}\) split Zariski locally on \(\operatorname{Gr}_{\mathcal{G}_{x},f}^{\leq \mu}\) when \(f > 0\), 
which we show in \cref{res:zariski-local-splittings} by giving a moduli interpretation of \(\operatorname{Gr}_{\mathcal{G}_{x},f}\) and identifying the special fibers of the group schemes \(\mathcal{G}_{x,f}\). 
In the case of \cref{res:lm-local-model}, we also need to know that these charts are stable under both actions of \(\widehat{T}\), which we accomplish by giving explicit constructions of charts in terms of a negative loop group \(L^{--}\mathcal{G}_{x}\) that we define in \cref{sec:negative-loop-group}. 
The existence of the negative loop group also implies that the \(L^+\mathcal{G}_{x}\)-torsor \(L^{\leq \mu} \mathcal{G} \twoheadrightarrow \operatorname{Gr}_{\mathcal{G}_{x}}^{\leq \mu}\) splits Zariski locally. 
\begin{remark}
    We remind the reader about our standing assumption that \(\mathcal{G}_{x}\) has connected fibers, since our proofs of both \cref{res:cl-local-model} and \cref{res:lm-local-model} are under this assumption.
    When \(\widehat{G}\) is unramified and the derived subgroup \(\widehat{G}_{\text{der}}\) is simply connected, this condition is automatic by \cref{res:connected-fibers}, but there are also cases where \(\mathcal{G}_{x}\) does not have connected fibers as we show in \cref{ex:PGL3-fibers}. 
    The reader is referred to \cref{sec:pappas-zhu-identification} for a more detailed discussion. 
\end{remark}
\subsection{Other results}
In addition to the results above, we develop some tools that may be of independent interest:
\begin{itemize}
    \item In \cref{sec:equivariant-torsors}, we discuss \((\Gamma,H)\)-torsors in high generality and prove a generalization of \cite[Theorem 4.1.6]{balajiModuliParahoricMathcal2014} (whose statement needs to be modified as explained in \cref{rem:balaji-seshadri-error}).
    \item In \cref{sec:galois-types}, we give a building theoretic interpretation of Galois types, following \cite{pappas-rapoport-tamely-ramified-bundles}.
    \item In \cref{res:pappas-zhu-concave-function} we extend \cite[Theorem 4.1]{pappas-zhu} to accomodate group schemes of the form \(\mathcal{G}_{x,f}\), where \(f \geq 0\) is a concave function on the root system \(\widehat{\Phi}(\widehat{G},\widehat{T})\) in the sense of \cite[Definition 7.3.1]{kaletha-prasad}.
    \item In \cref{sec:negative-loop-group-charts} and \cref{sec:moduli-interpretation}, we show that some affine Grassmannians are Zariski sheaves or even just presheaf quotients. This extends some results of \cite{cesnaviciusAffineGrassmannianPresheaf2024} in a special case. See in particular Remark 2.6 of loc. cit.. 
\end{itemize}

\subsection{Future work}
It would be interesting to work out the monodromy condition in order to develop local models for the Emerton--Gee stacks constructed by Lin \cite{EG-tame-1} \cite{EG-tame-II}, 
thus carrying out the second step formulated in \cref{sec:crystalline-eg-stacks}. 
As mentioned there, this would lead the way to many interesting applications. 

We mention that all of our results generalize in a straightforward manner to a connected reductive group \(G\) defined over an unramified extension \(K\) over \(\mathbb{Q}_{p}\).
This corresponds to replacing \(G\) by \(\operatorname{Res}^K_{\mathbb{Q}_{p}} G\), which is now a connected reductive group over \(\mathbb{Q}_{p}\). 
In fact, we did not consider more general \(K\) because we subscribe to the philosophy of encoding any added complexity of a more general field into a more general group over \(\mathbb{Q}_{p}\), which makes certain statements seem more natural, like the fact that \(\mu\) as it appears above is a cocharacter of \(\widehat{G}\) (and not of an associated Weil restriction or product defined in terms of \(\widehat{G}\)). 

\subsection*{Acknowledgements}
The existence of this work owes a lot to the brilliant and generous guidance of my advisor Bao Le Hung, for which I am grateful.
I would also like to thank Daniel Le, Brandon Levin, Zhongyipan Lin, and Stefano Morra for illuminating conversations, and Stephan Snegirov for helpful comments on a draft of my thesis. 
Beyond that, I'm grateful to many family members, friends, and members of the math department at Northwestern University, for all their support during grad school.
A longer (but still incomplete) list of acknowledgements can be found in my thesis \cite{thesis}.

\section{Preliminaries}
\label{sec:notation-and-preliminaries}
In this chapter we will introduce notation that will be used throughout, and cover various topics that will be needed. 

\subsection{Fields and rings}
\label{sec:fields-and-rings}
Let \(L\) be a finite tamely ramified Galois extension of \(\mathbb{Q}_{p}\), with ring of integers \(\mathcal{O}_{L}\) and residue field \(\mathbb{F}_{q}\), where \(q = p^r\).
The maximal unramified extension of \(\mathbb{Q}_{p}\) inside \(L\) is \(L_{0} = \mathbb{Q}_{q}\), which is of degree \(r\) over \(\mathbb{Q}_{p}\), and with ring of integers \(\mathcal{O}_{L_{0}} = \mathbb{Z}_{q}\).
Let \(e = [L : L_{0}]\) be the ramification index, which is not divisible by \(p\) by the tameness assumption.
We assume that \(\# \mu_{e}(\mathbb{F}_{q}) = e\) (i. e. \(e\) divides \(q - 1\)), which can always be arranged by replacing \(L\) by an unramified extension of \(L\). 
Then we may choose a uniformizer \(\varpi_{L} = (-p)^{1/e}\) of \(\mathcal{O}_{L}\), in which case \(L = \mathbb{Q}_{q}(\varpi_{L})\).
Note that the minimal polynomial for \(\varpi_{L}\) over \(L_{0}\) is \(u^e + p\). 

Let \(I = \operatorname{Gal}(L / L_{0})\), \(\Gamma = \operatorname{Gal}(L / \mathbb{Q}_{p})\), and \(\Gamma_{0} = \operatorname{Gal}(L_{0} / \mathbb{Q}_{p})\).
We identify \(\Gamma_{0}\) with \(\operatorname{Gal}(\mathbb{F}_{q} / \mathbb{F})\) via the natural isomorphism \(\Gamma_{0} \isoto \operatorname{Gal}(\mathbb{F}_{q} / \mathbb{F})\), 
with respect to which the arithmetic Frobenius \(\sigma : \mathbb{F}_{q} \to \mathbb{F}_{q}\), \(x \mapsto x^p\), corresponds to an element of \(\Gamma_{0}\) which we also denote by \(\sigma\).
The element \(\sigma \in \Gamma_{0}\) lifts to a unique element of \(\operatorname{Gal}(L / \mathbb{Q}_{p}(\varpi_{L})) \subset \operatorname{Gal}(L / \mathbb{Q}_{p}) = \Gamma\), which we also denote by \(\sigma\), and which provides a splitting of the map \(\Gamma \to \Gamma_{0}\). 
Note that \(\sigma(\varpi_{L}) = \varpi_{L}\). 
We have a natural isomorphism 
\begin{align*}
    \omega : I & \isoto \mu_{e} \left( \mathbb{Z}_{q} \right) \\
    \theta & \mapsto \frac{\theta(\varpi_{L})}{\varpi_{L}} ,
\end{align*}
and we will sometimes identify the image of \(\omega\) with \(\mu_{e}(\mathbb{F}_{q}) \cong \mu_{e}(\mathbb{Z}_{q})\).
Via \(\omega\), a choice of primitive \(e\)'th root of unity of \(\mathbb{F}_{q}\) corresponds to a cyclic generator of \(I\), which we will denote by \(\gamma\). 
We will sometimes conflate \(\omega\) with \(\omega(\gamma)\) when it seems easier to state things in terms of generators and relations, although any such statement 
has a more natural interpretation purely in terms of \(\omega\) (where natural means depending on fewer choices). 

With respect to the above choices we have a presentation 
\[
\Gamma = \left\lbrace \gamma, \sigma : \gamma^e = 1, \sigma^r = 1, \sigma \gamma \sigma^{-1} = \gamma^q \right\rbrace = I \rtimes \Gamma_{0} . 
\]

\subsubsection{Coefficients fields}
Let \(E\) be a finite extension of \(\mathbb{Q}_{p}\) with ring of integers \(\mathcal{O}\) and residue field \(\mathbb{F}\). 
Then \(\mathcal{O} \otimes_{\mathbb{Z}_{p}} \mathbb{Z}_{q}\) is an étale \(\mathcal{O}\)-algebra of degree \(r\) over \(\mathcal{O}\), so 
we have an isomorphism 
\[
\mathcal{O} \otimes_{\mathbb{Z}_{p}} \mathbb{Z}_{q} \cong \prod_{j \in \mathcal{J}} \mathcal{O}_{j} ,
\]
where each \(\mathcal{O}_{j}\) is a \textit{local} finite étale algebra over \(\mathcal{O}\). By inverting \(p\) and reducing modulo \(p\), respectively, we obtain 
\[
E \otimes_{\mathbb{Q}_{p}} \mathbb{Q}_{q} \cong \prod_{j \in \mathcal{J}} E_{j}, \,\,\,\,\, \mathbb{F} \otimes_{\mathbb{F}_{p}} \mathbb{F}_{q} \cong \prod_{j \in \mathcal{J}} \mathbb{F}_{j} .
\]
We can identify \(\mathcal{J}\) with the set of isomorphism classes of composita of \(E\) and \(\mathbb{Q}_{q}\) in \(\overline{\mathbb{Q}}_{p}\), or of \(\mathbb{F}\) and \(\mathbb{F}_{q}\) in \(\overline{\mathbb{F}}_{p}\).

We will employ the notation 
\[
\mathcal{O}^\mathcal{J} = \prod_{j \in \mathcal{J}} \mathcal{O}_{j} , \,\,\, E^\mathcal{J} = \prod_{j \in \mathcal{J}} E_{j}, \,\,\, \mathbb{F}^\mathcal{J} = \prod_{j \in \mathcal{J}} \mathbb{F}_{j} . 
\]

We have the following two examples in mind: 
\begin{enumerate}
    \item If \(E\) contains \(\mathbb{Q}_{q}\), then we can identify \(\mathcal{J}\) with the set of \(\mathbb{Q}_{p}\)-embeddings \(\mathbb{Q}_{q} \hookrightarrow E\), or of \(\mathbb{F}_{p}\)-embeddings \(\mathbb{F}_{q} \hookrightarrow \mathbb{F}\).
        In this case, we have for each \(j \in \mathcal{J}\), \(\mathcal{O}_{j} = \mathcal{O}\), \(E_{j} = E\), \(\mathbb{F}_{j} = \mathbb{F}\). 
    \item If \(E = \mathbb{Q}_{p}\), then we have \(\mathcal{O}^\mathcal{J} = \mathbb{Z}_{p}, E^\mathcal{J} = \mathbb{Q}_{p}, \mathbb{F}^\mathcal{J} = \mathbb{F}_{p}\). 
\end{enumerate}
Although situation (2) seems more natural in some respects, it is often necessary to work with a larger coefficient field in order to resolve rationality issues. 
But our notation is sufficiently general to handle both cases. 

The group \(\Gamma\) acts on \(\mathcal{O} \otimes_{\mathbb{Z}_{p}} \mathbb{Z}_{q}\) through the action of the quotient \(\Gamma \twoheadrightarrow \Gamma_{0}\), 
which is the trivial one on \(\mathcal{O}\) and the natural one on \(\mathbb{Z}_{q}\).
Similarly, \(\Gamma\) acts on \(E \otimes_{\mathbb{Q}_{p}} \mathbb{Q}_{q}\) and \(\mathbb{F} \otimes_{\mathbb{F}_{p}} \mathbb{F}_{q}\).

We can identify \(\mathcal{J}\) with the set of idempotents in the ring \(\mathcal{O} \otimes_{\mathbb{Z}_{p}} \mathbb{Z}_{q}\). 
Via this identification, we obtain an action of \(\Gamma\) on \(\mathcal{J}\). 
If the idempotent corresponding to \(j \in \mathcal{J}\) is denoted by \(e_{j}\), then this action is characterized by the relation \(\theta(e_{j}) = e_{\theta(j)}\) for \(\theta \in \Gamma\) and \(j \in \mathcal{J}\).
Since multiplication by \(e_{j}\) corresponds to projection onto \(\mathcal{O}_{j} = e_{j}(\mathcal{O} \otimes_{\mathbb{Z}_{p}} \mathbb{Z}_{q})\), we use the notation \(z_{j} := e_{j} z\) for \(j \in \mathcal{J}\) and \(z \in \mathcal{O} \otimes_{\mathbb{Z}_{p}} \mathbb{Z}_{q}\). 
The action of \(\theta \in \Gamma\) yields a map \(\mathcal{O}_{\theta^{-1}j} = e_{\theta^{-1}j} (\mathcal{O} \otimes_{\mathbb{Z}_{p}} \mathbb{Z}_{q}) \isoto e_{j}(\mathcal{O} \otimes_{\mathbb{Z}_{p}} \mathbb{Z}_{q}) = \mathcal{O}_{j}\) for each \(j \in \mathcal{J}\), and we have the formula 
\begin{equation}
    \label{eq:characterization-of-action-on-J}
    \left( \theta(z) \right)_{j} = \theta\left(z_{\theta^{-1} j}\right) \,\,\, \text{ for } j \in \mathcal{J}, z \in \mathcal{O} \otimes_{\mathbb{Z}_{p}} \mathbb{Z}_{q} .
\end{equation} 
\begin{example}
    \label{ex:large-coefficient-field}
    Assume that \(\mathbb{F} \supset \mathbb{F}_{q}\), and we identify \(\mathcal{J}\) with the set of embeddings \(\mathbb{Z}_{q} \hookrightarrow \mathcal{O}\). 
    Then the isomorphism \(\mathcal{O} \otimes_{\mathbb{Z}_{p}} \mathbb{Z}_{q} \cong \prod_{j \in \mathcal{J}} \mathcal{O}\) is characterized by \((x \otimes y)_{j} = x j(y)\).
    Hence \((x \otimes \theta(y))_{j} = x j(\theta(y)) = (x \otimes y)_{j \theta}\), so by comparing with \eqref{eq:characterization-of-action-on-J} we see that \(\theta^{-1}j = j \theta\) in this case, 
    and moreover the map \(\theta : \mathcal{O} = \mathcal{O}_{\theta^{-1}j} \isoto \mathcal{O}_{j} = \mathcal{O}\) is the identity (so \(\theta\) acts by permuting the factors).
    Even more concretely we can upon a choice of embedding \(\iota : \mathbb{Z}_{q} \hookrightarrow \mathcal{O}\) identify \(\mathcal{J} = \lbrace \iota, \iota\varphi^{-1}, \dots, \iota \varphi^{-r + 1} \rbrace\),
    and then \(\iota \varphi^{-j} \sigma = \iota \varphi^{-j+1}\) (since both \(\varphi\) and \(\sigma\) denote the same automorphism of \(\mathbb{Z}_{q}\) in this case).
\end{example}
Note that the \(\Gamma\)-action on \(\mathcal{J}\) factors through \(\Gamma_{0}\).
Moreover, \(\Gamma_{0}\) acts transitively on \(\mathcal{J}\) because \(\bigsqcup_{j \in \mathcal{J}} \operatorname{Spec}\mathcal{O}_{j} \to \mathcal{O}\)
is an étale \(\Gamma_{0}\)-torsor, whose \(\Gamma_{0}\) action is therefore transitive on fibers. 
For a fixed \(j \in \mathcal{J}\) let \(\Gamma_{0,\mathcal{J}}\) denote the stabilizer of \(j\) in \(\Gamma_{0,\mathcal{J}}\). 
Since \(\Gamma_{0}\) is abelian and acts trasitively on \(\mathcal{J}\), this stabilizer is independent of the choice of \(j\). 
Let \(\Gamma_{0}' = \Gamma_{0} / \Gamma_{0,\mathcal{J}}\), in which case the étale \(\Gamma_{0}\)-torsor factors as 
\[
    \bigsqcup_{j \in \mathcal{J}} \operatorname{Spec}\mathcal{O}_{j} \to \bigsqcup_{j \in \mathcal{J}} \operatorname{Spec} \mathcal{O} \to \operatorname{Spec} \mathcal{O} ,
\]
the first map being an étale \(\Gamma_{0,\mathcal{J}}\)-torsor where \(\Gamma_{0,\mathcal{J}}\) preserves each component, and the second an étale \(\Gamma_{0}'\)-torsor where \(\Gamma_{0}'\) permutes the factors.

\subsubsection{The variables \(u\) and \(v\)}
The relation between the formal variables \(u\) and \(v\) is always that \(u^e = v\).
For any field \(F\), we view \(F\llp v \rrp\) as a discretely valued field where \(v\) has valuation 1. 
With respect to the extension of this valuation to \(F\llp u \rrp \supset F \llp v \rrp\), the variable \(u\) has valuation \(1/e\).

\subsubsection{Rings from \(p\)-adic Hodge theory}
\label{sec:rings-from-p-adic-hodge-theory}
For any \(\mathbb{Z}_{p}\)-algebra \(R\), define 
\[
\mathfrak{S}_{R} = R \llb v \rrb , \,\,\,\, \mathfrak{S}_{L,R} = \left( R \otimes_{\mathbb{Z}_{p}} \mathbb{Z}_{q} \right) \llb u \rrb . 
\]
When \(R\) is an \(\mathcal{O}\)-algebra, we can identify \(\mathfrak{S}_{L,R} = R^\mathcal{J} \llb u \rrb\), where \(R^\mathcal{J} = \prod_{j \in \mathcal{J}} R_{j}\) and \(R_{j} = R \otimes_{\mathcal{O}} \mathcal{O}_{j}\).
We also define 
\[
\mathcal{E}_{R} = R \llp v \rrp^{\wedge \varpi}, \,\,\,\, \mathcal{E}_{L,R} = \left( R \otimes_{\mathbb{Z}_{p}} \mathbb{Z}_{q} \right) \llp u \rrp^{\wedge \varpi} , 
\]
where the superscript \(\wedge \varpi\) denotes \(p\)-adic completion, i. e. \(A^{\wedge \varpi} := \clim A / p^a A\). 
Similarly to the above, we can when \(R\) is an \(\mathcal{O}\)-algebra identify \(\mathcal{E}_{L,R} = R^\mathcal{J} \llp u \rrp^{\wedge \varpi}\). 
Note that the \(p\)-adic completion is redundant when \(p\) is nilpotent in \(R\).

The group \(\Gamma\) acts on \(\mathfrak{S}_{L,R}\) and \(\mathcal{E}_{L,R}\) as summarized by the following table.
\begin{table}
    \centering 
    \begin{center}
        \begin{tabular}{| l | c | c | c |}
            \hline & \(R\) & \(\mathbb{Z}_{q}\) & \(u\) \\ \hline 
            \(\gamma\) & \(1\) & \(1\) & scale by \(\omega(\gamma) \in \mathbb{Z}_{q}^\times\) (not in \(R\)!) \\ \hline 
            \(\sigma\) & \(1\) & \(\varphi_{\mathbb{Z}_{q}}\) & 1 \\ \hline 
        \end{tabular}
    \end{center}
    \caption{
        How \(\Gamma = \lbrace \gamma, \sigma | \gamma^e = 1, \sigma^r = 1, \sigma\gamma\sigma^{-1} = \gamma^q \rbrace\) acts on \(\mathfrak{S}_{L,R} = \left( R \otimes_{\mathbb{Z}_{p}} \mathbb{Z}_{q} \right)\llb u \rrb\) and various other rings.
        Here, \(\varphi_{\mathbb{Z}_{q}} : \mathbb{Z}_{q} \to \mathbb{Z}_{q}\) denotes the (arithmetic) Witt vector Frobenius automorphism. 
    }
    \label{tab:action}
\end{table}
The \(\Gamma\)-actions on \(\mathfrak{S}_{L,R}\) and \(\mathcal{E}_{L,R}\) induce left actions on \(\operatorname{Spec}\mathfrak{S}_{L,R}\) and \(\operatorname{Spec}\mathcal{E}_{L,R}\) as follows: 
If \(\theta \in \Gamma\) acts on for example \(\mathfrak{S}_{L,R}\) via \(a_{\theta} : \mathfrak{S}_{L,R} \to \mathfrak{S}_{L,R}\), 
then \(\theta\) acts on \(\operatorname{Spec}\mathfrak{S}_{L,R}\) via \((a_{\theta}^{-1})^* : \operatorname{Spec} \mathfrak{S}_{L,R} \to \operatorname{Spec}\mathfrak{S}_{L,R}\).

For some constructions it is natural to work over the affine line rather than the power series ring. 
For any \(\mathbb{Z}_{p}\)-algebra \(R\) we define 
\[
\mathbb{A}^1_{R} = \operatorname{Spec} R [v] , 
\,\,\,\, \breve{\mathbb{A}}^{1}_{R} = \operatorname{Spec} \left( R \otimes_{\mathbb{Z}_{p}} \mathbb{Z}_{q} \right)[v] , 
\,\,\,\, \widetilde{\mathbb{A}}^1_{R} = \operatorname{Spec} \left( R \otimes_{\mathbb{Z}_{p}} \mathbb{Z}_{q} \right)[u] . 
\]
Of course, \(\breve{\mathbb{A}}^1_{R} \cong \widetilde{\mathbb{A}}^1_{R}\) as abstract schemes, but they serve different purposes for us. 
The group \(\Gamma\) acts on \(\widetilde{\mathbb{A}}^1_{R}\) in the same way as on \(\operatorname{Spec}\mathfrak{S}_{L,R}\), and \(\Gamma_{0}\) acts on \(\widetilde{\mathbb{A}}^1_{0,R}\) via its action on \(\mathbb{Z}_{q}\).
We have maps 
\[
\begin{tikzcd}
\pi : \widetilde{\mathbb{A}}^1_{R} \arrow[r,"\pi_{I}"] & \breve{\mathbb{A}}^1_{R} \arrow[r,"\pi_{0}"] & \mathbb{A}^1_{R} ,
\end{tikzcd}
\]
corresponding to the inclusion \(R[v] \subset \left(R \otimes_{\mathbb{Z}_{p}} \mathbb{Z}_{q} \right)[v]\) and the map \(\left(R \otimes_{\mathbb{Z}_{p}} \mathbb{Z}_{q} \right)[v] \hookrightarrow \left(R \otimes_{\mathbb{Z}_{p}} \mathbb{Z}_{q} \right)[u]\) given by \(v \mapsto u^e\).
We remind the reader that the zero section \(v = 0 : \operatorname{Spec}R \hookrightarrow \mathbb{A}^1_{R}\) is a closed subscheme which we will denote by \(\lbrace 0 \rbrace\), and the complement \(\mathbb{A}^1_{R} \setminus \lbrace 0 \rbrace \subset \mathbb{A}^1_{R}\) is therefore an open subscheme. 
Similarly, we have open subschemes \(\breve{\mathbb{A}}^1_{R} \setminus \lbrace 0 \rbrace \subset \breve{\mathbb{A}}^1_{R}\) and \(\widetilde{\mathbb{A}}^1_{R} \setminus \lbrace 0 \rbrace \subset \widetilde{\mathbb{A}}^1_{R}\). 
Note that \(\pi\) is finite flat and \(\Gamma\)-equivariant for the trivial action on \(\mathbb{A}^1_{R}\), and in fact restricts to an étale \(\Gamma\)-torsor \(\widetilde{\mathbb{A}}^1_{R} \setminus \lbrace 0 \rbrace \to \mathbb{A}^1_{R}\setminus \lbrace 0 \rbrace\). 
In addition, \(\pi_{0}\) is an étale \(\Gamma_{0}\)-torsor, \(\pi_{I}\) is finite flat, and \(\pi_{I}\) restricts to an étale \(I\)-torsor \(\widetilde{\mathbb{A}}^1_{R} \setminus \lbrace 0 \rbrace \to\breve{\mathbb{A}}^1_{R} \setminus \lbrace 0 \rbrace \).
\begin{remark}
    \label{rem:cartesian-diagram}
    Note that the diagram 
    \[
    \begin{tikzcd}
        \operatorname{Spec} R \otimes_{\mathbb{Z}_{p}} \mathcal{O}_{L} \arrow[r,hookrightarrow,"u=\varpi_{L}"] \arrow[d] & \widetilde{\mathbb{A}}^1_{R} \arrow[d,"\pi"] \\ 
        \operatorname{Spec} R \arrow[r, hookrightarrow, "v=-p"] & \mathbb{A}^1_{R}
    \end{tikzcd}
    \]
    is cartesian. Indeed, since base change preserves cartesian diagrams it suffices to know this for \(R = \mathcal{O} = \mathbb{Z}_{p}\), 
    in which case it is a restatement of the fact that \(\mathbb{Z}_{q}[u] / (u^e + p) \cong \mathcal{O}_{L}\) via the map sending \(u\) to \(\varpi_{L} = (-p)^{1/e}\). 
\end{remark}

\subsubsection{The variables \(u\) and \(v\) versus \(u^e+p\) and \(v+p\)}
\label{sec:E-vs-v}
For any commutative ring \(A\) and ideal \(I \subset A\) we define \(A^{\wedge I} := \clim A / I^n\).
Given \(f(t) \in A[t]\) we define \(A \llb f(t) \rrb := R[t]^{\wedge (f(t))}\) and \(A \llp f(t) \rrp := R \llb f(t) \rrb\left[ f(t)^{-1} \right]\).
This notation could be misleading if the map \(A[t] \to A[t]\), \(t \mapsto f(t)\) is not an isomorphism, but we will only use it in the following lemma. 

\begin{lemma}
    Let \(A\) be a \(\mathbb{Z}_{p}\)-algebra and \(E(t) \in \mathbb{Z}_{p}[t]\) an Eisenstein polynomial. Then:
    \begin{enumerate}
        \item \(A \llb E(t) \rrb^{\wedge \varpi} = A [t]^{\wedge (p, E(t))} = A[t]^{\wedge (p, t)} = A \llb t \rrb^{\wedge \varpi} = \left( A^{\wedge \varpi} \right) \llb t \rrb\), 
        \item \(A \llp E(t) \rrp^{\wedge \varpi} = A \llp t \rrp^{\wedge \varpi}\). 
    \end{enumerate}
\end{lemma}
\begin{proof}
    Let \(e = \deg E(t)\).

    We first prove (1). 
    The only non-trivial identification is \(A[t]^{\wedge (p, E(t))} = A[t]^{\wedge (p,t)}\). 
    This in turn is a consequence of the observations that \(E(t) \in (p, t)\) and 
    \[
    (p,t)^e = (p^e, p^{e-1}t, \dots, p t^{e-1}, t^e) \subset (p, E(t)) . 
    \]

    It remains to prove (2). 
    In view of (1) we can set $B = A \llb t \rrb^{\wedge \varpi} = A \llb E(t) \rrb^{\wedge \varpi}$ to simplify notation.
    It suffices to prove that $E(t)^{-1} \in B[t^{-1}]^{\wedge \varpi}$ and $t^{-1} \in B[E(t)^{-1}]^{\wedge \varpi}$.
    This is a consequence of the fact that \(E(t) \equiv t^e \mod p\) and that \(p\) is in the Jacobson radical of any \(p\)-adically complete ring.
    For example, if we write \(E(t) = t^e + p x\), where \(x \in \mathbb{Z}_{p}[t]\), then \(E(t) = t^e(1+pxt^{-e})\) in \(B[t^{-1}]^{\wedge \varpi}\), and \(E(t)^{-1} = t^{-e}(1+pxt^{-e})^{-1} = t^{-e}\sum_{n=0}^\infty (-1)^n p^n x^n t^{-en}\). 
\end{proof}
\begin{example}
    When \(E(t) = t + p\), then \((t+p)^{-1} = t^{-1} \sum_{n=0}^\infty (-1)^n t^{-n}p^n\). 
\end{example}
The following is immediate. 
\begin{lemma}
    \label{res:inverting-different-variables}
    Let \(R\) be a \(\mathbb{Z}_{p}\)-algebra.
    \begin{enumerate}
        \item If \(R\) is \(p\)-adically complete, then \(\mathfrak{S}_{R} = R \llb v + p \rrb\), and \(\mathfrak{S}_{L,R} = \left( R \otimes_{\mathbb{Z}_{p}} \mathbb{Z}_{q} \right)[u]^{\wedge (u^e + p)}\). 
        \item If there exists \(a \in \mathbb{Z}_{\geq 1}\) such that \(p^a R = 0\), then \(\mathfrak{S}_{R} \left[ (v + p)^{-1} \right] = \mathcal{E}_{R} = R \llp v \rrp = R \llp v + p \rrp\) and \(\mathfrak{S}_{L,R} \left[ (u^e + p)^{-1} \right] = \mathcal{E}_{L,R} = \left( R \otimes_{\mathbb{Z}_{p}} \mathbb{Z}_{q} \right)\llp u \rrp\).
    \end{enumerate}
\end{lemma}

\subsection{Dual groups}
\label{sec:dual-groups}
Let \((\widehat{G}, \widehat{B}, \widehat{T}, e)\) be a Chevalley group over \(\mathbb{Z}\).
For any scheme \(X\), let \(\widehat{G}_{X}, \widehat{B}_{X}, \widehat{T}_{X}\) denote the base change to \(X\) of \(\widehat{G}, \widehat{B}, \widehat{T}\), respectively. 

We assume that \(\Gamma\) acts on \(\widehat{G}\) via pinned automorphism, and hence preserve the subgroups \(\widehat{B}\) and \(\widehat{T}\). 
We denote the automorphism of \(\widehat{G}\) by \(\sigma\) by 
\[
\psi : \widehat{G} \to \widehat{G} .
\]
If \(X\) is a scheme with a \(\Gamma\)-action, then we view \(\widehat{G}_{X}, \widehat{B}_{X}, \widehat{T}_{X}\) as having the ``diagonal'' \(\Gamma\)-action.
Writing \(\widehat{G}_{X} = \widehat{G} \times X\), this means that \(\sigma\) acts by \(\psi \times \sigma\).
If \(I\) acts trivially on \(\widehat{G}\), then \(\gamma\) acts on \(\widehat{G} \times X\) by \(1 \times \gamma\), but in general \(\gamma\) may also act through a group automorphism of \(\widehat{G}\). 

If \(\widetilde{X} \to X\) is an étale \(\Gamma\)-torsor, then the \(\Gamma\)-action on \(\widehat{G}_{\widetilde{X}}, \widehat{B}_{\widetilde{X}}, \widehat{T}_{\widetilde{X}}\)
describe descent data for the étale covering. 
Hence the groups \(\widehat{G}_{\widetilde{X}}, \widehat{B}_{\widetilde{X}}, \widehat{T}_{\widetilde{X}}\) define certain groups over \(X\) by étale descent. 
We apply this to the map \(\pi : \widetilde{\mathbb{A}}^1_{\mathcal{O}} \setminus \lbrace 0 \rbrace \to \mathbb{A}^1_{\mathcal{O}} \setminus \lbrace 0 \rbrace\), and 
denote the resulting groups by \(G^*, B^*, T^*\). 
Via base change we obtain the following groups:
\begin{enumerate}
    \item The groups \(G^*_{E}, B^*_{E}, T^*_{E}\) via base change along the map \(v = -p : \operatorname{Spec} E \hookrightarrow \widetilde{\mathbb{A}}^1_{\mathcal{O}} \setminus \lbrace 0 \rbrace\). 
    \item The groups \(G^*_{\mathbb{F}\llp v \rrp}, B^*_{\mathbb{F}\llp v \rrp}, T^*_{\mathbb{F}\llp v \rrp}\) via base change along the map \(\mathcal{O}[v^{\pm 1}] \to \mathbb{F}\llp v \rrp\) given by reduction modulo \(\varpi\). 
    \item The groups \(G^*_{E\llp v \rrp}, B^*_{E\llp v \rrp}, T^*_{E\llp v \rrp}\) via base change along the map \(\mathcal{O}[v^{\pm 1}] \hookrightarrow E\llp v \rrp\). 
\end{enumerate}
\begin{remark}
    It follows from \cref{rem:cartesian-diagram} that the diagram 
    \[
    \begin{tikzcd}
    \operatorname{Spec} E \otimes_{\mathbb{Q}_{p}} L \arrow[r,hookrightarrow,"u=\varpi_{L}"] \arrow[d] & \widetilde{\mathbb{A}}^1_{\mathcal{O}} \setminus \lbrace 0 \rbrace \arrow[d,"\pi"] \\ 
    \operatorname{Spec} E \arrow[r,hookrightarrow,"v = -p"] & \mathbb{A}^1_{\mathcal{O}} \setminus \lbrace 0 \rbrace 
    \end{tikzcd}
    \]
    is cartesian. As a consequence, the groups \(G^*_{E}, B^*_{E}, T^*_{E}\) are isomorphic to those defined from \(\widehat{G}_{E \otimes_{\mathbb{Q}_{p}} L}, \widehat{B}_{E \otimes_{\mathbb{Q}_{p}} L}, \widehat{T}_{E \otimes_{\mathbb{Q}_{p}} L}\) via étale descent.
    Similarly, the groups \(G^*_{\mathbb{F}\llp v \rrp}, B^*_{\mathbb{F}\llp v \rrp}, T^*_{\mathbb{F}\llp v \rrp}\) are isomorphic to those defined from \(\widehat{G}_{\mathbb{F}^\mathcal{J}\llp u \rrp}, \widehat{B}_{\mathbb{F}^\mathcal{J}\llp u \rrp}, \widehat{T}_{\mathbb{F}^\mathcal{J}\llp u \rrp}\) via étale descent, 
    and the groups \(G^*_{E\llp v \rrp}, B^*_{E\llp v \rrp}, T^*_{E\llp v \rrp}\) are isomorphic to those defined from \(\widehat{G}_{E^\mathcal{J}\llp u \rrp}, \widehat{B}_{E^\mathcal{J}\llp u \rrp}, \widehat{T}_{E^\mathcal{J}\llp u \rrp}\) via étale descent. 
\end{remark}
\begin{remark}
    \label{rem:G-star-is-G-hat}
    If \(\mathbb{F} \supset \mathbb{F}_{q}\), then there is no ``twisting'' in the unramified direction. 
    To make this precise, note that we can identify each factor of \(\mathcal{O}^\mathcal{J} = \prod_{j \in \mathcal{J}} \mathcal{O}_{j}\) as \(\mathcal{O}_{j} =\mathcal{O}\) in this case. 
    So we can identify the composition \(\operatorname{Spec} \mathcal{O}_{j}[v^{\pm 1}] \stackrel{j}{\hookrightarrow} \breve{\mathbb{A}}^1_{\mathcal{O}} \stackrel{\pi_{0}}{\twoheadrightarrow} \mathbb{A}^1_{\mathcal{O}} \setminus \lbrace 0 \rbrace\) with the identity map, 
    and consequently, when we pull back \(G^*\) along this map, we can identify the result with \(G^*\).
    Since \(G^* \times_{\mathbb{A}^1_{\mathcal{O}}\setminus \lbrace 0 \rbrace} (\breve{\mathbb{A}}^1_{\mathcal{O}}\setminus \lbrace 0 \rbrace) \cong (\pi_{I})_{*}^I \widehat{G}_{\mathbb{A}^1_{\mathcal{O}}\setminus \lbrace 0 \rbrace}\),
    we can therefore identify \(G^*\) with 
    \[
        G^* = \tilde{\pi}_{*}^I \left( \widehat{G}_{\mathcal{O}[u^{\pm 1}]} \right), \,\,\,\,\, \text{ where } \,\,\, \tilde{\pi} : \operatorname{Spec} \mathcal{O}[u^{\pm 1}] \to \operatorname{Spec}\mathcal{O}[v^{\pm 1}]  \text{ identifies } u^e = v . 
    \]
    Similarly, we can identify \(B^*\) and \(T^*\) with the invariant pushforwards of \(\widehat{B}\) and \(\widehat{T}\) along \(\tilde{\pi}\). 
    In particular, if in addition \(I\) acts trivially on \(\widehat{G}\), then we can identify \(G^* = \widehat{G}_{\mathcal{O}[v^{\pm 1}]}\), \(B^* = \widehat{B}_{\mathcal{O}[v^{\pm 1}]}\), and \(T^* = \widehat{T}_{\mathcal{O}[v^{\pm 1}]}\). 
    This picture is clearly compatible with the various base changes (1), (2), (3) above. 
\end{remark}

\subsubsection{Normalizers and the Iwahori--Weyl group}
\label{sec:iwahori-weyl}
We define \(\widehat{N}\) as the normalizer of \(\widehat{T}\) in \(\widehat{G}\) (as in \cite[Section 2.1]{Conrad2014ReductiveGS}, for example). 
Similarly, \(N^*\) denotes the normalizer of \(T^*\) in \(G^*\).
We consider the \emph{Iwahori--Weyl group} 
\[
    \widetilde{W}^\mathcal{J} := \widehat{N}(\mathbb{F}^\mathcal{J}\llp u \rrp) / \widehat{T}(\mathbb{F}^\mathcal{J}\llb u \rrb) = \prod_{j \in \mathcal{J}} \widehat{N}(\mathbb{F}_{j}\llp u \rrp) / \widehat{T}(\mathbb{F}_{j}\llb u \rrb) . 
\]
As noted in \cite[Section 4.1.1]{pappas-zhu}, we can also identify \(\widetilde{W}^\mathcal{J}\) with the quotients \(\widehat{N}(E^\mathcal{J}\llp u \rrp) / \widehat{T}(E^\mathcal{J}\llb u \rrb)\) and \(\widehat{N}(E^\mathcal{J}) / \widehat{T}(\mathcal{O}_{E}^\mathcal{J})\).
Similarly, we let \(\widetilde{W}^* := N^* (\mathbb{F}\llp v \rrp) / T^* (\mathbb{F}\llp v \rrp)^0\) denote the Iwahori--Weyl group for \(G^*\), 
where \(T^*(\mathbb{F}\llp v \rrp)^0 \subset T^*(\mathbb{F}\llp v \rrp)\) is the image of the \(I\)-norm map \(\widehat{T}(\mathbb{F}\llb u \rrb) \to T^*(\mathbb{F}\llp v \rrp)\) (this is consistent with \cite[Definition 2.5.13]{kaletha-prasad}). 
We can also identify \(\widetilde{W}^*\) with the quotient \(N^*(E\llp v \rrp)/T^*(E\llp v \rrp)^0\) by \cite[Section 9.2.2]{pappas-zhu}.
\begin{lemma}
    \label{res:iwahori-weyl-group-as-translations-rx-W}
    We have semidirect product decompositions
    \begin{enumerate}
        \item \(\widetilde{W}^\mathcal{J} \cong X_{*}(\widehat{T})^\mathcal{J} \rtimes W^\mathcal{J}\), where \(X_{*}(\widehat{T})^\mathcal{J} = \prod_{j \in \mathcal{J}} X_{*}(\widehat{T})\)
            and \(W^\mathcal{J} = \widehat{N}(\mathbb{F}^\mathcal{J}\llp u \rrp) / \widehat{T}(\mathbb{F}^\mathcal{J}\llp u \rrp) = \prod_{j \in \mathcal{J}} \widehat{N}(\mathbb{F}_{j}\llp u \rrp) / \widehat{T}(\mathbb{F}_{j}\llp u \rrp) \) is the finite Weyl group.
        \item \(\widetilde{W}^* \cong X_{*}(T^*)_{I} \rtimes \widetilde{W}^*\), where \(X_{*}(T^*)_{I}\) denote the \(I\)-coinvariants of \(X_{*}(T^*)\). 
    \end{enumerate}
\end{lemma}
\begin{proof}
    This is \cite[Proposition 13]{hainesAppendixParahoricSubgroups2008}.
    We provide a short argument in case (1). Namely, we have a short exact sequence 
    \begin{equation}
        \label{eq:iwahori-weyl-group-ses}
        1 \to X_{*}(\widehat{T}) \to \widetilde{W}_{j} \to W_{j} \to 1 , 
    \end{equation}
    in which the first map is the given by the isomorphism \(X_{*}(\widehat{T}) \isoto \widehat{T}(\mathbb{F}_{j}\llp u \rrp) / \widehat{T}(\mathbb{F}_{j}\llb u \rrb)\) which maps \(\lambda \mapsto u^\lambda\) (whose inverse is given by the valuation), 
    and \(W_{j} = \widehat{N}(\mathbb{F}_{j}\llp u \rrp) / \widehat{T}(\mathbb{F}_{j}\llp u \rrb)\) is the finite Weyl group. 
    Since \(\underline{W} := \widehat{N} / \widehat{T}\) is a constant group scheme over \(\operatorname{Spec}\mathbb{Z}\) by \cite[Proposition 5.1.6]{Conrad2014ReductiveGS}, we can identify \(W \cong \widehat{N}(\mathbb{F}_{j}\llb u \rrb) / \widehat{T}(\mathbb{F}_{j}\llb u \rrb)\) (note that \(\underline{W}(R) = \widehat{N}(R) / \widehat{T}(R)\) whenever \(H^1_{\et}(\operatorname{Spec} R, \widehat{T}) = 0\)).
    This gives a section of the second map in \eqref{eq:iwahori-weyl-group-ses} and the semidirect product decomposition follows. 
\end{proof}

\begin{remark}
    We remark that \(\widetilde{W}^*\) also admits a semidirect product decomposition of the affine Weyl group with the stabilizer of an alcove in the building \cite[Lemma 14]{hainesAppendixParahoricSubgroups2008},
    which can be used to equip \(\widetilde{W}^*\) with a quasi-Coxeter structure. 
    We will return to this point in \cref{sec:admissible-set}. 
\end{remark}

\subsubsection{The algebraic fundamental group}
\label{sec:algebraic-fundamental-group}
We recall that the algebraic fundamental group of \(\widehat{G}\) is  
\[
    \pi_{1}(\widehat{G}) := X_{*}(\widehat{T}) / \Phi^{\vee} (\widehat{G},\widehat{T})\mathbb{Z} . 
\]
Note that this definition is purely combinatorial, and as such does not depend on the field of definition. 
An alternative (but equivalent) definition can be found in \cite[Definition 11.3.2]{kaletha-prasad}. 

Since \(\Gamma\) already acts on \(\pi_{1}(\widehat{G})\) via pinned automorphism, we define \(\pi_{1}(G^*) := \pi_{1}(\widehat{G})\)
viewed as a \(\Gamma\)-group. 
This action will be used to form the coinvariants \(\pi_{1}(G^*)_{I}\). 
Note that if \(I\) acts trivially on \(\widehat{G}\), then \(\pi_{1}(G^*)_{I} = \pi_{1}(\widehat{G})\).

\begin{remark}
    The group \(\pi_{1}(G^*)_{I}\) has various other incarnations. 
    It can be identified with:
    \begin{enumerate}
        \item \(X^*(Z(G)^I)\), where \(Z(G)\) is the center of \(G\), the (split form of the) dual group of \(\widehat{G}\) \cite[196]{hainesAppendixParahoricSubgroups2008}.
        \item The stabilizer \(\widetilde{W}^*_{\mathcal{C}}\) of the ``base alcove'' \(\mathcal{C}\) in the Iwahori--Weyl group \(\widetilde{W}^*\) \cite[Lemma 14]{hainesAppendixParahoricSubgroups2008}. 
        \item \(\pi_{0}(L_{0}G^*)\), where \(L_{0}G^*(R) = G^*(R\llp v \rrp)\) \cite[Theorem 0.1]{pappasTwistedLoopGroups2008}. By loc. cit. we can also identify \(\pi_{1}(G^*)_{I}\) the connected components of the affine flag variety for \(G^*(\mathbb{F}\llp v \rrp)\) with respect to any parahoric subgroup. In particular, we can identify \(\pi_{1}(G^*)_{I} = \pi_{0}\left( \operatorname{Gr}_{\mathcal{G},\mathbb{F}} \right)\), where \(\operatorname{Gr}_{\mathcal{G}}\) is the affine Grassmannian defined in \cref{sec:affine-grassmannian} (we will not use this result). 
    \end{enumerate} 
\end{remark}

We will encounter the condition that \(\pi_{1}(G^*)_{I}\) is torsion-free. 
This condition admits alternative characterizations: 
\begin{lemma}
    \label{res:simply-connected-derived-group}
    \begin{enumerate}
        \item If \(\widehat{G}\) is semisimple, then \(\pi_{1}(\widehat{G})\) is a finite group. 
        \item We can identify \(\pi_{1}(\widehat{G})_{\text{tor}} \cong \pi_{1}(\widehat{G}_{\text{der}})\), where \(\widehat{G}_{\text{der}}\) is the derived group of \(\widehat{G}\). 
        \item If \(I\) acts trivially on \(\widehat{G}\), then \(\pi_{1}(G^*)_{I} = \pi_{1}(\widehat{G})\) is torsion-free if and only if \(\widehat{G}_{\text{der}}\) is simply connected. 
    \end{enumerate}
\end{lemma}
\begin{proof}
    (1) is well known (see for example \cite[Proposition 21.48]{milneAlgebraicGroupsTheory2017}).
    (3) is an immediate consequence of (2), so it remains to show (2).
    This follows from the proof of \cite[Lemma 11.5.2]{kaletha-prasad}, which we now recall. 
    Let \(\widehat{G}_{\text{der}} \subset \widehat{G}\) denote the derived subgroup of \(\widehat{G}\), and let \(p : \widehat{G}_{\text{sc}} \to \widehat{G}_{\text{der}}\) denote the simply connected covering.
    Let \(\widehat{T}_{\text{der}} = \widehat{T} \cap \widehat{G}_{\text{der}}\) and let \(\widehat{T}_{\text{sc}} = p^{-1} \widehat{T}_{\text{der}}\). 
    Then we have a short exact sequence 
    \[
        0 \to X_{*}(\widehat{T}_{\text{der}}) / X_{*}(\widehat{T}_{\text{sc}}) \to X_{*}(\widehat{T}) / X_{*}(\widehat{T}_{\text{sc}}) \to X_{*}(\widehat{T}) / X_{*}(\widehat{T}_{\text{der}}) \to 0 . 
    \]
    We can identify the leftmost group with \(\pi_{1}(\widehat{G}_{\text{der}})\), the middle group with \(\pi_{1}(\widehat{G})\), and the rightmost group with \(X_{*}(\widehat{T} / \widehat{T}_{\text{der}})\).
    The claim follows.  
\end{proof}

\subsection{Apartments and buildings}
\label{sec:buildings}
In this section we recall some facts from the theory of Bruhat--Tits buildings. 
This theory organizes the structure of bounded subgroups of a group over a hensellian local field by exhibiting a geometric object, called the Bruhat--Tits building, on which 
the group acts, and for which bounded subgroups correspond to bounded subsets by taking stabilizers. 
We refer the reader to \cite{kaletha-prasad} for additional details about Bruhat--Tits theory. 

Let \(\widehat{G}, \widehat{B}, \widehat{T}\) be as in \cref{sec:dual-groups}.
In this section we will for \(j \in \mathcal{J}\) let \(\widehat{G}_{j} = \widehat{G}|_{\mathbb{F}_{j}\llp u \rrp}\), \(\widehat{B}_{j} = \widehat{G}|_{\mathbb{F}_{j}\llp u \rrp}\), and \(\widehat{T}_{j} = T|_{\mathbb{F}_{j\llp u \rrp}}\). 

\subsubsection{The apartment}
Let \(j \in \mathcal{J}\). 
Then the enlarged apartment \(\widetilde{\mathcal{A}}\left( \widehat{G}_{j}, \mathbb{F}_{j}\llp u \rrp \right)\), as defined in \cite[Definition 6.1.27]{kaletha-prasad}, is a certain affine space under \(X_{*}(\widehat{T}_{j})_{\mathbb{R}} := X_{*}(\widehat{T}_{j}) \otimes_{\mathbb{Z}} \mathbb{R}\). 
The pinning of \(\widehat{G}\) gives for each \(j \in \mathcal{J}\) a distinguished point \(o_{j} \in \widetilde{\mathcal{A}}\left( \widehat{G}_{j}, \mathbb{F}_{j} \llp u \rrp \right)\) called a Chevalley valuation \cite[Definition 6.1.9]{kaletha-prasad},\footnote{
    The Chevalley valuation as defined in loc. cit. is really a point of the reduced apartment, so this is a slight abuse of notation. 
    We give further justification for this abuse of notation in \cref{sec:reduced-apartment} below. 
}
and this gives rise to an isomorphism \cite[Section 4.3.4]{kaletha-prasad}
\begin{align}
    \label{eq:coroots-apartment-iso-j}
    X_{*}(\widehat{T}_{j})_{\mathbb{R}} & \isoto \widetilde{\mathcal{A}}\left( \widehat{T}_{j}, \mathbb{F}_{j} \llp u \rrp \right) \\
    \lambda & \mapsto u^\lambda \cdot o_{j} = o_{j} - \frac{1}{e} \lambda ,  \nonumber 
\end{align}
which will have the property that the (hyperspecial) parahoric subgroup \(u^\lambda \widehat{G}(\mathbb{F}_{j}\llb u \rrb) u^{-\lambda} \subset \widehat{G}(\mathbb{F}_{j}\llp u \rrp)\)
corresponds to \(u^\lambda \cdot o_{j}\) (see \cref{sec:valuations-of-the-root-datum}).
\begin{remark}
    We note that the apartment \(\widetilde{\mathcal{A}}\left( \widehat{T}_{j}, \mathbb{F}_{j} \llp u \rrp \right)\) depends on the discrete valuation on \(\mathbb{F}_{j}\llp u \rrp\) (although any choice of discrete valuation yield apartments, and more generally buildings, which are abstractly isomorphic). 
    With respect to our normalizations, \(v\) has valuation \(1\) and \(u\) has valuation \(1/e\), and this is why there is a factor of \(1/e\) in \eqref{eq:coroots-apartment-iso-j}. 
\end{remark}

Define the enlarged apartment of \(\widehat{G}\) corresponding to \(\widehat{T}\) by
\[
\widetilde{\mathcal{A}}\left( \widehat{T}, \mathbb{F}^{\mathcal{J}}\llp u \rrp \right) = \prod_{j \in \mathcal{J}} \widetilde{\mathcal{A}}\left( \widehat{T}_{j}, \mathbb{F}_{j} \llp u \rrp \right) . 
\]
Let us denote \(X_{*}(\widehat{T})^\mathcal{J}_{\mathbb{R}} := \prod_{j \in \mathcal{J}}X_{*}(\widehat{T}_{j})_{\mathbb{R}} = \operatorname{Hom}_{\mathbb{F}^\mathcal{J}\llp u \rrp \text{-groups}} \left( \mathbb{G}_{m}, \widehat{T} \right)\). 
Then by taking products, \eqref{eq:coroots-apartment-iso-j} yields another isomorphism
\begin{equation}
    \label{eq:coroots-apartment}
    X_{*}(\widehat{T})^\mathcal{J}_{\mathbb{R}} \isoto \widetilde{\mathcal{A}}\left( \widehat{T}, \mathbb{F}^\mathcal{J}\llp u \rrp \right) . 
\end{equation}
The point \(0\) on the left hand side corresponds to \(o = (o_{j})_{j \in \mathcal{J}}\) on the right hand side.

\begin{example}
    When \(\mathcal{O} = \mathbb{Z}_{p}\), the apartment \(\widetilde{\mathcal{A}}\left( \widehat{T}, \mathbb{F}^\mathcal{J}\llp u \rrp \right) = \widetilde{\mathcal{A}}\left( \widehat{T}, \mathbb{F}_{q}\llp u \rrp \right)\) is exactly the enlarged apartment as described in \cite[Definition 6.1.27]{kaletha-prasad}. 
\end{example}

\subsubsection{The reduced apartment and simplicial structure}
\label{sec:reduced-apartment}
We will also need the (reduced) apartments \(\mathcal{A}\left( \widehat{T}_{j}, \mathbb{F}_{j}\llp u \rrp \right)\) (for \(j \in \mathcal{J}\)) and \(\mathcal{A}\left(\widehat{T}, \mathbb{F}^\mathcal{J}\llp u \rrp \right)\) (defined similarly to the enlarged apartment \(\widetilde{\mathcal{A}}\left(\widehat{T},\mathbb{F}^\mathcal{J}\llp u \rrp \right)\) in terms of the components labeled by \(j \in \mathcal{J}\)).
Let \(j \in \mathcal{J}\). 
By definition, the apartment \(\mathcal{A}\left( \widehat{T}_{j}, \mathbb{F}_{j}\llp u \rrp \right)\) is the enlarged apartment as defined above for the derived subgroup \(\widehat{G}_{j,\text{der}} \subset \widehat{G}_{j}\), 
and the corresponding maximal torus \(\widehat{T}_{j,\text{der}} \subset \widehat{G}_{j,\text{der}}\). 
That is,
\[
    \mathcal{A}\left( \widehat{T}_{j}, \mathbb{F}_{j}\llp u \rrp \right) := \widetilde{\mathcal{A}}\left( \widehat{T}_{j,\text{der}}, \mathbb{F}_{j}\llp u \rrp \right) = \mathcal{A}\left( \widehat{T}_{j,\text{der}}, \mathbb{F}_{j} \llp u \rrp \right) . 
\]
Now let \(\widehat{T}_{j,\text{ad}} \subset \widehat{G}_{j,\text{ad}}\) denote the maximal torus of the adjoint group which is the image of \(\widehat{T}_{j}\). 
By noting that the map \(X_{*}(\widehat{T}_{j,\text{der}})_{\mathbb{R}} \isoto X_{*}(\widehat{T}_{j,\text{ad}})_{\mathbb{R}}\) is an isomorphism (although it is not before tensoring with \(\mathbb{R}\)), 
it follows that the map \(\mathcal{A}\left( \widehat{T}_{j,\text{der}},\mathbb{F}_{j} \llp u \rrp \right) \isoto \mathcal{A}\left( \widehat{T}_{j,\text{ad}}, \mathbb{F}_{j}\llp u \rrp \right)\) is an isomorphism, 
which provides a canonical splitting
\begin{equation}
    \label{eq:enlarged-apartment-splitting}
    \widetilde{\mathcal{A}}\left( \widehat{T}_{j}, \mathbb{F}_{j}\llp u \rrp \right) = \mathcal{A}\left( \widehat{T}_{j}, \mathbb{F}_{j}\llp u \rrp \right) \oplus X_{*}(\widehat{A}_{j}, \mathbb{F}_{j}\llp u \rrp)_{\mathbb{R}} ,    
\end{equation}
where \(\widehat{A}_{j} \subset \widehat{G}_{j}\) denotes the center.

The reduced apartment \(\mathcal{A}\left(\widehat{T}_{j}, \mathbb{F}_{j} \llp u \rrp \right)\) carries a simplicial structure, which is described in \cite[Section 6.3.18]{kaletha-prasad}. 
We remark that loc. cit. calls maximal simplices by the name \emph{chambers}, whereas we call them \emph{alcoves}.
The Chevalley valuation \(o_{j}\) is then a vertex in \(\mathcal{A}\left( \widehat{T}_{j}, \mathbb{F}_{j} \llp u \rrp \right)\); in fact it is a \emph{hyperspecial} vertex.
Via the natural inclusion \(\mathcal{A}\left( \widehat{T}_{j}, \mathbb{F}_{j}\llp u \rrp \right) \subset \widetilde{\mathcal{A}}\left( \widehat{T}_{j}, \mathbb{F}_{j}\llp u \rrp \right)\) (corresponding to \(0 \in X_{*}(\widehat{A}_{j}, \mathbb{F}_{j}\llp u \rrp)_{\mathbb{R}}\) via \eqref{eq:enlarged-apartment-splitting}),
we may then view \(o_{j}\) as a point of \(\widetilde{\mathcal{A}}\left( \widehat{T}_{j}, \mathbb{F}_{j}\llp u \rrp \right)\) as we did above. 

Finally, we define 
\[
    \mathcal{A}\left( \widehat{T}, \mathbb{F}^\mathcal{J} \llp u \rrp \right) := \prod_{j \in J} \mathcal{A}\left( \widehat{T}_{j}, \mathbb{F}_{j}\llp u \rrp \right) . 
\]
We view this as carrying the simplicial structure of the product, so a vertex is a product of vertices, an alcove is a product of alcoves, and so on. 
The splitting \eqref{eq:enlarged-apartment-splitting} induces a canonical splitting 
\[
    \widetilde{\mathcal{A}} \left( \widehat{T}, \mathbb{F}^\mathcal{J}\llp u \rrp \right) = \mathcal{A}\left( \widehat{T}, \mathbb{F}^\mathcal{J}\llp u \rrp \right) \oplus \prod_{j \in \mathcal{J}} X_{*}(\widehat{A}_{j})_{\mathbb{R}} . 
\]
Abusively, we will say that a point \(\widetilde{y} \in \widetilde{\mathcal{A}}\left( \widehat{T}, \mathbb{F}^\mathcal{J}\llp u \rrp \right)\) is a vertex its projection \(y \in \mathcal{A}\left( \widehat{T},\mathbb{F}^\mathcal{J}\llp u \rrp \right)\) is a vertex.
Similarly, we will say that \(\widetilde{y}\) is in a certain alcove if its projection \(y\) is, and so on.

\subsubsection{Valuations of the root datum and parahoric subgroups}
\label{sec:valuations-of-the-root-datum}
In \cite[Chapter 6]{kaletha-prasad}, the points of the (reduced) apartment are defined as valuations of the root datum.\footnote{
    More precisely, the points of the apartment are defined as \emph{equipollence classes} of valuations of the root datum. 
    Since the role of equipollence classes is to have a definition which doesn't depend on a choice of pinning, we simply ignore this, and note that 
    the valuations of root datum we discuss are all with respect to the chosen Chevalley valuation \(o\). 
}
We will now give a brief explanation of how a point \(\widetilde{y} \in \widetilde{\mathcal{A}}\left( \widehat{T}, \mathbb{F}^\mathcal{J}\llp u \rrp \right)\) gives rise to a valuation of the root datum, 
and how this is used to describe parahoric subgroups of \(\widehat{G}(\mathbb{F}^\mathcal{J}\llp u \rrp) = \prod_{j \in \mathcal{J}} \widehat{G}_{j}(\mathbb{F}_{j}\llp u \rrp)\). 
For simplicity, we fix \(j \in \mathcal{J}\) and focus on a single factor.

Let \(\widetilde{y} \in \widetilde{\mathcal{A}}\left( \widehat{T}_{j}, \mathbb{F}_{j}\llp u \rrp \right)\), and let \(y \in \mathcal{A}\left( \widehat{T}_{j}, \mathbb{F}_{j}\llp u \rrp \right)\)
denote the projection of \(\widetilde{y}\) via \eqref{eq:enlarged-apartment-splitting}. 
The ``valuation of the root datum'' (with respect to \(o_{j}\)) described by \(\widetilde{y}\) only depends on \(y\), and it is given as follows. 
For \(a \in \Phi(\widehat{G},\widehat{T})\) a root and \(U_{a} \hookrightarrow \widehat{G}\) the corresponding root group,
the pinning of \(\widehat{G}\) provides an isomorphism \(u_{a} : \mathbb{G}_{a} \isoto U_{a}\), and we can 
define 
\begin{align*}
    & \varphi_{o,a} : U_{a} \left( \mathbb{F}_{j}\llp u \rrp \right) \stackrel{u_{a}^{-1}}{\longrightarrow} \mathbb{F}_{j}\llp u \rrp \stackrel{\operatorname{val}}{\longrightarrow} \frac{1}{e} \mathbb{Z} , & \text{where }\operatorname{val}\text{ is the discrete valuation,} \\
    & \varphi_{y,a}(u) := \varphi_{o,a}(u) + \langle a, y - o \rangle . &
\end{align*}
The collection \(\varphi_{y} = \lbrace \varphi_{y,a} \rbrace_{a \in \Phi(\widehat{G},\widehat{T})} \) is called a \emph{valuation of the root datum}.
\begin{remark}
    Note that the notation ``\(\varphi\)'' here has nothing to do with Frobenius, but we use it to be consistent with \cite[Chapter 6]{kaletha-prasad}.
    We only use \(\varphi\) to denote valuations of the root datum in this subsection; elsewhere in this work \(\varphi\) refers to a Frobenius morphism. 
\end{remark}

We can use \(\varphi_{y}\) to describe subgroups of the root groups. Namely, for \(a \in \Phi(\widehat{G},\widehat{T})\) let 
\begin{equation}
    \label{eq:filtrations-of-root-groups}
    U_{a,y} := \varphi_{y,a}^{-1}\left( [0, \infty] \right) = u_{a} \left( u^{\lceil - e\langle a, y - o \rangle \rceil} \right) \subset U_{a} \left( \mathbb{F}_{j}\llp u \rrp \right) . 
\end{equation}
Here, the second equality depends on the fact that \(u\) has valuation \(1/e\).
The parahoric subgroup 
\[
    P_{y} := \widehat{G}(\mathbb{F}_{j}\llp u \rrp)_{y}^0 \subset \widehat{G}(\mathbb{F}_{j}\llp u \rrp)
\]
attached to \(y\), as defined in \cite[Definition 7.3.3]{kaletha-prasad}, is the subgroup generated by the subgroups \(\widehat{T}(\mathbb{F}_{j}\llb u \rrb)\) and \(U_{a,x}\), \(a \in \Phi(\widehat{G},\widehat{T})\).

More generally, we define for any \(r \in \mathbb{R}\) (and even more generally, \(r \in \widetilde{\mathbb{R}}\) with \(\widetilde{\mathbb{R}}\) as in \cite[Section 1.6]{kaletha-prasad})
\[
    U_{a,y,r} := \varphi_{y,a}^{-1}\left( [r,\infty] \right) = u_{a}\left( u^{\lceil - e \langle a, y - o \rangle + re \rceil} \right) \subset U_{a}\left( \mathbb{F}_{j}\llp u \rrp \right) . 
\]
If \(r \geq 0\) we can also define \(\mathbb{G}_{m,r}(\mathbb{F}_{j}\llp u \rrp) = \left( u^{\lceil e r \rceil} \mathbb{F}_{j} \llp u \rrp \right)^\times\), and by taking products, we obtain for the split torus \(\widehat{T}\) a subgroup \(\widehat{T}_{r} \subset \widehat{T}(\mathbb{F}_{j}\llp u \rrp)\).
We note that \(\widehat{T}_{r}\) also admits a more ``canonical'' description in terms of a valuation morphism \(\operatorname{val} : \widehat{T}(\mathbb{F}_{j}\llp u \rrp) \to \frac{1}{e} \mathbb{Z}\) \cite[Section 7.2]{kaletha-prasad}.\footnote{In the first edition of \cite{kaletha-prasad}, there is a sign error in the definition of this valuation morphism in \cite[Section 2.5.b]{kaletha-prasad}, as pointed out in the errata available on Kaletha's homepage.}
Given \(r \geq 0\), the subgroup \(\widehat{G}(\mathbb{F}_{j}\llp u \rrp)_{y,r}^0 \subset \widehat{G}(\mathbb{F}_{j}\llp u \rrp)\) is by definition generated by \(U_{a,y,r}\) and \(\widehat{T}_{r}\). 

Even more generally, let \(f : \widehat{\Phi}(\widehat{G},\widehat{T}) \to \widetilde{\mathbb{R}}\) denote a concave function in the sense of \cite[Definition 7.3.1]{kaletha-prasad}, where \(\widehat{\Phi}(\widehat{G},\widehat{T}) = \Phi(\widehat{G},\widehat{T}) \cup \lbrace 0 \rbrace\). 
Then the subgroup \(\widehat{G}(\mathbb{F}_{j}\llp u \rrp)_{y,f}^0\) of \cite[Definition 7.3.3]{kaletha-prasad} is by definition generated by \(\widehat{T}_{f(0)}\) and \(U_{a,y,f(a)}\), \(a \in \Phi(\widehat{G},\widehat{T})\).
When \(f = r\) is constant we recover the groups \(\widehat{G}(\mathbb{F}_{j}\llp u \rrp)_{y,r}^0\) and when \(f = 0\) we recover the parahoric subgroups \(\widehat{G}(\mathbb{F}_{j}\llp u \rrp)_{y,0}^0\). 

As shown in \cite[Theorem 8.5.2]{kaletha-prasad}, there exists for \(y\) and \(f\) a concave function as above a smooth group scheme \(\mathcal{G}_{y,f}\) over \(\operatorname{Spec}\mathbb{F}_{j}\llb u \rrb\) with connected fibers, whose generic fiber is \(\widehat{G}\) and where the subgroup \(\mathcal{G}_{y,f}(\mathbb{F}_{j}\llb u \rrb) \subset \mathcal{G}_{y,f}(\mathbb{F}_{j}\llp u \rrp) = \widehat{G}(\mathbb{F}_{j}\llp u \rrp)\) equals \(\widehat{G}(\mathbb{F}_{j}\llp u \rrp)_{y,f}^0\).

\begin{example}
    \label{ex:valuation-of-the-root-datum}
    Suppose \(\widehat{G} = \GL_{3}\), \(\widehat{T} \subset \widehat{G}\) is the diagonal torus,
    and we identify \(X_{*}(\widehat{T}) = \mathbb{Z}^3\) in the usual way. 
    Let \(\lambda = - \frac{1}{3}(1,0,-1)\) and \(y = u^\lambda \cdot o = o + \frac{1}{3e}(1,0,-1)\).
    Then the parahoric subgroup \(P_{y} = \widehat{G}(\mathbb{F}_{j}\llp u \rrp)_{y}^0 \subset \widehat{G}(\mathbb{F}_{j}\llp u \rrp)\) is generated by \(\widehat{T}(\mathbb{F}_{j}\llb u \rrb)\) and \(U_{a,y}\), \(a \in \Phi(\widehat{G},\widehat{T})\). 
    The positive roots are \(\alpha_{12} = (1,-1,0)\), \(\alpha_{23} = (0,1,-1)\), and \(\alpha_{13} = (1,0,-1)\). 
    Note that \(\langle \alpha_{12},y - o \rangle = \langle \alpha_{23},y-o \rangle = 1/3e\) and \(\langle \alpha_{13},y-o \rangle = 2/3e\), so by \eqref{eq:filtrations-of-root-groups} we can see that 
    \begin{align*}
        & U_{\alpha_{12}} = u_{\alpha_{12}}(\mathbb{F}_{j}\llp u \rrp), & U_{-\alpha_{12}} = u_{-\alpha_{12}}(u \mathbb{F}_{j}\llp u \rrp) , \\ 
        & U_{\alpha_{23}} = u_{\alpha_{23}}(\mathbb{F}_{j}\llp u \rrp), & U_{-\alpha_{23}} = u_{-\alpha_{23}}(u \mathbb{F}_{j}\llp u \rrp) , \\
        & U_{\alpha_{13}} = u_{\alpha_{13}}(\mathbb{F}_{j}\llp u \rrp), & U_{-\alpha_{13}} = u_{-\alpha_{13}}(u \mathbb{F}_{j}\llp u \rrp) . 
    \end{align*}
    To summarize, we have 
    \[
        P_{y} = \left\lbrace \begin{pmatrix}
            \mathbb{F}_{j}\llp u \rrp^\times & \mathbb{F}_{j} \llp u \rrp & \mathbb{F}_{j}\llp u \rrp \\
            u \mathbb{F}_{j}\llp u \rrp & \mathbb{F}_{j}\llp u \rrp^\times & \mathbb{F}_{j}\llp u \rrp \\
            u \mathbb{F}_{j}\llp u \rrp & u \mathbb{F}_{j}\llp u \rrp & \mathbb{F}_{j}\llp u \rrp^\times 
        \end{pmatrix} \right\rbrace . 
    \]
\end{example}

\subsubsection{Tamely ramified descent}
\label{sec:tamely-ramified-descent}
Assume that the group \(\Gamma\) acts on \(\widehat{G}\) via pinned automorphisms, which in particular restrict to an action on \(\widehat{T}\).
Then \(\Gamma\) acts by ``conjugation'' on \(X_{*}(\widehat{T})^\mathcal{J}_{\mathbb{R}} = \prod_{j \in \mathcal{J}} X_{*}(\widehat{T}_{j})_{\mathbb{R}}\), meaning that if \(x = (x_{j})_{j \in \mathcal{J}}\) is an element, then 
\begin{equation}
    \label{eq:conjugation-action-on-cocharacters}
    \theta (x)_{j} = \theta (x_{\theta^{-1}j})  ,
\end{equation}
where \(\theta (y)\) for \(y \in X_{*}(\widehat{T}_{j})_{\mathbb{R}}\) is the action via pinned automorphisms. 
We let \(\Gamma\) act on \(\widetilde{\mathcal{A}}\left( \widehat{T}, \mathbb{F}^\mathcal{J}\llp u \rrp \right)\) via \eqref{eq:coroots-apartment}. 

For any \(n \in \widehat{N}(\mathbb{F}^\mathcal{J}\llp u \rrp)\) and \(x \in \widetilde{\mathcal{A}}\left( \widehat{T}, \mathbb{F}^\mathcal{J}\llp u \rrp \right)\), we have
\[
\theta(n \cdot x) = {^\theta n} \cdot \theta(x) \,\,\,\,\, \text{ for } \theta \in \Gamma . 
\]
Indeed, this shows why the above \(\Gamma\)-action on the apartment is the ``correct'' one. 

Recall that the torus \(\widehat{T}\) descends to a torus \(T^*_{\mathbb{F}\llp v \rrp}\) of \(G^*_{\mathbb{F}\llp v \rrp}\) over \(\mathbb{F}\llp v \rrp\) by étale descent. 
It follows from tamely ramified descent for Bruhat--Tits buildings \cite[Chapter 12]{kaletha-prasad} that we have an identification
\[
\widetilde{\mathcal{A}}\left( T^*, \mathbb{F}\llp v \rrp \right) = \widetilde{\mathcal{A}}\left( \widehat{T}, \mathbb{F}^\mathcal{J} \llp u \rrp \right)^\Gamma . 
\]
Let us briefly explain why this is true. 
By tamely ramified descent as in loc. cit., we can identify \(\widetilde{\mathcal{A}} \left( T^*, \mathbb{F}^\mathcal{J} \llp v \rrp \right) = \widetilde{\mathcal{A}} \left( \widehat{T}, \mathbb{F}^\mathcal{J}\llp u \rrp \right)^I\),
using the fact that \(I\) acts diagonally. 
Then we can perform unramified descent in two steps, first perform ``diagonal'' unramified descent \cite[Chapter 9]{kaletha-prasad} with respect to the stabilizer group \(\Gamma_{0,\mathcal{J}}\),
and finally take fixed points with respect to the quotient \(\Gamma_{0}'\) which permutes the factors (in a possibly twisted way).

\begin{remark}
    Even when \(I\) acts trivially on \(\widehat{G}\), the identification \(\widetilde{\mathcal{A}}\left( \widehat{T}, \mathbb{F}^\mathcal{J}\llp v \rrp \right) = \widetilde{\mathcal{A}}\left( \widehat{T}, \mathbb{F}^\mathcal{J}\llp u \rrp \right)^I = \mathcal{A}\left( \widehat{T}, \mathbb{F}^\mathcal{J}\llp u \rrp \right)\)
    typically does not identify the simplicial structures (as one typically \emph{defines} the simplicial structure on the \(I\)-fixed points so that the second equality identifies the simplicial structures). 
    If we consider the points of the apartment as being valuations of root data, the walls are determined by where there are jumps in the filtration of a root subgroup. 
    But since the valuation of \(u\) is \(1/e\) and the valuation of \(v\) is \(1\), these jumps occur at different stages.
    See \cref{fig:alcoves-inside-alcoves} on page \pageref{fig:alcoves-inside-alcoves} for an illustration of what this might look like. 
\end{remark}
\begin{figure}[h!]
    \centering
    \resizebox{\textwidth}{!}{
    \begin{tikzpicture}[scale = 12, rotate = 180]
        % Set the number of segments and the prime
        \pgfmathsetmacro{\numSegments}{6}

        % Function to draw the divided triangle
        \newcommand{\dividedTriangle}{
            % Define the coordinates of the large triangle
            \coordinate (A) at (0, 0);
            \coordinate (B) at (1, 0);
            \coordinate (C) at (0.5, 0.866); % height of equilateral triangle with side length 1

            % Divide each side into equal segments and draw the smaller triangles
            \foreach \i in {0,...,\numSegments} {
                \coordinate (P) at ($(A)!\i/\numSegments!(B)$);
                \coordinate (Q) at ($(A)!\i/\numSegments!(C)$);
                \coordinate (R) at ($(B)!\i/\numSegments!(C)$);
                \coordinate (S) at ($(C)!\i/\numSegments!(B)$);
                
                % Draw lines parallel to the sides of the large triangle
                \draw[line width=0.05pt,color=gray] (P) -- (Q);
                \draw[line width=0.05pt,color=gray] (P) -- (S);
                \draw[line width=0.05pt,color=gray] (Q) -- (R);
            }

            % Draw the large triangle
            \draw[thick] (A) -- (B) -- (C) -- cycle;
        }

        % Define the clipping region to create a square "photograph"
        \clip (-0.2,-0.173) rectangle (1.2,1.039);

        % Draw triangles which are shifts of divided triangle
        \foreach \x / \y in {-0.5/-0.866, 0.5/-0.866, -1/0, 0/0, 1/0, -0.5/0.866, 0.5/0.866} {
            \begin{scope}[shift={(\x,\y)}]
                \dividedTriangle
            \end{scope}
        }

        % Draw traingles which are shifts of upside down divided triangle
        \foreach \x / \y in {-0.5/-0.866, 0.5/-0.866, -1/0, 0/0, 1/0, 0/-1.732} {
            \begin{scope}[yscale=-1,shift={(\x,\y)}]
                \dividedTriangle
            \end{scope}
        }

        % Label corners of triangle
        \node[circle, fill=black, inner sep=1.5pt, label={[fill=white, fill opacity=0.7, text opacity=1, rounded corners] above:{$o$}}] at (0.5,0.866) {};
        \node[circle, fill=black, inner sep=1.5pt, label={[fill=white, fill opacity=0.7, text opacity=1, rounded corners] above:{$o + (1,0,0)$}}] at (0,0) {};
        \node[circle, fill=black, inner sep=1.5pt, label={[fill=white, fill opacity=0.7, text opacity=1, rounded corners] above:{$o + (1,1,0)$}}] at (1,0) {};
    \end{tikzpicture}}
    \caption{
        A snapshot of the apartment \(\widetilde{\mathcal{A}}\left( \widehat{T}, \mathbb{F}\llp v \rrp \right) = \widetilde{\mathcal{A}}\left( \widehat{T}, \mathbb{F}\llp u \rrp \right)^I\), when \(\widehat{T} \subset \widehat{G} = \GL_{3}\) is the diagonal torus and \(e = 6\).
        The small triangles bounded by think gray lines are the alcoves of \(\widetilde{\mathcal{A}}\left( \widehat{T}, \mathbb{F}\llp u \rrp \right)\), whereas the large triangles bounded by thick black lines are the alcoves of \(\widetilde{\mathcal{A}}\left( \widehat{T}, \mathbb{F}\llp v \rrp \right)\). 
    }
    \label{fig:alcoves-inside-alcoves}
\end{figure}
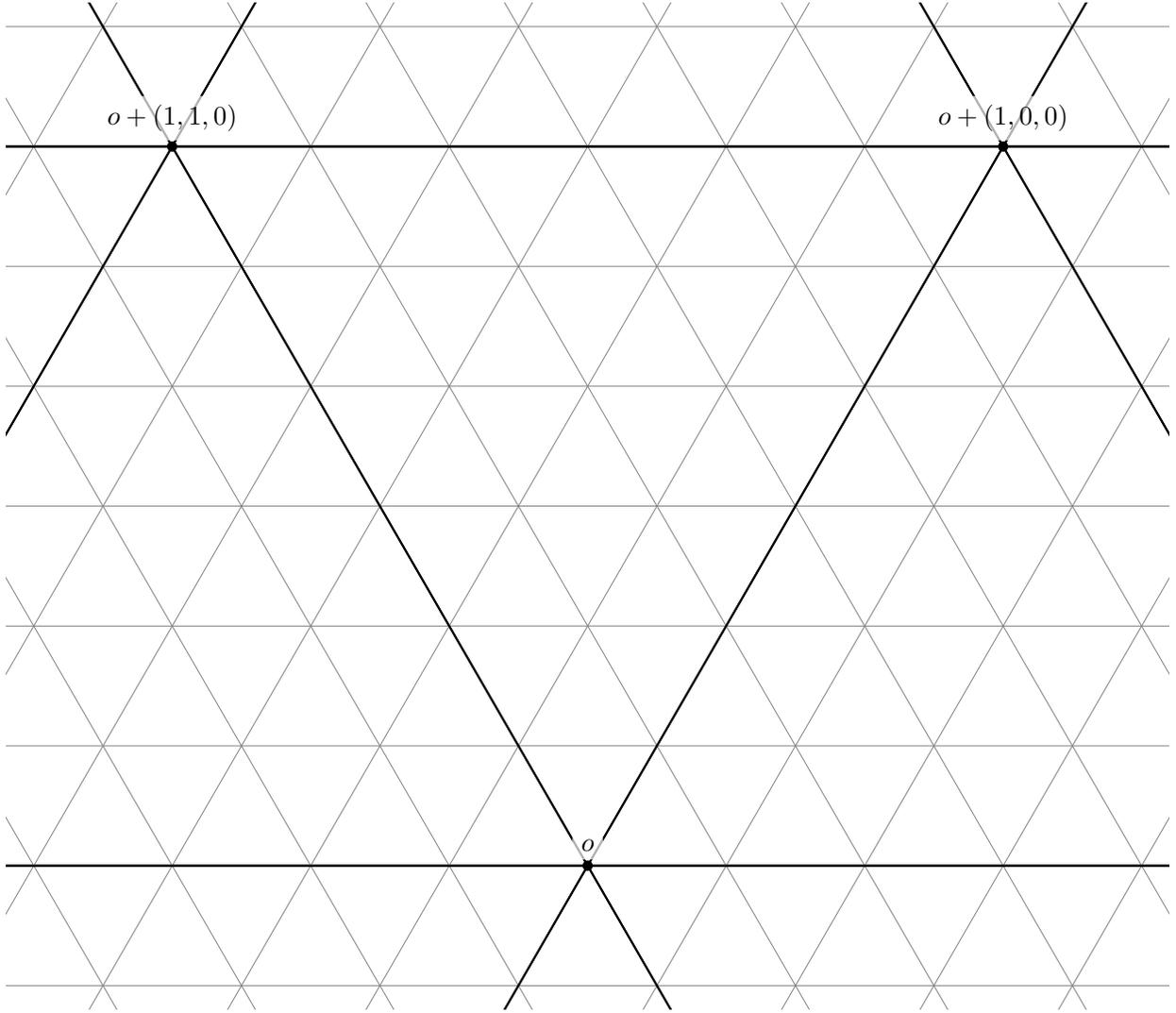

\subsubsection{The building}
\label{sec:the-building}
The enlarged building \(\widetilde{\mathcal{B}}\left( \widehat{G}, \mathbb{F}^\mathcal{J}\llp u \rrp \right)\) is the quotient of \(\widehat{G}(\mathbb{F}^\mathcal{J}\llp u \rrp) \times \widetilde{\mathcal{A}}\left( \widehat{T}, \mathbb{F}^\mathcal{J}\llp u \rrp \right)\)
by the equivalence relation \((g,x) \sim (g',x')\) if there exists \(n \in \widehat{N}(\mathbb{F}^\mathcal{J}\llp u \rrp)\) such that \(x' = n \cdot x\) and \(g' = n g n^{-1}\).
This will satisfy tamely ramified descent for the same reason as above. 

One can also define the (reduced) building \(\mathcal{B}\left( \widehat{G}, \mathbb{F}^\mathcal{J}\llp u \rrp \right)\), and as in the case of the apartment we have a canonical splitting
\[
    \widetilde{\mathcal{B}}\left( \widehat{G}, \mathbb{F}^\mathcal{J}\llp u \rrp \right) = \mathcal{B}\left( \widehat{G}, \mathbb{F}^\mathcal{J}\llp u \rrp \right) \oplus \prod_{j \in J} X_{*}(\widehat{A}_{j})_{\mathbb{R}} , 
\]
similarly to \eqref{eq:enlarged-apartment-splitting}.
Abusively we will use the simplicial structure of the building to talk about the enlarged building, so for example we will call a point \(\widetilde{y} \in \widetilde{\mathcal{B}}\left( \widehat{G}, \mathbb{F}^\mathcal{J}\llp u \rrp \right)\)
a vertex if its projection \(y \in \mathcal{B}\left( \widehat{G}, \mathbb{F}^\mathcal{J}\llp u \rrp \right)\) is a vertex. 

\subsection{Various Frobenius endomorphisms}
\label{sec:frobenius}
We warn the reader that the notation \(\varphi\) will be severely overloaded.
First of all we will use \(\varphi\) to denote the usual (arithmetic, \(x \mapsto x^{p}\) modulo \(p\)) Frobenius automorphism of \(\mathbb{F}_{p^k}\), \(\mathbb{Z}_{p^k}\), and \(\mathbb{Q}_{p^k}\).
For any \(\mathcal{O}\)-algebra \(R\) we have the Frobenius endomorphism \(\varphi: \left( R \otimes_{\mathbb{Z}_{p}} \mathbb{Z}_{q} \right) [u] \to \left( R \otimes_{\mathbb{Z}_{p}} \mathbb{Z}_{q} \right)[u]\), 
which acts as the usual on \(\mathbb{Z}_{q}\) and \(u \mapsto u^{p}\).
We similarly define endomorphisms of \(\left(R \otimes_{\mathbb{Z}_{p}} \mathbb{Z}_{q} \right)[v]\), \(R[v]\), \(\mathfrak{S}_{L,R}\), \(\mathfrak{S}_{R}\), \(\mathcal{E}_{L,R}\), \(\mathcal{E}_{R}\), \(E^\mathcal{J}\llp u \rrp\), \(\mathbb{F}^\mathcal{J}\llp u \rrp\), and we denote all of them by \(\varphi\). 
There are induced endomorphisms of the schemes \(\widetilde{\mathbb{A}}^1_{R}\), \(\breve{\mathbb{A}}^1_{R}\), and \(\mathbb{A}^1_{R}\), which we also denote by \(\varphi\). 

Let \(H\) be a group scheme defined over \(\mathcal{O}\) (we will apply this to \(H = \widehat{G}\)). 
Let \(\widetilde{H} := H \times_{\mathcal{O}} \widetilde{\mathbb{A}}^1_{\mathcal{O}}\).
Then we have an endomorphism \(\varphi = 1 \otimes \varphi : H \times_{\mathcal{O}} \widetilde{\mathbb{A}}^1_{\mathcal{O}} \to H \times_{\mathcal{O}} \widetilde{\mathbb{A}}^1_{\mathcal{O}}\), 
which corresponds to an identification \(\varphi^* \widetilde{H} \cong \widetilde{H}\).

Given an affine \(\widetilde{\mathbb{A}}^1_{\mathcal{O}}\) scheme \(\operatorname{Spec}A\), we also let \(\varphi\) denote the effect of pulling back along \(\varphi\), i. e. 
\[
\varphi : \widetilde{H}(A) \to \varphi^* \widetilde{H} \left( \varphi^* A \right) . 
\]
Typically, \(A\) is equipped with its own Frobenius endomorphism, giving a natural map \(A \to \varphi^* A\). 
(For example, we could take \(A = \mathfrak{S}_{L,R}\).) 
In this case, we also identify \(\varphi\) with the composite 
\[
\varphi : \widetilde{H}(A) \to \varphi^* \widetilde{H} \left( \varphi^* A \right) \cong \widetilde{H}\left(\varphi^* A \right) \to \widetilde{H}(A) . 
\]
\begin{example}
    If \(H = \mathbb{G}_{a}\) and \(A\) has a Frobenius endomorphism \(\varphi : A \to A\), then \(\varphi : \widetilde{H}(A) \to \widetilde{H}(A)\) is identified with the usual Frobenius \(\varphi : A \to A\). 
\end{example}
\begin{remark}
    If \(f : H_{1} \to H_{2}\) is a homomorphism of group schemes over \(\mathcal{O}\), then the induced homomorphism \(\widetilde{f} : \widetilde{H}_{1} \to \widetilde{H}_{2}\)
    commutes with \(\varphi\). For example, we will apply this to the case when \(\mathbb{G}_{a} \isoto U_{a} \hookrightarrow \widehat{G}\) is the inclusion of a root group. 
\end{remark}

\subsubsection{Equivariance}
The \(\Gamma\)-action on \(\widetilde{H}\) commutes with \(\varphi\), because the \(\Gamma\)-action on \(\widetilde{\mathbb{A}}^1_{\mathcal{O}}\) commutes with \(\varphi\). 
This is also true when \(H\) has an action of \(\Gamma\) defined over \(\mathcal{O}\), and \(\widetilde{H}\) is given the ``diagonal'' action.

The resulting map on points will then also be \(\Gamma\)-equivariant. 
For example, the map on global sections \(\varphi : \widetilde{H}(\mathfrak{S}_{L,R}) \to \widetilde{H}(\mathfrak{S}_{L,R})\) is \(\Gamma\)-equivariant.
In particular, there is an induced map \(\varphi : H^1 \left( \Gamma, \widetilde{H}(\mathfrak{S}_{L,R})\right) \to H^1 \left( \Gamma, \widetilde{H}(\mathfrak{S}_{L,R})\right)\).

\subsubsection{Endomorphism of the building}
\label{sec:frobenius-on-building}
We follow \cite[Section 6.b.2]{pappasPhiModulesCoefficient2009} in giving \(\varphi\) as an endomorphism of the apartments and buildings we considered in \cref{sec:buildings}.
For simplicity we assume that \(\mathbb{F} \supseteq \mathbb{F}_{q}\), so that we are in the situation of \cref{ex:large-coefficient-field}, and we identify \(\mathcal{J}\) with embeddings \(\mathbb{F}_{q} \hookrightarrow \mathbb{F}\). 

We first define an endomorphism \(\varphi\) of the apartment \(\widetilde{\mathcal{A}}\left( \widehat{T}, \mathbb{F}^\mathcal{J} \llp u \rrp \right)\) by 
\[
\varphi(x)_{j} = o_{j \varphi} + p (x_{j \varphi} - o_{j \varphi}) . 
\]
That is, \(\varphi\) acts by cyclically permutating the factors and scaling by \(p\). 
Note that if \(\lambda \in X_{*}(\widehat{T})^\mathcal{J}\), then 
\[
\varphi( u^\lambda \cdot x)_{j} = \varphi\left(x - \frac{1}{e}\lambda\right)_{j}
= o_{j \varphi} + p(x_{j\varphi} - o_{j \varphi}) - \frac{1}{e} p \lambda_{j \varphi} = \varphi(u^\lambda)_{j} \cdot \varphi(x)_{j} . 
\]
In fact, \(\varphi(n \cdot x) = \varphi(n) \cdot \varphi(x)\) for all \(n \in \widehat{N}(\mathbb{F}^\mathcal{J}\llp u \rrp)\) and \(x \in \widetilde{\mathcal{A}}\left( \widehat{T}, \mathbb{F}^\mathcal{J}\llp u \rrp \right)\). 

Note that \(\varphi\) commutes with the \(\Gamma\)-action, and in particular yields an endomorphism of \(\widetilde{\mathcal{A}}\left( T^*, \mathbb{F}\llp v \rrp \right) = \widetilde{\mathcal{A}}\left( \widehat{T}, \mathbb{F}^\mathcal{J}\llp u \rrp \right)^\Gamma\).
This is precisely the endomorphism described in \cite[Section 6.b.2]{pappasPhiModulesCoefficient2009}. 

Since the diagonal \(\varphi\) respects the equivalence relation used to define the enlarged building, there is an induced endomorphism of the enlarged building. 

\subsection{Galois types}
\label{sec:galois-types}
In this section we will discuss \emph{Galois types}, by which we mean 1-cocycles 
\[
\tau : \Gamma \to \widehat{G}(\mathbb{F}^\mathcal{J}) . 
\]
Here, we view \(\Gamma\) as acting on \(\widehat{G}(\mathbb{F}^\mathcal{J})\) simultaneously via group automorphisms on \(\widehat{G}\) and on the coefficients \(\mathbb{F}^\mathcal{J}\) (as when studying Galois cohomology).
We use the term ``Galois type'' loosely, in that we will refer to closely related 1-cocycles taking values in related groups such as \(\widehat{G}(\mathbb{F}^\mathcal{J}\llb u \rrb)\) by the same name. 

\begin{remark}
    The terminology ``Galois type'' is not used in \cite{local-models}, which works with \emph{inertial types} as classified by \cite[Section 9]{GHS}, 
    but it is actually the Galois type which dictates descent data for Breuil--Kisin modules. 
    A key difference between our notion of Galois types and inertial types is that \(\Gamma\) acts non-trivially on the coefficients (e. g. \(\mathbb{F}^\mathcal{J}\) or \(\mathbb{F}^\mathcal{J}\llb u \rrb\)) in the Galois type situation, whereas in the case of inertial types there is no non-trivial action on the coefficients.
    In \cref{sec:inertial-types} we will elucidate the relation between Galois types and inertial types, and in \cref{sec:weil-restriction-of-GL3} we give an explicit example of how an inertial type gives rise to a Galois type. 
\end{remark}

\subsubsection{Group cohomology}
The following results are well known. 
\begin{lemma}[Shapiro's Lemma]
    \label{res:shapiros-lemma}
    Assume that \(\mathbb{F} \supset \mathbb{F}_{q}\).
    Fix a factor \(\mathbb{F}_{i} \subset \prod_{j \in \mathcal{J}} \mathbb{F}_{j} \cong \mathbb{F}^\mathcal{J}\). 
    Then the projection \(\mathbb{F}^\mathcal{J} \to \mathbb{F}_{i}\) induces bijections 
    \begin{align*}
        \widehat{G} \left( \mathbb{F}^\mathcal{J} \right)^\Gamma & \cong \widehat{G}(\mathbb{F}_{i})^I \\ 
        H^1 \left( \Gamma, \widehat{G} \left( \mathbb{F}^\mathcal{J} \right) \right) & \cong H^1 \left( I, \widehat{G}(\mathbb{F}_{i}) \right) . 
    \end{align*}
\end{lemma}
\begin{proof}
    Recall the \(\Gamma\)-group 
    \[
        \operatorname{Ind}_{I}^\Gamma \widehat{G}(\mathbb{F}) = \left\lbrace f : \Gamma \to \widehat{G}(\mathbb{F}) | f(\delta \theta) = {^\delta f(\theta)} \text{ for all } \delta \in I, \, \theta \in \Gamma \right\rbrace , 
    \]
    with group operation given by pointwise multiplication, and 
    with \(\Gamma\)-action given by right translations, i. e. \((\eta f)(\theta) = f(\theta \eta)\) for \(\eta \in \Gamma\), \(f \in \operatorname{Ind}_{I}^\Gamma \widehat{G}(\mathbb{F})\). 
    By the non-abelian version of Shaprio's Lemma \cite[Section I.5.8b]{galois-cohomology}, 
    it suffices to show that \(\widehat{G}\left( \mathbb{F}^\mathcal{J}\right) \cong \operatorname{Ind}_{I}^{\Gamma} \widehat{G}(\mathbb{F})\)
    as \(\Gamma\)-groups.
    Recall from \cref{ex:large-coefficient-field} that we can identify \(\mathcal{J} = \lbrace \iota, \iota \varphi^{-1}, \dots, \iota \varphi^{-r+1} \rbrace\) for some fixed embedding \(\iota : \mathbb{F}_{q} \hookrightarrow \mathbb{F}\). 
    Let us furthermore identify the embedding \(\iota \varphi^{-j}\) with the integer \(j\), so we can write \(\widehat{G}\left(\mathbb{F}^\mathcal{J}\right) = \prod_{j=0}^{r-1} \widehat{G}(\mathbb{F})\). 
    Consider the map
    \begin{align*}
        \operatorname{Ind}_{I}^\Gamma \widehat{G}(\mathbb{F}) & \to \prod_{j=0}^{r-1} \widehat{G}(\mathbb{F})  \\ 
        f & \mapsto g , \,\, \text{ where } \,\, g_{j} = {^{\sigma^j} f(\sigma^{-j})} , 
    \end{align*}
    where we remind the reader that \(\sigma \in \Gamma\) is the Frobenius element. 
    This map is clearly bijective, since any \(f \in \operatorname{Ind}_{I}^\Gamma \widehat{G}(\mathbb{F})\) is uniquely specified by the values \(f(1), f(\sigma^{-1}), \dots, f(\sigma^{-r+1})\) by the equivariance property.
    It is also easy to see that this map is a group homomorphism. 
    So we need only check that this map is equivariant: 
    \begin{align*}
        & {^{\sigma^j}(\gamma f)(\sigma^{-j})} = {^{\sigma^j}f(\sigma^{-j}\gamma)} = {^{\sigma^j}f(\sigma^{-j} \gamma \sigma^j \sigma^{-j})}
        = {^{\sigma^j \sigma^{-j} \gamma \sigma^j}f(\sigma^{-j})} = {^{\gamma \sigma^{j}} f(\sigma^{-j})} = {^\gamma g_{j}}  \\ 
        & {^{\sigma^j}(\sigma f)(\sigma^{-j})} = {^{\sigma \sigma^{j-1}}f(\sigma^{-j+1})} = {^\sigma(g_{\sigma^{-1}j})} = (^\sigma g)_{j} . 
    \end{align*}
    This completes the proof.
\end{proof}
\begin{remark}
    It is a general fact that \(\operatorname{Ind}_{I}^\Gamma \widehat{G}(\mathbb{F}) \cong \operatorname{Hom}_{\text{Sets}}\left( \Gamma / I, \widehat{G}(\mathbb{F}) \right)\) (see for example \cite[Page 59]{neukirchCohomologyNumberFields2008})
    and the right hand side is easily seen to be isomorphic to \(\prod_{j=0}^{r-1} \widehat{G}(\mathbb{F})\) as a \(\Gamma\)-group. The proof above makes this explicit.
\end{remark}
\begin{lemma}
    \label{res:deforming-galois-types}
    If \(\mathbb{F} \supset \mathbb{F}_{q}\), then the reduction modulo \(u\) map
    \(H^1 \left( \Gamma, \widehat{G}(\mathbb{F}^\mathcal{J}\llb u \rrb) \right) \isoto H^1 \left( \Gamma, \widehat{G}(\mathbb{F}^\mathcal{J}) \right)\)
    is a bijection. 
\end{lemma}
\begin{proof}
    We have a commutative diagram 
    \[
    \begin{tikzcd}
        H^1 \left( \Gamma, \widehat{G}(\mathbb{F}^\mathcal{J}\llb u \rrb) \right) \arrow[r] \arrow[d,"\sim"] & H^1 \left( \Gamma, \widehat{G}(\mathbb{F}^\mathcal{J}) \right) \arrow[d,"\sim"] \\ 
        H^1\left( I , \widehat{G}(\mathbb{F}\llb u \rrb) \right) \arrow[r] & H^1 \left( I, \widehat{G}(\mathbb{F}) \right) ,
    \end{tikzcd}
    \]
    where the horizontal maps are given by reduction modulo \(u\) and the vertical maps are the bijections from Shapiro's \cref{res:shapiros-lemma}.
    So it suffices to show that the bottom horizontal map is a bijection, which will follow from deformation theory and the fact that \(\# I = e\) is prime to \(p\). 
    Note that the reduction modulo \(u\) map \(H^1\left( I, \widehat{G}(\mathbb{F}\llb u \rrb)\right) \to H^1\left( I, \widehat{G}(\mathbb{F})\right)\) is surjective, since it admits a section induced by the inclusion of constants \(\mathbb{F} \subset \mathbb{F}\llb u \rrb\), so we need only check injectivity.
    Since \(H^1\left( I ,\widehat{G}(\mathbb{F}\llb u \rrb)\right) = \clim_{n} H^1\left( I,  \widehat{G}(\mathbb{F}[u]/u^n) \right)\), it suffices to show that \(H^1\left( I,  \widehat{G}(\mathbb{F}[u]/u^{n+1}) \right) \to H^1\left( I, \widehat{G}(\mathbb{F}[u]/u^n)\right)\) is injective for \(n \in \mathbb{Z}_{\geq 0}\). 
    The square zero extension \(J := u^n \mathbb{F}[u] / u^{n+1} \hookrightarrow \mathbb{F}[u]/u^{n+1} \twoheadrightarrow \mathbb{F}[u]/u^n\) yields a short exact sequence of \(I\)-groups \cite[Theorem II.4.3.5]{demazureIntroductionAlgebraicGeometry1980}
    \[
        0 \to \operatorname{Lie}\widehat{G}(J) \hookrightarrow \widehat{G}(\mathbb{F}[u] / u^{n+1}) \twoheadrightarrow \widehat{G}(\mathbb{F}[u]/u^n) \to 1 . 
    \]
    For any \([\rho] \in H^1\left( I, \widehat{G}(\mathbb{F}[u]/u^{n+1})\right)\), let \({_{\rho} \operatorname{Lie}\widehat{G}(J)}\) denote \(\operatorname{Lie}\widehat{G}(J)\) with \(\rho\)-twisted \(I\)-action \(\gamma\cdot v = \rho(\gamma)(^\gamma v )\rho(\gamma)^{-1}\). 
    Then by \cite[Section I.5.5]{galois-cohomology} there is an exact sequence of pointed sets 
    \[ H^1 \left( I, {_{\rho} \operatorname{Lie}\widehat{G}(J)} \right) \to H^1\left( I, \widehat{G}(\mathbb{F}[u]/u^{n+1}) \right) \to H^1\left( I , \widehat{G}(\mathbb{F}[u]/u^n)\right) ,\] 
    where \([\rho]\) is the basepoint of the middle set (exactness then means that any class in the middle having mapping to the same class as \([\rho]\) is in the image of the left map). 
    The point is now that \({_{\rho} \operatorname{Lie}\widehat{G}(J)}\) is an \(\mathbb{F}\)-module, and since \(e = \# I\) is invertible in \(\mathbb{F}\), 
    multiplication by \(e\) acts as an automorphism of \({_{\rho} \operatorname{Lie}\widehat{G}(J)}\). 
    It then follows that \(H^1 \left( I, {_{\rho} \operatorname{Lie}\widehat{G}(J)} \right) = 0\) by \cite[Section I.2.4]{galois-cohomology}, and since \(\rho\) was arbitrary, that \(H^1\left( I,  \widehat{G}(\mathbb{F}[u]/u^{n+1}) \right) \hookrightarrow H^1\left( I, \widehat{G}(\mathbb{F}[u]/u^n)\right)\) is injective. 
\end{proof}

\subsubsection{Building theoretic classification of Galois types}
\label{sec:description-of-types}
Following \cite[Proposition 5.4]{pappas-rapoport-tamely-ramified-bundles}, we give a classification of types in terms of the enlarged Bruhat--Tits building from \cref{sec:the-building}. 

\begin{proposition}
    \label{res:building-classification-of-types}
    There is a bijection 
    \[
    H^1 \left( \Gamma, \widehat{G}(\overline{\mathbb{F}}^\mathcal{J} \llbracket u \rrbracket) \right) \isoto G^*(\overline{\mathbb{F}}\llp v \rrp) \backslash \left( \widehat{G}(\overline{\mathbb{F}}^\mathcal{J} \llp u \rrp) \cdot o \cap \widetilde{\mathcal{B}} \left( G^*, \overline{\mathbb{F}}\llp v \rrp \right) \right) ,
    \]
    where the intersection takes place in \(\widetilde{\mathcal{B}}\left( \widehat{G}, \overline{\mathbb{F}}^\mathcal{J} \llp u \rrp \right)\), using tamely ramified descent for Bruhat--Tits buildings to identify \(\widetilde{\mathcal{B}} \left( G^*,\overline{\mathbb{F}} \llp v \rrp \right) = \widetilde{\mathcal{B}}\left( \widehat{G}, \overline{\mathbb{F}}^\mathcal{J} \llp u \rrp \right)^\Gamma\). 
\end{proposition}
\begin{proof}
    Note that \(H^1 \left( \Gamma, \widehat{G}(\overline{\mathbb{F}}^\mathcal{J}\llp u \rrp) \right) \cong H^1 \left( I, \widehat{G}(\overline{\mathbb{F}}\llp u \rrp) \right)\) by Shapiro's \cref{res:shapiros-lemma}, 
    and \(H^1 \left( I , \widehat{G}(\overline{\mathbb{F}}\llp u \rrp) \right)\) classifies \(G^*\)-torsors over \(\overline{\mathbb{F}}\llp v \rrp\) which trivialize over \(\overline{\mathbb{F}}\llp u \rrp\).
    By Steinberg's theorem \cite[Theorem 2.3.3]{kaletha-prasad} we do in fact have \(H^1 \left( I , \widehat{G}(\overline{\mathbb{F}}\llp u \rrp) \right) = 0\). 
    
    Applying \cite[I.5.4 Corollary 1]{galois-cohomology} (non-abelian version of the long exact sequence in group cohomology induced by a short exact sequence) to the subgroup \(\widehat{G}(\overline{\mathbb{F}}^\mathcal{J}\llb u \rrb) \subset \widehat{G}(\overline{\mathbb{F}}^\mathcal{J}\llp u \rrp)\) yields
    \[
    H^1 \left( \Gamma, \widehat{G}(\overline{\mathbb{F}}^\mathcal{J}\llb u \rrb) \right) \cong G^*(\overline{\mathbb{F}}\llp v \rrp) \backslash \left( \widehat{G} (\overline{\mathbb{F}}^\mathcal{J}\llp u \rrp) /\widehat{G}(\overline{\mathbb{F}}^\mathcal{J}\llb u \rrb) \right)^{\Gamma} .
    \]
    The subgroup \(\widehat{G}(\overline{\mathbb{F}}^\mathcal{J}\llb u \rrb) \subset \widehat{G}(\overline{\mathbb{F}}^\mathcal{J}\llp u \rrp)\) is the stabilizer of \(o\) by \cite[Lemma 4.3.3, Lemma 7.7.10, Lemma 7.7.11, Proposition 8.3.16]{kaletha-prasad},
    and the result follows.
\end{proof}
\begin{remark}
    \label{rem:explicit-type-building-correspondence}
    From the proof and unpacking of \cite[I.5.4 Corollary 1]{galois-cohomology}, the bijection of \cref{res:building-classification-of-types} can be made explicit as follows.  
    Let \(\tau : \Gamma \to \widehat{G}(\overline{\mathbb{F}}^\mathcal{J} \llb  u  \rrb)\) be a 1-cocycle.
    Choose \(b \in \widehat{G}(\overline{\mathbb{F}}^\mathcal{J} \llp u \rrp)\) such that \(\tau(\theta) = b^{-1} (^\theta b)\) for all \(\theta \in \Gamma\) (and such a coboundary always exists).
    Then \(\tau\) is mapped to the \(G^*(\overline{\mathbb{F}}\llp v \rrp)\)-orbit of \(b \cdot o\).
\end{remark}
\begin{remark}
    In practice, we work with \(\tau : \Gamma \to \widehat{G}(\mathbb{F}^\mathcal{J}\llb u \rrb)\), where \(\mathbb{F}\) is finite but large enough to trivialize \(\tau\) in the 
    sense that there exists a coboundary \(b \in \widehat{G}(\mathbb{F}^\mathcal{J}\llp u \rrp)\) such that \(\tau(\theta) = b^{-1}(^\theta b)\) for all \(\theta \in \Gamma\). 
    Via \cref{res:building-classification-of-types}, \(\tau\) can then be described as the \(G^*(\mathbb{F}\llp v \rrp)\)-orbit of a point \(\widetilde{x} \in \widehat{G}(\mathbb{F}^\mathcal{J}\llp u \rrp) \cdot o \cap \mathcal{B}(G^*, \mathbb{F}\llp v \rrp)\). 
\end{remark}
\begin{remark}
    For any cohomology class \([\tau] \in H^1 \left( \Gamma, \widehat{G}(\overline{\mathbb{F}}^\mathcal{J}\llb u \rrb)\right)\)
    and any alcove \(C \subset \widetilde{\mathcal{B}}\left( G^*, \overline{\mathbb{F}}\llp v \rrp \right)\),
    \cref{res:building-classification-of-types} together with the fact that \(G^*(\overline{\mathbb{F}}\llp v \rrp)\) acts transitively on alcoves
    implies that there exists \(g \in \widehat{G}(\overline{\mathbb{F}}^\mathcal{J}\llp u \rrp)\) 
    for which the point \(x = g \cdot o\) lies in the closure of \(C\) and \([\tau]\) is represented by \(\tau(\theta) = g^{-1}(^\theta g)\).
    Moreover, if \(C \subset \widetilde{\mathcal{A}}\left( T^*, \overline{\mathbb{F}} \llp v \rrp \right)\), we will see in \cref{res:types-in-the-normalizer} below 
    that we can take \(g \in \widehat{N}(\overline{\mathbb{F}}^\mathcal{J}\llp u \rrp)\).
\end{remark}
\begin{remark}
    \label{rem:building-classification-is-varphi-equivariant}
    The bijection of \cref{res:building-classification-of-types} is \(\varphi\)-equivariant: 
    If \(g \cdot o\) corresponds to \(\tau(\theta) = g^{-1}(^\theta g)\), 
    then \(\varphi(g \cdot o) = \varphi(g) \cdot o\) corresponds to \(\varphi \tau(\theta) = \varphi(g^{-1}(^\theta g)) = \varphi(g)^{-1} (^\theta\varphi(g))\).
\end{remark}
As a part of the above \cref{res:building-classification-of-types}, we have that any Galois type \([\tau] \in H^1\left( \Gamma, \widehat{G}(\overline{\mathbb{F}}^\mathcal{J}\llb u \rrb) \right)\) is represented by 
some \(\tau\) of the form \(\tau(\theta) = g^{-1}(^\theta g)\), where \(g \in \widehat{G}(\overline{\mathbb{F}}^\mathcal{J}\llp u \rrp)\).
The following result shows that we may even take \(g \in \widehat{N}(\overline{\mathbb{F}}^\mathcal{J}\llp u \rrp)\). 
\begin{lemma}
    \label{res:N-orbits-vs-G-orbits}
    The natural map 
    \[
        N^*(\overline{\mathbb{F}}\llp v \rrp) \backslash \left( \widehat{N}(\overline{\mathbb{F}}^\mathcal{J}\llp u \rrp) \cdot o \cap \widetilde{\mathcal{A}}(T^*,\overline{\mathbb{F}}\llp v \rrp) \right)
        \isoto 
        G^*(\overline{\mathbb{F}}\llp v \rrp) \backslash \left( \widehat{G}(\overline{\mathbb{F}}^\mathcal{J} \llp u \rrp) \cdot o \cap \widetilde{\mathcal{B}} \left( G^*, \overline{\mathbb{F}}\llp v \rrp \right) \right)
    \]
    is bijective.
\end{lemma}
\begin{proof}
    We start by proving injectivity.
    This follows from the fact \cite[Corollary 4.2.25]{kaletha-prasad} that if \(x, x' \in \widetilde{\mathcal{A}}\left( T^*, \mathbb{F}\llp v \rrp \right)\) and there exists \(g \in G^*(\overline{\mathbb{F}}\llp v \rrp)\) such that \(g \cdot x' = x\), then there exists \(n \in N^*(\overline{\mathbb{F}}\llp v \rrp)\) such that \(n \cdot x' = x\).

    It remains to show surjectivity. 
    Let \(x = b \cdot o\) be fixed by \(\Gamma\), where \(b \in \widehat{G}(\overline{\mathbb{F}}^\mathcal{J}\llp u \rrp)\).
    Since \(G^*(\overline{\mathbb{F}}\llp v \rrp)\) acts transitively on the alcoves of \(\widetilde{\mathcal{B}}(G^*, \overline{\mathbb{F}} \llp v \rrp)\), 
    there is some \(g \in G^*(\overline{\mathbb{F}}\llp v \rrp)\) such that \(g \cdot x\) is in the closure of any given alcove, and in particular we can ensure that \(x' = g\cdot x = gb \cdot o \in \widetilde{\mathcal{A}}(T^*, \overline{\mathbb{F}}\llp v \rrp) = \widetilde{\mathcal{A}}(\widehat{T},\overline{\mathbb{F}}^\mathcal{J}\llp u \rrp)^\Gamma\).
    By \cite[Corollary 4.2.25]{kaletha-prasad} again there exists \(n \in N^*(\overline{\mathbb{F}}\llp v \rrp)\) such that \(n \cdot o = gb \cdot o = x'\). 
\end{proof}
\begin{proposition}
    \label{res:types-in-the-normalizer}
    Consider the commutative diagram 
    \[
    \begin{tikzcd}
        N^*(\overline{\mathbb{F}}\llp v \rrp) \backslash \left( \widehat{N}(\overline{\mathbb{F}}^\mathcal{J}\llp u \rrp) \cdot o \cap \widetilde{\mathcal{A}}(T^*,\overline{\mathbb{F}}\llp v \rrp) \right) \arrow[r,"\sim"] \arrow[d,hookrightarrow] & 
        G^*(\overline{\mathbb{F}}\llp v \rrp) \backslash \left( \widehat{G}(\overline{\mathbb{F}}^\mathcal{J} \llp u \rrp) \cdot o \cap \widetilde{\mathcal{B}} \left( G^*, \overline{\mathbb{F}}\llp v \rrp \right) \right) \arrow[d,"\sim"] \\
        H^1 \left( \Gamma, \widehat{N}(\overline{\mathbb{F}}^\mathcal{J} \llb u \rrb) \right) \arrow[r,twoheadrightarrow] & H^1 \left( \Gamma, \widehat{G}(\overline{\mathbb{F}}^\mathcal{J}\llb u \rrb)\right) ,
    \end{tikzcd}
    \]
    where the vertical maps are given by sending the class of \(g \cdot o\) to \(\tau(\theta) = g^{-1}(^\theta g)\).
    Then the arrows marked \(\sim\) are bijective, the left vertical arrow is injective and the bottom horizontal arrow is surjective. 
\end{proposition}
\begin{proof}
    Immediate from \cref{res:building-classification-of-types} and \cref{res:N-orbits-vs-G-orbits}. 
\end{proof}
The above result implies that any Galois type \([\tau] \in H^1 \left( \Gamma, \widehat{G}(\overline{\mathbb{F}}^\mathcal{J}\llb u \rrb) \right)\) is represented by a \(1\)-cocycle \(\tau\) of the form \(\tau(\theta) = n^{-1}(^\theta n)\), where \(n \in \widehat{N}(\overline{\mathbb{F}}^\mathcal{J}\llp u \rrp)\). 
It will be convenient to have an even more control over what \(n\) can be. 
\begin{proposition}
    \label{res:types-in-smaller-normalizer}
    Let \(\dot{N} \subset \widehat{N}(\overline{\mathbb{F}}^\mathcal{J}\llp u \rrp)\) be a subgroup which is stable under the action of \(\Gamma\), and let \(\dot{N}^0 := \dot{N} \cap \widehat{N}(\overline{\mathbb{F}}^\mathcal{J}\llb u \rrb)\). 
    Consider the map 
    \begin{equation}
        \label{eq:smaller-N}
        \dot{N}^\Gamma \backslash \left( \dot{N} \cdot o \cap \widetilde{\mathcal{A}}(T^*,\overline{\mathbb{F}}\llp v \rrp) \right) 
        \to 
        N^*(\overline{\mathbb{F}}\llp v \rrp) \backslash \left( \widehat{N}(\overline{\mathbb{F}}^\mathcal{J}\llp u \rrp) \cdot o \cap \widetilde{\mathcal{A}}(T^*,\overline{\mathbb{F}}\llp v \rrp) \right) . 
    \end{equation}
    Then we have the following:
    \begin{enumerate}[(A)]
        \item If \(\dot{N}\) contains the image of the map \(X_{*}(\widehat{T})^I \cong \left( X_{*}(\widehat{T})^\mathcal{J} \right)^\Gamma \subset X_{*}(\widehat{T})^\mathcal{J} \to \widehat{T}(\overline{\mathbb{F}}^\mathcal{J}\llp u \rrp)\)\footnote{The identification \(X_{*}(\widehat{T})^I \cong \left( X_{*}(\widehat{T})^\mathcal{J} \right)^\Gamma\) depends on a choice of \(j \in \mathcal{J}\), just like in \cref{res:shapiros-lemma}.} given by \(\lambda \mapsto u^\lambda\) (where the action of \(\Gamma\) on \(X_{*}(\widehat{T})^\mathcal{J}\) given by \eqref{eq:conjugation-action-on-cocharacters}), then \eqref{eq:smaller-N} is surjective. 
        \item If the map \(\dot{N}^\Gamma \subset \widehat{N}(\overline{\mathbb{F}}^\mathcal{J}\llp u \rrp)^\Gamma = N^*(\overline{\mathbb{F}}\llp v \rrp) \to \widetilde{W}^*\) is surjective (where \(\widetilde{W}^*\) denotes the Iwahori--Weyl group for \(G^*(\overline{\mathbb{F}}\llp v \rrp)\) from \cref{sec:iwahori-weyl}), then \eqref{eq:smaller-N} is injective. 
    \end{enumerate}
    Consequently, if the assumptions in both (A) and (B) are satisfied, so that \eqref{eq:smaller-N} is a bijection, 
    then the map \(H^1\left( \Gamma, \dot{N}^0 \right) \twoheadrightarrow H^1 \left( \Gamma, \widehat{G}(\overline{\mathbb{F}}^\mathcal{J}\llb u \rrb) \right)\) is surjective and admits a canonical section. 
    In concrete terms, given any \(1\)-cocycle \(\tau : \Gamma \to \widehat{G}(\overline{\mathbb{F}}^\mathcal{J}\llb u \rrb)\), there exists \(n \in \dot{N}\) for which \(\tau\) is coboundarous to the \(1\)-cocycle \(\tau' : \Gamma \to \dot{N}^0\) given by \(\tau'(\theta) = n^{-1}(^\theta n)\). 
    Moreover, if \(n_{1}, n_{2} \in \dot{N}\) and \(\tau_{1}(\theta) = n_{1}^{-1}(^\theta n_{1})\), \(\tau_{2}(\theta) = n_{2}^{-1}(^\theta n_{2})\), then \(\tau_{1}\) and \(\tau_{2}\) are coboundarous by an element of \(\widehat{G}(\overline{\mathbb{F}}^\mathcal{J}\llb u \rrb)\) if and only if \(\tau_{1}\) and \(\tau_{2}\) are coboundarous by an element of \(\dot{N}^0\). 
\end{proposition}
\begin{proof}
    We first show (A). 
    So let \(n \in \widehat{N}(\overline{\mathbb{F}}^\mathcal{J}\llp u \rrp)\) and assume that the point \(x = n \cdot o\) is fixed by \(\Gamma\). 
    By \cref{res:iwahori-weyl-group-as-translations-rx-W} there exists \(\lambda \in X_{*}(\widehat{T})^\mathcal{J}\) and \(w \in W^\mathcal{J}\) such that the image of \(n\) in the Iwahori--Weyl group \(\widetilde{W}^\mathcal{J} = X_{*}(\widehat{T})^\mathcal{J} \rtimes W^\mathcal{J}\) is \(u^\lambda w\). 
    Since the action of \(\widehat{N}(\overline{\mathbb{F}}^\mathcal{J}\llp u \rrp)\) on \(\widetilde{\mathcal{A}}\left( \widehat{T}, \overline{\mathbb{F}}^\mathcal{J}\llp u \rrp \right)\) factors through \(\widetilde{W}^\mathcal{J}\), 
    we see that
    \[ 
        x = n \cdot o = u^\lambda w \cdot o = u^\lambda \cdot o = o - \lambda ,
    \]
    using \(w \cdot o = o\). 
    Now, recall from \cref{sec:tamely-ramified-descent} that action map \(X_{*}(\widehat{T})^\mathcal{J} \to \widetilde{\mathcal{A}}\left( \widehat{T}, \overline{\mathbb{F}}^\mathcal{J}\llp u \rrp \right)\) given by \(\nu \mapsto u^\nu \cdot o = o - \nu\)
    is injective and \(\Gamma\)-equivariant with respect to the action of \(\Gamma\) on \(X_{*}(\widehat{T})^\mathcal{J}\) given by \eqref{eq:conjugation-action-on-cocharacters}. 
    The condition that \(x\) is fixed by \(\Gamma\) is therefore equivalent to the condition that \(\lambda\) is fixed by \(\Gamma\).
    In sum, we have shown that \(x = u^\lambda \cdot o\) for some \(\lambda \in X_{*}(\widehat{T})^I\), which under the assumption of (A) implies that \eqref{eq:smaller-N} is surjective. 

    We now turn to (B). 
    That is, suppose \(n_{1} ,n_{2} \in \dot{N}\) and there is \(g \in N^*(\overline{\mathbb{F}}\llp v \rrp)\) such that \(n_{1} \cdot o = g n_{2} \cdot o\). 
    The assumption of (B) implies that we can find \(h \in \dot{N}^\Gamma\) such that the images of \(g,h\) in \(\widetilde{W}^*\) concide. 
    Since the action of \(N^*(\overline{\mathbb{F}}\llp v \rrp)\) factors through \(\widetilde{W}^*\), it follows that \(n_{1} \cdot o = h n_{2} \cdot o\). 
    This shows that \eqref{eq:smaller-N} is injective.

    The additional statements are immediate from the fact that we now have a diagram as in \cref{res:types-in-the-normalizer} with the respective roles of \(\widehat{N}(\overline{\mathbb{F}}^\mathcal{J}\llp u \rrp)\), \(\widehat{N}(\overline{\mathbb{F}}^\mathcal{J}\llb u \rrb)\), and \(N^*(\overline{\mathbb{F}}\llp v \rrp)\) being replaced with \(\dot{N}\), \(\dot{N}^0\), and \(\dot{N}^\Gamma\).
\end{proof}
\begin{remark}
    \label{rem:types-in-normalizer-simple}
    The above \cref{res:types-in-smaller-normalizer} shows that all Galois types are represented by \(1\)-cocycles of the form \(\tau(\theta) = u^{-\lambda}(^\theta u^\lambda)\) for some \(\lambda \in \left(X_{*}(\widehat{T})^\mathcal{J}\right)^\Gamma\).
    In particular, \(\tau(\sigma) = 1\) and \(\tau(\gamma) = \omega(\gamma)^{\lambda}\), where \(\omega^{\lambda}\) denotes the composition 
    \[
        \begin{tikzcd}
            I \arrow[r,"\omega"] & \mathbb{F}_{q}^\times \arrow[r,"a \mapsto 1 \otimes a",hookrightarrow] & \left( \mathbb{F} \otimes_{\mathbb{F}_{p}} \mathbb{F}_{q} \right)^\times 
            \arrow[r,equals] & \prod_{j \in \mathcal{J}} \mathbb{F}_{j} = \mathbb{F}^\mathcal{J} \arrow[r,"\lambda"] & 
            \widehat{T}(\mathbb{F}^\mathcal{J}) . 
        \end{tikzcd}
    \]
    If \(j \in \mathcal{J}\) is identified with an embedding \(\iota_{j} : \mathbb{F}_{q}\hookrightarrow \overline{\mathbb{F}}\), then
    the projection of \(\omega^{\lambda}\) onto the \(j\)'th factor \(\widehat{G}(\overline{\mathbb{F}}_{j}) = \widehat{G}(\overline{\mathbb{F}})\) is \(\iota_{j}(\omega)^{\lambda_{j}}\) (because \({^\gamma u} = \iota_{j}(\omega(\gamma))u\) in \(\overline{\mathbb{F}}_{j}\llp u \rrp\)). 
    For generic \(\lambda\), this depends on the value \(j\). 
    As we will see in \cref{res:strictly-frobenius-invariant-coboundary-in-normalizer}, we can often make the value of the projection of \(\tau(\gamma)\) onto the \(j\)'th factor independent of \(j\) at the cost of changing \(\tau(\sigma)\) into a non-trivial element. 
\end{remark}
\begin{example}
    What is a minimal \(\dot{N} \subset \widehat{N}(\overline{\mathbb{F}}^\mathcal{J}\llp u \rrp)\) satisfying the conditions (A) and (B) in \cref{res:types-in-smaller-normalizer}?
    An obvious choice is to let \(\dot{N}\) be the subgroup generated by the image of \(\left(X_{*}(\widehat{T})^\mathcal{J}\right)^\Gamma\) and \(\widehat{N}(\mathbb{Z}^\mathcal{J})\), because \(\widehat{N}(\mathbb{Z})\) maps surjectively onto \(W^\mathcal{J}\) (the finite Weyl group) as we saw in the proof of \cref{res:iwahori-weyl-group-as-translations-rx-W}. 
    But more optimally, we can replace \(\widehat{N}(\mathbb{Z}^\mathcal{J})\) by another minimal subgroup which maps surjectively onto \(W^\mathcal{J}\).
    For example, if \(\widehat{G} = \GL_{n}\) or \(\operatorname{PGL}_{n}\), then we can embed \(W^\mathcal{J} \subset \widehat{N}(\mathbb{Z}^\mathcal{J})\), and consequently also \(\widetilde{W}^\mathcal{J} \subset \widehat{N}(\overline{\mathbb{F}}^\mathcal{J}\llp u \rrp)\), but this is not possible if \(\widehat{G} = \operatorname{SL}_{n}\). 
\end{example}

We are not going to use the following result, but include it here for completeness. It gives an upper bound on the number of Galois types. 
\begin{proposition}
    \label{res:types-via-tate-cohomology}
    The map \(X_{*}(\widehat{T})^I \subset X_{*}(\widehat{T})^\mathcal{J} \to H^1 \left( \Gamma, \widehat{G}(\overline{\mathbb{F}}^\mathcal{J}\llb u \rrb ) \right)\) given by sending \(\lambda\) to the class represented by the \(1\)-cocycle \(\tau(\theta) = u^{-\lambda}(^\theta u^\lambda)\) is surjective. 
    Moreover, this map factors through the Tate cohomology group \(\widehat{H}^0\left( I, X_{*}(\widehat{T}) \right) = N \left( X_{*}(\widehat{T})_{I} \right)\backslash X_{*}(\widehat{T})^I\), where \(N : X_{*}(\widehat{T})_{I} \to X_{*}(\widehat{T})^I\) is the norm map. 
    In particular, if \(I\) acts trivially on \(\widehat{T}\), we have a surjection \(e X_{*}(\widehat{T}) \backslash X_{*}(\widehat{T}) \twoheadrightarrow H^1 \left( \Gamma, \widehat{G}(\overline{\mathbb{F}}^\mathcal{J}\llb u \rrb ) \right)\).
\end{proposition}
\begin{proof}
    The first statement (about surjectivity) is immediate from \cref{res:types-in-smaller-normalizer}.
    Using Shapiro's \cref{res:shapiros-lemma} we make an identification \(H^1\left( I, \widehat{T}(\overline{\mathbb{F}}\llb u \rrb) \right) \cong H^1 \left( \Gamma, \widehat{T}(\overline{\mathbb{F}}^\mathcal{J}\llb u \rrb) \right)\) which depends on a choice of \(j \in \mathcal{J}\). 
    We then see that the given map factors through the map \(\delta : X_{*}(\widehat{T})^I \to H^1\left( I, \widehat{T}(\overline{\mathbb{F}}\llb u \rrb) \right)\), which maps \(\lambda \in X_{*}(\widehat{T})^I\) to the class \(\delta(\lambda)\) represented by the \(1\)-cocycle \(\tau(\gamma) = u^{-\lambda}(^\gamma u^\lambda)\).
    Therefore, we need only show that \(\delta\) factors through \(\widehat{H}^0\left( I, X_{*}(\widehat{T}) \right)\). 
    Recall that the discrete valuation of \(\overline{\mathbb{F}}\llp u \rrp\) yields a short exact sequence 
    \[ 
        0 \to \widehat{T}(\overline{\mathbb{F}}\llb u \rrb) \to \widehat{T}(\overline{\mathbb{F}}\llp u \rrp) \stackrel{e \cdot \operatorname{val}}{\longrightarrow} X_{*}(\widehat{T}) \to 0 , 
    \]
    where we need to scale the valuation by \(e\) because \(\operatorname{val}(u) = 1/e\).
    The long exact sequence in \(I\)-group cohomology then gives the exact sequence
    \[
        \widehat{T}(\overline{\mathbb{F}}\llp u \rrp)^I = T^*(\overline{\mathbb{F}}\llp v \rrp) \stackrel{e \cdot \operatorname{val}}{\longrightarrow} X_{*}(\widehat{T})^I \stackrel{\delta}{\longrightarrow} H^1\left( I, \widehat{T}(\overline{\mathbb{F}}\llb u \rrb) \right) \to H^1 \left( I, \widehat{T}(\overline{\mathbb{F}}\llp u \rrp) \right) = 0 , 
    \]
    where the last equality is due to Steinberg's Theorem (or just Hilbert 90 if \(I\) acts trivially on \(\widehat{T}\)). 
    It now follows from \cite[Corollary 11.7.3]{kaletha-prasad} that \(e \cdot \operatorname{val}(T^*(\overline{\mathbb{F}}\llp v \rrp)) \subset X_{*}(\widehat{T})^I\) coincides with the image of the norm map \(N : X_{*}(\widehat{T})_{I} \to X_{*}(\widehat{T})^I\), 
    so indeed \(\delta\) yields an isomorphism 
    \[
        N(X_{*}(\widehat{T})_{I}) \backslash X_*(\widehat{T})^I = \widehat{H}^0\left( I , X_{*}(\widehat{T}) \right) \isoto H^1\left( I, \widehat{T}(\overline{\mathbb{F}}\llb u \rrb) \right) . 
    \]
    Finally, note that if \(I\) acts trivially on \(\widehat{T}\), then \(X_{*}(\widehat{T})^I = X_{*}(\widehat{T})\) and \(N(X_{*}(\widehat{T})) = e X_{*}(\widehat{T})\) since \(e = \# I\). 
\end{proof}

\subsubsection{Frobenius invariant types}
\label{sec:frobenius-invariant-types}
Galois types arising from Breuil--Kisin modules with descent data are not completely arbitrary,
since they have the property of being ``Frobenius invariant'' in the following sense. 
\begin{definition}
    We say that a cohomology class \([ \tau ] \in H^1 \left(\Gamma, \widehat{G}(\mathbb{F}^\mathcal{J}\llb u \rrb) \right)\) is 
    \textit{Frobenius invariant} if \(\varphi([\tau]) = [\varphi\tau] = [\tau]\). We say that a \(1\)-cocycle \(\tau \in Z^1 \left(\Gamma, \widehat{G}(\mathbb{F}^\mathcal{J}\llb u \rrb) \right)\)
    is \emph{strictly Frobenius invariant} if \(\varphi \tau = \tau\) on the level of \(1\)-cocycles. 
    (This definition also works with other rings such as \(\mathcal{O}\) or \(\overline{\mathbb{F}}\) in place of \(\mathbb{F}\).)
\end{definition}

The following result shows that any type arising from Breuil--Kisin modules with descent data indeed is Frobenius invariant.
The classical case of this result is discussed in \cite[Remark 3.3.3]{CEGS}. 
\begin{proposition}
    \label{res:galois-types-are-frobenius-invariant}
    Assume that \(I\) acts trivially on \(\widehat{G}\). 
    Then the map 
    \[
        H^1 \left( \Gamma, \widehat{G}(\mathcal{O}^{\text{ur},\mathcal{J}} \llb u \rrb) \right) \hookrightarrow H^1 \left( \Gamma, \widehat{G}(\mathcal{O}^{\text{ur},\mathcal{J}} \llb u \rrb[(u^e+p)^{-1}]) \right)
    \]
    is injective, where \(\mathcal{O}^{\text{ur}} \subset E^{\text{ur}}\) is the ring of integers in the maximal unramified extension of \(E\). 
\end{proposition}
\begin{proof}
    By Shapiro's \cref{res:shapiros-lemma} we may replace \(\Gamma\) by \(I\) and \(\mathcal{O}^{\text{ur},\mathcal{J}}\) by \(\mathcal{O}^{\text{ur}}\). 
    Consider the composition 
    \[
    H^1 \left( I, \widehat{G}(\mathcal{O}^{\text{ur}}) \right) \isoto H^1 \left( I, \widehat{G}(\mathcal{O}^{\text{ur}} \llb u \rrb) \right) \to H^1 \left( I, \widehat{G}(\mathcal{O}^{\text{ur}} \llb u \rrb[(u^e+p)^{-1}]) \right)
    \stackrel{\mod u}{\longrightarrow} H^1 \left( I, \widehat{G}(\overline{E}) \right) . 
    \]
    The first isomorphism is a consequence of deformation theory and the fact that \(e = \# I\) is prime to \(p\) (more precisely, by the argument in the proof of \cref{res:deforming-galois-types}). 
    The map \(H^1 \left( I, \widehat{G}(\mathcal{O}^{\text{ur}} \llb u \rrb) \right) \to H^1 \left( I, \widehat{G}(\mathcal{O}^{\text{ur}} \llb u \rrb[(u^e+p)^{-1}]) \right)\) is injective provided that the composition is injective, which is what we will show. 

    Let \(\rho_{1}, \rho_{2} : I \to \widehat{G}(\mathcal{O}^{\text{ur}})\) be two 1-cocycles for which the composites \(\rho_{1},\rho_{2} : I \to \widehat{G}(\overline{E})\) are coboundarous. 
    Since \(I\) acts trivially on \(\widehat{G}(\mathcal{O}^{\text{ur}})\), \(\rho_{1},\rho_{2}\) are just homomorphisms, 
    and being coboundarous means being conjugate. 
    The question is then whether the two elements \(\rho_{1}(\gamma), \rho_{2}(\gamma) \in \widehat{G}(\mathcal{O}^{\text{ur}})\), which are conjugate by an element of \(\widehat{G}(\overline{E})\), must be conjugate by an element of \(\widehat{G}(\mathcal{O}^{\text{ur}})\). 
    
    This is where we invoke Lafforgue's theory of pseudo-characters.
    As explained in \cite[Section 3]{thorne-potential-automorphy}, one can construct certain functions \(\operatorname{tr}\rho_{1}\) and \(\operatorname{tr}\rho_{2}\) (analogous to the trace), 
    and the fact that \(\rho_{1}\) and \(\rho_{2}\) are conjugate by an element of \(\widehat{G}(\overline{E})\) implies that \(\operatorname{tr}\rho_{1} = \operatorname{tr}\rho_{2}\) (where we use that \(\mathcal{O}^\text{ur} \subset \overline{E}\)). 
    On the other hand, we will show in a moment that reduction modulo \(\varpi\) yields a bijection 
    \begin{equation}
        \label{eq:deformation-for-lafforgue}
        H^1\left( I, \widehat{G}(\mathcal{O}^{\text{ur}})\right)\cong H^1\left(I, \widehat{G}(\overline{\mathbb{F}}) \right) . 
    \end{equation}
    If we let \(\overline{\rho}_{1}, \overline{\rho}_{2} : I \to \widehat{G}(\overline{\mathbb{F}})\) denote the reductions modulo \(\varpi\) of \(\rho_{1}, \rho_{2}\), 
    then the equality \(\operatorname{tr}\rho_{1} = \operatorname{tr}\rho_{2}\) implies that \(\operatorname{tr}\overline{\rho}_{1} = \operatorname{tr}\overline{\rho}_{2}\)
    by universality of the construction of these pseudo-characters. 
    By \cite[Theorem 3.3]{thorne-potential-automorphy} applied to homomorphisms \(I \to \widehat{G}(\overline{\mathbb{F}})\), it then follows that \(\overline{\rho}_{1}\) and \(\overline{\rho}_{2}\) are coboundarous. 
    By \eqref{eq:deformation-for-lafforgue}, so are \(\rho_{1}\) and \(\rho_{2}\). 

    It remains to show the bijection \eqref{eq:deformation-for-lafforgue}.
    This is again a consequence of deformation theory and the fact that \(e = \#I\) is prime to \(p\). 
    However, since \(\mathcal{O}^{\text{ur}}\) is not complete, the argument from the proof of \cref{res:deforming-galois-types} is not directly applicable. 
    But for each finite extension \(\mathbb{F}' \supset \mathbb{F}\) and cooresponding extension \(\mathcal{O}' \supset \mathcal{O}\), we have a bijection \(H^1\left( I, \widehat{G}(\mathcal{O}')\right) \cong H^1\left(I, \widehat{G}(\mathbb{F}')\right)\), 
    and by \cite[Section I.2.2]{galois-cohomology} we then obtain \(H^1\left( I, \widehat{G}(\mathcal{O}^{\text{ur}})\right) \cong \colim H^1\left( I, \widehat{G}(\mathcal{O}')\right) \cong \colim H^1\left(I, \widehat{G}(\mathbb{F}')\right) \cong H^1\left(I, \widehat{G}(\overline{\mathbb{F}}) \right)\).
\end{proof}
\begin{remark}
    Whereas the above shows that \(H^1 \left( I, \widehat{G}(\mathcal{O}^{\text{ur}} \llb u \rrb) \right) \hookrightarrow H^1 \left( I, \widehat{G}(\mathcal{O}^{\text{ur}} \llb u \rrb[(u^e+p)^{-1}]) \right)\)
    is injective, it is not true that \(H^1 \left( I, \widehat{G}(\overline{\mathbb{F}} \llb u \rrb) \right) \to H^1 \left( I, \widehat{G}(\overline{\mathbb{F}} \llb u \rrb[(u^e+p)^{-1}]) \right)\) is injective.
    The reasons are that (1) \(\overline{\mathbb{F}} \llb u \rrb[(u^e+p)^{-1}] = \overline{\mathbb{F}}\llp u \rrp\), so there is no ``reduction modulo \(u\)'' map, (2) \cite[Theorem 3.3]{thorne-potential-automorphy} doesn't apply to \(1\)-cocycles \(I \to \widehat{G}(\overline{\mathbb{F}}\llp u \rrp)\) because \(I\) acts non-trivially on \(\overline{\mathbb{F}}\llp u \rrp\) and (3) because \(\overline{\mathbb{F}}\) is not algebraically closed. 
    In fact, \(H^1\left( I, \widehat{G}(\overline{\mathbb{F}}\llp u \rrp)\right) = 0\) by Steinberg's Theorem \cite[Theorem 2.3.3]{kaletha-prasad}, 
    which implies that it's possible for Breuil--Kisin \((\Gamma,\widehat{G})\)-torsors over \(\overline{\mathbb{F}}\) to have any ``mixed'' type in the sense of \cite{CEGS}. 
    An explicit example illustrating this fact can be found in \cite[Remark 3.7]{caraiani-levin}.
\end{remark}

We can interpret Frobenius invariance of Galois types in terms of our classification given by \cref{res:building-classification-of-types}. 
\begin{lemma}
    \label{res:frobenius-invariant-type-as-x}
    Let \([\tau] \in H^1\left( \Gamma, \widehat{G}(\overline{\mathbb{F}}^\mathcal{J}\llb u \rrb) \right)\) be a Galois type which via the bijection of 
    \cref{res:building-classification-of-types} is represented by a point \(x = g \cdot o\), where \(g \in \widehat{G}(\overline{\mathbb{F}}^\mathcal{J}\llp u \rrp)\). 
    \begin{enumerate}
        \item Then \([\tau]\) is Frobenius invariant if and only if there exists \(c \in G^*(\overline{\mathbb{F}}\llp v \rrp)\) such that \(c \cdot \varphi(x) = x\).
        \item After passing to a finite unramified extension, there exists a representative \(\tau\) for the class \([\tau]\) which is \emph{strictly} Frobenius invariant, i. e. satisfies \(\varphi \tau = \tau\) on the level of \(1\)-cocycles. 
            More precisely, for a finite unramified extension \(L' \supset L\), let \(\Gamma' = \operatorname{Gal}(L' / \mathbb{Q}_{p})\), \(q'\) be the order of the residue field of \(L'\), and let \([\tau]' \in H^1\left( \Gamma', \widehat{G}(\left( \overline{\mathbb{F}} \otimes_{\mathbb{F}_{p}} \mathbb{F}_{q'}\right)\llb u \rrb)\right)\) denote the image of \([\tau]\) (induced by the map on \(1\)-cocycles which maps a representative \(1\)-cocycle \(\tau : \Gamma \to \widehat{G}((\overline{\mathbb{F}} \otimes_{\mathbb{F}_{p}} \mathbb{F}_{q})\llb u \rrb)\) to the composite \(\Gamma' \twoheadrightarrow \Gamma \stackrel{\tau}{\longrightarrow} \widehat{G}((\overline{\mathbb{F}} \otimes_{\mathbb{F}_{p}} \mathbb{F}_{q})\llb u \rrb) \subset \widehat{G}((\overline{\mathbb{F}} \otimes_{\mathbb{F}_{p}} \mathbb{F}_{q'})\llb u \rrb)\)).
            Then there exists such an extension \(L' \supset L\) and a representative \(1\)-cocycle \(\tau' : \Gamma' \to \widehat{G}(\overline{\mathbb{F}}^\mathcal{J}\llb u \rrb)\) for the class \([\tau]'\) which is strictly Frobenius invariant. 
    \end{enumerate}
\end{lemma}
\begin{proof}
    (1) is obvious from \cref{rem:building-classification-is-varphi-equivariant}, so let us show (2).
    Let \(\tilde{\tau}\) be some representative for the class \([\tau]\).
    Note that by \cref{res:deforming-galois-types} and \cref{rem:types-in-normalizer-simple} we may assume that \(\tilde{\tau}\) takes values in \(\widehat{G}(\mathbb{F}^\mathcal{J})\) as long as \(\mathbb{F} \supset \mathbb{F}_{q}\).
    The assumption that \(\tilde{\tau}\) is Frobenius invariant means that there exists a coboundary \(b \in \widehat{G}(\mathbb{F}^\mathcal{J}\llb u \rrb)\) such that \(b \varphi(\tilde{\tau}(\theta))(^\theta b)^{-1} = \tilde{\tau}(\theta)\) for all \(\theta \in \Gamma\).
    Note that if \(H = \operatorname{Res}^{\mathbb{F}}_{\mathbb{F}_{p}} \widehat{G}\), then we can identify \(\varphi : \widehat{G}(\mathbb{F}^\mathcal{J}) \to \widehat{G}(\mathbb{F}^\mathcal{J})\) with the relative Frobenius of \(H\) applied to \(\mathbb{F}_{q}\)-points.
    Hence the map \(\widehat{G}(\mathbb{F}^\mathcal{J}) \to \widehat{G}(\mathbb{F}^\mathcal{J})\) given by \(c \mapsto c^{-1} \varphi(c)\) can be identified with the Lang isogeny of \(H\), which must be surjective by Lang's theorem \cite[Theorem 17.96]{milneAlgebraicGroupsTheory2017}. 
    Therefore, there exists some finite extension \(\mathbb{F}_{q'} \supset \mathbb{F}_{q}\) and \(c \in \widehat{G}(\mathbb{F} \otimes_{\mathbb{F}_{p}} \mathbb{F}_{q'})\) such that \(b = c^{-1}\varphi(c)\). 
    Our strictly Frobenius invariant \(1\)-cocycle is then \(\tau'(\theta) := c \tilde{\tau}(\theta)(^\theta c)^{-1}\) for \(\theta \in \Gamma\). 
    Indeed, we now have 
    \[
        \varphi \tau'(\theta) = \varphi(c) \varphi \tilde{\tau}(\theta) (^\theta \varphi(c))^{-1} = c b \varphi \tilde{\tau}(\theta)(^\theta b)^{-1}(^\theta c)^{-1} = c \tilde{\tau}(\theta)(^\theta c)^{-1} = \tau'(\theta) 
    \]
    for all \(\theta \in \Gamma\). 
\end{proof}
\begin{remark}
    \label{rem:strictly-frobenius-invariant-explicit}
    The proof above relied on Lang's theorem to find a solution \(c\) of the equation \(b = c^{-1} \varphi(c)\).
    We now give an explicit alternative method. 
    Namely, choose \(\iota_{0} \in \mathcal{J}\), in which case we can identify \(\mathcal{J} = \lbrace \iota_{0}, \iota_{0} \varphi^{-1}, \dots, \iota_{0}\varphi^{-r+1} \rbrace\) as in \cref{ex:large-coefficient-field}. 
    Given an element \(g \in \widehat{G}(\mathbb{F}^\mathcal{J})\) and \(j \in \lbrace 0, 1, \dots, r-1 \rbrace\) let \(g_{j} \in \widehat{G}(\mathbb{F}_{\iota_{0}\varphi^{-j}})\) denote its projection onto the factor corresponding to \(\iota_{0}\varphi^{-j}\).
    We then have \(\varphi(g)_{j} = g_{j-1}\). 
    We find a solution by choosing \(c_{0}\) arbitrarily, and recursively defining \(c_{j} = b_{j}^{-1}c_{j-1}\) for \(j =1,2,\dots,r-1\), i. e. \(c_{j} = b_{j}^{-1} b_{j-1}^{-1} \cdots b_{1}^{-1}c_{0}\). 
    Then \(c\) solves \(b = c^{-1}\varphi(c)\) if and only if \(c_{0} = b_{0}^{-1}c_{r-1}\), i. e. if and only if \(b_{0} b_{1} \cdots b_{r-1} = 1\). 
    Now the point is that even if \(\tilde{b} := b_{0} b_{1} \cdots b_{r-1} \neq 1\), this is an element of a finite group, hence \(\tilde{b}^s = 1\) for some \(s \in \mathbb{Z}_{\geq 0}\).
    Therefore we can solve the equation by passing from coefficients in \(\mathbb{F}^\mathcal{J} = \mathbb{F} \otimes_{\mathbb{F}_{p}} \mathbb{F}_{q}\) to \(\mathbb{F} \otimes_{\mathbb{F}_{p}} \mathbb{F}_{q^s}\) (and further extending \(\mathbb{F}\) if needed). 
\end{remark}

\begin{lemma}
    \label{res:strictly-frobenius-invariant-coboundary-in-normalizer}
    Let \([\tau] \in H^1\left( \Gamma, \widehat{G}(\mathbb{F}^\mathcal{J}\llb u \rrb) \right)\) be a Frobenius invariant Galois type.
    After passing to a finite unramified extension (in the sense of \cref{res:frobenius-invariant-type-as-x} (2)), there exists \(\lambda \in X_{*}(\widehat{T})^I \cong \left(X_{*}(\widehat{T})^\mathcal{J}\right)^\Gamma\) and \(w \in \widehat{N}(\mathbb{Z}^\mathcal{J})\)
    such that if we set \(n := w^{-1} u^\lambda\), then the \(1\)-cocycle \(\tau(\theta) = n^{-1}(^\theta n)\) represents the class \([\tau]\) and is \emph{strictly} Frobenius invariant (i. e. \(\varphi \tau = \tau\)).
\end{lemma}
\begin{proof}
    Let \(\dot{N} \subset \widehat{N}(\overline{\mathbb{F}}^\mathcal{J}\llp u \rrp)\) be the subgroup generated by the image of \(X_{*}(\widehat{T})^I\) and \(\widehat{N}(\mathbb{Z}^\mathcal{J})\). 
    By \cref{res:types-in-smaller-normalizer} it follows that there exists \(n' = w'^{-1} u^{\lambda'} \in \dot{N}\) such that \(\tau'(\theta) = n'^{-1}(^\theta n')\) is a representative for the class \([\tau]\), 
    and also that there exists \(b \in \dot{N}^0 = \dot{N} \cap \widehat{N}(\overline{\mathbb{F}}^\mathcal{J}\llb u \rrb)\) such that \(b \varphi \tau' (\theta) (^\theta b)^{-1} = \tau'(\theta)\) for all \(\theta \in \Gamma\).
    As in the proof of \cref{res:frobenius-invariant-type-as-x} it suffices to find \(c \in \dot{N}\) such that \(b = c^{-1} \varphi(c)\). 
    Although Lang's theorem doesn't apply, we can still use the explicit method of \cref{rem:strictly-frobenius-invariant-explicit}.
\end{proof}
\begin{remark}
    \label{rem:strictly-frobenius-invariant-vs-fixed}
    If \(n \in \widehat{N}(\overline{\mathbb{F}}^\mathcal{J}\llp u \rrp)\) and \(\tau(\theta) = n^{-1}(^\theta n)\) is the associated \(1\)-cocycle, then \(\varphi\) is strictly Frobenius invariant (i. e. \(\varphi \tau = \tau\)) if and only if \(n \varphi(n)^{-1} \in N^*(\overline{\mathbb{F}}\llp v \rrp)\) (i. e. \(n \varphi(n)^{-1}\) is fixed by \(\Gamma\)). 
\end{remark}

\begin{example}
    \label{ex:SL2}
    Suppose \(L = \mathbb{Q}_{p^2}\left( (-p)^{1/e} \right)\) and \(\widehat{G} = \operatorname{SL}_{2}\). Since \(\operatorname{SL}_{2}\) is semisimple and simply connected, \(\operatorname{SL}_{2}(\mathbb{F}\llp v \rrp)\) acts simply transitively on the chambers of \(\mathcal{A}\left( \widehat{T}, \mathbb{F}\llp v \rrp \right)\), where \(\widehat{T} \subset \operatorname{SL}_{2}\) is the diagonal torus. 
    If \(\mathcal{C} \subset \mathcal{A}\left( \widehat{T}, \mathbb{F}\llp v \rrp \right)\) denotes the base alcove, it follows that Galois types \([\tau] \in H^1 \left(\Gamma, \widehat{G}(\mathbb{F}^\mathcal{J}\llb u \rrb) \right)\)
    are classified by vertices \(\operatorname{SL}_{2}\left( \mathbb{F}\llp u \rrp \right)\cdot o \cap \mathcal{C}\). For \(p = 7\) and \(e = 24\) this is illustrated by the colored nodes in \cref{fig:SL2-alcove}.

    Consider \(x = u^{(n,-n)} \cdot o = o - \frac{1}{e}(n,-n)\), classifying a type \([\tau]\). 
    For \(p = 7\) and \(e = 24\) the colored nodes in \cref{fig:SL2-alcove} correspond to such \(x\) where \(n \in \lbrace 0, -1, \dots, -12 \rbrace\).
    The point \(x\) classified the type \([\tau]\) represented by \(\tau(\sigma) = u^{-(n,-n)}(^\sigma u^{(n,-n)}) = 1\) and \(\tau(\gamma) = u^{-(n,-n)}(^\gamma u^{(n,-n)}) = \omega(\gamma)^{(n,-n)}\). 
    Here, \(u^{(n,-n)}\) represents the element of \(\operatorname{SL}_{2}(\mathbb{F}^\mathcal{J}\llp u \rrp) = \operatorname{SL}_{2}(\mathbb{F}_{0}\llp u \rrp) \times \operatorname{SL}_{2}(\mathbb{F}_{1}\llp u \rrp)\) for which both factors is what is typically denoted \(u^{(n,-n)}\). 
    In particular, note that \(\omega(\gamma)^{(n,-n)} = \left( \iota_{0}(\omega(\gamma))^{(n,-n)}, \iota_{1} (\omega(\gamma))^{(n,-n)} \right) = \left( \iota_{0}(\omega(\gamma))^{(n,-n)}, \iota_{0}(\omega(\gamma))^{(pn, -pn)} \right)\), where \(\iota_{0},\iota_{1} : \mathbb{F}_{p^2} \hookrightarrow \mathbb{F}\) are the two embeddings. 

    \begin{figure}[h]
        \centering
        \begin{tikzpicture}[scale=10]
            % Define the variables
            \def\e{24}
            \def\p{7}
        
            \clip (-0.2,-0.2) rectangle (1.2,0.2);
        
            \draw (0,0) -- (1,0);
        
            % Loop to draw the circles
            \foreach \n in {1,...,\e} {
                \pgfmathsetmacro{\x}{\n / \e}
                \pgfmathparse{
                    ifthenelse(mod(\n, 2) == 1, "white",
                    ifthenelse(mod(floor(\p*\n / \e), 2) == 0 && mod((\p-1)*\n, \e) == 0, "green",
                    ifthenelse(mod(floor(\p*\n / \e), 2) == 1 && mod((\p+1)*\n, \e) == 0, "green",
                    "red")))
                }
                \edef\color{\pgfmathresult}
                \draw[fill = \color] (\x,0) circle (0.1mm);
        
                % Remember the last color and position
                \ifnum\n=\e
                    \xdef\lastcolor{\color}
                \fi
            }
        
            % Label corners of base alcove
            \node[draw, circle, fill=green, minimum size=0.15cm, label={above:{$o$}}] at (0,0) {};
            \node[draw, circle, fill=\lastcolor, minimum size=0.15cm, label={above:{$o + (\frac{1}{2},-\frac{1}{2})$}}] at (1,0) {};
        \end{tikzpicture}
        \caption{
            The base alcove of \(\mathcal{A}\left( \widehat{T}, \mathbb{F}\llp v \rrp \right)\) where \(\widehat{T} \subset \operatorname{SL}_{2}\) is the diagonal torus, for \(p = 7\) and \(e = 24\).
            The large nodes are vertices of \(\mathcal{A}\left( \widehat{T}, \mathbb{F} \llp v \rrp \right)\) and the small nodes are vertices of \(\mathcal{A}\left( \widehat{T}, \mathbb{F}\llp u \rrp \right)\). 
            The white nodes are vertices of \(\mathcal{A}\left( \widehat{T}, \mathbb{F}\llp u \rrp \right)\) which are not in the orbit of \(o\).
            The green nodes classify types which are Frobenius invariant, and the red nodes classify types which are not Frobenius invariant.
            Since the base alcove is a fundamental domain for the \(\operatorname{SL}_{2}(\mathbb{F}\llp v \rrp)\) action, we see that there are (under the chosen parameters) exactly 13 classes of Galois types \(\Gamma \to \widehat{G}(\overline{\mathbb{F}}^\mathcal{J}\llb u \rrb)\), of which 7 are Frobenius invariant.
        }
        \label{fig:SL2-alcove}
    \end{figure}
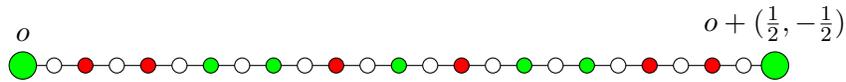

    What does it mean for \(x = u^{(n,-n)}\cdot o \in \mathcal{A}(\widehat{T}, \mathbb{F}\llp v \rrp)\) to represent a Frobenius invariant type? As we have meantioned, this is equivalent to \(x\) and \(\varphi(x) = u^{(pn, -pn)}\cdot o\) being in the same \(\operatorname{SL}_{2}(\mathbb{F}\llp v \rrp)\) orbit.
    There are two ways this can happen:
    \begin{enumerate}
        \item If there exists \(m \in \mathbb{Z}\) such that \(v^{(m,-m)} \cdot \varphi(x) = x\), i. e. if \((p-1)n \equiv 0 \mod e\). 
        \item If there exists \(m \in \mathbb{Z}\) such that \(w_{0}v^{(m,-m)} \cdot \varphi(x) = x\), where \(w_{0}\) is the reflection across \(o\) (corresponding to the unique element of the finite Weyl group of \(\operatorname{SL}_{2}\)). 
            This is equivalent to \((p+1)n \equiv 0 \mod e\). 
    \end{enumerate}
    This is how one can determine the coloring of the nodes in \cref{fig:SL2-alcove}. 
    
    As a concrete example, consider \(x = u^{(-3,3)}\cdot o\) for \(p = 7\) and \(e = 24\), corresponding to the second green node from the left (including \(o\)) in \cref{fig:SL2-alcove}.
    Then since \((7+1)(-3) \equiv 0 \mod 24\), the type classified by \(x\), which is represented by \(\tau(\theta) = u^{-(-3,3)}(^\theta u^{(-3,3)})\) for \(\theta \in \Gamma\), is Frobenius invariant. 
    To make this even more explicit, note that \(\tau(\sigma) = 1\) and \(\tau(\gamma) = \left( \iota_{0}(\omega(\gamma))^{(-3,3)}, \iota_{0}(\omega(\gamma))^{(-21,21)} \right)\). 
    Using the observation that \(\varphi\) acts on \(\operatorname{SL}_{2}(\mathbb{F}_{0}) \times \operatorname{SL}_{2}(\mathbb{F}_{1})\) by swapping the factors, we find that 
    \(\varphi \tau(\sigma) = 1 = \tau(\sigma)\) and \(\varphi \tau(\gamma) = \left( \iota_{0}(\omega(\gamma))^{(-21,21)}, \iota_{0}(\omega(\gamma))^{(-3,3)}\right)\).
    Note that \(\iota_{0}(\omega(\gamma))^{(-21,21)} = \iota_{0}(\omega(\gamma))^{(3,-3)}\) since \(\gamma^e = 1\) and \(-21 \equiv 3 \mod e = 24\). 
    Therefore, if we let \(b = \left( \begin{pmatrix} 0 & 1 \\ -1 & 0 \end{pmatrix}, \begin{pmatrix} 0 & 1 \\ -1 & 0 \end{pmatrix} \right)\), then \(b \varphi \tau(\theta) (^\theta b^{-1}) = \tau(\theta)\) for all \(\theta \in \Gamma\), showing that \([\tau]\) is Frobenius invariant. 

    Note that our representative \(\tau\) for the class \([\tau]\) corresponding to \(x\) is Frobenius invariant, but not strictly so. 
    However, we can follow the method of \cref{rem:strictly-frobenius-invariant-explicit} to produce a representative which is strictly Frobenius invariant.
    Since \(b_{0} b_{1} =  \begin{pmatrix} 0 & 1 \\ -1 & 0 \end{pmatrix}^2 = \begin{pmatrix} -1 & 0 \\ 0 & -1 \end{pmatrix}\) has order \(2\), we should according to \cref{rem:strictly-frobenius-invariant-explicit} pass to an unramified extension of order \(2\). 
    Let us first see what goes wrong if we don't:
    Define \(g = (g_{0},g_{1}) \in \operatorname{SL}_{2}(\mathbb{F}^\mathcal{J}\llp u \rrp)\) by choosing \(g_{0} = 1 \in \operatorname{SL}_{2}(\mathbb{F}_{0}\llp u \rrp)\) and let \(g_{1} = b_{1}^{-1} g_{0} = \begin{pmatrix} 0 & 1 \\ -1 & 0 \end{pmatrix}\).
    Consider now the Galois type \(\tilde{\tau} : \Gamma \to \operatorname{SL}_{2}(\mathbb{F}^\mathcal{J}\llp u \rrp)\) with \(\tilde{\tau}(\theta) = g \tau(\theta)(^\theta g)^{-1}\) for \(\theta \in \Gamma\).
    Then we have \(\tilde{\tau}(\gamma) = \left( \iota_{0}(\omega(\gamma))^{(-3,3)}, \iota_{0}(\omega(\gamma))^{(21,-21)} \right) = \varphi\tilde{\tau}(\gamma)\) (since \(-3 \equiv 21 \mod 24\)), but \(\tilde{\tau}(\sigma) = \left( \begin{pmatrix} 0 & -1 \\ 1 & 0 \end{pmatrix}, \begin{pmatrix} 0 & 1 \\ -1 & 0 \end{pmatrix} \right) \neq \left( \begin{pmatrix} 0 & 1 \\ -1 & 0 \end{pmatrix}, \begin{pmatrix} 0 & -1 \\ 1 & 0 \end{pmatrix} \right) = \varphi \tilde{\tau}(\sigma)\).
    So let us now pass to \(L' = \mathbb{Q}_{p^4}\left( (-p)^{1/e} \right)\), in which case the type \(\tau\) gives rise to 
    \(\tau' : \Gamma' := \operatorname{Gal}(L' / \mathbb{Q}_{p}) \to \operatorname{SL}_{2}\left( \mathbb{F}_{0} \right) \times \operatorname{SL}_{2}(\mathbb{F}_{1}) \times \operatorname{SL}_{2}(\mathbb{F}_{2}) \times \operatorname{SL}_{2}(\mathbb{F}_{3})\)
    with \(\tau'(\gamma) = \left( \iota_{0}(\omega(\gamma))^{(-3,3)}, \iota_{0}(\omega(\gamma))^{(3,-3)}, \iota_{0}(\omega(\gamma))^{(-3,3)}, \iota_{0}(\omega(\gamma))^{(3,-3)} \right)\)
    and \(\tau'(\sigma) = 1\). 
    Then \(b = \left( \begin{pmatrix} 0 & 1 \\ -1 & 0 \end{pmatrix}, \begin{pmatrix} 0 & 1 \\ -1 & 0 \end{pmatrix}, \begin{pmatrix} 0 & 1 \\ -1 & 0 \end{pmatrix}, \begin{pmatrix} 0 & 1 \\ -1 & 0 \end{pmatrix} \right)\)
    gives a coboundary \(b : \varphi \tau' \isoto \tau'\). 
    But now \(b_{0}b_{1}b_{2}b_{3} = (b_{0}b_{1})^2 = 1\), so according to \cref{rem:strictly-frobenius-invariant-explicit}, we can define 
    \(c_{0} = 1\), \(c_{1} = b_{1}^{-1}c_{0} = \begin{pmatrix} 0 & -1 \\ 1 & 0 \end{pmatrix}\), \(c_{2} = b_{2}^{-1} c_{1} = \begin{pmatrix} 0 & -1 \\ 1 & 0 \end{pmatrix}^2 = \begin{pmatrix} -1 & 0 \\ 0 & -1 \end{pmatrix}\), \(c_{3} = b_{3}^{-1}c_{2} = \begin{pmatrix} 0 & 1 \\ -1 & 0 \end{pmatrix}\), in which case we also have \(b_{0}^{-1} c_{3} = 1 = c_{0}\). 
    Then \(\tilde{\tau}(\theta) = c \tau'(\theta)(^\theta c)^{-1}\) is strictly Frobenius invariant, because \(\tilde{\tau}(\gamma) = \left( \iota_{0}(\omega(\gamma))^{(-3,3)}, \iota_{0}(\omega(\gamma))^{(-3,3)}, \iota_{0}(\omega(\gamma))^{(-3,3)}, \iota_{0}(\omega(\gamma))^{(-3,3)} \right)\)
    and \(\tilde{\tau}(\sigma) = \left( c_{0}c_{3}^{-1}, c_{1}c_{0}^{-1}, c_{2}c_{1}^{-1}, c_{3}c_{2}^{-1} \right) = \left( \begin{pmatrix} 0 & 1 \\ -1 & 0 \end{pmatrix}, \begin{pmatrix} 0 & 1 \\ -1 & 0 \end{pmatrix}, \begin{pmatrix} 0 & 1 \\ -1 & 0 \end{pmatrix}, \begin{pmatrix} 0 & 1 \\ -1 & 0 \end{pmatrix} \right)\).
\end{example}

\subsubsection{Inertial types vs. Galois types}
\label{sec:inertial-types}
In the literature (e.g. \cite{GHS}, \cite{local-models}, \cite{CEGS}), the descent datum for a Breuil--Kisin module is usually specified in terms on an inertial type.
We spell out how this is related to our Galois type. 
For this it is convenient to use the following enlargement of \(\Gamma\), namely 
\[
\Gamma_{W} := I \rtimes \mathbb{Z} = \lbrace \gamma, \sigma : \gamma^e = 1, \sigma \gamma \sigma^{-1} = \gamma^p \rbrace . 
\]
The following is consistent with \cite[Section 9]{GHS}. 
\begin{definition}
    An \emph{inertial type} is a homomorphism \(\tau' : I \to \widehat{G}(\mathbb{F})\) which extends to a 1-cocycle \(\Gamma_{W} \to \widehat{G}(\mathbb{F})\).
    Here, we view \(\mathbb{F}\) as having trivial \(\Gamma_{W}\)-action, so \(\Gamma_{W}\) acts on \(\widehat{G}(\mathbb{F})\) purely via group automorphisms on \(\widehat{G}\).
\end{definition}
If \(\tau'\) is an inertial type, note that the extension \(\Gamma_{W} \to \widehat{G}(\mathbb{F})\) necessarily factors through a finite quotient \(\Gamma'\) in which \(\sigma^{r'} = 1 \), 
and conversely, any 1-cocycle \(\Gamma' \to \widehat{G}(\mathbb{F})\) yields by restriction an inertial type. 
In principle one can first classify all inertial types, and let the splitting field \(L\) be dictated by the inertial type. 
Note also that since Breuil--Kisin modules and the property of being crystalline satisfy unramified descent, it is harmless to modify \(L\) by an unramified extension.

There is a discrepancy between 1-cocycles \(\tau' : \Gamma_{W} \to \widehat{G}(\mathbb{F})\)
(which come from inertial types) and \(\tau :\Gamma_{W} \to \widehat{G}(\mathbb{F}^\mathcal{J})\) (which correspond to our notion of Galois types).
Although \(\widehat{G}(\mathbb{F}^\mathcal{J}) \cong \prod_{j \in \mathcal{J}} \widehat{G}(\mathbb{F}_{j})\), the issue is that \(\sigma\) has a permuting effect on the factors. 
However, under the assumption of Frobenius invariance, one can pass between these notions by using the isomorphism \(\varphi \tau \to \varphi\) to ``linearize'' the Galois type, yielding an inertial type. 
This is akin to the usual construction of a Weil--Deligne representation from the semilinear descent data attached to a potentially crystalline representation. 
We record this construction in the following lemma. 
\begin{lemma}
    \label{res:linearizing-galois-types}
    Assume that \(I\) acts trivially on \(\widehat{G}\). 
    Let \(\mathbb{F}^{\mathcal{J},\text{triv}}\) denote \(\mathbb{F}^{\mathcal{J}} = \mathbb{F} \otimes_{\mathbb{F}_{p}} \mathbb{F}_{q}\) equipped with the trivial \(\Gamma\)-action. 
    Then for any choice of \(j \in \mathcal{J}\) we have equivalences of groupoids
    \begin{align*}
        Z^1\left( \Gamma_{W}, \widehat{G}\left( \mathbb{F}^\mathcal{J} \right) \right)^{\varphi = 1} & \cong Z^1 \left( \Gamma_{W}, \widehat{G}\left( \mathbb{F}^{\mathcal{J},\text{triv}} \right) \right)^{\varphi = 1} \cong Z^1 \left( \Gamma_{W}, \widehat{G}(\mathbb{F}_{j}) \right)  \\ 
        \left( \tau, b : \varphi \tau \isoto \tau \right) & \mapsto \,\,\,\,\, \left( \tau', b : \varphi \tau' \isoto \tau' \right) \,\,\,\,\, \mapsto \tau'_{j}
    \end{align*}
    where \(\tau'(\gamma) = \tau(\gamma)\) and \(\tau'(\sigma) = \varphi^{-1}(b^{-1} \tau(\sigma)) = \tau(\sigma) \psi(b)^{-1}\), and we view \(\mathbb{F}_{j}\) as having trivial \(\Gamma\)-action.
    Here, we recall that \(\psi : \widehat{G} \isoto \widehat{G}\) denotes the group automorphism of \(\widehat{G}\) determined by \(\sigma\), 
    and the superscript \(\varphi = 1\) denotes homotopy fixed points of groupoids (so elements of for example \(Z^1\left( \Gamma_{W}, \widehat{G}\left( \mathbb{F}^\mathcal{J} \right) \right)^{\varphi = 1}\)
    are pairs \((\tau, b)\), where \(\tau : \Gamma_{W} \to \widehat{G}(\mathbb{F}^\mathcal{J})\) is a \(1\)-cocycle and \(b : \varphi \tau \isoto \tau\) is a coboundary). 
\end{lemma}
\begin{proof}
    Let \(H\) be any \(\Gamma_{W}\)-group on which \(I \subset \Gamma_{W}\) acts trivially. 
    Then we have a bijection 
    \begin{align*}
        Z^1\left( \Gamma_{W}, H \right) & \isoto \lbrace (x,y) \in H \times H | x^e = 1, y(^\sigma x)y^{-1} = x^p \rbrace \\ 
        \rho & \mapsto (\rho(\gamma), \rho(\sigma)) . 
    \end{align*}
    We will apply this to \(H = \widehat{G} \left( \mathbb{F}^\mathcal{J} \right), \widehat{G}(\mathbb{F}^{\mathcal{J},\text{triv}}), \widehat{G}(\mathbb{F}_{j})\). 
    Note that when \(H = \widehat{G}\left( \mathbb{F}^\mathcal{J} \right)\), we have \({^\sigma x} = \varphi \psi (x) = \psi \varphi (x)\), 
    but when \(H = \widehat{G}(\mathbb{F}^{\mathcal{J},\text{triv}})\) or \(\widehat{G}(\mathbb{F}_{j})\) we have \({^\sigma x} = \psi(x)\). 
    
    Now let \(\left(\tau, b \right) \in Z^1\left( \Gamma_{W}, \widehat{G}\left( \mathbb{F}^\mathcal{J} \right) \right)^{\varphi = 1}\) and let \(\tau'\) be as in the statement of the lemma. 
    We first show that \(\tau'\) is a \(1\)-cocycle. 
    Let \(x = \tau(\gamma) = \tau'(\gamma)\), \(y = \tau(\sigma)\), and \(y' = \varphi^{-1}(b^{-1}y)\).
    Note that \(b : \varphi \tau \isoto \tau\) being a coboundary means that \(b \varphi(x) b^{-1} = x\) and \(b \varphi(y) = y (^{\sigma} b)^{-1} = y \psi (\varphi(b))^{-1}\). 
    We have 
    \begin{align*}
        y' \psi(x) y'^{-1} & = \varphi^{-1}(b^{-1} y ) \psi(x) \varphi^{-1}(b^{-1}y)^{-1} \\
        & = \varphi^{-1}\left( b^{-1} y \varphi \psi(x) y^{-1} b \right) \\ 
        & = \varphi^{-1} \left( b^{-1} x^p b \right) \\
        & = \varphi^{-1} \left( \varphi(x)^p \right) \\
        & = x^p ,
    \end{align*}
    which shows that \(\tau'\) is indeed a \(1\)-cocycle.
    We leave it to the reader to show that \(b : \varphi \tau' \isoto \tau'\) is a \(1\)-coboundary, and therefore that \((\tau,b)\mapsto (\tau',b)\) is well-defined. 
    Once the map is seen to be well-defined, it is easily seen to be an equivalence, since we can solve for \(y\) in terms of \(y'\) to define the inverse map. 

    It remains to see that projection onto the \(j\)'th factor yields an equivalence \(Z^1 \left( \Gamma_{W}, \widehat{G}\left( \mathbb{F}^{\mathcal{J},\text{triv}} \right) \right)^{\varphi = 1} \cong Z^1 \left( \Gamma_{W}, \widehat{G}(\mathbb{F}_{j}) \right)\).
    The point is now that since \(\Gamma_{W}\) acts trivially on \(\widehat{G}(\mathbb{F}^{\mathcal{J},\text{triv}}) = \prod_{j \in \mathcal{J}}\widehat{G}(\mathbb{F}_{j})\), 
    a \(1\)-cocycle \(\tau' : \Gamma_{W} \to \widehat{G}(\mathbb{F}^{\mathcal{J},\text{triv}})\) corresponds to a collection of \(1\)-cocycles \(\lbrace \tau_{j}' : \Gamma_{W} \to \widehat{G}(\mathbb{F}_{j}) \rbrace_{j \in \mathcal{J}}\), 
    and the coboundary \(b\) gives equivalences between all the different \(1\)-cocycles \(\tau_{j}'\), \(j \in \mathcal{J}\).
    We leave the remaining details to the reader. 
\end{proof}
\begin{remark}
    The above \cref{res:linearizing-galois-types} together with \cref{res:galois-types-are-frobenius-invariant} makes precise the statement that 
    a Galois type arising as descent data for Breuil--Kisin modules have the same inertial type ``in each factor'', i. e. is ``unmixed'' in 
    the terminology of \cite{CEGS}. 
\end{remark}
\begin{remark}
    \label{rem:lm-inertial-type-vs-galois-type}
    In \cite{local-models} they use the parametrization of tame inertial types from \cite[Section 9]{GHS} giving 1-cocycle \(\Gamma_{W} \to \widehat{G}(\mathbb{F})\). 
    The corresponding Galois type (they do not use this terminology) is then the corresponding element on the left hand side in \cref{res:linearizing-galois-types}. 
\end{remark}

\subsubsection{Genericity}
According to \cref{sec:description-of-types}, we describe our inertial type \(\tau\) as a point \(x\) in the apartment. 
Our notion of \textit{genericity} of \(\tau\) will be given in terms of a notion of genericity of \(x\).

\begin{definition}
    \label{def:generic-type}
    Let \(y \in \widetilde{\mathcal{A}}(T^*, \mathbb{F}\llp v \rrp)\) and \(d \in \mathbb{R}\). 
    We say that \(y\) is \textit{\(d\)-generic} if
    for all \(a \in \Phi\left( G^*,T^* \right)^{+}\) there exists
    \(n_{a} \in \mathbb{Z}\) such that
    \[ n_{a} + \frac{d}{p} < \langle a, y - o \rangle < n_{a} + 1 - \frac{d}{p} . \]
\end{definition}
\begin{remark}
    \label{rem:generic-type-inequality}
    The above definition depends on the prime \(p\), and it is
    only possible to be \(d\)-deep if \(d < \frac{p}{2}\). This bound is not sharp: For example, if the root system is of type \(A_{2}\), then it 
    is only possible to be \(d\)-deep if \(d < \frac{p}{3}\) (see \cref{fig:alcove-genericity} on page \pageref{fig:alcove-genericity}).
\end{remark}
\begin{remark}
    When \(d < 0\), it is vacous to say that \(y\) is \(d\)-generic. 
    However, it will be convient to allow negative values of \(d\) in order to have uniform statements, as for example in \cref{res:contraction}.
\end{remark}
\begin{remark}
    \label{rem:genericity-for-explicit-coboundary}
    If \(x = w^{-1} u^\lambda \cdot o\), where \(w \in \widehat{N}(\mathbb{F}^\mathcal{J})\) and \(\lambda \in X_{*}(\widehat{T})^\mathcal{J}\), 
    then \( x - o = - \frac{1}{e} w^{-1}\lambda\), so \(\langle a, x - o \rangle = - \frac{1}{e} \langle w a , \lambda \rangle\). 
    So we see that \(x\) is \(d\)-generic if and only if for each \(a \in \Phi(G^*, T^*)\) there exists an integer \(n_{a} \in \mathbb{Z}\)
    such that \(n_{a} + \frac{d}{p} < \frac{1}{e}\langle a, \lambda \rangle < n_{a} + 1 - \frac{d}{p}\).
\end{remark}

We will relate our notion of genericity to that of \cite{local-models}, specializing to the case where \(\widehat{G} = \prod_{j = 0}^{f-1} \GL_{n}\) is the dual group of 
\(G = \operatorname{Res}_{\mathbb{Q}_{p}}^{K} \GL_{n}\) is the Weil restriction of \(\GL_{n}\), where \(\mathbb{Q}_{p} \subset K\) is an unramified extension of degree \(f\). 
According to Definition 2.1.10 in loc. cit., 
a cocharacter \(\mu \in X_{*}(\widehat{T}) = X^*(T)\) is said to be \emph{\(d\)-deep} if for each positive root \(a \in \Phi(\widehat{G},\widehat{T})^+\) there exists \(n_{a} \in \mathbb{Z}\) such that 
\begin{equation}
    \label{eq:lm-deep}
    p n_{a} + d < \langle a, \mu + \eta \rangle < p (n_{a}+1) - d , 
\end{equation}
where \(\eta\) is the sum of of the fundamental dominant coweights \(\omega_{a}^\vee\), satisfying \(\langle a, \omega_{a'}^\vee \rangle = \delta_{a a'}\) (Kronecker \(\delta\)). 
Given \(\mu \in X_{*}(\widehat{T})\) and \(s \in W\) (the finite Weyl group for \(\widehat{G}\)), 
one associates an inertial type \(\tau(s,\mu + \eta) : I \to \widehat{G}(\mathbb{F})\) as in \cite[Example 2.4.1]{local-models}.
Following \cref{rem:lm-inertial-type-vs-galois-type} and \cref{res:deforming-galois-types} this gives rise to a Galois type \( \widetilde{\tau} : \Gamma \to \widehat{G}(\mathbb{F}^\mathcal{J}\llb u \rrb)\), 
which assuming \(\mathbb{F}\) is sufficiently large is classified by a point \(\widetilde{x} = n \cdot o \in \widetilde{\mathcal{A}}\left( T^*, \mathbb{F}\llp v \rrp \right)\) using \cref{res:types-in-the-normalizer}.
We then have the following. 
\begin{lemma}
    Assume that \(\mu\) is \((d+1)\)-deep according to \eqref{eq:lm-deep} as well as ``lowest alcove'', meaning that 
    \[
        d + 1 < \langle a, \mu + \eta \rangle < p - d - 1
    \]
    for all positive roots \(a \in \Phi(\widehat{G},\widehat{T})^+\). 
    Then \(\widetilde{x}\) is \(d\)-generic in the sense of \cref{def:generic-type}.
\end{lemma}
\begin{proof}
    We will follow \cite[Example 2.4.1]{local-models}, noting that what their \(fr\) is our \(r = [\mathbb{Q}_{q} : \mathbb{Q}_{p}] = [L_{0} : \mathbb{Q}_{p}]\).
    Following the conventions of loc. cit. we have \(e = p^r - 1\). 
    We can for some fixed embedding \(\iota : \mathbb{F}_{q} \hookrightarrow \mathbb{F}\) write \(\tau(s,\mu + \eta)|_{I} = \iota(\omega)^\lambda\), where \(\lambda \in X_{*}(\widehat{T})\) (this \(\lambda\) corresponds to ``\(\sum_{k=0}^{fr-1} (F^* \circ s^{-1})^j(\mu + \eta)\)'' in loc. cit.). 
    For any fixed \(0 \leq i \leq f -1\), we can moreover write 
    \[
        \lambda_{i} = \sum_{j=0}^{r-1} \beta_{j}p^j ,
    \]
    where each \(\beta_{j} = s'_{j} (\mu_{j'} + \eta_{j'})\) for some \(0 \leq j' \leq f -1\) and \(s'_{j} \in W\). 
    The assumption that \(\mu + \eta\) is \((d+1)\)-deep lowest alcove implies that for each \(0 \leq j \leq r-1\) and \(a \in \Phi^\vee(\GL_{n},T)^{+}\), we have 
    \begin{align*}
        d + 1 < \langle a, \beta_{j} \rangle < p - d - 1 \,\,\,\,\,\,\, \text{ or } \,\,\,\,\,\,\, - p + d + 1 < \langle a , \beta_{j} \rangle < - d - 1 . 
    \end{align*}
    Take now any \(a \in \Phi^\vee(\GL_{n},T)^{+}\); by \cref{rem:genericity-for-explicit-coboundary} we have to show that there exists \(n_{a} \in \mathbb{Z}\) such that \(n_{a} + \frac{d}{p} < \frac{1}{e} \langle a, \lambda_{i} \rangle < n_{a} + 1 - \frac{d}{p}\). 
    We consider two cases:
    \begin{enumerate}
        \item \(d + 1 < \langle a, \beta_{r-1} \rangle < p - d - 1\), 
        \item \(- p + d + 1 < \langle a, \beta_{r-1} \rangle < - d - 1\). 
    \end{enumerate}
    Assume we are in case (1).
    It will be useful to observe that 
    \[
        (p-d-1) \sum_{j=0}^{r-1}p^j = (p-d-1)\frac{p^r-1}{p-1} = \left( 1 - \frac{d}{p-1}\right)(p^r - 1) . 
    \]
    We have 
    \begin{align*}
        \langle a, \lambda_{i} \rangle & = \sum_{j=0}^{r-1} \langle a, \beta_{j} \rangle p^j \\
        & < (p-d-1)\sum_{j=0}^{r-1}p^j \\ 
        & = \left( 1 - \frac{d}{p-1}\right) (p^r - 1) . 
    \end{align*}
    Dividing by \(e = p^r - 1\), it follows that 
    \[
        \frac{1}{e} \langle a , \lambda_{i} \rangle < 1 - \frac{d}{p-1} < 1 - \frac{d}{p} . 
    \]
    On the other hand, we have 
    \begin{align*}
        \langle a, \lambda_{i} \rangle & = \sum_{j=0}^{r-1} \langle a, \beta_{j} \rangle p^j \\ 
        & > (- p + d + 1) \sum_{j=0}^{r-2} p^j + (d+1)p^{r-1} \\
        & = -(p - d - 1)\sum_{j=0}^{r-1}p^j + p^r \\
        & = - \left( 1 - \frac{d}{p-1}\right)(p^r - 1) + p^r \\
        & = 1 + (p^r - 1)\frac{d}{p-1} . 
    \end{align*}
    Dividing by \(e = p^r - 1\) again, we obtain 
    \[
        \frac{1}{e}\langle a, \lambda_{i} \rangle > \frac{1}{e} + \frac{d}{p-1} > \frac{d}{p} . 
    \]
    This completes the proof in case (1), and case (2) is similar. 
\end{proof}
\begin{figure}[h!]
    \centering
    \resizebox{\textwidth}{!}{
    \begin{tikzpicture}[scale = 12, rotate = 180]
        % Set the number of segments and the prime
        \pgfmathsetmacro{\numSegments}{36}
        \pgfmathsetmacro{\p}{19}
        \pgfmathsetmacro{\kbound}{floor(\p/3)}

        % Function to draw the divided triangle
        \newcommand{\dividedTriangle}{
            % Define the coordinates of the large triangle
            \coordinate (A) at (0, 0);
            \coordinate (B) at (1, 0);
            \coordinate (C) at (0.5, 0.866); % height of equilateral triangle with side length 1

            % Draw the red triangles representing genericity zones
            \foreach \k in {0,...,\kbound} {
                \pgfmathparse{1.5*\k/\p}
                \let\xi\pgfmathresult
                \pgfmathparse{0.866*\k/\p}
                \let\yi\pgfmathresult
                \pgfmathparse{1 - (1.5*\k/\p)}
                \let\xii\pgfmathresult
                \pgfmathparse{0.866*(1 - (2*\k/\p))}
                \let\yiii\pgfmathresult
                % \draw[line width=0.05pt,color=red] (\xi,\yi) -- (\xii,\yi) -- (0.5,\yiii) -- cycle;
                \fill[color=red,fill opacity=0.15] (\xi,\yi) -- (\xii,\yi) -- (0.5,\yiii) -- cycle;
            }

            % Divide each side into equal segments and draw the smaller triangles
            \foreach \i in {0,...,\numSegments} {
                \coordinate (P) at ($(A)!\i/\numSegments!(B)$);
                \coordinate (Q) at ($(A)!\i/\numSegments!(C)$);
                \coordinate (R) at ($(B)!\i/\numSegments!(C)$);
                \coordinate (S) at ($(C)!\i/\numSegments!(B)$);
                
                % Draw lines parallel to the sides of the large triangle
                \draw[line width=0.05pt,color=gray] (P) -- (Q);
                \draw[line width=0.05pt,color=gray] (P) -- (S);
                \draw[line width=0.05pt,color=gray] (Q) -- (R);
            }

            % Draw the large triangle
            \draw[thick] (A) -- (B) -- (C) -- cycle;
        }

        % Define the clipping region to create a square "photograph"
        \clip (-0.2,-0.173) rectangle (1.2,1.039);

        % Draw triangles which are shifts of divided triangle
        \foreach \x / \y in {-0.5/-0.866, 0.5/-0.866, -1/0, 0/0, 1/0, -0.5/0.866, 0.5/0.866} {
            \begin{scope}[shift={(\x,\y)}]
                \dividedTriangle
            \end{scope}
        }

        % Draw traingles which are shifts of upside down divided triangle
        \foreach \x / \y in {-0.5/-0.866, 0.5/-0.866, -1/0, 0/0, 1/0, 0/-1.732} {
            \begin{scope}[yscale=-1,shift={(\x,\y)}]
                \dividedTriangle
            \end{scope}
        }

        % Place a dot labeled "x" at (0.5, 0.3)
        \node[circle, fill=black, inner sep=1pt, label={[fill=white, fill opacity=0.5, text opacity=1, rounded corners] above:{$x$}}] at (0.319, 0.217) {};

        % Label corners of triangle
        \node[circle, fill=black, inner sep=1.5pt, label={[fill=white, fill opacity=0.7, text opacity=1, rounded corners] above:{$o$}}] at (0.5,0.866) {};
        \node[circle, fill=black, inner sep=1.5pt, label={[fill=white, fill opacity=0.7, text opacity=1, rounded corners] above:{$o + (1,0,0)$}}] at (0,0) {};
        \node[circle, fill=black, inner sep=1.5pt, label={[fill=white, fill opacity=0.7, text opacity=1, rounded corners] above:{$o + (1,1,0)$}}] at (1,0) {};
    \end{tikzpicture}}
    \caption{
        A snapshot of the apartment \(\widetilde{\mathcal{A}}\left( \widehat{T}, \mathbb{F}\llp v \rrp \right)\), where \(\widehat{T} \subset \GL_{3}\) is the diagonal torus (the same one that we saw in \cref{fig:alcoves-inside-alcoves}). 
        We have chosen (arbitrary) parameters \(p = 19\) and \(e = 36\).
        The genericity is represented by the shade of red, where darker red means more generic. 
        We have marked an (arbitrary) \(2\)-generic point \(x\) which classifies a Galois type via \cref{res:building-classification-of-types}, and in general any Galois type \(\tau\) will be classified by a special vertex for the apartment \(\widetilde{\mathcal{A}}\left( \widehat{T}, \mathbb{F}\llp u \rrp \right)\) (i. e. at a point where the thin gray lines intersect).
    }
    \label{fig:alcove-genericity}
\end{figure}
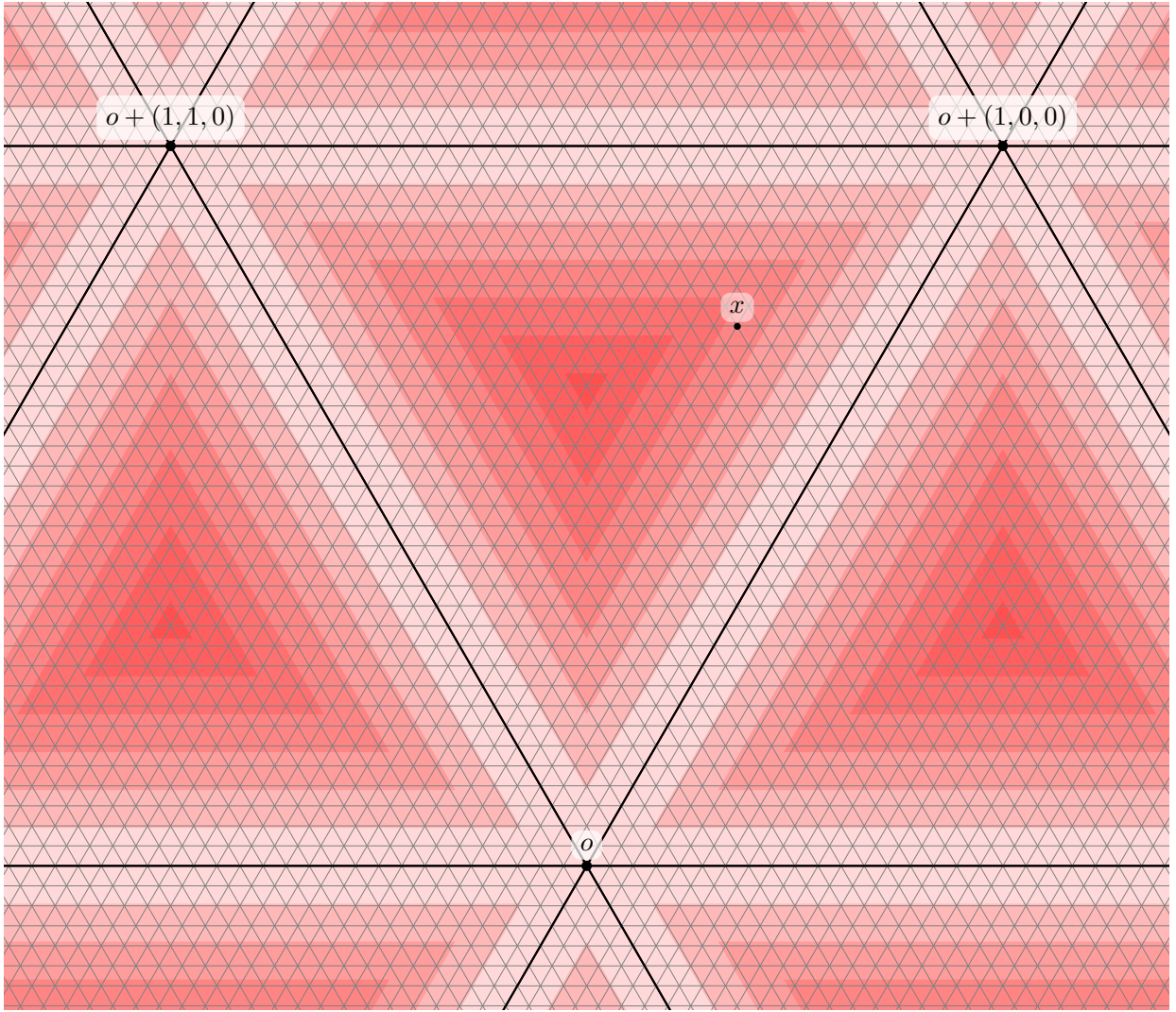

\subsection{Invariant pushforward and \((\Gamma,\widehat{G})\)-torsors}
\label{sec:specialized-equivariant-torsors}
In this section we specialize some of the discussion of \cref{sec:equivariant-torsors}, and in particular \cref{sec:equivariant-torsors-in-algebraic-geometry}, 
to the situation of \((\Gamma,\widehat{G})\)-torsors over \(\widetilde{\mathbb{A}}^1_{\mathcal{O}}\).
For the definition of \((\Gamma,\widehat{G})\)-torsors, we refer the reader to \cref{sec:equivariant-torsors}. 
We will use the notation and assumptions from \cref{sec:rings-from-p-adic-hodge-theory} and \cref{sec:dual-groups}. 

\begin{proposition}
    \label{res:invariant-pushforward-is-smooth}
    Let \(X\) be a \(\Gamma\)-equivariant smooth affine scheme over \(\widetilde{\mathbb{A}}^1_{\mathcal{O}}\). 
    Then \((\pi_{I})_{*}^I X\) is a smooth affine group scheme over \(\breve{\mathbb{A}}^1_{\mathcal{O}}\), 
    and \(\pi_{*}^{\Gamma} X \cong (\pi_{0})_{*}^{\Gamma_{0}}\left( (\pi_{I})_{*}^I X \right)\) is a smooth affine group scheme over \(\mathbb{A}^1_{\mathcal{O}}\).
\end{proposition}
\begin{proof}
    This follows from \cref{res:smooth-invariant-pushforward-in-general}. 
\end{proof}
Note that by descent for affine morphisms, any \((\Gamma,\widehat{G})\)-torsor over \(\widetilde{\mathbb{A}}^1_{R}\) in the fpqc topology is representable by a smooth affine scheme over \(\widetilde{\mathbb{A}}^1_{R}\). 
The following lemma shows that all reasonable definitions of \emph{type} for \((\Gamma,\widehat{G})\)-torsors over \(\widetilde{\mathbb{A}}^1_{R}\) are equivalent and coincide with \cref{def:torsor-type}. 
\begin{proposition}
    \label{res:equivalent-definitions-of-type}
    Let \(R\) be an \(\mathcal{O}\)-algebra, \(\widehat{G}\) a \(\Gamma\)-equivariant smooth affine group scheme over \(\widetilde{\mathbb{A}}^1_{R}\), and let \(\mathcal{P}\) and \(\mathcal{Q}\) be \((\Gamma,\widehat{G})\)-torsors over \(\widetilde{\mathbb{A}}^1_{R}\). 
    Then the following are equivalent:
    \begin{enumerate}
        \item There exists an fpqc covering \(\lbrace V_{i} \to \mathbb{A}^1_{R} \rbrace_{i \in \mathcal{I}}\) and isomorphisms of \((\Gamma,\widehat{G})\)-torsors \(\mathcal{P}_{V_{i} \times_{\mathbb{A}^1_{R}} \widetilde{\mathbb{A}}^1_{R}} \isoto \mathcal{Q}_{V_{i} \times_{\mathbb{A}^1_{R}} \widetilde{\mathbb{A}}^1_{R}}\). 
        \item There exists an étale covering \(\lbrace V_{i} \to \mathbb{A}^1_{R} \rbrace_{i \in \mathcal{I}}\) and isomorphisms of \((\Gamma,\widehat{G})\)-torsors \(\mathcal{P}_{V_{i} \times_{\mathbb{A}^1_{R}} \widetilde{\mathbb{A}}^1_{R}} \isoto \mathcal{Q}_{V_{i} \times_{\mathbb{A}^1_{R}} \widetilde{\mathbb{A}}^1_{R}}\). 
        \item For every geometric point \(\overline{x} : \operatorname{Spec} \overline{F} \to \mathbb{A}^1_{R}\) there exists an isomorphism of \((\Gamma,\widehat{G})\)-torsors \(\mathcal{P}_{\overline{x} \times_{\mathbb{A}^1_{R}} \widetilde{\mathbb{A}}^1_{R}} \isoto \mathcal{Q}_{\overline{x} \times_{\mathbb{A}^1_{R}} \widetilde{\mathbb{A}}^1_{R}}\). 
        \item For every geometric point \(\overline{x} : \operatorname{Spec} \overline{F} \to \operatorname{Spec}(R ) \hookrightarrow \mathbb{A}^1_{R}\) (the zero section) there exists an isomorphism of \((\Gamma,\widehat{G})\)-torsors \(\mathcal{P}_{\overline{x} \times_{\mathbb{A}^1_{R}} \widetilde{\mathbb{A}}^1_{R}} \isoto \mathcal{Q}_{\overline{x} \times_{\mathbb{A}^1_{R}} \widetilde{\mathbb{A}}^1_{R}}\).
        \item There exists an étale covering \(R \to R'\) and isomorphisms of \((\Gamma,\widehat{G})\)-torsors \(\mathcal{P}_{R' \otimes_{\mathbb{Z}_{p}} \mathbb{Z}_{q}} \isoto \mathcal{Q}_{R' \otimes_{\mathbb{Z}_{p}} \mathbb{Z}_{q}}\), where subscript \(R' \otimes_{\mathbb{Z}_{p}} \mathbb{Z}_{q}\) means base change to the zero section \(R \otimes_{\mathbb{Z}_{p}} \mathbb{Z}_{q}\) then to \(R' \otimes_{\mathbb{Z}_{p}} \mathbb{Z}_{q}\). 
        \item (Assuming \(\mathbb{F} \supseteq \mathbb{F}_{q}\).) There exists an étale covering \(R \to R'\), \(j \in \mathcal{J}\), and isomorphisms of \((I,\widehat{G})\)-torsors \(\mathcal{P}_{R'_{j}} \isoto \mathcal{Q}_{R'_{j}}\), where subscript \(R'_{j}\) denotes base change to \(R' \otimes_{\mathbb{Z}_{p}} \mathbb{Z}_{q} = \prod_{j \in \mathcal{J}} R_{j}'\) as in the previous point and then to the \(j\)'th component. 
    \end{enumerate}
\end{proposition}
\begin{proof}
    This is a consequence of \cref{res:equivalent-definitions-of-type-in-general}.
\end{proof}
\begin{remark}
    The analogue of \cref{res:equivalent-definitions-of-type} for \(\mathfrak{S}_{L,R} \to \mathfrak{S}_{R}\) in place of \(\widetilde{\mathbb{A}}^1_{R} \to \mathbb{A}^1_{R}\) is also true, with the same proof.
\end{remark}

We explain how \cref{res:equivalent-definitions-of-type} relates to the study of Breuil--Kisin \((\Gamma,\widehat{G})\)-torsors.
We fix a Galois type \([\tau] \in H^1(\Gamma,\widehat{G}(\mathbb{F}^\mathcal{J})) \cong H^1 (\Gamma,\widehat{G}(\mathcal{O}^\mathcal{J})) \cong H^1(\Gamma,\widehat{G}(\mathfrak{S}_{L,\mathcal{O}}))\), 
and as in \cref{res:equivariant-torsors-and-group-cohomology} this gives rise to a \((\Gamma,\widehat{G})\)-torsor \({^\tau \mathcal{E}^0}\) over \(\mathfrak{S}_{L,R}\) whose underlying \(\widehat{G}\)-torsor is trivial. 
For any \(\mathcal{O}\)-algebra \(R\), the \((\Gamma,\widehat{G})\)-torsors on \(\mathfrak{S}_{L,R}\) of \emph{type \(\tau\)} are those which are equivalent to (the base change of) \({^\tau \mathcal{E}^0}\) in the sense of any of the equivalent conditions of \cref{res:equivalent-definitions-of-type}.
As a consequence of \cref{res:twisted-fundamental-result-about-equivariant-torsors}, the groupoid of \((\Gamma,\widehat{G})\)-torsors over \(\mathfrak{S}_{L,R}\) of type \(\tau\) is equivalent to the groupoid of
\(\mathcal{G}\)-torsors over \(\mathfrak{S}_{R}\), where \(\mathcal{G} = \pi_{*}^\Gamma {_{\tau} \widehat{G}}\) is the invariant pushforward. 
In \cref{sec:bruhat-tits-group-schemes} we will identify \(\mathcal{G}\) with a Bruhat--Tits group scheme over \(\mathfrak{S}_{R}\) in the sense of Pappas--Zhu \cite{pappas-zhu}. 
The Bruhat--Tits group scheme \(\mathcal{G}\) will depend on a choice of \(x\) corresponding to \(\tau\) via \cref{res:building-classification-of-types}.  
\subsection{Assumptions and fixed choices}
\label{sec:setup}
In this section we summarize some of the assumptions and fixed choices of parameters that will be used in the remainder of the paper, excluding the appendices.

Throughout, we will rely on the notation introduced in \cref{sec:notation-and-preliminaries}, in particular \cref{sec:fields-and-rings} and \cref{sec:dual-groups}.
We emphasize that we only deal with \emph{tamely} ramified extensions over \(\mathbb{Q}_{p}\) (no wildly ramified extensions of \(\mathbb{Q}_{p}\) occur in this paper). 

We allow the possibility that \(I\) acts non-trivially on \(\widehat{G}\) (corresponding to the case when the dual group \(G\) over \(\mathbb{Q}_{p}\) is not unramified), 
and we will state whenever we make the assumption that \(I\) acts trivially on \(\widehat{G}\) (corresponding to the case when \(G\) is unramified). 

At various points we need to invoke certain technical assumptions on the group \(\mathcal{G}_{x}\) that we construct in \cref{sec:bruhat-tits-group-scheme-as-invariant-pushforward}. 
From \cref{sec:root-groups} onwards we invoke \cref{ass:connected-fibers-assumptions} which is that \(\mathcal{G}_{x}\) has connected fibers (which is automatic by \cref{res:connected-fibers} when \(\pi_{1}(G^*)_{I}\) from \cref{sec:algebraic-fundamental-group} is torsion-free). 
For any result relying on the negative loop group \(L^{--}\mathcal{G}_{x}\) (introduced in \cref{sec:negative-loop-group}) we invoke \cref{ass:pappas-zhu-groups-as-dilations} (including in  \cref{res:explicit-smooth-equivalence}, which is one of our main results).

\subsubsection{Fixed choices}
\label{sec:galois-type-setup}
We fix a \textit{Galois type}, by which we mean a class \([\tau] \in H^1 \left( \Gamma , \widehat{G}(\mathcal{O}^{\mathcal{J}}\llb u \rrb)\right)\).
We assume that \(\tau\) is \textit{Frobenius invariant}, meaning that \(\varphi[\tau] \sim [\tau]\). 
This is a condition which is satisfied by all Galois types which describe descent data for Breuil--Kisin modules in characteristic zero by \cref{res:galois-types-are-frobenius-invariant} (at least when \(I\) acts trivially). 

We let \(w \in \widehat{N}(\mathbb{Z}^\mathcal{J})\) and \(\lambda \in \left( X_{*}(\widehat{T})^\mathcal{J} \right)^\Gamma \cong X_{*}(\widehat{T})^I\) (where \(\Gamma\) acts on \(X_{*}(\widehat{T})^\mathcal{J}\) as in \eqref{eq:conjugation-action-on-cocharacters})
be elements such that \(n := w^{-1}u^\lambda \in \widehat{N}(\mathbb{Z}^\mathcal{J}[u^{\pm 1}])\) trivializes \(\tau\) in the sense that
\(\tau(\theta) = n^{-1}(^\theta n)\) for all \(\theta \in \Gamma\). 
Up to isomorphism, such \(w,\lambda\) always exists by \cref{res:strictly-frobenius-invariant-coboundary-in-normalizer} (after enlarging \(E\) if necessary).
Moreover, we may assume that \(\tau\) is \emph{strictly} Frobenius invariant (i. e. \(\varphi \tau = \tau\) on the level of \(1\)-cocycles), or equivalently that \(n \varphi(n)^{-1}\) is fixed by \(\Gamma\) (see \cref{rem:strictly-frobenius-invariant-vs-fixed}). 
Note that if the action of \(I\) fixes \(w\), then \(\tau|_{I} = \omega^\lambda\) as explained in \cref{rem:types-in-normalizer-simple}. 

Recall the Chevalley valuation \(o \in \widetilde{\mathcal{A}}\left( \widehat{T}, \mathbb{F}^{\mathcal{J}}\llp u \rrp \right)\), and let 
\[
x := n \cdot o = o - \frac{1}{e} w^{-1} \lambda . 
\]
Note that \(x\) is fixed by \(\Gamma\), since \(\theta(x) = (^\theta n) \cdot \theta(o) = n \tau(\theta) \cdot o = n \cdot o\), where we use that \(\tau(\theta) \in \widehat{G}(\mathbb{F}^\mathcal{J}\llb u \rrb)\) stabilizes \(o\). 

\begin{remark}
    Note that \(n\) being in the normalizer of \(\widehat{T}\) implies that \(\widehat{T} = n \widehat{T} n^{-1}\) is \(\Gamma\)-stable. 
\end{remark}
\begin{remark}
    For readers familiar with \cite[Section 5.1]{local-models}, it might be instructive to consult some of the explicit examples, in particular \cref{sec:weil-restriction-of-GL3}. 
    As pointed out in \cref{rem:local-model-notation-in-weil-restriction-of-GL3}, our \(\lambda\) corresponds to \(\mathbf{a}'\) in \cite[Section 5.1]{local-models}, 
    and similarly our \(w\) corresponds to \(s_{\text{or}}'\) in loc. cit.. 
\end{remark}

The assumption \(\varphi [\tau] = [\tau]\) implies by \cref{res:frobenius-invariant-type-as-x} 
that there exists \(c \in G^*(\mathcal{O}[v^{\pm 1}])\) such that 
\[ c \cdot \varphi(x) = x . \] 
Since our choice of \(\tau\) is strictly Frobenius invariant we may (and will) take \(c = n \varphi(n)^{-1}\). 

\begin{remark}
    Almost no results rely on the fact that \(c = n \varphi(n)^{-1}\), except 
    for \cref{res:moduli-interpretation-as-BK-modules}, which gives the interpretation of the quotient stack \(\mathcal{Y} := \left[ L \mathcal{G}_{x}^{\wedge \varpi} /^{c,\varphi} L^+ \mathcal{G}_{x}^{\wedge \varpi} \right]\) as the moduli stack of Breuil--Kisin \((\Gamma,\widehat{G})\)-torsors of type \(\tau\). 
\end{remark}

We also fix a dominant cocharacter \(\mu \in X_{*}(T^*_{E})^{\text{dom}}\).

\cref{tab:choices} summarizes the most important choices that are fixed throughout the remainder of this work, excluding appendices. 

\begin{table}[h]
    \centering
    \begin{tabular}{| c | c | c |} 
        \hline 
        \(n \in \widehat{N}(\mathcal{O}^\mathcal{J}[u^{\pm 1}])\) & \(n = w^{-1} u^\lambda\) \\
        \hline 
        \(\tau : \Gamma \to \widehat{G}(\mathcal{O}^\mathcal{J}[ u ])\) & \(\tau(\theta) = n^{-1} (^\theta n)\) \\
        \hline 
        \(x \in \widetilde{\mathcal{A}}\left( T^*, \mathbb{F}\llp v \rrp \right)\) & \(x = n \cdot o = o - \frac{1}{e} w^{-1}\lambda\) \\
        \hline 
        \(c \in G^*(\mathcal{O}[v^{\pm 1}])\) & \(c = n \varphi(n)^{-1}\), \(c \cdot \varphi(x) = x\) \\
        \hline 
    \end{tabular}
    \caption{Choices that are fixed for the reminder of the paper. Note that the choice of \(n\) controls all other entries of the table.}
    \label{tab:choices}
\end{table}

\section{Bruhat--Tits Group Schemes and Loop Groups}
\label{sec:bruhat-tits-group-schemes-and-loop-groups}
The purpose of this chapter is to introduce the various Bruhat--Tits group schemes and loop groups which play a role in the study of the stack \(\mathcal{Y}\). 

\subsection{Bruhat--Tits group schemes over \(\mathbb{A}^1_{\mathcal{O}}\)}
\label{sec:bruhat-tits-group-schemes}
In this section we introduce the group scheme \(\mathcal{G}_{x}\) over \(\mathbb{A}^1_{\mathcal{O}}\) and certain related group schemes.
We provide three different ways of thinking about \(\mathcal{G}_{x}\):
\begin{enumerate}
    \item In \cref{sec:bruhat-tits-group-scheme-as-invariant-pushforward} we define \(\mathcal{G}_{x}\) as a certain invariant pushforward, making the connection to Breuil--Kisin \((\Gamma,\widehat{G})\)-torsors transparent (in view of \cref{sec:moduli-of-BK-modules}, as well as \cref{sec:specialized-equivariant-torsors} and \cref{sec:equivariant-torsors}).
    \item In \cref{sec:pappas-zhu-identification}, we identify \(\mathcal{G}_{x}\) with the group schemes of Pappas--Zhu \cite{pappas-zhu}, under the assumption that \(\pi_{1}(\widehat{G})\) is torsion-free. 
    \item In \cref{sec:bruhat-tits-group-scheme-as-dilation}, we show that we can identify \(\mathcal{G}_{x}\) with a dilation of \(\widehat{G}\) in some cases. 
\end{enumerate}
We carefully describe root groups in \cref{sec:root-groups}, as this will be important in later computations. 
In addition, we extend the construction of \cite[Theorem 4.1]{pappas-zhu} in \cref{res:pappas-zhu-concave-function} to obtain groups of the form \(\mathcal{G}_{x,f}\) where \(f\) is a concave function on the root system. 

\begin{remark}
    We mention some related work that we will not address further. 
    In addition to \cite{pappas-zhu}, the paper \cite{levinLocalModelsWeilrestricted2016} constructs the groups \(\mathcal{G}_{x}\) for possibly wildly ramified Weil restrictions, 
    and related Pappas--Zhu local models.  
    As the group scheme of Pappas--Zhu is a Bruhat--Tits group scheme over a 2-dimensional base, the paper \cite{balajiBruhatTitsTheoryHigher2024} also seems relevant, as it develops a theory of Bruhat--Tits group schemes over higher dimensional bases (although they focus on the case of multivariate power series rings over fields). 
\end{remark}

\subsubsection{\(\mathcal{G}\) as invariant pushforward}
\label{sec:bruhat-tits-group-scheme-as-invariant-pushforward}
View \(\widehat{G}, \widehat{B}, \widehat{T}, e\) as a Chevalley group
scheme over \(\widetilde{\mathbb{A}}^1_{\mathcal{O}}\) with the diagonal action
of the group \(\Gamma\). 
Let \(n = w^{-1}u^\lambda \in \widehat{N}(\mathcal{O}^\mathcal{J}[u^{\pm 1}])\), \(\tau\), and \(x\) be as in \cref{sec:galois-type-setup}.
Consider the groups over \(\widetilde{\mathbb{A}}^1_{\mathcal{O}}\) given by\footnote{
    Suppose \(\widehat{G} = \operatorname{Spec} \Lambda\) and \(\widehat{G}[u^{-1}] = \operatorname{Spec}\Lambda[u^{-1}]\), where \(\widehat{G}[u^{-1}]:= \widehat{G}_{\widetilde{\mathbb{A}}^1_{\mathcal{O}} \setminus \lbrace 0 \rbrace}\) and \(\Lambda[u^{-1}] := \Lambda \otimes_{\mathcal{O}^\mathcal{J}[u]} \mathcal{O}^\mathcal{J}[u^{\pm 1}]\). 
    Then the automorphism \(\operatorname{Ad}_{n} : \widehat{G}[u^{-1}] \isoto \widehat{G}[u^{-1}]\) corresponds to an automorphism \(\operatorname{Ad}_{n}^* : \Lambda[u^{-1}] \isoto \Lambda[u^{-1}]\). 
    What we mean by \(n \widehat{G}n^{-1}\) is then the spectrum of the subring \(\operatorname{Ad}_{n}^* (\Lambda) \subset \Lambda[u^{-1}]\).
    Similar remarks apply to \(n \widehat{B} n^{-1}\) and \(n \widehat{T} n^{-1}\). 
} 
\begin{align*}
\widetilde{\mathcal{G}}_{x} & := n \widehat{G} n^{-1}, \\
\widetilde{\mathcal{B}}_{x} & := n \widehat{B} n^{-1} , \\
\widetilde{\mathcal{T}}_{x} & := n \widehat{T} n^{-1} = \widehat{T}.
\end{align*}
We will often suppress the subscript \(x\). 
Over \(\widetilde{\mathbb{A}}^1_{\mathcal{O}}\) we have an isomorphism of
\(\Gamma\)-equivariant group schemes
\begin{align}
    \label{eq:tau-twisted-equivariant-group-scheme}
    \operatorname{Ad}_{n} : {_{\tau}}\widehat{G} & \isoto \widetilde{\mathcal{G}}, \\
    g & \mapsto n g n^{-1} \nonumber 
\end{align}
where \({_{\tau}}\widehat{G}\) denotes the group scheme \(\widehat{G}\) with the \(\tau\)-twisted
\(\Gamma\)-action
\(\theta \cdot g = \tau(\theta) (^{\theta}g)\tau(\theta)^{-1}\).
We have a natural identification
\(\widetilde{\mathcal{G}}[u^{-1}] \cong \widehat{G}[u^{-1}]\) over
\(\widetilde{\mathbb{A}}^1_{\mathcal{O}}\setminus \left\{ 0 \right\}\), via
which we can identify the \(\mathcal{O}^\mathcal{J}[u]\)-points of
\(\widetilde{\mathcal{G}}\) as 
\[
\widetilde{\mathcal{G}}\left( \mathcal{O}^\mathcal{J}[u] \right) = n \widehat{G}\left( \mathcal{O}^\mathcal{J}[u] \right) n^{-1} \subset \widehat{G} \left( \mathcal{O}^\mathcal{J}[u^{\pm 1}] \right) = \widetilde{\mathcal{G}} \left( \mathcal{O}^\mathcal{J}[u^{\pm 1}] \right) .
\] 
More generally, we have similar identifications on \(R\)-points,
where \(R\) is any \(\mathcal{O}^\mathcal{J}[u]\)-algebra in which \(u\) is not a zerodivisor.

Note that 
\(\widetilde{\mathcal{G}}, \widetilde{\mathcal{B}}, \widetilde{\mathcal{T}}\)
are all stable under the \(I\)-action since \(\tau(\gamma) = \omega(\gamma)^\lambda\). 
We define
\begin{align*}
\breve{\mathcal{G}}_{x} & := (\pi_{I})_{*}^{I} \widetilde{\mathcal{G}} , \\
\breve{\mathcal{B}}_{x} & := (\pi_{I})_{*}^{I} \widetilde{\mathcal{B}} , \\
\breve{\mathcal{T}} & := (\pi_{I})_{*}^{I} \widetilde{\mathcal{T}} .
\end{align*}
These are smooth affine group schemes over \(\breve{\mathbb{A}}^1_{\mathcal{O}}\) by \cref{res:invariant-pushforward-is-smooth}.

The group schemes \(\breve{\mathcal{G}}_{x}\) and \(\breve{\mathcal{T}}\) have residual actions of
\(\Gamma_{0} = \Gamma / I\) which are equivariant over the base
\(\breve{\mathbb{A}}^{1}_{\mathcal{O}}\). 
We can then let
\begin{align*}
\mathcal{G}_{x} & := (\pi_0)_{*}^{\Gamma_{0}} \breve{\mathcal{G}}_{x} \cong \pi_{*}^{\Gamma} \widetilde{\mathcal{G}} , \\
\mathcal{T} & := (\pi_0)_{*}^{\Gamma_{0}}\breve{\mathcal{T}} \cong \pi_{*}^{\Gamma}\widetilde{\mathcal{T}} .
\end{align*}
Again, these are smooth affine group schemes over \(\mathbb{A}^1_{\mathcal{O}}\) by \cref{res:invariant-pushforward-is-smooth}. 

As a consequence of \eqref{eq:tau-twisted-equivariant-group-scheme} and \cref{res:twisted-fundamental-result-about-equivariant-torsors}, we obtain: 
\begin{corollary}
    \label{res:G-as-invariant-pushforward-corollary}
    For any \(\mathcal{O}\)-algebra \(R\) we have equivalences of groupoids  
    \begin{align*} 
        \lbrace (\Gamma,\widehat{G})\text{-torsors over }\widetilde{\mathbb{A}}^1_{R}\text{ of type }\tau \rbrace 
        \simeq & 
        \lbrace (\Gamma, \widetilde{\mathcal{G}}_{x})\text{-torsors over }\widetilde{\mathbb{A}}^1_{R}\text{ of trivial type}\rbrace \\
        \simeq &
        \lbrace (\Gamma_{0},\breve{\mathcal{G}}_{x})\text{-torsors over }\breve{\mathbb{A}}^1_{R} \rbrace \\
        \simeq &
        \lbrace \mathcal{G}_{x}\text{-torsors over }\mathbb{A}^1_{R} \rbrace ,
    \end{align*}
    where being of type \(\tau\) can be interpreted in terms of \cref{res:equivalent-definitions-of-type}.
\end{corollary}
\begin{remark}
    These equivalences can be made explicit, by inspecting the ingredients going into \cref{res:twisted-fundamental-result-about-equivariant-torsors}. 
    The first equivalence is given by the associated bundle construction using \({_{\tau}\mathcal{E}^0}\), the trivial \(\widehat{G}\) torsor with \(\tau\)-twisted \(\Gamma\)-action, 
    and the remaining equivalences are given by invariant pushforward. 
    Note that these equivalences are functorial in \(R\). 
\end{remark}
\begin{remark}
    There is an analogous result of the above in which \(\widetilde{\mathbb{A}}^1_{R}\) is replaced with \(\mathfrak{S}_{L,R}\), \(\breve{\mathbb{A}}^1_{R}\) is replaced with \(\operatorname{Spec}\left( R \otimes_{\mathbb{Z}_{p}} \mathbb{Z}_{q}\right)\llb v \rrb\), and \(\mathbb{A}^1_{R}\) is replaced with \(\mathfrak{S}_{R}\).
    This is the case of interest in the study of Breuil--Kisin \((\Gamma,\widehat{G})\)-torsors, as in \cref{sec:moduli-of-BK-modules}. 
\end{remark}

\subsubsection{\(\mathcal{G}\) as the group scheme of Pappas--Zhu}
\label{sec:pappas-zhu-identification}
The goal of this section is to identify \(\mathcal{G}_{x}\) as defined above with the group schemes of \cite[Theorem 4.1]{pappas-zhu}, or more precisely give a criterion for when we can do so. 
That is, we will prove the following result: 
\begin{proposition}
    \label{res:pappas-zhu-identification-basic}
    Assume that \(\pi_{1}(G^*)_{I}\) is torsion-free. 
    Then \(\mathcal{G}_{x}\) can be identified with the Bruhat--Tits group scheme for \(G^*\) over \(\mathbb{A}^1_{\mathcal{O}}\) attached to \(x\) in the sense of \cite[Theorem 4.1]{pappas-zhu}.
    More precisely, we have first of all identifications of buildings 
    \[
        \widetilde{\mathcal{B}}\left( G^*_{\mathbb{F}\llp v \rrp}, \mathbb{F}\llp v \rrp \right) = \widetilde{\mathcal{B}}\left( G^*_{E\llp v \rrp}, E\llp v \rrp \right) .
    \]
    with respect to which we will view \(x\) interchangeably as a member of any one of the buildings.  
    Now \(\mathcal{G}_{x}\) is a smooth affine group scheme over \(\mathbb{A}^1_{\mathcal{O}}\) with connected fibers, satisfying the following properties:
    \begin{enumerate}
        \item \(\mathcal{G}_{x}[v^{-1}] = G^*\);
        \item The base change of \(\mathcal{G}_{x}\) to \(E\llb v \rrb\) or \(\mathbb{F}\llb v \rrb\) (along the obvious maps) is the parahoric group scheme for \(G^*_{E\llp v \rrp}\) or \(G^*_{\mathbb{F}\llp v \rrp}\) (respectively) attached to \(x\). 
     \end{enumerate}
     Moreover, \(\mathcal{G}_{x}\) is the unique smooth affine group scheme over \(\mathbb{A}^1_{\mathcal{O}}\) with connected fibers for which \(\mathcal{G}_{x}[v^{-1}] = G^*\) and \(\mathcal{G}_{x}(E\llb v \rrb) \subset G^*(E\llp v \rrp)\) is the parahoric subgroup determined by \(x\). 
\end{proposition}
The proof of this result will be given at the end of the subsection, and as we will see the assumption on \(\pi_{1}(G^*)_{I}\) is used to ensure connectedness of fibers of \(\mathcal{G}_{x}\).
The claim about identifications of buildings is covered in \cite{pappas-zhu}, so we will not provide a complete proof of this fact. 
We will use the terminology ``Bruhat--Tits group scheme in the sense of Pappas--Zhu'' multiple times in the remainder of this subsection without spelling out all the required properties as we have done in \cref{res:pappas-zhu-identification-basic}.

We will begin our discussion over \(\widetilde{\mathbb{A}}^1_{\mathcal{O}}\), then over \(\breve{\mathbb{A}}^1_{\mathcal{O}}\), and finally over \(\mathbb{A}^1_{\mathcal{O}}\).
In this subsection, we will for \(j \in \mathcal{J}\) let \(\widetilde{\mathcal{G}}_{j}, \widetilde{\mathcal{B}}_{j}, \widetilde{\mathcal{T}}_{j}\)
denote the base changes of \(\widetilde{\mathcal{G}}, \widetilde{\mathcal{B}}, \widetilde{\mathcal{T}}\) to \(\mathcal{O}_{j}[u]\), 
and let \(\breve{\mathcal{G}}_{j}, \breve{\mathcal{B}}_{j}, \breve{\mathcal{T}}_{j}\) denote the base changes of \(\breve{\mathcal{G}}, \breve{\mathcal{B}}, \breve{\mathcal{T}}\) to \(\mathcal{O}_{j}[v]\). 

As a first step towards understanding \cref{res:pappas-zhu-identification-basic}, we begin by discussing identifications of the apartments from \cref{sec:buildings} with analogously defined apartments over \(E^\mathcal{J}\llp u \rrp\). 
Following \cite[Section 4.1]{pappas-zhu} we can for each \(j \in \mathcal{J}\) identify \(\widetilde{\mathcal{A}}\left( \widehat{T}, \mathbb{F}_{j}\llp u \rrp \right) = \widetilde{\mathcal{A}}\left( \widehat{T}, E_{j}\llp u \rrp \right)\), and therefore identify 
\[
    \widetilde{\mathcal{A}}\left( \widehat{T}, \mathbb{F}^\mathcal{J}\llp u \rrp \right) = \widetilde{\mathcal{A}}\left( \widehat{T}, E^\mathcal{J}\llp u \rrp \right) . 
\]
This identifies the Chevalley valuation \(o\) on the left hand side with the obvious Chevalley valuation \(o\) on the right hand side, 
and it is compatible with the obvious actions of \(\widehat{N}(\mathcal{O}^\mathcal{J}\llp u \rrp)\) on the left and right hand sides, 
as well as the \(\Gamma\)-actions.
By taking \(\Gamma\)-fixed points, we have a resulting identification \(\widetilde{\mathcal{A}}\left( T^*_{\mathbb{F}\llp v \rrp}, \mathbb{F}\llp v \rrp \right) = \widetilde{\mathcal{A}}\left( T^*_{E\llp v \rrp}, E \llp v \rrp \right)\).

Note that the point \(x \in \widetilde{\mathcal{A}}\left( \widehat{T}, \mathbb{F}^\mathcal{J} \llp u \rrp \right)\) can be interpreted as a point of \(\widetilde{\mathcal{A}}\left( \widehat{T}, E^\mathcal{J}\llp u \rrp \right)\), 
and \(x_{j} \in \widetilde{\mathcal{A}} \left( \widehat{T}, \mathbb{F}_{j} \llp u \rrp \right)\) as a point of \(\widetilde{\mathcal{A}}\left( \widehat{T}, E_{j} \llp u \rrp \right)\). 

\begin{proposition}
    Let \(j \in \mathcal{J}\). 
    Then \(\widetilde{\mathcal{G}}_{j}\) is the Bruhat--Tits group scheme for \(\widehat{G}_{j}\) attached to \(x_{j}\) in the sense of Pappas--Zhu. 
\end{proposition}
\begin{proof}
    We note that \(\widehat{G}(\mathcal{O}_{j}\llp u \rrp)\) acts on the building \(\widetilde{\mathcal{B}}\left( \widehat{G}, E_{j}\llp u \rrp \right)\)
    by restricting the action of \(\widehat{G}(E_{j}\llp u \rrp)\) via the natural map \(\widetilde{\mathcal{G}}_{j}(\mathcal{O}_{j}\llp u \rrp) = \widehat{G}(\mathcal{O}_{j}\llp u \rrp) \hookrightarrow \widehat{G}(E_{j}\llp u \rrp)\). 
    Since \(\widetilde{\mathcal{G}}_{j}(E_{j}\llb u \rrb) = n_{j}\widehat{G}(E_{j}\llb u \rrb)n_{j}^{-1}\) and the action of \(\widehat{G}(E_{j}\llp u \rrp)\) on the building is compatible with the conjugation on subgroups,
    it follows that \(\widetilde{\mathcal{G}}_{j} \times_{\mathcal{O}_{j}[u]} \operatorname{Spec}E_{j}\llb u \rrb\) is the parahoric group scheme attached to the hyperspecial vertex \(x_{j} = n_{j} \cdot o_{j}\). 
    The claim then follows from the uniqueness part of the proof of \cite[Theorem 4.1]{pappas-zhu} (in Section 4.2.1).
\end{proof}

Since the \(I\)-covering 
\(\pi_{I} : \widetilde{\mathbb{A}}^1_{\mathcal{O}} \to \breve{\mathbb{A}}^{1}_{\mathcal{O}}\) restricts to an 
étale \(I\)-torsor away from 0, we see that the fibers of
\(\breve{G} := \breve{\mathcal{G}}_{x}[v^{-1}]\) are connected reductive groups. On
the other hand, the fibers of \(\breve{\mathcal{G}}\) and \(\breve{\mathcal{T}}\)
lying over the zero section \(v=0\) may be disconnected.
This turns out to be the only obstruction to identifying \(\breve{\mathcal{G}}\) with the group schemes of Pappas--Zhu. 
Before proceeding, we give a criterion for connectedness of fibers.
\begin{proposition}
    \label{res:connected-fibers}
    If \(\pi_{1}(G^*)_{I}\) (see \cref{sec:algebraic-fundamental-group} for the definition) is torsion-free, then \(\breve{\mathcal{G}}_{x}\) has connected fibers.
\end{proposition}
\begin{proof}
    Since \(I\) acts diagonally on \(\mathcal{O}^\mathcal{J} [u] = \prod_{j \in \mathcal{J}} \mathcal{O}_{j}[u]\) we may without loss of generality assume that \(\# \mathcal{J} = 1\) and ignore the \(\mathcal{J}\) in what follows. 
    By the observations above, it suffices to show that the fibers of \(\breve{\mathcal{G}}_{x}\) over \(v = 0 : \operatorname{Spec}\mathcal{O} \hookrightarrow \breve{\mathbb{A}}^1_{\mathcal{O}}\) are connected, 
    which by \cite[\href{https://stacks.math.columbia.edu/tag/055J}{Lemma 055J}]{stacks-project} are connected if and only if the special fiber over \(\operatorname{Spec}\mathbb{F} \hookrightarrow \operatorname{Spec}\mathcal{O}\) is connected.

    Let \(\widetilde{\mathcal{H}} = \widetilde{\mathcal{G}}_{x} \times \operatorname{Spec} \mathbb{F} \llb u \rrb\), which is a parahoric group scheme for \(\widehat{G}_{\mathbb{F}\llp u \rrp}\) attached to \(x\). 
    Let \(\pi' : \operatorname{Spec} \mathbb{F}\llb u \rrb \to \operatorname{Spec} \mathbb{F}\llb v \rrb\) sending \(v\) to \(u^e\), and let \(\mathcal{H} = (\pi')_{*}^I \widetilde{\mathcal{H}}\).
    Note that \(\mathcal{H}\) is identified with the fiber of \(\breve{\mathcal{G}}_{x}\) over \(\mathbb{F} \llb v \rrb\) since \(\mathbb{F}\llb u \rrb \cong \mathcal{O}[u] \otimes_{\mathcal{O}[v]}\mathbb{F}\llb v \rrb\), 
    so we have to show that the special fiber of \(\mathcal{H}\) is connected.
    For this question, we may and will replace \(\mathbb{F}\) by its algebraic closure. 
    In particular, we can then identify \(\breve{G}_{\mathbb{F}\llp v \rrp} = G^*_{\mathbb{F}\llp v \rrp}\). 

    We will now use the notation \(\widehat{G}(\mathbb{F}\llp u \rrp)^1, \widehat{G}(\mathbb{F}\llp u \rrp)^0, G^*(\mathbb{F}\llp v \rrp)^1, G^*(\mathbb{F}\llp v \rrp)^0\) as in \cite{kaletha-prasad}. 
    It follows from \cite[Lemma 2.6.16 (4)]{kaletha-prasad} that 
    \begin{equation}
        \label{eq:descent-for-superscript-1}
        (\widehat{G}(\mathbb{F}\llp u \rrp)^1)^I = \widehat{G}(\mathbb{F}\llp u \rrp)^1 \cap \widehat{G}(\mathbb{F}\llp u \rrp)^I = G^*(\mathbb{F}\llp v \rrp)^1 . 
    \end{equation}

    Let \(x'\) denote the projection of \(x\) (which is a point of the enlarged building) to the restricted building \(\mathcal{B}(\widehat{G},\mathbb{F}\llp v \rrp) \subset \widetilde{\mathcal{B}}(\widehat{G}, \mathbb{F}\llp u \rrp)\). 
    By \cite[Lemma 7.7.10, Proposition 7.7.11]{kaletha-prasad} the subgroup \(\widetilde{\mathcal{H}}(\mathbb{F}\llb u \rrb) \subset \widehat{G}(\mathbb{F}\llp u \rrp)\) is the stabilizer of \(x'\) in \(\widehat{G}(\mathbb{F}\llp u \rrp)^1\).
    So by \eqref{eq:descent-for-superscript-1}, it follows that 
    \[
        \mathcal{H}(\mathbb{F}\llb v \rrb) = \widetilde{\mathcal{H}}(\mathbb{F}\llb u \rrb)^I = \widetilde{\mathcal{H}}(\mathbb{F}\llb u \rrb) \cap G^*(\mathbb{F}\llp v \rrp)^1
    \]
    is the stabilizer of \(x'\) in \(G^*(\mathbb{F}\llp v \rrp)^1\), i. e. \(\mathcal{H}(\mathbb{F}\llb v \rrb) = G^*(\mathbb{F}\llp v \rrp)^1_{x'}\).
    Since \(\mathcal{H}\) is a smooth integral model for \(G^*_{\mathbb{F}\llp v \rrp}\), it follows by \cite[Corollary 2.10.11, Axiom 4.1.20]{kaletha-prasad}
    that \(\mathcal{H}\) has connected fibers if and only if \(G^*(\mathbb{F}\llp v \rrp)^1_{x'} = G^*(\mathbb{F}\llp v \rrp)^0_{x'}\).

    The assumption that \(\pi_{1}(G^*)_{I}\) is torsion free implies that \(G^*(\mathbb{F}\llp v \rrp)^0 = G^*(\mathbb{F} \llp v \rrp)^1\) by \cite[Corollary 11.6.2]{kaletha-prasad}. 
    In particular, it is implied that we have an equality of stabilizers \(G^*(\mathbb{F}\llp v \rrp)^1_{x'} = G^*(\mathbb{F}\llp v \rrp)^0_{x'}\), 
    so \(\mathcal{H}\) has connected fibers.
\end{proof}
\begin{remark}
    It is natural to ask whether \cref{res:connected-fibers} has a converse. 
    Let us say a few word about this in the case that \(x\) is \(0\)-generic.
    We will use the assumptions and notation from the proof of \cref{res:connected-fibers} without comment. 

    Let \(\mathcal{C}\) denote the alcove (called chamber in \cite{kaletha-prasad}) containing \(x'\), and let \(G^*(\mathbb{F}\llp v \rrp)_{\mathcal{C}}\) denote the stabilizer of \(\mathcal{C}\) in \(G^*(\mathbb{F}\llp v \rrp)\) (but not necessarily pointwise stabilizer). 
    Define \(G^*(\mathbb{F}\llp v \rrp)^\dag_{\mathcal{C}} := G^*(\mathbb{F}\llp v \rrp)_{\mathcal{C}} \cap G^*(\mathbb{F}\llp v \rrp)^1\) (which is consistent with \cite[Section 7.7]{kaletha-prasad}), 
    and \(G^*(\mathbb{F}\llp v \rrp)^0_{\mathcal{C}} := G^*(\mathbb{F}\llp v \rrp)_{\mathcal{C}} \cap G^*(\mathbb{F}\llp v \rrp)^0\).
    Since the action of \(G^*(\mathbb{F}\llp v \rrp)\) on \(\mathcal{B}(G^*, \mathbb{F} \llp v \rrp)\) is via simplicial automorphisms, we have 
    \(G^*(\mathbb{F}\llp v \rrp)^1_{x'} \subseteq G^*(\mathbb{F}\llp v \rrp)^\dag_{\mathcal{C}}\) (this is where we use that \(x'\) is \(0\)-generic).
    On the other hand we have \(G^*(\mathbb{F} \llp v \rrp)^0_{x'} = G^*(\mathbb{F}\llp v \rrp)^0_{\mathcal{C}}\) by \cite[Lemma 7.4.4, Proposition 7.6.4]{kaletha-prasad} (see also \cite[Remark 7.7.7]{kaletha-prasad}). 
    It follows from \cite[Proposition 11.6.1, Corollary 11.6.2]{kaletha-prasad} that the Kottwitz homomorphism induces an isomorphism 
    \(G^*(\mathbb{F}\llp v \rrp)^\dag_{\mathcal{C}} / G^*(\mathbb{F}\llp v \rrp)^0_{\mathcal{C}} \cong \pi_{1}(G^*)_{I,\text{tor}}\). 
    Combining the above, we obtain a short exact sequence 
    \begin{equation}
        \label{eq:connected-fibers-SES}
        1 \to G^*(\mathbb{F}\llp v \rrp)^1_{x'} / G^*(\mathbb{F}\llp v \rrp)^0_{x'} \to \pi_{1}(G^*)_{I,\text{tor}} \to G^*(\mathbb{F}\llp v \rrp)^\dag_{\mathcal{C}} / G^*(\mathbb{F}\llp v \rrp)^1_{x'} \to 1 . 
    \end{equation}
    As in the proof of \cref{res:connected-fibers}, connectivity of the fibers of \(\breve{\mathcal{G}}_{x}\) is equivalent to having an equality \(G^*(\mathbb{F}\llp v \rrp)^1_{x'} = G^*(\mathbb{F}\llp v \rrp)^0_{x'}\).
    The exact sequence \eqref{eq:connected-fibers-SES} shows that \(\breve{\mathcal{G}}_{x}\) may have connected fibers even if \(\pi_{1}(G^*)_{I,\text{tor}} \neq 0\), but this discrepancy is accounted for 
    by the existence of elements \(g \in G^*(\mathbb{F}\llp v \rrp)^1\) which stabilize \(\mathcal{C}\) but not \(x'\). 
\end{remark}
\begin{example}
    Suppose \(\widehat{G} = \GL_{n}\) with trivial \(\Gamma\)-action, in which case we have \(G^* = \widehat{G} = \GL_{n}\).
    In this case, \cref{res:connected-fibers} applies. 
    We can see this by noting that \(\pi_{1}(\widehat{G})_{I} = \pi_{1}(\GL_{n}) = \mathbb{Z}\). 
    Alternatively, by \cref{res:simply-connected-derived-group} \(\pi_{1}(\widehat{G})_{I}\) being torsion free is equivalent to \(\widehat{G}_{\text{der}} = \operatorname{SL}_{n}\) being simply connected. 
\end{example}
\begin{example}
    \label{ex:PGL3-fibers}
    Let \(\widehat{G} = \operatorname{PGL}_{3}\), with \(\Gamma\) acting trivially on \(\widehat{G}\), so again \(G^* = \widehat{G} = \operatorname{PGL}_{3}\). 
    Then we have \(\pi_{1}(\widehat{G})_{I} = \pi_{1}(\widehat{G}) = \mathbb{Z} / 3\), and therefore \cref{res:connected-fibers} does not apply. 
    For the convenience of the reader we briefly recall the computation \(\pi_{1}(\widehat{G}) = \mathbb{Z} / 3\). 
    According to \cite[Definition 11.3.2]{kaletha-prasad}, \(\pi_{1}(\widehat{G}) = X_{*}(\widehat{T}) / X_{*}(\widehat{T}')\), where \(\widehat{T} \subset \operatorname{PGL}_{3}\) is the diagonal torus, and \(\widehat{T}' \subset \operatorname{SL}_{3}\) is the diagonal torus.
    Corresponding to the identification of the cocharacters of the diagonal torus of \(\GL_{3}\) with \(\mathbb{Z}^3\), we have \(X_{*}(\widehat{T}') = \mathbb{Z}^3 / \mathbb{Z}\) (the quotient by the ``diagonal'' \(\mathbb{Z}\)), and \(X_{*}(\widehat{T}') = \lbrace (a_{1},a_{2},a_{3}) \in \mathbb{Z}^3 | a_{1} + a_{2} + a_{3} = 0 \rbrace \subset \mathbb{Z}^3\). 
    It is now straightforward to verify that the map \(X_{*}(\widehat{T}) / X_{*}(\widehat{T}') \isoto \mathbb{Z}/ 3\) given by the sum, i. e. \((a_{1},a_{2},a_{3}) \mapsto a_{1} + a_{2} + a_{3}\), is an isomorphism. 
    
    Although we can't use \cref{res:connected-fibers}, we can still reason about connectivity of fibers of \(\breve{\mathcal{G}}_{x}\) for various possible \(x\). 
    We will do this shortly. 
    As in the proof of \cref{res:connected-fibers}, we work over a fixed connected component corresponding to a choice of \(j \in \mathcal{J}\).

    We remark that the \emph{Kottwitz homomorphism} \(\kappa : \operatorname{PGL}_{3}(\mathbb{F}\llp v \rrp) \to \pi_{1}(\operatorname{PGL}_{3}) = \mathbb{Z} / 3\) is realized as the composition of the determinant \(\operatorname{det} : \operatorname{PGL}_{3}(\mathbb{F}\llp v \rrp) \to \mathbb{F}\llp v \rrp^\times / \left( \mathbb{F} \llp v \rrp^\times \right)^3\) followed by the valuation map \(\operatorname{val} : \mathbb{F}\llp v \rrp^\times / \left( \mathbb{F} \llp v \rrp^\times \right)^3 \to \mathbb{Z} / 3\mathbb{Z}\).
    Since \(\operatorname{PGL}_{3}(\mathbb{F}\llp v \rrp)^0 = \ker \kappa\) by \cite[Proposition 11.5.4]{kaletha-prasad}, we have a fairly explicit description of \(\operatorname{PGL}_{3}(\mathbb{F}\llp v \rrp)^0\);
    for example, \(\operatorname{diag}(1,v,v) \notin \operatorname{PGL}_{3}(\mathbb{F}\llp v \rrp)^0\), because \(\kappa(g) = 2\).
    We also have \(\operatorname{PGL}_{3}(\mathbb{F}\llp v \rrp)^1 = \kappa^{-1}\left( \pi_{1}(\operatorname{PGL}_{3})_{I,\text{tor}} \right) = \operatorname{PGL}_{3}(\mathbb{F}\llp v \rrp)\) by \cite[Lemma 11.5.2]{kaletha-prasad}.

    Following the proof of \cref{res:connected-fibers}, connectivity of the fibers of \(\breve{\mathcal{G}}_{x}\) boils down to whether 
    \begin{equation}
        \label{eq:PGL3-equation}
        \operatorname{PGL}_{3}(\mathbb{F}\llp v \rrp)^0_{x} = \operatorname{PGL}_{3}(\mathbb{F}\llp v \rrp)^1_{x} = \operatorname{PGL}_{3}(\mathbb{F}\llp v \rrp)_{x}
    \end{equation}
    We consider a few values of \(x\):
    \begin{enumerate}
        \item \(x = o\) the given Chevalley valuation. Then \eqref{eq:PGL3-equation} is satisfied by \cite[Lemma 7.7.10, Proposition 7.7.11]{kaletha-prasad}, so the fibers of \(\breve{\mathcal{G}}_{x}\) are connected. 
        \item \(x = o + (1/3, 0, -1/3)\). Now \(x\) is the barycenter of the alcove \(\mathcal{C}\) containing it. Therefore, any \(g \in \operatorname{PGL}_{3}(\mathbb{F} \llp v \rrp)\) which stabilizes \(\mathcal{C}\) also stabilizes \(x\).
            So \(\operatorname{PGL}_{3}(\mathbb{F}\llp v \rrp)^1_{x} = \operatorname{PGL}_{3}(\mathbb{F}\llp v \rrp)_{x} = \operatorname{PGL}_{3}(\mathbb{F}\llp v \rrp)_{\mathcal{C}} = \operatorname{PGL}_{3}(\mathbb{F}\llp v \rrp)_{\mathcal{C}}^\dag\). 
            Using the short exact sequence \eqref{eq:connected-fibers-SES} (which holds by the same proof since \(x\) is \(0\)-generic) it follows that \(\operatorname{PGL}_{3}(\mathbb{F}\llp v \rrp)^1_{x} / \operatorname{PGL}_{3}(\mathbb{F}\llp v \rrp)^0_{x} \cong \mathbb{Z} / 3\), so in particular \eqref{eq:PGL3-equation} is not satisfied, and \(\breve{\mathcal{G}}_{x}\) has disconnected fibers. 
        \item Let \(\mathcal{C}\) be as in (2), and let \(x\) be any point of \(\mathcal{C}\) which is different from the barycenter.
            Consider the element 
            \[
                g = \begin{pmatrix} 0 & 0 & v^{-1} \\ 1 & 0 & 0 \\ 0 & 1 & 0 \end{pmatrix} \in \operatorname{PGL}_{3}(\mathbb{F}\llp v \rrp) . 
            \]
            Then the action of \(g\) on \(\mathcal{A}(\operatorname{PGL}_{3}, \mathbb{F}\llp v \rrp)\) is a rotation through \(2 \pi / 3\) centered at the barycenter of \(\mathcal{C}\). 
            Therefore, \(g \in \operatorname{PGL}_{3}(\mathbb{F}\llp v \rrp)_{\mathcal{C}} = \operatorname{PGL}_{3}(\mathbb{F}\llp v \rrp)_{\mathcal{C}}^\dag\), but since \(x\) is not the barycenter, \(g \notin \operatorname{PGL}_{3}(\mathbb{F}\llp v \rrp)^1_{x}\). 
            In fact, \(g\) generates a subgroup of \(\operatorname{PGL}_{3}(\mathbb{F}\llp v \rrp)^\dag_{\mathcal{C}} / \operatorname{PGL}_{3}(\mathbb{F}\llp v \rrp)^1_{x}\) of order 3, 
            so the map \(\pi_{1}(\operatorname{PGL}_{3}) \to \operatorname{PGL}_{3}(\mathbb{F}\llp v \rrp)^\dag_{\mathcal{C}} / \operatorname{PGL}_{3}(\mathbb{F}\llp v \rrp)^1_{x}\) from \eqref{eq:connected-fibers-SES} must be an isomorphism,
            and therefore \(\operatorname{PGL}_{3}(\mathbb{F}\llp v \rrp)^0 = \operatorname{PGL}_{3}(\mathbb{F}\llp v \rrp)^1\).
            So in this case, the fibers of \(\breve{\mathcal{G}}_{x}\) are connected.
    \end{enumerate}
    Observe that with \(x\) as in (2), the element \(g\) from (3) generates \(\operatorname{PGL}_{3}(\mathbb{F}\llp v \rrp)_{x} / \operatorname{PGL}_{3}(\mathbb{F}\llp v \rrp)^0_{x}\).
\end{example}

Having discussed connectedness of fibers at length, we return to the issue of relating our group schemes to those of Pappas--Zhu.
Recall that \(\breve{G} = (\pi_{I})_{*}^I \widehat{G}\) is a connected reductive group scheme over \(\breve{\mathbb{A}}^1_{\mathcal{O}}\),
For \(j \in \mathcal{J}\) let \(\breve{G}_{j}\) denote the base change to the \(j\)'th component \(\mathcal{O}_{j}[v^{\pm 1}]\), 
and recall that \(\breve{\mathcal{G}}_{j}\) denotes the base change to the \(j\)'th component \(\mathcal{O}_{j}[v]\). 

\begin{proposition}
    \label{res:pappas-zhu-identification}
    Assume that \(\breve{\mathcal{G}}\) has connected fibers (which holds if \(\pi_{1}(G^*)_{I}\) is torsion-free by \cref{res:connected-fibers}).
    \begin{enumerate}
        \item Then for each \(j\in \mathcal{J}\) the group scheme \(\breve{\mathcal{G}}_{j}\) is the Bruhat--Tits group scheme for \(G^*_{j}\) over \(\mathcal{O}_{j}[v]\) attached to \(x_{j} \in \widetilde{\mathcal{B}}\left( G^*_{j}, \mathbb{F}_{j}\llp v \rrp \right)\) in the sense of Pappas--Zhu.
        \item \(\mathcal{G}_{x}\) is the Bruhat--Tits group scheme for \(G^*\) over \(\mathbb{A}^1_{\mathcal{O}}\) attached to \(x\) in the sense of Pappas--Zhu.
    \end{enumerate}
\end{proposition}
\begin{proof}
    By the proof of \cite[Theorem 4.1]{pappas-zhu}, the group scheme over \(\mathcal{O}_{j}[v]\) for \(\widehat{G}_{E_{j}}\) attached to \(x_{j}\) is the netural connected component of \(\breve{\mathcal{G}}\). 
    Claim (1) is then obvious.
    Since \(\pi_{0} : \breve{\mathbb{A}}^1_{\mathcal{O}} \to \mathbb{A}^1_{\mathcal{O}}\) is an étale \(\Gamma_{0}\)-torsor, \(\mathcal{G}_{x} = (\pi_{0})_{*}^{\Gamma_{0}} \breve{\mathcal{G}}_{x}\) is a smooth affine group scheme with connected fibers. 
    Moreover, \(\mathcal{G}_{x}[v^{-1}] = G^*\) by construction and \(\mathcal{G}_{x}(E\llb v \rrb) = \breve{\mathcal{G}}_{x}(E^\mathcal{J}\llb v \rrb)^{\Gamma_{0}} = \breve{\mathcal{G}}_{x}(E^\mathcal{J}\llb v \rrb) \cap G^*(\mathbb{F}\llp v \rrp) \subset \breve{G}(\mathbb{F}^\mathcal{J}\llp v \rrp)\) is 
    the parahoric subgroup attached to \(x\) by part (1). 
    Claim (2) then follows from the uniqueness part of the proof of \cite[Theorem 4.1]{pappas-zhu}. 
\end{proof}
\begin{proof}[{Proof of \cref{res:pappas-zhu-identification-basic}}]
    Immediate from \cref{res:connected-fibers} and \cref{res:pappas-zhu-identification}. 
\end{proof}

\subsubsection{\(\mathcal{G}\) as dilations}
\label{sec:bruhat-tits-group-scheme-as-dilation}
In some cases, the group schemes of Pappas--Zhu coincide with certain dilations, see e. g. \cite[Example 3.3]{mayeuxNeronBlowupsLowdegree2023}.
The approach via dilations also used in \cite{local-models} and \cite{leeEmertonGeeStacks2023}. 
In \cref{rem:pappas-zhu-dilation} below we describe conditions ensuring that the group schemes from \cref{sec:bruhat-tits-group-scheme-as-invariant-pushforward} are dilations.

We need the following definition. 
\begin{definition}
    \label{def:lowest-alcove}
    A point \(x \in \widetilde{\mathcal{A}}\left( T^*, \mathbb{F}\llp v \rrp \right)\) is said to be \emph{lowest alcove} if 
    \[
        0 \leq \langle a , x - o \rangle < 1 
    \]
    for all \(a \in \Phi(G^*,T^*)^+\), or equivalently, if via \eqref{eq:coroots-apartment-iso-j} \(x\) corresponds to an element \(\beta \in X_{*}(T^*)\) with \(-1 < \langle a, \beta \rangle \leq 0 \rbrace\). 
\end{definition}
\begin{remark}
    Note that \(x\) is \(0\)-generic lowest alcove if and only if \(0 < \langle a, x - o \rangle < 1\) for all \(a \in \Phi(G^*,T^*)^+\). 
\end{remark}
\begin{remark}
    In \cref{def:lowest-alcove}, we have non-strict inequality \(0 \leq \langle a, x - o \rangle\) but strict inequality \(\langle a, x - o \rangle < 1\). 
    This is intentional, because it is what we need in \cref{rem:pappas-zhu-dilation}.
    Note that there is a unique vertex which is lowest alcove, namely the Chevalley valuation \(o\). 
\end{remark}
\begin{remark}
    \label{rem:lowest-alcove-assumption-is-mild}
    Although the group scheme \(\mathcal{G}_{x}\) depends on \(x\) in absolute terms, for our purposes (studying the stack \(\mathcal{Y}\) which does not depend on the choice of \(x\)) we are free to replace \(x\) by any point in its \(G^*(\mathbb{F}\llp v \rrp)\)-orbit by \cref{res:building-classification-of-types}. 
    Since \(G^*(\mathbb{F}\llp v \rrp)\) acts transitively on alcoves, it is often the case that the orbit of \(x\) intersects the lowest alcove. 
    For example, this is always true when \(x\) is \(0\)-generic. 
    On the other hand, this is not true if \(x \neq o\) is a vertex in the closure of the lowest alcove and \(\widehat{G}\) is a semisimple and simply connected group such as \(\operatorname{SL}_{n}\). 
\end{remark}

\begin{proposition}
    \label{rem:pappas-zhu-dilation}
    Assume that \(I\) acts trivially on \(\widehat{G}\), that \(\mathbb{F} \supseteq \mathbb{F}_{q}\), \(\mathcal{G}_{x}\) has connected fibers (e. g. if \(\pi_{1}(\widehat{G})\) is torsion-free by \cref{res:connected-fibers}), and \(x\) is lowest alcove. 
    Then \(\mathcal{G}_{x}\) is isomorphic to the dilation of \(\widehat{G}\) (as a group scheme over \(\mathbb{A}^1_{\mathcal{O}}\)) in a parabolic subgroup \(\widehat{P}_{x} \subset \widehat{G}\) along the zero section \(v = 0\). 
    This parabolic \(\widehat{P}_{x}\) is generated by \(\widehat{T}\) and root groups \(U_{a}\) where \(\langle a, x-o \rangle > 0\).\footnote{
        This \(\widehat{P}_{x}\) can also be described using the dynamic method as \(\underline{P}_{\widehat{G}}(x - o)\) in the notation of \cite[Theorem 4.1.7]{Conrad2014ReductiveGS}. 
    }
    In particular, if \(x\) is \(0\)-generic, then \(\widehat{P}_{x} = \widehat{B}\) is the Borel subgroup.
\end{proposition}
\begin{proof}
    Since \(I\) acts trivially on \(\widehat{G}\) and \(\mathbb{F} \supseteq \mathbb{F}_{q}\), we have by \cref{rem:G-star-is-G-hat} that \(G^* = \widehat{G}\), 
    and since \(\mathcal{G}_{x}\) has connected fibers, \(\mathcal{G}_{x}\) is the group scheme of Pappas--Zhu by \cref{res:pappas-zhu-identification}.
    Now let \(\widehat{P}_{x}\) be as in the statement of the proposition, and define \(\mathcal{G}_{x}'\) as the dilation of \(\widehat{G}\) in \(\widehat{P}_{x}\) along \(v = 0\).
    Then \(\mathcal{G}_{x}'\) is a smooth affine group scheme over \(\mathbb{A}^1_{\mathcal{O}}\) with connected fibers by \cite[Theorem 3.2]{mayeuxNeronBlowupsLowdegree2023}, 
    and \(\mathcal{G}_{x}'(E \llb v \rrb) \subset \widehat{G}(E \llp v \rrp)\) consists of elements which belong to \(\widehat{P}_{x}\) modulo \(v\). 
    By interpreting \(x\) as a valuation of the root datum as in \cref{sec:valuations-of-the-root-datum}, it is straightforward to check that \(\mathcal{G}_{x}'(E \llb v \rrb)\)
    is exactly the parahoric subgroup attached to \(x\) (and this is where the lowest alcove assumption is used). 
    The result then follows from the proof of uniqueness in \cite[Theorem 4.1]{pappas-zhu}.
\end{proof}
\begin{remark}
    The use of the term ``parabolic'' in \cref{rem:pappas-zhu-dilation} is slightly abusive, because we allow the possibility that \(x = o\) in which case \(\widehat{G} = \widehat{P}_{x}\).
\end{remark}

\subsubsection{Some assumptions}
By default, our \(\mathcal{G}_{x}\) is defined via an invariant pushforward as in \cref{sec:bruhat-tits-group-scheme-as-invariant-pushforward}. 
In most of what follows, we impose assumptions on \(\mathcal{G}_{x}\) to ensure that we can identify \(\mathcal{G}_{x}\) with the group schemes of Pappas--Zhu or as a dilation.
Let us state these assumptions precisely. 
\begin{AssumptionCon}\refstepcounter{AssumptionCon}
    \label{ass:connected-fibers-assumptions}
    We assume that \(\mathcal{G}_{x}\) has connected fibers.
\end{AssumptionCon}
Unless otherwise stated, we will assume \cref{ass:connected-fibers-assumptions} from now on. 

It follows from the discussion of \cref{sec:pappas-zhu-identification}
that \(\mathcal{G}_{x}\) has connected fibers if and only if \(\breve{\mathcal{G}}_{x}\) has connected fibers, 
which is always satisfied when \(\pi_{1}(G^*)_{I}\) is torsion-free. 

We will also have the following assumption, which is stronger than \cref{ass:connected-fibers-assumptions}. 
\begin{AssumptionDil}\refstepcounter{AssumptionDil}
    \label{ass:pappas-zhu-groups-as-dilations}
    We assume that we are in the situation of \cref{rem:pappas-zhu-dilation}. 
    That is, there is an explicitly defined parabolic \(\widehat{P}_{x} \subset \widehat{G}\) (where we also allow \(\widehat{P}_{x} = \widehat{G}\))
    such that \(\mathcal{G}_{x}\) is the dilation of \(\widehat{G}\) in \(\widehat{P}_{x}\) along the zero section \(v = 0\).
\end{AssumptionDil}
By \cref{rem:pappas-zhu-dilation}, this assumption is satisfied when \(I\) acts trivially on \(\widehat{G}\) (corresponding to the assumption that \(G\) is unramified), 
\(\mathbb{F} \supset \mathbb{F}_{q}\), \(\mathcal{G}_{x}\) has connected fibers, and \(x\) is lowest alcove in the sense of \cref{def:lowest-alcove} (which is actually mild, see \cref{rem:lowest-alcove-assumption-is-mild}).
Under \cref{ass:pappas-zhu-groups-as-dilations}, we will write \(\widehat{P} = \widehat{P}_{x}\), 
we will write \(\widehat{P} = \widehat{M} \widehat{U}\) for the Levi decomposition, and \(\widehat{U}^{\text{op}} \subset \widehat{G}\) for the unipotent subgroup opposite to \(\widehat{U}\) (all defined over \(\mathcal{O}\)). 

The above \cref{ass:pappas-zhu-groups-as-dilations} will be used in the construction of the negative loop group in \cref{sec:negative-loop-group}, 
and consequently any result using the negative loop group will also impose this assumption. 

\subsubsection{Root groups}
\label{sec:root-groups}
We are now going to describe certain root groups of \(\widetilde{\mathcal{G}}\) which are defined over \(\widetilde{\mathbb{A}}^1_{\mathcal{O}}\). 
Under the assumption that \(I\) acts trivially on \(\widehat{G}\) we will also describe root groups of \(\breve{\mathcal{G}}\) which are defined over \(\breve{\mathbb{A}}^1_{\mathcal{O}}\). 

For each root \(a \in \Phi \left( \widehat{G}, \widehat{T} \right)\) we 
have a root group \(\widetilde{U}_{a} \subset \widehat{G}\), which we view as being over \(\widetilde{\mathbb{A}}^1_{\mathcal{O}}\). 
The pinning of \(\widehat{G}\) means that
we have a chosen identification
\(\widetilde{U}_{a} \cong \mathbb{G}_{a} = \operatorname{Spec} \mathcal{O}^\mathcal{J}[u,t]\),
and for any \(n \in \mathbb{Z}\) it is then natural to identify 
\[
u^{n} \widetilde{U}_{a} = \operatorname{Spec} \mathcal{O}^\mathcal{J}[u, u^{-n} t] ,
\] where
\(\mathcal{O}^\mathcal{J}[u, u^{-n} t] \subset \operatorname{Frac}\left( \mathcal{O}^\mathcal{J}[u,t] \right)\)
is a subring. 
More generally, we allow \(n \in \mathbb{Z}^\mathcal{J}\), in which case we allow the following abuse of notation:
\[
    \mathcal{O}^\mathcal{J}[u, u^{-n}t] := \prod_{j \in \mathcal{J}} \mathcal{O}_{j}[u, u^{-n_{j}}t] ,
\]
and consequently 
\[u^n \widetilde{U}_{a} = \operatorname{Spec} \mathcal{O}^\mathcal{J}[u, u^{-n}t] = \bigsqcup_{j \in \mathcal{J}} \operatorname{Spec} \mathcal{O}_{j}[u, u^{-n_{j}}t] .\]
The same remarks apply to \(\breve{U}_{a}\), for dilations with respect to powers of \(v\). 

The root group \(\widetilde{U}_{a} \subset \widehat{G}\) yields a root group 
\[
\widetilde{\mathcal{U}}_{a,x} := n \widetilde{U}_{wa} n^{-1} = u^{w^{-1}\lambda} \widetilde{U}_{a} u^{- w^{-1}\lambda} = u^{\langle a, w^{-1} \lambda \rangle} \widetilde{U}_{a} = u^{\langle w a, \lambda \rangle} \widetilde{U}_{a} 
\]
for \(\widetilde{\mathcal{G}}_{x}\).

Now assume that \(I\) acts trivially on \(\widehat{G}\), in which case we identify \(\breve{G}:= (\pi_{I})_{*}^I \widehat{G}_{\widetilde{\mathbb{A}}^1_{\mathcal{O}}\setminus \lbrace 0\rbrace} = \breve{\mathcal{G}}_{x}[v^{-1}]\) with \(\widehat{G}_{\breve{\mathbb{A}}^1_{\mathcal{O}}\setminus \lbrace 0 \rbrace}\). 
Hence \(\breve{G}\) admits a canonical extension to \(\breve{\mathbb{A}}^1_{\mathcal{O}}\), namely \(\widehat{G}_{\breve{\mathbb{A}}^1_{\mathcal{O}}}\). 
The root group \(U_{a} \subset \widehat{G}\) then yields a root group \(\breve{U}_{a} \subset \widehat{G}_{\breve{\mathbb{A}}^1_{\mathcal{O}}}\) over \(\breve{\mathbb{A}}^1_{\mathcal{O}}\). 
For each \(a \in \Phi(\widehat{G},\widehat{T})\), we define 
\[
\breve{\mathcal{U}}_{a,x} := (\pi_{I})_{*}^I \left( \widetilde{\mathcal{U}}_{a,x} \right) . 
\]
We then have 
\[
\breve{\mathcal{U}}_{a,x} = \left( u^{\langle a,w^{-1}\lambda \rangle} \widetilde{U}_{a} \right)^I = v^{\lceil \frac{1}{e} \langle a, w^{-1}\lambda \rangle \rceil} \breve{U}_{a} = v^{\lceil - \langle a, x - o \rangle \rceil} \breve{U}_{a} .
\]
\begin{remark}
    One can define ``root groups'' for \(\breve{\mathcal{G}}_{x}\) more generally when 
    \(I\) does not act trivially on \(\widehat{G}\), analogously to the description of root groups for quasi-split connected reductive groups. 
    See for example \cite[Section 5.1]{zhuCoherenceConjecturePappas2014}. 
\end{remark}

\subsubsection{The open cell}
\label{sec:open-cell}
The map 
\begin{equation}
    \label{eq:open-cell}
    \widetilde{\mathcal{V}}_{x} := \prod_{a \in \Phi^{-}} \widetilde{\mathcal{U}}_{a,x} \times \widetilde{\mathcal{T}} \times \prod_{a \in \Phi^{+}} \widetilde{\mathcal{U}}_{a,x} \hookrightarrow \widetilde{\mathcal{G}}_{x}
\end{equation}
is an open immersion, and
\(\widetilde{\mathcal{V}}_{x} \subset \widetilde{\mathcal{G}}_{x}\) is
fiberwise dense \cite[2.2.10, 3.9.4]{bruhat-tits-II} (see also \cite[Proof of Theorem 4.1]{pappas-zhu}).
If \(I\) acts trivially on \(\widehat{G}\), we define 
\[
\breve{\mathcal{V}}_{x} := \widetilde{\pi}_{*}^{I} \left( \widetilde{\mathcal{V}}_{x} \right) = \prod_{a \in \Phi^{-}} \breve{\mathcal{U}}_{a,x} \times \breve{\mathcal{T}} \times \prod_{a \in \Phi^{+}} \breve{\mathcal{U}}_{a,x} .
\] 
By \cite[Proposition A.5.2, Proposition A.8.10]{CGP} (see also \cref{res:invariant-pushforward-is-smooth})
it follows from \eqref{eq:open-cell} that the natural map \(\breve{\mathcal{V}}_{x} \hookrightarrow \breve{\mathcal{G}}_{x}\) is also an open immersion. 

\subsubsection{Root groups attached to concave functions}
\label{sec:root-groups-attached-to-concave-functions}

For any \(m \in \widetilde{\mathbb{R}}\) (notation as in \cite[Section 1.6]{kaletha-prasad})
we also have the root group 
\[
\widetilde{\mathcal{U}}_{a,x,m} := u^{\lceil - e \langle a, x - o \rangle + em  \rceil} \widetilde{U}_{a} = u^{\langle w a, \lambda \rangle + \lceil e m \rceil} \widetilde{U}_{a} .
\] Clearly
\(\widetilde{\mathcal{U}}_{a,x,0} = \widetilde{\mathcal{U}}_{a,x}\).

Let
\(f : \widehat{\Phi}\left( \widehat{G},\widehat{T} \right) \to \widetilde{\mathbb{R}}\) be
a concave function \cite[Definition 7.3.1]{kaletha-prasad}.
We allow the possibility \(f(0) > 0\) which corresponds to a valuation
condition on the torus, which we now describe. 
Given \(m' \in \widetilde{\mathbb{R}}_{\geq 0}\), we define 
\[
    \mathbb{G}_{m, m'} := \operatorname{Spec} \mathcal{O}^\mathcal{J}\left[ u, t^{-1}, u^{- \lceil e m' \rceil}(t-1) \right] ,
\]
and for the split torus
\(\widetilde{\mathcal{T}} = \mathbb{G}_{m}^{r}\) we define
\(\widetilde{\mathcal{T}}_{m'} = \mathbb{G}_{m,m'}^r\). 
Going back to the concave function \(f\), we define 
\[
    \widetilde{\mathcal{V}}_{x,f} := \prod_{a \in \Phi^{-}} \widetilde{\mathcal{U}}_{a,x,f(a)} \times \widetilde{\mathcal{T}}_{f(0)} \times \prod_{{a} \in \Phi ^{+} } \widetilde{\mathcal{U}}_{a,x,f(a)} .
\]

Now assume that \(I\) acts trivially on \(\widehat{G}\).
Then we define 
\[
\breve{\mathcal{U}}_{a,x,f} := (\pi_{I})_{*}^{I} \left( \widetilde{\mathcal{U}}_{a,x, f(a)} \right) = v^{\lceil - \langle a, x- o \rangle + f(a) \rceil } \breve{U}_{a} . 
\]
We also define 
\[
\breve{\mathcal{T}}_{m} := \widetilde{\pi}_{*}^{I}\left( \widetilde{\mathcal{T}}_{m} \right)  .
\] 
Again, since \(I\) acts trivially we can actually identify \(\breve{\mathcal{T}}_{m'} = \mathbb{G}_{m, \lceil m' \rceil}^r\), where \(r\) is the rank of \(\widehat{T}\). 
Indeed, we can verify this as follows when \(\widehat{T} = \mathbb{G}_{m}\):
\[ \breve{\mathcal{T}}_{m}\left( \mathcal{O}^{\mathcal{J}}\llb v \rrb \right) = \left( \widetilde{\mathcal{T}}_{m}\left( \mathcal{O}^{\mathcal{J}}\llb u \rrb \right)  \right)^{I} = \left( 1 + u^{\lceil e m \rceil} \mathcal{O}^{\mathcal{J}}\llb u \rrb \right)^{I} = 1 + v^{\lceil m \rceil}\mathcal{O}^{\mathcal{J}}\llb v \rrb . \]
Finally, we define 
\begin{align*}
\breve{\mathcal{V}}_{x,f} & := (\pi_{I})_{*}^{I} \left( \widetilde{\mathcal{V}}_{x,f} \right) = \prod_{a \in \Phi^{-}} \breve{\mathcal{U}}_{a,x,f} \times \breve{\mathcal{T}}_{f(0)} \times \prod_{a \in \Phi^{+}} \breve{\mathcal{U}}_{a,x,f} .
\end{align*}

\subsubsection{The group schemes \(\mathcal{G}_{x,f}\)}
We can construct group schemes \(\mathcal{G}_{x,f}\) over \(\mathbb{A}^1_{\mathcal{O}}\) for a concave function \(f\), extending \cite[Theorem 4.1]{pappas-zhu}
in our situation. 

\begin{proposition}
    \label{res:pappas-zhu-concave-function}
    Assume that \(f \geq 0\). Then we can construct a smooth affine group scheme \(\mathcal{G}_{x,f}\) over \(\mathbb{A}^1_{\mathcal{O}}\) with connected fibers with the following properties:
    \begin{enumerate}
        \item The group scheme \(\mathcal{G}_{x,f}|_{\mathbb{A}^1_{\mathcal{O}} \setminus \lbrace 0 \rbrace}\) is the group scheme \(G^*\). 
        \item The base change of \(\mathcal{G}_{x,f}\) to \(\mathbb{F}\llb v \rrb\) or \(E \llb v \rrb\) is the smooth integral model of \(G^*_{\mathbb{F}\llb v \rrb}\) or 
            \(G^*_{E\llb v \rrb}\) (respectively) attached to \(x\) and \(f\) (in the sense of \cref{sec:valuations-of-the-root-datum} and \cite[Theorem 8.5.2]{kaletha-prasad}). 
    \end{enumerate}
    Moreover, \(\mathcal{G}_{x,f}\) is the unique smooth affine group scheme over \(\mathbb{A}^1_{\mathcal{O}}\) with connected fibers such that \(\mathcal{G}_{x,f}[v^{-1}] = G^*\) 
    and \(\mathcal{G}_{x,f}(E\llb v \rrb) \subset G^*(E\llp v \rrp)\) is the subgroup which is denoted \(G^*(E\llp v \rrp)_{x,f}\) in \cite[Definition 7.3.3]{kaletha-prasad}. 
\end{proposition}
\begin{proof}
    The uniqueness follows from the proof of uniqueness in \cite[Theorem 4.1]{pappas-zhu}. 
    As the proof there shows, \(\mathcal{G}_{x,f}\) is already uniquely characterized by (1) and the \(E\) part of (2).

    For the remainder of the proof, we show the analogous statement for \(\breve{\mathcal{G}}_{x,f}\) over \(\breve{\mathbb{A}}^1_{\mathcal{O}}\), from which the statement will follow by taking invariant pushforward.
    To simplify notation we assume that \(\# \mathcal{J}= 1\) and ignore the \(\mathcal{J}\).
    We also carry out the proof under the assumption that \(I\) acts trivially on \(\widehat{G}\), so that the root groups of \(\breve{\mathcal{G}}_{x,f}\) are defined. 
    But as in \cite[Theorem 4.1]{pappas-zhu}, the general case follows by arguing instead with \(\widetilde{\mathcal{G}}_{x,f}\) and passing to the netural connected component invariant pushforward (although the passage to neutral connected components is likely redundant by \cref{ass:connected-fibers-assumptions}). 

    By the induction argument used in the proof of \cite[Theorem 4.1]{pappas-zhu}, based on \cite{Yu2002SMOOTHMA},\footnote{
        In fact, the motivation of \cite{Yu2002SMOOTHMA} appears to have been the construction of smooth integral models like \(\mathcal{G}_{x,f}\), where we allow \(f\geq 0\) to be a concave function. 
        The proof of \cite[Theorem 8.5.2]{kaletha-prasad} is also based on Yu's method. 
    }
    it suffices to show that if \(0 \leq f \leq f' \leq f+1\) then we can describe \(\breve{\mathcal{G}}_{x,f'}\) as a dilation of \(\breve{\mathcal{G}}_{x,f}\). 
    More precisely, our induction hypothesis is that:
    \begin{enumerate}[(A)]
        \item \(\breve{\mathcal{G}}_{x,f}\) satisfies the properties of the proposition, 
        \item \(\breve{\mathcal{V}}_{x,f} \hookrightarrow \breve{\mathcal{G}}_{x,f}\) is open, and dense in geometric fibers, 
    \end{enumerate}
    The base case with \(f(0) = 0\) is taken care of by the proof of \cite[Theorem 4.1]{pappas-zhu}, so we will assume that \(f'(0) > 0\).

    Let \(\overline{\mathcal{G}}_{x,f}\) denote the base change of \(\breve{\mathcal{G}}_{x,f}\) along the section \(v = 0 : \operatorname{Spec}\mathcal{O} \hookrightarrow \breve{\mathbb{A}}^1_{\mathcal{O}}\), and similarly for other related schemes over \(\breve{\mathbb{A}}^1_{\mathcal{O}}\). 
    Let \(Q\) be the set theoretic image of \(\overline{\mathcal{V}}_{x,f'} \to \overline{\mathcal{G}}_{x,f}\). 
    The crux of the proof is to show that \(Q \subset \overline{\mathcal{G}}_{x,f}\) is a closed subgroup, from which it will follow similarly to the argument in \cite[177-178]{pappas-zhu}
    that \(\mathcal{G}_{x,f'}\) is the dilation of \(\mathcal{G}_{x,f}\) in \(Q\) along \(v=0\). 

    Let \(F \in \lbrace \overline{E}, \overline{\mathbb{F}} \rbrace\), and note that we have a commutative diagram 
    \[
        \begin{tikzcd}
        \breve{\mathcal{V}}_{x,f'}(F \llb v \rrb) \arrow[r,twoheadrightarrow] \arrow[d,hookrightarrow] & \overline{\mathcal{V}}_{x,f'}(F) \arrow[d] \\ 
        \breve{\mathcal{G}}_{x,f}(F \llb v \rrb) \arrow[r,twoheadrightarrow] & \overline{\mathcal{G}}_{x,f}(F) ,
        \end{tikzcd}
    \]
    where the hooked arrows are injective and the two-headed arrows are surjective (by smoothness). 
    The assumption \(f'(0) > 0\) implies that \(\breve{\mathcal{V}}_{x,f'}(F\llb v \rrb) \subset \breve{\mathcal{G}}_{x,f}(F\llb v \rrb)\) is a subgroup by \cite[Proposition 7.3.12]{kaletha-prasad}. 
    Therefore the points underlying \(Q\) are closed under the group operations. 

    Observe that \(\breve{\mathcal{U}}_{a,x,f'} \hookrightarrow \breve{\mathcal{U}}_{a,x,f}\) is a dilation (or Néron blowup \cite[Definition 3.1]{mayeuxNeronBlowupsLowdegree2023}) for each \(a \in \widehat{\Phi}(\widehat{G},\widehat{T})\),\footnote{
        The hat over \(\widehat{\Phi}\) means that we allow \(a = 0\), in which case \(\breve{\mathcal{U}}_{a,x,m} = \breve{\mathcal{T}}_{m}\). 
    } 
    using the explicit descriptions from \cref{sec:root-groups-attached-to-concave-functions} and the assumption \(f' \leq f + 1\). 
    More precisely, it is the dilation in the identity section along \(v = 0\) if there occurs a jump in the filtration at \(a\), and the trivial dilation otherwise. 
    Since dilations are compatible with products it follows that \(\breve{\mathcal{V}}_{x,f'} \subset \breve{\mathcal{V}}_{x,f}\) is also a dilation.
    In fact, it is precisely the dilation of \(\breve{\mathcal{V}}_{x,f}\) in \(Q\) along \(v=0\). 

    By making the previous paragraph explicit we can see that the image \(Q_{a}\) of \(\overline{\mathcal{U}}_{a,x,f'} \to \overline{\mathcal{U}}_{a,x,f}\) is closed for each \(a \in \widehat{\Phi}(\widehat{G},\widehat{T})\).
    It follows that \(Q' = \prod_{a< 0} Q_{a} \times Q_{0} \times \prod_{a>0} Q_{a} \subseteq \overline{\mathcal{V}}_{x,f}\) is also closed.
    Note that \(Q\) is now the image of the multiplication map \(Q' \subset \overline{\mathcal{V}}_{x,f} \hookrightarrow \overline{\mathcal{G}}_{x,f}\).
    If \(f(0) > 0\), then \(\overline{\mathcal{V}}_{x,f} \cong \overline{\mathcal{G}}_{x,f}\) by (B) and \cite[Proposition 7.3.12]{kaletha-prasad}, so \(Q\) is closed in \(\overline{\mathcal{G}}_{x,f}\) by the above. 
    But if \(f(0) = 0\) we need an additional argument.
    
    Let \(\widetilde{Q}\) be the closure of \(Q\) in \(\overline{\mathcal{G}}_{x,f}\).
    Then \(Q_{E} = \widetilde{Q}_{E}\) and \(Q_{\mathbb{F}}\) is the neutral connected component of \(\widetilde{Q}_{\mathbb{F}}\), since 
    the fibers of \(Q\) are closed by \cite[I Proposition 1.3]{borelLinearAlgebraicGroups1991}, and \(Q\) is open in \(\widetilde{Q}\) since \(Q\) is closed in \(\overline{\mathcal{V}}_{x,f}\)
    (which is open in \(\overline{\mathcal{G}}_{x,f}\) by (B)). 
    It therefore suffices to show that the fibers of \(\widetilde{Q}\) are geometrically connected, but this follows from the argument in \cite[177]{pappas-zhu}.
    This completes the proof that \(Q\) is a closed subgroup of \(\overline{\mathcal{G}}_{x,f}\). 
    
    Now define \(\breve{\mathcal{G}}_{x,f'}\) as the dilation of \(\breve{\mathcal{G}}_{x,f}\) in \(Q\) along \(v=0\). 
    We have to check that \(\breve{\mathcal{G}}_{x,f'}\) satisfies properties (A), (B). 
    General properties of Néron blowups \cite[Theorem 3.2]{mayeuxNeronBlowupsLowdegree2023} dictate that \(\breve{\mathcal{G}}_{x,f'}\) is a smooth affine group scheme over \(\breve{\mathbb{A}}^1_{\mathcal{O}}\) with connected fibers, and \(\breve{\mathcal{G}}_{x,f'}[v^{-1}] = \breve{\mathcal{G}}_{x,f}[v^{-1}] = G^*\).
    Note that \(\breve{\mathcal{V}}_{x,f'}\) is the dilation of \(\breve{\mathcal{V}}_{x,f}\) in \(Q\) along \(v=0\). 
    Since \(\breve{\mathcal{V}}_{x,f} \hookrightarrow \breve{\mathcal{G}}_{x,f}\) is open and dense in geometric fibers, it follows by \cite[Corollary 2.8]{mayeuxNeronBlowupsLowdegree2023}
    that so is \(\breve{\mathcal{V}}_{x,f'} \hookrightarrow \breve{\mathcal{G}}_{x,f'}\), verifying property (B). 
    In fact, the universal property of dilations \cite[Proposition 2.6]{mayeuxNeronBlowupsLowdegree2023} and the fact that \(Q \subset \overline{\mathcal{V}}_{x,f}\) imply that if \(R\) is a \(\mathcal{O}[v]\)-algebra in which \(v\) is not a zerodivisor and which is \(v\)-adically complete,\footnote{
        We can relax the assumption that \(R\) is \(v\)-adically complete -- it suffices that (the image of) \(v\) is contained in the Jacobson radical of \(R\). 
        The crucial part is that the only open subset of \(\operatorname{Spec}R\) containing \(\operatorname{Spec}R / vR\) is \(\operatorname{Spec}R\) itself.
        Since \(\breve{\mathcal{V}}_{x,f} \subset \breve{\mathcal{G}}_{x,f}\) is open by (B), it then follows that \(\operatorname{Spec}R \to \breve{\mathcal{G}}_{x,f}\) factors through \(\breve{\mathcal{V}}_{x,f}\) 
        if and only if \(\operatorname{Spec}R / vR \to \operatorname{Spec}R \to \breve{\mathcal{G}}_{x,f}\) factors through \(\breve{\mathcal{V}}_{x,f}\). Consequently, we have \(\breve{\mathcal{V}}_{x,f}(R) = \overline{\mathcal{V}}_{x,f}(R/vR) \times_{\overline{\mathcal{G}}_{x,f}(R/vR)} \breve{\mathcal{G}}_{x,f}(R)\), so 
        \(\breve{\mathcal{V}}_{x,f'}(R) = Q(R/vR) \times_{\overline{\mathcal{V}}_{x,f}(R/vR)} \breve{\mathcal{V}}_{x,f}(R) = Q(R/vR) \times_{\overline{\mathcal{G}}_{x,f}(R/vR)} \breve{\mathcal{G}}_{x,f}(R) = \breve{\mathcal{G}}_{x,f'}(R)\).
    }
    then \(\breve{\mathcal{V}}_{x,f'}(R) = \breve{\mathcal{G}}_{x,f'}(R)\).
    In particular, for \(F \in \lbrace E, \mathbb{F} \rbrace\) we have by definition of \(\breve{\mathcal{V}}_{x,f'}\) and \cite[Proposition 7.3.12]{kaletha-prasad} that \(\breve{\mathcal{V}}_{x,f'}(F\llb v \rrb) \subset \widehat{G}(F\llp v \rrp)\) is the subgroup attached to \(x\) and \(f\). 
    This shows that property (2) holds, finalizing the proof of (A). 
\end{proof}
\begin{remark}
    We will use the notation \(\breve{\mathcal{G}}_{x,f}\) and \(\widetilde{\mathcal{G}}_{x,f}\) which are to \(\breve{\mathcal{G}}_{x}\) and \(\widetilde{\mathcal{G}}_{x}\), respectively, what \(\mathcal{G}_{x,f}\) is to \(\mathcal{G}_{x}\). 
    In fact, \(\breve{\mathcal{G}}_{x,f}\) is exactly as in the proof above. 
\end{remark}
We point out a few special cases of \cref{res:pappas-zhu-concave-function}.
In fact, these are more or less the only cases we will need. 
\begin{example}
    \label{ex:congruence-pappas-zhu-group-scheme}
    Suppose \(f = n\). For any \(\mathcal{O}[v]\)-algebra \(R\) we then have 
    \[
        \mathcal{G}_{x,n}(R) = \ker \left( \mathcal{G}_{x}(R) \to \mathcal{G}_{x}(R / v^n R) \right) . 
    \]
    To see this, note that the right hand side is a smooth affine group scheme over \(\mathbb{A}^1_{\mathcal{O}}\) with connected fibers by the proof of \cref{res:congruence-loop-group-lemma}. 
    It has the required property over \(E \llb v \rrb\) by \cite[Proposition 8.5.16]{kaletha-prasad}, so by the uniqueness in \cref{res:pappas-zhu-concave-function} it must agree with \(\mathcal{G}_{x,n}\).
\end{example}
\begin{example}
    \label{ex:dilation-in-unipotent-radical}
    Impose \cref{ass:pappas-zhu-groups-as-dilations}. 
    Recall that under this assumption, the group scheme \(\mathcal{G}_{x} = \mathcal{G}_{x,0}\) is the dilation of \(\widehat{G}\) in \(\widehat{P}_{x}\) along \(v = 0\). Then for \(f = 0^+\), the group scheme \(\mathcal{G}_{x,0^+}\) is the dilation of \(\widehat{G}\) in \(\widehat{U}\) along \(v=0\). 
\end{example}
When \(f > 0\), the group schemes \(\mathcal{G}_{x,f}\) has unipotent special fiber:
\begin{proposition}
    \label{res:unipotent-special-fiber}
    Let \(f > 0\) be a concave function. 
    Then \(\mathcal{G}_{x,f} \times_{v=0} \operatorname{Spec} \mathbb{F}\) is a split unipotent group scheme over \(\mathbb{F}\).\footnote{
        Recall that being \emph{split unipotent} means being an iterated extension of \(\mathbb{G}_{a}\). 
    } 
\end{proposition}
\begin{proof}
    By construction, \(\mathcal{G}_{x,f} \times \operatorname{Spec} \mathbb{F}\llb v \rrb\) is the smooth integral model for \(G^*_{\mathbb{F}\llp v \rrp}\) attached to \(x\) and \(f\)
    in the sense of \cite[Theorem 8.5.2]{kaletha-prasad}
    The special fiber of this group scheme, which coincides with \(\mathcal{G}_{x,f} \times_{v=0} \operatorname{Spec} \mathbb{F}\), is unipotent by \cite[Corollary 8.5.12]{kaletha-prasad}, 
    and in fact split unipotent by \cite[Corollary 15.5.ii]{borelLinearAlgebraicGroups1991} since \(\mathbb{F}\) is finite (and in particular perfect). 
\end{proof}
\begin{proposition}
    \label{res:unipotent-group-has-trivial-torsors}
    Let \(f > 0\) be a concave function. 
    Let \(R\) be a \(p\)-adically complete \(\mathcal{O}\)-algebra. 
    Then all \(\mathcal{G}_{x,f}\)-torsors on \(R\llb v \rrb \) are trivial.
    Moreover, if \(\overline{\mathcal{G}}_{x,f} := \mathcal{G}_{x,f} \times_{v=0} \operatorname{Spec}\mathcal{O}\), then all \(\overline{G}_{x,f}\)-torsors on \(R\) are trivial. 
\end{proposition}
\begin{proof}
    Let \(\mathcal{P}\) be a \(\mathcal{G}_{x,f}\)-torsor over \(\operatorname{Spec}R\llb v \rrb\).
    Then \(\mathcal{P}\) is smooth affine by descent, so \(\mathcal{P}(R\llb v \rrb) \neq \emptyset\)
    only if  \(\mathcal{P}(R) \neq \emptyset\) only if \(\mathcal{P}(R / \varpi R) \neq \emptyset\) (by \(p\)-adic completeness of \(R\)).
    Note that \(\mathcal{P}(R / \varpi R)\) is also the set of sections of \(\mathcal{P} \times_{v=0} \operatorname{Spec} \mathbb{F}\), which is a torsor for the group \(C := \mathcal{G}_{x,f} \times_{v=0} \operatorname{Spec}\mathbb{F}\). 
    This set of sections is non-empty since \(C\) is unipotent by \cref{res:unipotent-special-fiber}, and all torsors for a unipotent group scheme over an affine space are trivial.
    This fact is proved in \cite[Propopsition 1]{grothendieckTorsionHomologiqueSections1958}, but for convenience of the reader we give a short argument. 
    The vanishing of \(H^1_{\et}\left( \operatorname{Spec}R / \varpi R,C \right)\) follows from the vanishing of \(H^1_{\et}\left( \operatorname{Spec}R / \varpi R, \mathbb{G}_{a} \right) = H^1 \left( \operatorname{Spec}R / \varpi R, \mathcal{O}_{R/ \varpi R} \right) = 0\), 
    where the first equality is due to \cite[\href{https://stacks.math.columbia.edu/tag/03DW}{Proposition 03DW}]{stacks-project}, and 
    the second due to Serre vanishing \cite[\href{https://stacks.math.columbia.edu/tag/01XB}{Lemma 01XB}]{stacks-project}.
\end{proof}

\subsection{Loop groups}
\label{sec:loop-groups}
We will now discuss various loop groups for the groups constructed in \cref{sec:bruhat-tits-group-schemes}. 
Some general properties of loop groups are also recalled in \cref{sec:loop-groups-generalities}.

\subsubsection{Positive loop groups}
\label{sec:positive-loop-groups}
Given a scheme \(X\) over \(\mathbb{A}^1_{\mathcal{O}}\), we define two functors on \(\mathcal{O}\)-algebras 
\begin{align*}
    L^+_{0} X (R) & = X (R \llb v \rrb) , \\ 
    L^+ X (R) & = X(R \llb v + p \rrb) , 
\end{align*}
where \(R \llb v + p \rrb\) is viewed as an \(\mathcal{O}[v]\)-algebra via \(v \mapsto v\). 
These loop schemes are the specializations of the global loop scheme \cite[Section 6.2.4]{pappas-zhu} at \(v = 0\) and \(v= -p\), respectively. 

Note that if \(R\) is \(p\)-adically complete,
then \(R\llb v \rrb = R \llb v+p \rrb\) by \cref{sec:E-vs-v}, 
and consequently \(L^+_{0}X(R) = L^+ X(R)\). 
But this property notably fails for \(R = E\) (and more generally if \(p \in R^\times\)).

Similarly, we define positive loop groups for schemes over \(\breve{\mathbb{A}}^1_{\mathcal{O}}\) and \(\widetilde{\mathbb{A}}^1_{\mathcal{O}}\) as well. 
That is, if \(X\) is defined over \(\breve{\mathbb{A}}^1_{\mathcal{O}}\), we set 
\begin{align*}
    L^+_{0} X (R) & = X\left( \left( R \otimes_{\mathbb{Z}_{p}} \mathbb{Z}_{q} \right) \llb v \rrb \right) , \\
    L^+ X (R) & = X\left( \left( R \otimes_{\mathbb{Z}_{p}} \mathbb{Z}_{q} \right) \llb v + p \rrb \right) . 
\end{align*}
Similarly, if \(X\) is defined over \(\widetilde{\mathbb{A}}^1_{\mathcal{O}}\), we set 
\begin{align*}
    L^+_{0} X (R) & = X\left( \left( R \otimes_{\mathbb{Z}_{p}} \mathbb{Z}_{q} \right) \llb u \rrb \right) , \\
    L^+ X (R) & = X\left( \left( R \otimes_{\mathbb{Z}_{p}} \mathbb{Z}_{q} \right) \llb u^e + p \rrb \right) , 
\end{align*}
where we remind the reader that \(\left( R \otimes_{\mathbb{Z}_{p}} \mathbb{Z}_{q} \right) \llb u^e + p \rrb := \left( R \otimes_{\mathbb{Z}_{p}} \mathbb{Z}_{q} \right) [u]^{\wedge (u^e + p)}\).
Note that when \(R\) is \(p\)-adically complete, we have \(\left( R \otimes_{\mathbb{Z}_{p}} \mathbb{Z}_{q} \right) \llb u^e + p \rrb = \left( R \otimes_{\mathbb{Z}_{p}} \mathbb{Z}_{q} \right) \llb u \rrb\) by \cref{res:inverting-different-variables}. 

\subsubsection{Congruence loop groups}
\label{sec:congruence-loop-groups}
If \(H\) is a smooth affine group scheme over \(\mathbb{A}^1_{\mathcal{O}}\), we define for any \(n \in \mathbb{Z}_{\geq 0}\) the jet groups 
\begin{align*}
    L^n_{0}H(R) & = H(R[v]/v^n) , \\
    L^nH(R) & = H(R[v+p]/(v+p)^n) . 
\end{align*} 
We then define the congruence loop groups \(L^{+n}_{0}H\) and \(L^{+n}H\) as the kernels of the short exact sequences 
\begin{align}
    1 \to L^{+n}_{0}H \to L^+_{0} H \to L^n_{0} H \to 1 , \label{eq:congruence-loop-group-exact-sequence} \\
    1 \to L^{+n}H \to L^+ H \to L^n H \to 1 , \nonumber
\end{align}
where we note that these sequences are right exact by smoothness of \(H\).

If \(H\) is defined over \(\breve{\mathbb{A}}^1_{\mathcal{O}}\), we define \(L^n_{0}H\) and \(L^n H\) similarly as above. 
But if \(H\) is defined over \(\widetilde{\mathbb{A}}^1_{\mathcal{O}}\), we define
\begin{align*}
    L^n_{0} H(R) & = H \left( \left( R \otimes_{\mathbb{Z}_{p}} \mathbb{Z}_{q} \right)[u]/u^{en} \right) , \\
    L^n H(R) & = H \left( \left( R \otimes_{\mathbb{Z}_{p}} \mathbb{Z}_{q} \right)[u]/(u^e+p)^n \right) . 
\end{align*}
In both cases (\(H\) defined over \(\breve{\mathbb{A}}^1_{\mathcal{O}}\) or \(\widetilde{\mathbb{A}}^1_{\mathcal{O}}\)), we define \(L^{+n}_{0}H\) and \(L^{+}H\) as kernels as above. 

Even when \(R \llb v \rrb = R \llb v + p \rrb\), we typically have \(L^{+n}_{0}H (R) \neq L^{+n}H (R) \), for the simple reason that \(R \llb v \rrb / v^n \neq R \llb v \rrb / (v+p)^n\). 
The following lemma enables us to compare these congruence subgroups when \(p^a R = 0\). 
\begin{lemma}
    \label{res:comparing-congruence-subgroups-at-v-and-v-plus-p}
    Let \(R\) be a \(\mathbb{Z} / p^{a}\)-algebra where
    \(a \in \mathbb{Z}_{\geq 1}\), and let \(n \in \mathbb{Z}_{\geq 1}\). If
    \(n \geq a\), then
    \begin{align*}
        L^{+n}H(R) & \subset L^{+n-a+1}_{0} H (R) , \\ 
        L^{+n}_{0} H(R) & \subset L^{+n - a + 1}H(R) .
    \end{align*}
    For any \(n \in \mathbb{Z}_{\geq 1}\), we also have
    \[ L^{+p^{a-1}n}_{0} H(R) = L^{+p^{a-1}n}H(R) . \]
\end{lemma}
\begin{proof}
    The first statement follows from the fact that
    \((v+p)^{n} \in v^{n-a+1}R[v]\) and \(v^{n} \in (v+p)^{n-a+1}R[v]\).
    Both of these facts are consequences of the binomial theorem. Another
    consequence of the binomial theorem is the identity
    \((v+p)^{p^{a-1}} \equiv v^{p^{a-1}} \mod p^{a}\), which implies the
    second statement.
\end{proof}

We will often use the following result without comment. 
\begin{lemma}
    \label{res:loop-of-pappas-zhu-dilation-as-congruence-loop-group}
    Let \(n \in \mathbb{Z}_{\geq 0}\) and let \(\mathcal{G}_{x,n}\) denote the group scheme from \cref{res:pappas-zhu-concave-function}. 
    Then \(L^+_{0}\mathcal{G}_{x,n} = L^{+n}_{0} \mathcal{G}_{x}\).
    Similarly, \(L^+_{0}\breve{\mathcal{G}}_{x,n} = L^{+n}_{0}\breve{\mathcal{G}}_{x}\) and \(L^+_{0}\widetilde{\mathcal{G}}_{x,n} = L^{+n}_{0}\widetilde{\mathcal{G}}_{x}\). 
\end{lemma}
\begin{proof}
    Immediate from \cref{ex:congruence-pappas-zhu-group-scheme}, \cref{res:congruence-loop-group-lemma} and the definitions. 
\end{proof}

\subsubsection{Lie algebras of congruence loop groups}
We also make some remarks about the Lie algebras of congruence loop groups.
Let \(J \to R \to \overline{R}\) be a square zero extension of
\(\mathcal{O}\)-algebras. The are two natural ways to interpret the
Lie algebra \(\operatorname{Lie}L^{+n}\mathcal{G}(J)\), either as the
Lie algebra of the congruence loop group \(L^{+n}\mathcal{G}\), or as
the kernel of
\(\operatorname{Lie}L^{+}\mathcal{G}(J) \to \operatorname{Lie}L^{n}\mathcal{G}(J)\).
In fact these two definitions agree by commutativity of the following
diagram: 
\[\begin{CD}
\operatorname{Lie}L^{+n}\mathcal{G}(J) @>>> L^{+n}\mathcal{G}(R) @>>> L^{+n}\mathcal{G}(\overline{R}) \\
@VVV @VVV @VVV \\
\operatorname{Lie} L^{+}\mathcal{G}(J) @>>> L^{+}\mathcal{G}(R) @>>> L^{+}\mathcal{G}(\overline{R}) \\
@VVV @VVV @VVV \\
\operatorname{Lie}L^{n} \mathcal{G}(J) @>>> L^{n}\mathcal{G}(R) @>>> L^{n}\mathcal{G}(\overline{R}) .
\end{CD}\] 
Indeed, by the 9-lemma the topmost horizontal row is exact if and only if the leftmost vertical column is exact. 

It is clear that the above remarks also apply to \(\operatorname{Lie}L^{+n}_{0}\mathcal{G}(J)\). 

\begin{lemma}
    \label{res:congruence-subgroups-of-lie-algebras}
    Let \(J \to R \to \overline{R}\) be a square zero extension, where \(pJ = 0\).
    Assume that \(R\) and \(\overline{R}\) are \(p\)-adically complete so that \(L^+\mathcal{G}(R) = L^+_{0}\mathcal{G}(R)\) as explained above.
    Then there is a natural identification
    \(\operatorname{Lie}L^{+n}\mathcal{G}(J) \cong \operatorname{Lie}L^{+n}_{0}\mathcal{G}(J)\)
    and in fact these are equal as subgroups of \(L^{+}\mathcal{G}(R)\).
\end{lemma}

Intuitively this is the fact that \(v + p = v\) when \(p=0\).

\begin{proof}
    The map
    \(\operatorname{Lie}L^{+}\mathcal{G}(J) \to \operatorname{Lie}L^{n}\mathcal{G}(J)\)
    corresponds to the map
    \(\operatorname{Lie}\mathcal{G}(J\llb v\rrb) \to \operatorname{Lie}\mathcal{G}(J[v]/(v+p)^{n})\)
    given by functoriality of \(\operatorname{Lie}\mathcal{G}\) in square
    zero extensions. When \(pJ = 0\), then \(J[v]/(v+p)^{n} = J \otimes_{\mathcal{O}} \mathcal{O}[v]/(v+p)^n = J \otimes_{\mathcal{O}} \mathcal{O}[v]/v^n = J[v]/v^{n}\),
    so \[
    1 \to \operatorname{Lie}L^{+n}\mathcal{G}(J) \to \operatorname{Lie}\mathcal{G}(J\llb v \rrb)  \to \operatorname{Lie}\mathcal{G}(J[v]/v^{n}) \to 1
    \] is exact, implying that \(\operatorname{Lie}L^{+n}\mathcal{G}(J)\) is
    also the congruence subgroup with respect to \(v^{n}\).
\end{proof}

\subsubsection{Iwahori decomposition}
\label{sec:iwahori-decomposition}
We will need Iwahori decomposition in the following form. 
\begin{proposition}
    \label{res:iwahori-decomposition}
    Let \(f \geq 0\) be a concave function (possibly zero). 
    The map \(L^+_{0}\widetilde{\mathcal{V}}_{x,f} \hookrightarrow L^+_{0} \widetilde{\mathcal{G}}_{x,f}\) is an open immersion, 
    and if \(f>0\) or \(x\) is \(0\)-generic, then it is an isomorphism. 
    Similarly, if \(I\) acts trivially on \(\widehat{G}\), then the map \(L^+_{0}\breve{\mathcal{V}}_{x,f} \hookrightarrow L^+_{0} \breve{\mathcal{G}}_{x,f}\) is an open immersion,
    and if \(f > 0\) or \(x\) is \(0\)-generic, then it is an isomorphism. 
\end{proposition}
\begin{proof}
    We will prove the latter statement under the assumption that \(I\) acts trivially on \(\widehat{G}\), and the former statement can be proved similarly. 
    If \(f = 0\) then the map \(\breve{\mathcal{V}}_{x,f} \hookrightarrow \breve{\mathcal{G}}_{x,f}\) is an open immersion by \cref{sec:open-cell}, 
    and in general \(\breve{\mathcal{V}}_{x,f} \hookrightarrow \breve{\mathcal{G}}_{x,f}\) is an open immersion by the proof of \cref{res:pappas-zhu-concave-function}. 
    By \cref{res:loop-group-lemma} it follows that \(L^+_{0}\breve{\mathcal{V}}_{x,f} \hookrightarrow L^+_{0} \breve{\mathcal{G}}_{x,f}\) is an open immersion.

    Now assume that \(f(0) > 0\) or \(x\) is \(0\)-generic. 
    To check if the open immersion \(L^+_{0}\breve{\mathcal{V}}_{x,f} \hookrightarrow L^+_{0} \breve{\mathcal{G}}_{x,f}\) is an isomorphism, 
    it suffices to show surjectivity on \(\overline{\mathbb{F}}\) and \(\overline{E}\)-points, where it is a consequence of \cite[Proposition 7.3.12]{kaletha-prasad} (Iwahori decomposition).
\end{proof}
\begin{corollary}
    \label{res:iwahori-decomposition-corollary}
    Let \(f \geq 0\) be a concave function. 
    Assume that \(f>0\) or \(x\) is \(0\)-generic.
    Then \(L^+_{0}\mathcal{G}_{x,f} = \left( L^+_{0}\widetilde{\mathcal{V}}_{x,f} \right)^\Gamma\). 
    If moreover \(I\) acts trivially on \(\widehat{G}\), 
    then \(L^+_{0}\mathcal{G}_{x,f} = (L^+_{0}\breve{\mathcal{V}}_{x,f})^{\Gamma_{0}}\). 
\end{corollary}

\subsubsection{Quotient loop groups}  
\label{sec:quotient-loop-groups}
We will identify the quotient \(L^+_{0}\mathcal{G}_{x} / L_{0}^+ \mathcal{G}_{x,f}\) in the following cases:
\begin{enumerate}[(a)]
    \item \(f = n \in \mathbb{Z}_{\geq 1}\), 
    \item \(f = 0^+\) and \(x\) is 0-generic.  
\end{enumerate}
For case (a), we have by \eqref{eq:congruence-loop-group-exact-sequence} and \cref{res:loop-of-pappas-zhu-dilation-as-congruence-loop-group}
\[
    L^+_{0}\mathcal{G}_{x} / L_{0}^+ \mathcal{G}_{x,n} \cong L^+_{0} \mathcal{G}_{x} / L^{+n}_{0} \mathcal{G}_{x} \cong L^n \mathcal{G}_{x} , 
\]
where \(L^n \mathcal{G}_{x} = \operatorname{Res}^{\mathcal{O}[v]/v^n}_{\mathcal{O}} \mathcal{G}_{x}\), and in particular this is a smooth affine group scheme. 

For case (b), consider the torus \(\overline{\mathcal{T}}\) which is the restriction of \(\mathcal{T}\) to the fiber \(v = 0\). 
Let \(h : L^+_{0} \mathcal{G}_{x} \to \overline{\mathcal{T}}\) be given as the composite 
\[
\begin{tikzcd}
    h : L^+_{0} \mathcal{G}_{x} \arrow[r,equal] 
    & \left( \prod_{a < 0} L^+_{0}\widetilde{\mathcal{U}}_{a,x} \times L_{0}^+\widetilde{\mathcal{T}} \times \prod_{a > 0} L_{0}^+\widetilde{\mathcal{U}}_{a,x} \right)^{\Gamma} \arrow[r, "\operatorname{pr}"]
    & L_{0}^+\widetilde{\mathcal{T}}^\Gamma \arrow[r,equal]
    & L_{0}^+\mathcal{T} \arrow[rr,"\mod v"]
    & & \overline{\mathcal{T}} ,
\end{tikzcd}
\]
where the first identification is a consequence of \cref{res:iwahori-decomposition-corollary} and the fact that \(x\) is \(0\)-generic.
Note that \(\operatorname{pr}\) is a group homomorphism because the conjugation action of the torus normalize the root groups, and consequently so is \(h\). 
The kernel of \(h\) is seen to be \(L^+_{0} \mathcal{G}_{x,0^+}\) by \cref{res:iwahori-decomposition-corollary} again,
and surjectivity of the reduction mod \(v\) map follows from smoothness of \(\mathcal{T}\).
To summarize, we have have an isomorphism 
\begin{equation}
    \label{eq:quotient-torus}
    L^+_{0} \mathcal{G}_{x} / L_{0}^+ \mathcal{G}_{x,0^+} \cong \overline{\mathcal{T}} .  
\end{equation}
Moreover, the map \(L^+_{0}\mathcal{G}_{x} \twoheadrightarrow \overline{\mathcal{T}}\) has a canonical section, which ultimately comes from the section \(\mathcal{O} \subset \mathcal{O}\llb v \rrb\)
of the reduction mod \(v\) map. 

\begin{remark}
    If \(I\) acts trivially on \(\widehat{G}\) and \(\mathbb{F} \supseteq \mathbb{F}_{q}\), then we identify \(\overline{\mathcal{T}} = \widehat{T}\). 
\end{remark}

\begin{remark}
    For a random concave function \(f \geq 0\), it is not always the case that \(L^+ \mathcal{G}_{x,f}\) is a normal subgroup of \(L^+ \mathcal{G}_{x}\), but 
    this is true whenever \(f\) is constant by \cite[Proposition 7.3.17]{kaletha-prasad}. 
\end{remark}

\subsubsection{Loop groups}
\label{sec:actual-loop-groups}
Recall that \(\mathcal{G}_{x}[v^{-1}] = G^*\) and \(\widetilde{\mathcal{G}}_{x}[u^{-1}] = \widehat{G}\).
We define a functors on \(\mathcal{O}\)-algebras
\begin{align*}
    & L_{0} \mathcal{G}_{x}(R)   = G^* \left( R \llp v \rrp \right) , \\
    & L \mathcal{G}_{x} (R) = \mathcal{G}_{x} \left( R \llp v + p \rrp \right), \,\,\,\,\, \text{where } R \llp v + p \rrp = R [v]^{\wedge (v+p)}\left[ (v+p)^{-1} \right] , \\ 
    & L_{0} \widetilde{\mathcal{G}}_{x}(R) = \widehat{G} \left( \left( R \otimes_{\mathbb{Z}_{p}} \mathbb{Z}_{q} \right) \llp u \rrp \right) . 
\end{align*}
Note that \(L_{0} \widetilde{\mathcal{G}}_{x} (R)\) has an obvious \(\Gamma\)-action, and \(L_{0}\mathcal{G}_{x}(R) = L_{0} \widetilde{\mathcal{G}}_{x}(R)^\Gamma\).

\begin{remark}
    \label{rem:loop-group-at-0-or-p}
    If \(p^a R = 0\) for some \(a \in \mathbb{Z}_{\geq 1}\), then \(R \llp v \rrp = R \llp v + p \rrp\) by \cref{res:inverting-different-variables}, and consequently \(L \mathcal{G}_{x}(R) = L_{0}\mathcal{G}_{x}(R) = L_{0} G^* (R)\). 
\end{remark}

\subsection{The negative loop group}
\label{sec:negative-loop-group}
In this section we introduce the negative loop group \(L^{--} \mathcal{G}_{x}\), which is in some sense ``complementary'' to the positive loop group \(L^+ \mathcal{G}_{x}\). 
We will use it to construct convenient open charts for the affine Grassmannian \(\operatorname{Gr}_{\mathcal{G}}\) that we will introduce later on.
For reasons that will soon become apparent, we impose \cref{ass:pappas-zhu-groups-as-dilations} throughout this section. 

For the constant group \(\widehat{G}\), one can define the negative loop group as 
\begin{equation}
    \label{eq:negative-loop-group-for-constant-group}
    L^{--}\widehat{G}(R) = \ker \left( \widehat{G}\left( R\left[ \frac{1}{v+p} \right] \right) \stackrel{v = \infty}{\longrightarrow} \widehat{G}(R) \right) ,
\end{equation}
where \(R\left[ \frac{1}{v+p} \right] \subset R \llp v + p \rrp \) is the subring of polynomials in the variable \(\frac{1}{v+p}\). 
Then the multiplication map \(L^+ \widehat{G} \times L^{--}\widehat{G}\to L \widehat{G}\) is an open immersion, as shown for example in \cite[Lemma 2.3.5]{zhu-affine-grassmannians}. 
This fact is useful, as it implies that \(L^{--}\widehat{G}\)-translates give open charts of the affine Grassmannian \(L^+ \widehat{G} \backslash L \widehat{G}\). 

In the present case, we want a negative loop group \(L^{--}\mathcal{G} \subset L \mathcal{G}\) such that the multiplication map \(L^+ \mathcal{G} \times L^{--}\mathcal{G} \hookrightarrow L \mathcal{G}\) is an open immersion.
But we can not use \eqref{eq:negative-loop-group-for-constant-group} since the symbol \(\mathcal{G}\left( R\left[ \frac{1}{v+p} \right] \right)\) is not a priori defined.
One possibility of making sense of this symbol is to make an appropriate extension of \(\mathcal{G}\) to \(\mathbb{P}^1_{\mathcal{O}}\). 
Even still, it is not clear which extension of \(\mathcal{G}\) to \(\mathbb{P}^1_{\mathcal{O}}\) one should choose; roughly speaking, such an extension corresponds to a choice of parahoric structure at \(v = \infty\). 
Our idea is to choose the parahoric structure at \(v=\infty\) which is perfectly complementary to the parahoric structure at \(v = 0\). 

The point \(x \in \widetilde{\mathcal{A}}(\widehat{T}, \mathbb{F}\llp v \rrp)\) gives a valuation of root datum as we explained in \cref{sec:valuations-of-the-root-datum}.
That is, for each root \(a \in \Phi(\widehat{G},\widehat{T})\), the point \(x\) describes \emph{lower bounds} for the valuations of elements in the root group \(U_{a}(\mathbb{F}\llp v \rrp) \subset \widehat{G}(\mathbb{F}\llp v \rrp)\), and the parahoric subgroup corresponding to \(x\) is generated by elements satisfying this bound. 
We can then think of this same point \(x\) as describing complementary \emph{upper bounds} for the ``antivaluations''\footnote{
    If \(\operatorname{val} : \mathbb{F}\llp v^{-1} \rrp \to \mathbb{Z}\) is the discrete valuation with \(\operatorname{val}(v^{-1}) = 1\), then the ``antivaluation'' is \(- \operatorname{val}\).}
of elements in the root group \(U_{a}(\mathbb{F}\llp v^{-1} \rrp) \subset \widehat{G}(\mathbb{F}\llp v^{-1} \rrp)\), giving a parahoric subgroup of \(\widehat{G}(\mathbb{F}\llp v^{-1} \rrp)\).
If \(\eta := x - o \in X(\widehat{T})_{\mathbb{R}}\), so that \(x_{0} := x = o + \eta \in \widetilde{\mathcal{A}}(\widehat{T},\mathbb{F}\llp v \rrp)\), 
then the above parahoric subgroup of \(\widehat{G}(\mathbb{F}\llp v^{-1} \rrp)\) is described by the point \(x_{\infty} := o - \eta \in \widetilde{\mathcal{A}}\left( \widehat{T}, \mathbb{F}\llp v^{-1} \rrp \right)\).

We define \(\mathcal{G}_{x}\) over \(\mathbb{P}^1_{\mathcal{O}}\) by gluing the group scheme \(\mathcal{G}_{x} = \mathcal{G}_{x_{0}}\) over \(\mathbb{A}^1_{\mathcal{O}} = \operatorname{Spec}\mathcal{O}[v]\) with 
the ``complemenary'' group scheme \(\mathcal{G}_{x_{\infty},0^+}\) over \(\mathbb{A}^1_{\mathcal{O},\infty} = \operatorname{Spec} \mathcal{O}[v^{-1}]\).
Because of \cref{ass:pappas-zhu-groups-as-dilations}, the group scheme \(\mathcal{G}_{x}\) over \(\mathbb{P}^1_{\mathcal{O}}\) is the dilation of the constant group scheme \(\widehat{G}\) over \(\mathbb{P}^1_{\mathcal{O}}\) in \(\widehat{P}\) at \(v = 0\) and in \(\widehat{U}^\text{op}\) at \(v = \infty\).

For any \(\mathcal{O}\)-algebra \(R\) we set 
\[
    L^{--}\mathcal{G}_{x} (R) := \mathcal{G}_{x}\left( R \left[ \frac{1}{v+p} \right] \right) . 
\]
Here, we think of \(\operatorname{Spec}R\left[ \frac{1}{v+p} \right]\) as \(\mathbb{P}^1_{R} \setminus \lbrace - p \rbrace\), 
and we will discuss this identification in more detail in the proof of \cref{res:explicit-negative-loop-group}.
The functor \(L^{--}\mathcal{G}_{x}\) is an Ind-scheme over \(\mathcal{O}\).
\begin{example}
    When \(x = o\) is the Chevalley valuation, \(\mathcal{G}_{x} = \widehat{G}\) as a group scheme over \(\mathbb{A}^1_{\mathcal{O}}\). 
    Over \(\mathbb{P}^1_{\mathcal{O}}\), we see that \(\mathcal{G}_{x}\) is the dilation of \(\widehat{G}\) in the identity section along \(v = \infty\). 
    By the universal property of dilations, \(L^{--}\mathcal{G}_{x} = L^{--}\widehat{G}\) is then compatible with \eqref{eq:negative-loop-group-for-constant-group}. 
\end{example}
\begin{remark}
    \label{rem:dilation-at-infty-doesnt-affect-loop-groups}
    Although we have now defined \(\mathcal{G}_{x}\) as a group scheme over \(\mathbb{P}^1_{\mathcal{O}}\), we remark that for any \(\mathcal{O}\)-algebra \(R\), 
    the maps from \(\operatorname{Spec}R\llb v \rrb\), \(\operatorname{Spec}R\llp v \rrp\), \(\operatorname{Spec}R\llb v + p \rrb\), or \(\operatorname{Spec}R \llp v + p \rrp\) into \(\mathbb{P}^1_{\mathcal{O}}\) factor through \(\mathbb{A}^1_{\mathcal{O}}\). 
    Hence the symbols \(L_{0}^+\mathcal{G}_{x}\), \(L_{0}\mathcal{G}_{x}\), \(L^+ \mathcal{G}_{x}\), and \(L \mathcal{G}_{x}\) are not ambiguous, i. e. the extension of \(\mathcal{G}_{x}\) to \(\mathbb{P}^1_{\mathcal{O}}\) does not affect these (loop group) functors in any way. 
\end{remark}
The following result shows that our definition is compatible with \cite[Definition 3.2.1]{local-models}.
\begin{lemma}
    \label{res:explicit-negative-loop-group}
    For any \(\mathcal{O}\)-algebra \(R\) we have\footnote{Note that \(\frac{v}{v+p} = 1 - \frac{p}{v+p} \in R \left[ \frac{1}{v+p} \right]\), and if we set \(\tilde{t} = \frac{1}{v+p}\), then \(R\left[\frac{1}{v+p}\right] / \frac{v}{v+p}R \left[\frac{1}{v+p}\right] = R [\tilde{t}] / (1 - p\tilde{t}) = R\left[\frac{1}{p}\right]\).}
    \begin{align*} 
        \mathcal{G}_{x} \left( R \left[ \frac{1}{v+p} \right] \right) 
        & = \left\lbrace A \in \widehat{G}\left( R \left[ \frac{1}{v+p} \right] \right) : A \mod \frac{1}{v+p} \in \widehat{U}^\text{op}(R) \,\,\text{ and }\,\, A \mod \frac{v}{v+p} \in \widehat{P}\left( R \left[ \frac{1}{p} \right] \right) \right\rbrace .
    \end{align*}
    In particular, if \(p\) is nilpotent in \(R\), we have
    \[
        \mathcal{G}_{x} \left( R \left[ \frac{1}{v+p} \right] \right)
        = \left\lbrace A \in \widehat{G}\left(R \left[\frac{1}{v+p}\right]\right) : A \mod \frac{1}{v+p} \in \widehat{U}^\text{op}(R) \right\rbrace . 
    \]
\end{lemma}
\begin{proof}
    The proof proceeds by computing sections in various affine charts of \(\mathbb{P}^1_{R}\).
    We have closed subschemes \(v = 0, v = -p, v = \infty : \operatorname{Spec} R \hookrightarrow \mathbb{P}^1_{R}\) which we will just denote \(\lbrace 0 \rbrace, \lbrace -p \rbrace, \lbrace \infty \rbrace\), respectively.
    Consider the two coordinates on the affine line \(\mathbb{A}^1_{R} = \mathbb{P}^1_{R} \setminus \lbrace \infty \rbrace\) given by \(v\) and \(t = v + p\) centered at \(0\) and \(-p\) (respectively).
    We can then identify: 
    \begin{align*}
        & \mathbb{P}^1_{R} \setminus \lbrace \infty \rbrace = \operatorname{Spec} R[v] = \operatorname{Spec}R[v+p] , \\
        & \mathbb{P}^1_{R} \setminus \lbrace 0 \rbrace = \operatorname{Spec}R \left[ \frac{1}{v} \right] , \\
        & \mathbb{P}^1_{R} \setminus \lbrace -p \rbrace = \operatorname{Spec} R \left[ \frac{1}{t} \right] =  \operatorname{Spec}R \left[ \frac{1}{v+p} \right] , \\
        & \mathbb{P}^1_{R} \setminus \lbrace \infty, -p \rbrace = \operatorname{Spec}R \left[ v+p, \frac{1}{v+p} \right] = \operatorname{Spec} R\left[ v, \frac{1}{v+p} \right] , \\
        & \mathbb{P}^1_{R} \setminus \lbrace \infty, 0, -p \rbrace = \operatorname{Spec} R\left[ v, \frac{1}{v}, \frac{1}{v+p} \right] , \\ 
        & \mathbb{P}^1_{R} \setminus \lbrace 0, -p \rbrace = \operatorname{Spec} R\left[ \frac{1}{v+p}, \frac{v+p}{v} \right] = \operatorname{Spec} R\left[ \frac{1}{v}, \frac{v}{v+p} \right] = \operatorname{Spec}R \left[ \frac{1}{v}, \frac{1}{v+p} \right] . 
    \end{align*}
    Note that the closed subscheme \(\lbrace 0 \rbrace \subset \mathbb{P}^1_{R} \setminus \lbrace -p \rbrace\) corresponds to the ideal \(\frac{v}{v+p}R\left[ \frac{1}{v+p}\right] \subset R \left[ \frac{1}{v+p}\right]\). 
    Indeed, the intersection \(R\left[ \frac{1}{v+p} \right] \cap v R \left[v, \frac{1}{v+p} \right] \subset R \left[v, \frac{1}{v+p} \right]\) is the principal ideal \(\left( \frac{v}{v+p} \right) = \left( 1 - \frac{p}{v+p} \right)\), 
    and we can see this as follows. 
    Suppose that \(a_{0},\dots, a_{n} \in R\) and \(a_{0} + \frac{a_{1}}{v+p} + \cdots + \frac{a_{n}}{(v+p)^n} \in v R \left[v, \frac{1}{v+p} \right]\), where we may assume that \(a_{n} \neq 0\). 
    Then we can write 
    \begin{equation}
        \label{eq:ideal-v}
        a_{0} + \frac{a_{1}}{v+p} + \cdots + \frac{a_{n}}{(v+p)^n} = \frac{v g(v)}{(v+p)^k} 
    \end{equation}
    for some \(k \in \mathbb{Z}\) and polynomial \(g(v) \in R[v] = R[v+p]\), where \(g(-p) \neq 0\).
    By rewriting the right hand side as a Laurent polynomial in \(v+p\), we find that \(n = k\). 
    Multiplying \eqref{eq:ideal-v} by \((v+p)^n\) and setting \(v = 0\) we obtain \(p^n a_{0} + p^{n-1}a_{1} + \cdots + a_{n} = 0\), and therefore 
    \[a_{0} + \frac{a_{1}}{v+p} + \cdots + \frac{a_{n}}{(v+p)^n} = \left(1 - \frac{p}{v+p}\right) \left( a_{0} + \frac{pa_{0} + a_{1}}{v+p} + \cdots + \frac{p^{n-1}a_{0} + p^{n-2}a_{1} + \cdots + a_{n-1}}{(v+p)^{n-1}} \right).\]

    We now compute sections of \(\mathcal{G}_{x}\), first in the affine chart \(\mathbb{P}^1_{R} \setminus \lbrace \infty, -p \rbrace = \operatorname{Spec}R\left[ (v+p)^{\pm 1} \right]\). 
    Since \(v\) is a nonzerodivisor in \(R \left[(v+p)^{\pm 1} \right] = R\left[v, (v+p)^{-1} \right]\), we have by the universal property of dilations that 
    \[
        \mathcal{G}_{x}\left( R \left[(v+p)^{\pm 1} \right] \right) = \left\lbrace A \in \widehat{G}\left( R \left[ (v+p)^{\pm 1} \right] \right) : A \mod v \in \widehat{P}\left( R \left[ \frac{1}{p} \right] \right) \right\rbrace . 
    \]
    Next, we compute sections in the affine chart \(\mathbb{P}^1_{R} \setminus \lbrace 0, -p \rbrace = \operatorname{Spec} R\left[ \frac{1}{v+p}, \frac{1}{v} \right]\). 
    Using the universal property of dilations again we see that 
    \[
        \mathcal{G}_{x} \left( R \left[ \frac{1}{v+p}, \frac{1}{v} \right] \right) = \left\lbrace A \in \widehat{G} \left( R \left[ \frac{1}{v+p}, \frac{1}{v} \right] \right) : A \mod \frac{1}{v+p} \in \widehat{U}^\text{op}\left( R \right) \right\rbrace , 
    \]
    where we note that \(\left( \frac{1}{v+p} \right) = \left( \frac{1}{v} \right) = \left( \frac{1}{v+p},\frac{1}{v} \right) \subset R \left[ \frac{1}{v+p}, \frac{1}{v} \right]\) (the ideal corresponding to \(v = \infty\)), since \(\frac{1}{v}\left( 1 - \frac{p}{v+p}\right) = \frac{1}{v+p}\) and \(\frac{1}{v} = \frac{1}{v+p}\left( 1 + \frac{p}{v}\right)\).
    By using the Zariski open cover \(\operatorname{Spec} R \left[ (v+p)^{\pm 1} \right] \cup \operatorname{Spec} R \left[ \frac{1}{v+p}, \frac{v+p}{v} \right] = \operatorname{Spec}R \left[ \frac{1}{v+p}\right]\), 
    we see that 
    \begin{align*} 
        \mathcal{G}_{x} \left( R \left[ \frac{1}{v+p} \right] \right) = \mathcal{G}_{x} \left( R \left[(v+p)^{\pm 1} \right] \right) \cap \mathcal{G}_{x} \left( R \left[ \frac{1}{v+p}, \frac{v+p}{v} \right] \right) ,
    \end{align*}
    where the intersection takes place in \(\mathcal{G}_{x} \left( R\left[ v, \frac{1}{v+p}, \frac{1}{v} \right] \right)\).
    Combining the above, we see that the sections of \(\mathcal{G}_{x}\) take the described form. 
\end{proof}

Our next goal is to show that \(L^{--} \mathcal{G}_{x}\) indeed has the expected property that \(L^{--}\mathcal{G} \times L^+ \mathcal{G} \hookrightarrow L \mathcal{G}\) is an open immersion. 
Following the strategy of \cite[Lemma 2.3.5]{zhu-affine-grassmannians}, we prove that (1) \(L^{--}\mathcal{G}\cap L^{+}\mathcal{G} = \lbrace 1 \rbrace\), with the intersection taking place in \(L \mathcal{G}\), 
and (2) \(\operatorname{Lie}L^{--}\mathcal{G} + \operatorname{Lie} L^+ \mathcal{G} = \operatorname{Lie}L \mathcal{G}\). 
\begin{lemma}
    \label{res:negative-loop-group-intersection}
    We have \(L^+\mathcal{G} \cap L^{--}\mathcal{G} = \lbrace 1 \rbrace\) on noetherian \(\mathcal{O}\)-algebras, i. e. \(L^+\mathcal{G}(R) \cap L^{--}\mathcal{G}(R) = \lbrace 1 \rbrace\) for any noetherian \(\mathcal{O}\)-algebra \(R\). 
\end{lemma}
\begin{proof}
    Let \(R\) be a noetherian \(\mathcal{O}\)-algebra.
    Then \(\operatorname{Spec} R\llb v \rrb \cup \operatorname{Spec} R\left[ \frac{1}{v} \right] \twoheadrightarrow \mathbb{P}^1_{R}\) is an fpqc covering by \cite[\href{https://stacks.math.columbia.edu/tag/00MB}{Lemma 00MB}]{stacks-project}, so we get a diagram 
    % https://q.uiver.app/#q=WzAsNixbMCwwLCJcXG1hdGhjYWx7R31feChcXG1hdGhiYntQfV4xX1IpIl0sWzEsMCwiXFx3aWRlaGF0e0d9XFxsZWZ0KFJcXGxlZnRbXFxmcmFjezF9e3Z9XFxyaWdodF1cXHJpZ2h0KSBcXHRpbWVzIFxcbWF0aGNhbHtHfV94KFJcXGxsYiB2IFxccnJiKSJdLFsyLDAsIlxcd2lkZWhhdHtHfShSXFxsbHAgdiBcXHJycCkiXSxbMiwxLCJcXHdpZGVoYXR7R30oUlxcbGxwIHYgXFxycnApIl0sWzEsMSwiXFx3aWRlaGF0e0d9XFxsZWZ0KFJcXGxlZnRbXFxmcmFjezF9e3Z9XFxyaWdodF1cXHJpZ2h0KSBcXHRpbWVzIFxcd2lkZWhhdHtHfShSXFxsbGIgdiBcXHJyYikiXSxbMCwxLCJcXHdpZGVoYXR7R30oUikiXSxbMCwxLCIiLDAseyJzdHlsZSI6eyJ0YWlsIjp7Im5hbWUiOiJob29rIiwic2lkZSI6InRvcCJ9fX1dLFsxLDIsIiIsMCx7Im9mZnNldCI6LTF9XSxbMSwyLCIiLDAseyJvZmZzZXQiOjF9XSxbNCwzLCIiLDAseyJvZmZzZXQiOi0xfV0sWzQsMywiIiwyLHsib2Zmc2V0IjoxfV0sWzIsMywiIiwyLHsic3R5bGUiOnsidGFpbCI6eyJuYW1lIjoiaG9vayIsInNpZGUiOiJ0b3AifX19XSxbMSw0LCIiLDIseyJzdHlsZSI6eyJ0YWlsIjp7Im5hbWUiOiJob29rIiwic2lkZSI6InRvcCJ9fX1dLFs1LDQsIiIsMix7InN0eWxlIjp7InRhaWwiOnsibmFtZSI6Imhvb2siLCJzaWRlIjoidG9wIn19fV0sWzAsNSwiIiwyLHsic3R5bGUiOnsidGFpbCI6eyJuYW1lIjoiaG9vayIsInNpZGUiOiJ0b3AifX19XV0=
    \[\begin{tikzcd}
        {\mathcal{G}_x(\mathbb{P}^1_R)} & {\mathcal{G}_{x}\left(R\left[\frac{1}{v}\right]\right) \times \mathcal{G}_x(R\llb v \rrb)} & {\mathcal{G}(R\llp v \rrp)} \\
        {\widehat{G}(R)} & {\widehat{G}\left(R\left[\frac{1}{v}\right]\right) \times \widehat{G}(R\llb v \rrb)} & {\widehat{G}(R\llp v \rrp)}
        \arrow[hook, from=1-1, to=1-2]
        \arrow[shift left, from=1-2, to=1-3]
        \arrow[shift right, from=1-2, to=1-3]
        \arrow[shift left, from=2-2, to=2-3]
        \arrow[shift right, from=2-2, to=2-3]
        \arrow[equal, from=1-3, to=2-3]
        \arrow[hook, from=1-2, to=2-2]
        \arrow[hook, from=2-1, to=2-2]
        \arrow[hook, dashed, from=1-1, to=2-1]
    \end{tikzcd}\]
    in which both rows are equalizers. Note that \(\widehat{G}(\mathbb{P}^1_{R}) = \widehat{G}(R)\) since \(\widehat{G}\) is a constant affine scheme. 
    By construction, \(\widehat{G}(R) \cap \mathcal{G}\left(R\left[\frac{1}{v}\right]\right) = \widehat{U}^\text{op}(R)\) and \(\widehat{G}(R) \cap \mathcal{G}_x(R\llb v \rrb) = \widehat{P}(R)\), 
    so \(\mathcal{G}_{x}(\mathbb{P}^1_{R}) = \widehat{U}^\text{op}(R) \cap \widehat{P}(R) = \lbrace 1 \rbrace\).
    Using the fpqc covering \(\operatorname{Spec}R \llb v + p \rrb \cup \operatorname{Spec}R \left[ \frac{1}{v+p}\right]\) instead, we find that 
    \[
        \mathcal{G}_{x}\left( R\llb v + p \rrb \right) \cap \mathcal{G}_{x} \left( R\left[\frac{1}{v+p}\right]\right) = \mathcal{G}_{x}(\mathbb{P}^1_{R}) = \lbrace 1 \rbrace . 
    \]
\end{proof}
We will now discuss the Lie algebra. 
Let \(J \to R \to \overline{R}\) be a square zero extension.
We define \(J \llp v + p \rrp \to R \llp v + p \rrp \to \overline{R} \llp v + p \rrp\) as the kernel, and similarly \(J \left[ \frac{1}{v+p} \right]\) and \(J \llb v + p \rrb\). 
Since \(J\llp v +p \rrp = \frac{1}{v+p}J\left[\frac{1}{v+p}\right] \oplus J\llb v + p \rrb\) and \(vJ\llp v + p \rrp = \frac{v}{v+p}J \left[ \frac{1}{v+p} \right] \oplus v J \llb v + p \rrb\), we have for any constant group \(H\) over \(\mathbb{P}^1_{R}\) the identities
\begin{align}
    \label{eq:lie-algebra-identity-1}
    \operatorname{Lie}H \left( J \llp v + p \rrp \right) & = \frac{1}{v+p}\operatorname{Lie} H \left( J \left[ \frac{1}{v+p} \right] \right) \oplus \operatorname{Lie} H \left( J \llb v + p \rrb \right) , \\
    \label{eq:lie-algebra-identity-2}
    v\operatorname{Lie}H \left( J \llp v + p \rrp \right) & = \frac{v}{v+p}\operatorname{Lie} H \left( J \left[ \frac{1}{v+p} \right] \right) \oplus v\operatorname{Lie} H \left( J \llb v + p \rrb \right) .
\end{align}
We will in particular apply these identities with \(H = \widehat{G}, \widehat{U}^\text{op}, \widehat{P}\). 
Note that the decomposition 
\begin{equation}
    \label{eq:lie-algebra-of-G-hat-decomposition}
    \operatorname{Lie}\widehat{G} = \operatorname{Lie}\widehat{U}^\text{op} \oplus \operatorname{Lie}\widehat{P}
\end{equation}
is compatible with \eqref{eq:lie-algebra-identity-1} and \eqref{eq:lie-algebra-identity-2}. 
\begin{lemma}
    \label{res:negative-loop-group-lie-algebra}
    Let \(J \to R \to \overline{R}\) be a square zero extension of noetherian \(\mathcal{O}\)-algebras. Then 
    \[
        \operatorname{Lie}\mathcal{G}_{x}\left(J \left[ \frac{1}{v+p} \right] \right) \oplus \operatorname{Lie}\mathcal{G}_{x}\left( J \llb v+p \rrb \right) = \operatorname{Lie}\mathcal{G}_{x}\left( J \llp v + p \rrp \right) . 
    \]
\end{lemma}
\begin{proof}
    The assumption that \(R\) is noetherian ensures that \(v\) is a non-zerodivisor in \(R\llb v + p \rrb\) \cite{fieldsZeroDivisorsNilpotent1971} (this property fails if \(R\) has infinite \(p\)-torsion), so \(v\) is also a non-zerodivisor in \(R \llp v + p \rrp\). 
    By \cref{rem:dilation-at-infty-doesnt-affect-loop-groups}, \cref{ass:pappas-zhu-groups-as-dilations}, and the universal property of dilations, we then have
    \[
        \mathcal{G}_{x}\left( R\llp v + p \rrp \right) = \left\lbrace A \in \widehat{G}\left( R \llp v + p \rrp \right) | A \mod v \in \widehat{P}\left( R\llp v + p \rrp / v R \llp v + p \rrp \right) \right\rbrace , 
    \]
    and we may describe \(\mathcal{G}_{x}\left( \overline{R}\llp v + p \rrp \right)\) similarly. 
    By the decomposition \eqref{eq:lie-algebra-of-G-hat-decomposition}, it follows that we can identify \(\operatorname{Lie}\mathcal{G}_{x}\left( J \llp v + p \rrp \right) \subset \operatorname{Lie} \widehat{G} \left( J \llp v + p \rrp \right)\) as the sub-Lie algebra
    \begin{equation}
        \label{eq:lie-algebra-of-the-group} 
        \operatorname{Lie}\mathcal{G}_{x}\left( J \llp v + p \rrp \right) = v \operatorname{Lie}\widehat{U}^\text{op}\left( J \llp v + p \rrp \right) 
        \oplus \operatorname{Lie} \widehat{P} \left( J \llp v + p \rrp \right) .
    \end{equation}
    We can also show (similarly to the above for the first identification, using \cref{res:explicit-negative-loop-group} for the second identification)
    \begin{align*}
        \operatorname{Lie}\mathcal{G}_{x}\left( J \llb v + p \rrb \right) 
        & = v \operatorname{Lie} \widehat{U}^\text{op} \left( J \llb v + p \rrb \right) + \operatorname{Lie} \widehat{P} \left( J \llb v + p \rrb \right) , \\ 
        \operatorname{Lie}\mathcal{G}_{x} \left( J \left[ \frac{1}{v+p} \right] \right) 
        & = \frac{v}{v+p} \operatorname{Lie} \widehat{U}^\text{op}\left( J \left[ \frac{1}{v+p} \right] \right) + \frac{1}{v+p} \operatorname{Lie} \widehat{P}\left( J \left[ \frac{1}{v+p} \right] \right) , \\ 
    \end{align*}
    The sum of these two equations is seen to be \eqref{eq:lie-algebra-of-the-group} using \eqref{eq:lie-algebra-identity-1} with \(H = \widehat{P}\) and \eqref{eq:lie-algebra-identity-2} with \(H = \widehat{U}^\text{op}\). 
\end{proof}
\begin{proposition}
    \label{res:formally-etale-monomorphism}
    The multiplication map \(m : L^+ \mathcal{G}_{x} \times L^{--} \mathcal{G}_{x} \hookrightarrow L \mathcal{G}_{x}\) is a formally étale monomorphism on noetherian \(\mathcal{O}\)-algebras.\footnote{
        By ``formally étale monomorphism on noetherian \(\mathcal{O}\)-algebras'' we mean that when we restrict these loop group functors to noetherian \(\mathcal{O}\)-algebras, the map is a monomorphism and satisfies the definition of being formally étale with respect to square zero extensions of noetherian \(\mathcal{O}\)-algebras.
    }
\end{proposition}
\begin{proof}
    This is just like \cite[Lemma 3.2.6]{local-models}. We reproduce the proof here for the convenience of the reader. 

    We first show that \(m\) is a monomorphism. Let \(R\) be an \(\mathcal{O}\)-algebra and \(g_{1},g_{2} \in L^+ \mathcal{G}_{x}(R)\), \(h_{1},h_{2} \in L^{--}\mathcal{G}_{x}(R)\). 
    If \(g_{1}h_{1} = g_{2}h_{2}\), then \(g_{2}^{-1}g_{1} = h_{2}h_{1}^{-1}\), so it is immediate from \cref{res:negative-loop-group-intersection} that \(g_{1} = g_{2}\) and \(h_{1} = h_{2}\).

    It remains to verify the formally étale property.
    Let \(J \to R \to \overline{R}\) be a square zero extension of \(\mathcal{O}\)-algebras, and consider the lifting problem 
    \[\begin{tikzcd}
        {\operatorname{Spec}\overline{R}} & {L^+\mathcal{G} \times L^{--} \mathcal{G} } \\
        {\operatorname{Spec}R} & {L \mathcal{G}} . 
        \arrow[from=1-1, to=1-2, "{(\overline{x},\overline{y})}"]
        \arrow[from=1-2, to=2-2, "m"]
        \arrow[from=2-1, to=2-2, "g"]
        \arrow[from=1-1, to=2-1]
        \arrow[dashed, from=2-1, to=1-2]
    \end{tikzcd}\]
    That is, we have an element \(g \in L \mathcal{G}(R)\) and a decomposition \(\overline{g} = \overline{x}\overline{y} \in L\mathcal{G}(\overline{R})\), with \(\overline{x} \in L^{+}\mathcal{G}(\overline{R})\) and \(\overline{y} \in L^{--}\mathcal{G}(\overline{R})\). 
    The solution to the lifting problem is given by lifts \(x \in L^+\mathcal{G}(R)\) of \(\overline{x}\) and \(y \in L^{--}\mathcal{G}(R)\) of \(\overline{y}\) such that \(g = xy\).
    Note that since \(m\) is a monomorphism, the solution is necessarily unique, so it suffices to show that one exists. 

    Since both \(L^+ \mathcal{G}\) and \(L^{--}\mathcal{G}\) are formally smooth (by smoothness of \(\mathcal{G}\)), we can choose \emph{some} lifts \(x' \in L^+\mathcal{G}(R)\) of \(\overline{x}\) and \(y' \in L^{--}\mathcal{G}(R)\) of \(\overline{y}\). 
    Then \((x')^{-1}g (y')^{-1} \in L\mathcal{G}(R)\) reduces to \(\overline{x}^{-1} \overline{g} \overline{y}^{-1} = \overline{x}^{-1}\overline{x}\overline{y}\overline{y}^{-1} = 1\) in \(L\mathcal{G}(\overline{R})\), 
    so \((x')^{-1}g(y')^{-1} \in \operatorname{Lie}L\mathcal{G}(J)\).
    By \cref{res:negative-loop-group-lie-algebra} there exists \(a \in \operatorname{Lie}L^{+}\mathcal{G}(J)\) and \(b \in \operatorname{Lie}L^{--}\mathcal{G}(J)\) such that 
    \((x')^{-1} g (y')^{-1} = ab\), and then \(g = (x'a)(by')\), so \(x = x'a\) and \(y = by'\) solves the lifting problem. 
\end{proof}

\subsection{Frobenius}
\label{sec:frobenius-on-loop-group}
Just as we constructed \(\mathcal{G}_{x}\) in \cref{sec:bruhat-tits-group-schemes}, we may construct \(\mathcal{G}_{\varphi(x)}\) in the same way.
From \cref{sec:frobenius-on-building} we see that \(\varphi(x) = \varphi(n) \cdot o\).
Note that 
\begin{equation}
    \label{eq:frobenius-pullback-of-G-tilde}
    \varphi^* \widetilde{\mathcal{G}}_{x} = \varphi^* (n \widehat{G} n^{-1}) \cong \varphi(n) \widehat{G} \varphi(n)^{-1} = \widetilde{\mathcal{G}}_{\varphi(x)} . 
\end{equation}
However, it does not follow trivially that \(\varphi^* \breve{\mathcal{G}}_{x} \cong \breve{\mathcal{G}}_{\varphi(x)}\), because the commutative diagram 
\[
\begin{tikzcd}
\widetilde{\mathbb{A}}^1_{\mathcal{O}} \arrow[r,"\varphi"] \arrow[d,"\pi_{I}"] & \widetilde{\mathbb{A}}^1_{\mathcal{O}} \arrow[d,"\pi_{I}"] \\
\breve{\mathbb{A}}^1_{\mathcal{O}} \arrow[r,"\varphi"] & \breve{\mathbb{A}}^1_{\mathcal{O}}
\end{tikzcd}
\]
is not cartesian.\footnote{However, note that the corresponding diagram away from the origin \(\lbrace 0 \rbrace\) (i. e. after inverting \(v\)) is cartesian. So is the diagram with \(\pi_{0} : \breve{\mathbb{A}}^1_{\mathcal{O}} \to \mathbb{A}^1_{\mathcal{O}}\) in place of \(\pi_{I}\).} 
There is a comparison map 
\[
\varphi^* \breve{\mathcal{G}}_{x} = \varphi^* (\pi_{I})_{*}^I \widetilde{\mathcal{G}}_{x} \to (\pi_{I})_{*}^I \varphi^* \widetilde{\mathcal{G}}_{x} = \breve{\mathcal{G}}_{\varphi(x)} ,
\]
but it is not necessarily an isomorphism, as the following (typical) example shows. 
\begin{example}
    \label{ex:frobenius-pullback-example}
    Consider the case that \(L = \mathbb{Q}_{p}\left( (-p)^{1 / (p-1)} \right)\) (where \(p > 2\)), \(\widehat{G} = \GL_{2}\), and \(x = u^{(0,1)} \cdot o\).
    Note that \(e = p-1\), \(\# \mathcal{J} = 1\), and \(\breve{\mathcal{G}}_{x} = \mathcal{G}_{x}\) in this case. 
    Let \(R\) be an \(\mathcal{O}[u]\)-algebra in which \(u\) is a non-zerodivisor.
    Then \(\widetilde{\mathcal{G}}_{x}(R) = u^{(0,1)} \GL_{2}(R) u^{-(0,1)}(R) = \begin{pmatrix} R & u^{-1} R \\ u R & R \end{pmatrix}\), 
    and \(\varphi^* \widetilde{\mathcal{G}}_{x}(R) = u^{(0,p)} \GL_{2}(R) u^{- (0, p)} = \begin{pmatrix} R & u^{-p} R \\ u^p R & R \end{pmatrix} = \widetilde{\mathcal{G}}_{\varphi(x)}\) (this notation is slightly abusive since \(\GL_{2}\) only consists of invertible matrices).
    On the other hand, if \(R\) is an \(\mathcal{O}[v]\)-algebra in which \(v = u^{p-1}\) is a non-zerodivisor, then 
    \(\mathcal{G}_{x}(R) = \begin{pmatrix} R & R \\ v R & R \end{pmatrix}\) and \(\varphi^* \mathcal{G}_{x}(R) = \begin{pmatrix} R & R \\ v^p R & R \end{pmatrix}\), 
    but \(\mathcal{G}_{\varphi(x)}(R) = \begin{pmatrix} R & v^{-1} R \\ v^2 R & R \end{pmatrix}\). 
\end{example}
\begin{remark}
    It is actually \(\breve{\mathcal{G}}_{\varphi(x)} = (\pi_{I})_{*}^I \varphi^* \widetilde{\mathcal{G}}_{x}\) as well as \(\mathcal{G}_{\varphi(x)} = \pi_{*}^\Gamma \varphi^* \widetilde{\mathcal{G}_{x}}\) that is relevant for our purposes, and not \(\varphi^* \breve{\mathcal{G}}_{x}\) and \(\varphi^* \mathcal{G}_{x}\).
    This is important for the presentation \eqref{res:presentation-of-moduli-stack-of-BKL-parameters} of the moduli stack of Breuil--Kisin \((\Gamma,\widehat{G})\)-torsors of fixed type. 
    We will see this again in \cref{sec:moduli-interpretation-of-Y}. 
\end{remark}
\begin{remark}
    Let us also remark that \(\varphi^* \widetilde{\mathcal{U}}_{a,x} = \widetilde{\mathcal{U}}_{a,\varphi(x)}\), as a special case of \cref{res:frobenius-pullback-of-root-groups-with-congruence}. 
    This implies that \(\varphi^* \widetilde{\mathcal{V}}_{x} = \widetilde{\mathcal{V}}_{\varphi(x)}\), which provides an alternative way to see the isomorphism \eqref{eq:frobenius-pullback-of-G-tilde}, by evaluating at \(E^\mathcal{J}\llb u \rrb\)-points and using the uniqueness part of the proof of \cite[Theorem 4.1]{pappas-zhu} again.
    One may similarly justify the description of \(\varphi^* \mathcal{G}_{x}\) in \cref{ex:frobenius-pullback-example}. 
\end{remark}

\subsubsection{Comparing \(\mathcal{G}_{\varphi(x)}\) and \(\mathcal{G}_{x}\)}
Recall from \cref{sec:galois-type-setup} that we have made a choice of \(c \in G^*(\mathcal{O}[v^{\pm 1}])\) such that \(c \cdot \varphi(x) = x\).
It is then clear that we have an isomorphism 
\[\operatorname{Ad}_{c} : \widetilde{\mathcal{G}}_{\varphi(x)} \isoto \widetilde{\mathcal{G}}_{x} \] 
of group schemes over \(\widetilde{\mathbb{A}}^1_{\mathcal{O}}\). 
This is \(\Gamma\)-equivariant since \(c\) is fixed by \(\Gamma\), so by functoriality of invariant pushforward we obtain isomorphisms 
\begin{align*}
    \operatorname{Ad}_{c} : \breve{\mathcal{G}}_{\varphi(x)} & \isoto \breve{\mathcal{G}}_{x} , \\
    \operatorname{Ad}_{c} : \mathcal{G}_{\varphi(x)} & \isoto \mathcal{G}_{x} . 
\end{align*}
\begin{example}
    In the situation of \cref{ex:frobenius-pullback-example}, the element \(c = v^{(0, -1)}\) satisfies \(c \cdot \varphi(x) = x\), since \(v^{(0,-1)}\begin{pmatrix} R & v^{-1} R \\ v^2 R & R \end{pmatrix}v^{-(0,-1)} = \begin{pmatrix} R & R \\ v R & R \end{pmatrix}\). 
\end{example}

\subsubsection{Frobenius endomorphism of loop groups}
\label{sec:frobenius-endomorphism-of-loop-groups}
We now turn to the Frobenius endomorphisms of loop groups.
The construction in \cref{sec:frobenius} yields a \(\Gamma\)-equivariant endomorphism \(\varphi : L_{0} \widehat{G} \to L_{0} \widehat{G}\), where \(L_{0}\widehat{G}(R) = \widehat{G}\left( (R \otimes_{\mathbb{Z}_{p}} \mathbb{Z}_{q})\llp u \rrp\right)\).
Taking \(\Gamma\)-fixed points we obtain an endomorphism
\[
\varphi : L_{0} G^* \to L_{0} G^* , 
\]
where \(L_{0}G^*(R) = G^* \left( R \llp v \rrp \right)\). 
Moreover, we can view \(L_{0}^+\mathcal{G}_{x} \subset L_{0}G^*\) as a subfunctor. 
We will need the operator 
\begin{align}
    \label{eq:varphi-c}
    \varphi_{c} : L_{0}G^* & \to L_{0} G^* \\
    A & \mapsto c \varphi(A) c^{-1} . \nonumber
\end{align}
\begin{lemma}
    \label{res:frobenius-on-parahoric}
    We have \(\varphi_{c} \left( L_{0}^+ \mathcal{G}_{x} \right) \subset L_{0}^+ \mathcal{G}_{x}\). 
\end{lemma}
\begin{proof}
    It is immediate from \eqref{eq:frobenius-pullback-of-G-tilde} that \(\varphi\left( L_{0}^+ \widetilde{\mathcal{G}}_{x} \right) \subset L_{0}^+ \widetilde{\mathcal{G}}_{\varphi(x)}\), where for \(y = x, \varphi(x)\) we have \(L_{0}^+ \widetilde{\mathcal{G}}_{y}(R) = \widetilde{\mathcal{G}}_{y}\left( (R \otimes_{\mathbb{Z}_{p}} \mathbb{Z}_{q})\llp u \rrp\right)\) and we view \(L_{0}^+ \widetilde{\mathcal{G}}_{y} \subset L_{0} \widehat{G}\) as a subfunctor. 
    Since \(c \cdot \varphi(x) = x\) it then follows that \(\varphi_{c} \left( L_{0}^+ \widetilde{\mathcal{G}}_{x} \right) = c \varphi \left( L_{0}^+ \widetilde{\mathcal{G}}_{x} \right) c^{-1} \subset c L_{0}^+\widetilde{\mathcal{G}}_{\varphi(x)} c^{-1} = L_{0}^+\widetilde{\mathcal{G}}_{x}\).
    Taking \(\Gamma\)-fixed points, the result follows.
\end{proof}
\section{The affine grassmannian and the stack \(\mathcal{Y}\)}
In this section we introduce the two main geometric objects of interest.
The first object is the affine Grassmannian, as described in \cite[Section Section 6.2.6]{pappas-zhu}. 
This is the global affine Grassmannian for \(\mathcal{G}\) specialized to \(v = - p\), and we will denote it by \(\operatorname{Gr}_{\mathcal{G}}\) and refer to it simply as ``the affine Grassmannian''.
The second object is the moduli stack of Breuil--Kisin \((\Gamma,\widehat{G})\)-torsors, denoted \(\mathcal{Y}\). 
We will work with \(\mathcal{Y}\) as a quotient stack. 

We keep the notation and assumptions from \cref{sec:setup} and \cref{sec:bruhat-tits-group-schemes-and-loop-groups}.
In particular, we make \cref{ass:connected-fibers-assumptions} by default. 

\subsection{The affine Grassmannian}
\label{sec:affine-grassmannian}
We define the affine Grassmannian is the fppf quotient sheaf over \(\mathcal{O}\)
\[
\operatorname{Gr}_{\mathcal{G}} := \left[ L^+ \mathcal{G}  \backslash L \mathcal{G} \right] ,
\]
i. e. the fppf sheafification of the functor \(R \mapsto \mathcal{G}\left( R \llb v + p \rrb \right) \backslash \mathcal{G}\left(R \llp v + p \rrp \right) \).
The fppf sheaf \(\operatorname{Gr}_{\mathcal{G}}\) is a sheaf in the fpqc topology by \cite[Proposition 1.6]{cesnaviciusAffineGrassmannianPresheaf2024}, and therefore agrees with the ``fpqc-sheafification'' definition often taken in the literature. 
This is the specialization at \(v = - p\) of the global affine Grassmannian for \(\mathcal{G}\) over \(\mathbb{A}^1_{\mathcal{O}}\) \cite[Section 6.2]{pappas-zhu}. 
By \cite[Proposition 6.5]{pappas-zhu} \(\operatorname{Gr}_{\mathcal{G}}\) is a projective Ind-scheme over \(\mathcal{O}\). 

In what follows we will often use the convention \(\operatorname{Gr} := \operatorname{Gr}_{\mathcal{G}}\) to simplify notation. 

\subsubsection{The special and generic fiber} 
\label{sec:special-and-generic-fibers}
We reiterate that \(\operatorname{Gr}\) (without any subscript) is the mixed characteristic affine Grassmannian for \(\mathcal{G}\). 
Following \cite[Corollary 6.6]{pappas-zhu}, we describe its generic fiber \(\operatorname{Gr}_{E}\) and special fiber \(\operatorname{Gr}_{\mathbb{F}}\).
It will be helpful to keep in mind that over \(\mathbb{A}^1_{\mathcal{O}}\), we have \(p\) is invertible in the generic fiber \(\mathbb{A}^1_{E}\), but \(0 = -p\) in the special fiber \(\mathbb{A}^1_{\mathbb{F}}\). 

Let \(\operatorname{Gr}_{E} = \left[ L^+ \mathcal{G}_{E} \backslash L \mathcal{G}_{E} \right]\) denote the generic fiber of \(\operatorname{Gr}\) over \(E\). 
Because of the identity \cite[Section 6.2.6]{pappas-zhu}
\[
    \mathcal{G} \times_{\mathcal{O}[v]} \operatorname{Spec} E \llb v + p \rrb \cong G^*_{E} \times_{E} \operatorname{Spec} E \llb v + p \rrb ,
\]
it follows that we can identify \(\operatorname{Gr}_{E}\) with the affine Grassmannian for the connected reductive group \(G^*_{E}\) over \(E\). 

Similarly, we let \(\operatorname{Gr}_{\mathbb{F}} = \left[ L^+ \mathcal{G}_{\mathbb{F}} \backslash L \mathcal{G}_{\mathbb{F}} \right]\) denote the special fiber of \(\operatorname{Gr}\) over \(\mathbb{F}\). 
This is the affine flag variety over \(\mathbb{F}\) for the connected reductive group \(G^*_{\mathbb{F}\llp v \rrp}\) over \(\mathbb{F}\llp v \rrp\) and the parahoric group scheme \(\mathcal{G}_{\mathbb{F}\llb v \rrb}\). 

\subsubsection{Pappas--Zhu local model}
\label{sec:pappas-zhu-local-model}
We first recall the definition of the Schubert variety for \(\mu\). 
By Cartan decomposition for \(G^*_{\overline{E}} = \widehat{G}_{\overline{E}}\) we have 
\begin{equation}
	\label{eq:cartan-decomposition}
	L \mathcal{G}(\overline{E}) = \widehat{G} \left( \overline{E} \llp v + p \rrp \right) = \bigsqcup_{\nu \in X_{*}(\widehat{T})^{\text{dom}}} L^+\mathcal{G}(\overline{E}) \cdot (v+p)^\nu \cdot L^+ \mathcal{G}(\overline{E}) .
\end{equation}
The Schubert variety \(\operatorname{Gr}_{E}^{\leq \mu} \subset \operatorname{Gr}_{E}\) (as described in for example \cite[Section 2.1]{zhu-affine-grassmannians}) 
is the reduced closed subscheme with \(\overline{E}\)-points 
\begin{equation}
	\label{eq:schubert-variety}
    \operatorname{Gr}_{E}^{\leq \mu}(\overline{E}) = \bigsqcup_{\nu \in X_{*}(T^*_{E})^{\text{dom}}, \nu \leq \mu} L^+\mathcal{G}(\overline{E}) \cdot (v+p)^\nu \cdot L^+ \mathcal{G}(\overline{E}) , 
\end{equation}
where \(\nu \leq \mu\) means that \(\mu - \nu\) is a sum of positive roots. 
\begin{definition}[{\cite[Definition 7.1]{pappas-zhu}}]
    \label{def:pappas-zhu-local-model}
    The \emph{Pappas--Zhu local model} for \(\mathcal{G}\) with respect to \(\mu\), denoted \(\operatorname{Gr}^{\leq \mu} = \operatorname{Gr}_{\mathcal{G}}^{\leq \mu}\), is the reduced closure of \(\operatorname{Gr}_{E}^{\leq \mu}\) in \(\operatorname{Gr}\).
\end{definition}
\begin{remark}
    The Pappas--Zhu local model \(\operatorname{Gr}^{\leq \mu}\) is often denoted \(M_{\mathcal{G},\mu}\), for example in \cite{pappas-zhu}.
\end{remark}

Note that \(\operatorname{Gr}^{\leq \mu}\) is a projective scheme over \(\mathcal{O}\) by (Ind-)projectivity of \(\operatorname{Gr}\).
The generic fiber of \(\operatorname{Gr}^{\leq \mu}\) is (clearly) the Schubert variety \(\operatorname{Gr}_{E}^{\leq \mu}\), and the special fiber is identified by \cite[Theorem 9.3]{pappas-zhu} in terms of the admissible set \(\operatorname{Adm}(\mu)\) of \(\mu\), which we recall below. 
\subsubsection{Digression: The admissible set and the special fiber}
\label{sec:admissible-set}
We will recall the description and significance of the admissible set \(\operatorname{Adm}(\mu)\). 
By definition, the admissible set \(\operatorname{Adm}(\mu)\) consists of elements of the Iwahori--Weyl group \(\widetilde{W}^*\) that we defined in \cref{sec:iwahori-weyl}, so let us first recall some facts about the Iwahori--Weyl group. 
When \(I\) acts trivially on \(\widehat{G}\) and \(\mathbb{F} \supset \mathbb{F}_{q}\), then \(G^* = \widehat{G}\) by \cref{rem:G-star-is-G-hat} and we have \(\widetilde{W}^* = \widehat{N}(\mathbb{F}\llp v \rrp) / \widehat{T}(\mathbb{F}\llb v \rrb)\). 
Now, if \(x\) is \(0\)-generic (so the connected stabilizer of \(x\) in \(G^*(\mathbb{F}\llp v \rrp)\) is an Iwahori subgroup), then the Iwahori--Weyl group parametrizes double cosets, in the sense that the map
\begin{align}
    \widetilde{W}^* & \isoto L^+ \mathcal{G}_{x, \mathbb{F}}\backslash L \mathcal{G}_{x, \mathbb{F}} / L^+ \mathcal{G}_{x, \mathbb{F}} = \operatorname{Gr}_{\mathbb{F}}  / L^+ \mathcal{G}_{x, \mathbb{F}} \nonumber \\ 
    \widetilde{w} & \mapsto L^+ \mathcal{G}_{x}\widetilde{w} L^+ \mathcal{G}_{x} \nonumber 
\end{align}
is a bijection \cite[Theorem 7.8.1]{kaletha-prasad}. More generally, if \(x\) is not \(0\)-generic, we can define \(W_{x} \subset \widetilde{W}^*\) as the stabilizer of the facet containing \(x\) in the affine Weyl group \(W^\text{aff} \subset \widetilde{W}^*\), and by loc. cit. we have a bijection 
\begin{align}
    \label{eq:IW-double-cosets} 
    W_{x} \backslash \widetilde{W}^* / W_{x} \isoto \operatorname{Gr}_{\mathbb{F}} / L^+ \mathcal{G}_{x} . 
\end{align}

We now recall some facts from the study of affine flag varieties.
The double coset \(S^\circ(\widetilde{w}) = L^+ \mathcal{G}_{x} \widetilde{w} L^+\mathcal{G}_{x}\) is called an (open) affine Schubert cell, and if \(W_{x}\widetilde{w}W_{x} \neq W_{x}\widetilde{w}'W_{x}\), the associated affine Schubert cells \(S^\circ(\widetilde{w})\) and \(S^\circ(\widetilde{w}')\) are disjoint by \eqref{eq:IW-double-cosets}. 
The closure of \(S^\circ (\widetilde{w})\) in \(\operatorname{Gr}_{\mathbb{F}}\) is the affine Schubert variety \(S(\widetilde{w}) \subset \operatorname{Gr}_{\mathbb{F}}\). 
The Bruhat ordering on \(\widetilde{W}^*\) is then characterized by the property that 
\[
    S(\widetilde{w}) = \bigsqcup_{W_{x}\widetilde{w}'W_{x} \in W_{x} \backslash \widetilde{W}^* / W_{x}, \widetilde{w}' \leq \widetilde{w}} S^\circ (\widetilde{w}') , 
\]
where this is the disjoint union on the underlying points (not as schemes). 
The Bruhat ordering also has a combinatorial description, which we now recall. 

First consider the affine Weyl group \(W^\text{aff} \cong \Phi(\widehat{G},\widehat{T})\mathbb{Z} \rtimes W\), which is a Coxeter group and can be identified with the Iwahori--Weyl group of the simply connected cover of the derived group of \(\widehat{G}\). 
Let \(\mathcal{C}\) denote an alcove whose closure contains \(x\). 
Let \(\lbrace s_{i} \rbrace_{i \in \mathcal{I}} \subset W^{\text{aff}}\) denote the affine reflections across walls bounding \(\mathcal{C}\).
For any \(\widetilde{z} \in W^\text{aff}\) we may then write \(\widetilde{z}\) as an expression 
\begin{equation}
    \label{eq:reduced-expression}
    \widetilde{z} = s_{i_{1}} s_{i_{2}} \cdots s_{i_{l}}, 
\end{equation}
which geometrically corresponds to a ``gallery walk'' from the alcove \(\mathcal{C}\) to the alcove \(\widetilde{z} \mathcal{C}\).
We call the expression \eqref{eq:reduced-expression} \emph{reduced} if \(l = l(\widetilde{z})\) is minimal.
If \(\widetilde{z}' \in W^\text{aff}\), then \(\widetilde{z}' \leq \widetilde{z}\) in the Bruhat order if and only if there exists a reduced expression for \(\widetilde{z}\) such that a reduced expression for \(\widetilde{z}'\) can be obtained by deleting some of the \(\widetilde{s}_{i_{j}}\)'s from the right hand side of \eqref{eq:reduced-expression}.

We have a decomposition 
\[
    \widetilde{W}^* = W^\text{aff} \rtimes \widetilde{W}^*_{\mathcal{C}} , 
\]
where \(\widetilde{W}^*_{\mathcal{C}} \subset \widetilde{W}^*\) is the stabilizer of \(\mathcal{C}\) in \(\widetilde{W}^*\). 
One way to observe this is that \(W^{\text{aff}}\) acts simply transitively on the set of alcoves in \(\mathcal{A}\left( T^*, \mathbb{F}\llp v \rrp \right)\), implying that \(W^{\text{aff}} \cap \widetilde{W}^*_{\mathcal{C}} = \lbrace 1 \rbrace\) in \(\widetilde{W}^*\)
and \(W^{\text{aff}}\) acts transitively on \(\widetilde{W}^* / \widetilde{W}^*_{\mathcal{C}}\) via left translations. 
The Bruhat ordering on \(\widetilde{W}^*\) is then induced from the Bruhat ordering on \(W^\text{aff}\). 
More precisely, if \(\widetilde{w} = \widetilde{z}\widetilde{t}\) and \(\widetilde{w}' = \widetilde{z}'\widetilde{t}'\), where \(\widetilde{z},\widetilde{z}' \in W^{\text{aff}}\) and \(\widetilde{t}, \widetilde{t}' \in \widetilde{W}^*_{\mathcal{C}}\),  then \(\widetilde{w}' \leq \widetilde{w}\) if and only if \(\widetilde{t} = \widetilde{t}'\) and \(\widetilde{z}' \leq \widetilde{z}\) (and if \(\widetilde{t} \neq \widetilde{t}'\) then \(\widetilde{w}\) and \(\widetilde{w}'\) are incomparable). 

We now return to the admissible set, giving a heuristic description in the unramified case. 
Consider \(\operatorname{Gr}^{\leq \mu}(\mathcal{O})\). This contains the classes represented by \((v+p)^\mu\), as well as \((v+p)^{w \mu}\) for any \(w \in W\).
Therefore, \(S^\circ(v^{w\mu}) = L^+ \mathcal{G}_{\mathbb{F}} v^{w\mu}  L^+ \mathcal{G}_{\mathbb{F}} \subset \operatorname{Gr}^{\leq \mu}_{\mathbb{F}}\) for every \(w \in W\). 
At the same time, \(\operatorname{Gr}^{\leq \mu}_{\mathbb{F}} \subset \operatorname{Gr}_{\mathbb{F}}\) is closed, so contains the closure 
\[
    S(v^{w\mu}) = \bigsqcup_{W_{x}\widetilde{w}'W_{x} \in W_{x} \backslash \widetilde{W}^* / W_{x}, \widetilde{w}' \leq v^{w\mu}} S^\circ (\widetilde{w}') . 
\]
The union of all these \(\widetilde{w}'\) as \(w\) ranges through \(W\) is precisely the admissible set of \(\mu\). That is, 
\[
    \operatorname{Adm}(\mu) := \lbrace \widetilde{w}' \in \widetilde{W}^* | \text{there exists }w \in W \text{ such that }\widetilde{w}' \leq v^{w \mu} \rbrace . 
\]
In \cref{sec:weil-res-example-admissible-set} we give a detailed description of the admissible set in one example. If \(x\) is not \(0\)-generic, we define 
\[
    \operatorname{Adm}^x(\mu) := W_{x} \operatorname{Adm}(\mu) W_{x} 
\]
as the image of \(\operatorname{Adm}(\mu)\) in \(W_{x} \backslash \widetilde{W}^* / W_{x}\). 

We have the following result, completely characterizing the special fiber \(\operatorname{Gr}_{\mathbb{F}}^{\leq \mu}\). 
\begin{theorem}
    \label{res:special-fiber-of-pappas-zhu-local-model}
    Suppose \(p \nmid \pi_{1}(G^*)\). 
    Then \(\operatorname{Gr}^{\leq \mu}_{\mathbb{F}}\) is geometrically reduced and we have \(\operatorname{Gr}^{\leq \mu}_{\overline{\mathbb{F}}} = \bigsqcup_{W_{x}\widetilde{w}'W_{x} \in \operatorname{Adm}^x(\mu)} S^\circ(\widetilde{w}')\).
    In particular, if \(x\) is \(0\)-generic, then \(\operatorname{Gr}^{\leq \mu}_{\overline{\mathbb{F}}} = \bigsqcup_{\widetilde{w}' \in \operatorname{Adm}(\mu)} S^\circ(\widetilde{w}')\). 
\end{theorem}
\begin{proof}
    This is \cite[Theorem 9.1, Theorem 9.3]{pappas-zhu}. 
    Note that the (easy) inclusion \(\supset\) is explained above.
\end{proof}
\begin{corollary}
    Suppose \(p \nmid \pi_{1}(G^*)\). 
    If \(I\) acts trivially on \(\widehat{G}\), then the irreducible components of \(\operatorname{Gr}^{\leq \mu}_{\overline{\mathbb{F}}}\) are exactly \(\lbrace S(v^{w \mu}) \rbrace_{w \in W}\).
\end{corollary}

\subsubsection{Bounded loop group}
\label{sec:bounded-loop-group}
We define a ``bounded'' subspace \(L^{\leq \mu} \mathcal{G} \subset L \mathcal{G}\) as the fiber product
% https://q.uiver.app/#q=WzAsNCxbMCwwLCJMXntcXGxlcSBcXG11fSBcXG1hdGhjYWx7R30iXSxbMSwwLCJMIFxcbWF0aGNhbHtHfSJdLFsxLDEsIlxcb3BlcmF0b3JuYW1le0dyfV97XFxtYXRoY2Fse0d9fSJdLFswLDEsIk1fe1xcbWF0aGNhbHtHfSxcXG11fSJdLFsxLDJdLFszLDJdLFswLDNdLFswLDFdLFswLDIsIiIsMSx7InN0eWxlIjp7Im5hbWUiOiJjb3JuZXIifX1dXQ==
\[\begin{tikzcd}
	{L^{\leq \mu} \mathcal{G}} & {L \mathcal{G}} \\
	\operatorname{Gr}^{\leq \mu} & {\operatorname{Gr}_{\mathcal{G}}} .
	\arrow[from=1-2, to=2-2]
	\arrow[from=2-1, to=2-2]
	\arrow[from=1-1, to=2-1]
	\arrow[from=1-1, to=1-2]
	\arrow["\lrcorner"{anchor=center, pos=0.125}, draw=none, from=1-1, to=2-2]
\end{tikzcd}\]
Since \(L \mathcal{G} \to \operatorname{Gr}_{\mathcal{G}}\) is a left \(L^+\mathcal{G}\)-torsor, so is \(L^{\leq \mu} \mathcal{G} \to \operatorname{Gr}^{\leq \mu}\). 
In particular, \(L^{\leq \mu}\mathcal{G}\) is representable by an affine scheme, using the fact that \(\operatorname{Gr}^{\leq \mu}\) is a projective scheme together with \cite[\href{https://stacks.math.columbia.edu/tag/01SG}{Lemma 01SG}]{stacks-project} and \cite[\href{https://stacks.math.columbia.edu/tag/0244}{Section 0244}]{stacks-project}.
This is in contrast to \(L \mathcal{G}\), which is only an Ind-scheme.

The scheme \(L^{\leq \mu}\mathcal{G}\) is reduced by \cref{res:reduced-times-positive-loop-group}. 
Note that the \(\overline{E}\)-points of \(L^{\leq \mu} \mathcal{G}\) are precisely given by \eqref{eq:schubert-variety} by Cartan decomposition, and the reduced 
closure of \(L^{\leq \mu}\mathcal{G}_{\overline{E}}\) in \(L \mathcal{G}\) is \(L^{\leq \mu} \mathcal{G}\) by \cref{res:reduced-closure-lemma}. 
Since \(L^{\leq \mu}\mathcal{G}_{\overline{E}}\) is evidently stable under both left and right translation by \(L^+ \mathcal{G}(\overline{E})\), 
its reduced closure \(L^{\leq \mu}\mathcal{G}\) must also be stable under both left and right translation by \(L^+ \mathcal{G}\). 

\subsubsection{Covers of the affine Grassmannian}
For a concave function \(f \geq 0\), we will consider the cover of the affine Grassmannian \(\operatorname{Gr}_{f}\) defined as the fppf quotient sheaf
\[
    \operatorname{Gr}_{f} := \left[ L^+ \mathcal{G}_{x,f} \backslash L \mathcal{G} \right] , 
\]
where we recall the group \(\mathcal{G}_{x,f}\) from \cref{res:pappas-zhu-concave-function}. 
The natural map \(\operatorname{Gr}_{f} \to \operatorname{Gr}\) is a \(L^+ \mathcal{G}_{x} / L^+ \mathcal{G}_{x,f}\)-torsor whenever \(L^+ \mathcal{G}_{x,f}\) is normal in \(L^+\mathcal{G}_{x}\). 

We recall the most important cases from \cref{sec:quotient-loop-groups}. If \(f = n \in \mathbb{Z}_{\geq 1}\), then \(\operatorname{Gr}_{n} \to \operatorname{Gr}\) is a \(L^n \mathcal{G}_{x}\)-torsor. 
If \(f = 0^+\) and \(x\) is \(0\)-generic, then \(\operatorname{Gr}_{0^+} \to \operatorname{Gr}\) is a \(\overline{\mathcal{T}}\)-torsor.

We also consider covers \(\operatorname{Gr}_{f}^{\leq\mu} := \left[ L^+ \mathcal{G}_{x,f} \backslash L^{\leq \mu}\mathcal{G} \right]\), and the above remarks still apply. 

\subsubsection{Affine charts in terms of the negative loop group}
\label{sec:negative-loop-group-charts}
In this subsection we invoke \cref{ass:pappas-zhu-groups-as-dilations}.
Define \(L^{--, \leq \mu}\mathcal{G}\) as the pullback 
% https://q.uiver.app/#q=WzAsNCxbMCwwLCJMXnstLSxcXGxlcVxcbXV9XFxtYXRoY2Fse0d9Il0sWzEsMCwiTF57LS19XFxtYXRoY2Fse0d9Il0sWzEsMSwiTFxcbWF0aGNhbHtHfSJdLFswLDEsIkxee1xcbGVxIFxcbXV9XFxtYXRoY2Fse0d9Il0sWzMsMl0sWzEsMl0sWzAsM10sWzAsMV0sWzAsMiwiIiwxLHsic3R5bGUiOnsibmFtZSI6ImNvcm5lciJ9fV1d
\[\begin{tikzcd}
	{L^{--,\leq\mu}\mathcal{G}} & {L^{--}\mathcal{G}} \\
	{L^{\leq \mu}\mathcal{G}} & {L\mathcal{G}}
	\arrow[from=2-1, to=2-2]
	\arrow[from=1-2, to=2-2]
	\arrow[from=1-1, to=2-1]
	\arrow[from=1-1, to=1-2]
	\arrow["\lrcorner"{anchor=center, pos=0.125}, draw=none, from=1-1, to=2-2]
\end{tikzcd}\]

\begin{lemma}
    \label{res:negative-loop-group-is-of-finite-type}
    \(L^{--,\leq \mu}\mathcal{G}\) is a scheme of finite type over \(\mathcal{O}\). 
\end{lemma}
\begin{proof}
    With respect to an embedding \(\widehat{G}\) into affine space over \(\mathbb{A}^{1}_{\mathcal{O}}\), 
    elements of \(L^{--, \leq \mu}\mathcal{G}(R)\) (for some \(\mathcal{O}\)-algebra \(R\)) is determined by its ``entries'' corresponding to the coordinates of affine space, 
    each being an element of \(R\left[ \frac{1}{v+p}\right]\). In fact, the coefficients of these entries gives the Ind-scheme structure of \(L^{--} \mathcal{G}\). 
    The \(\leq \mu\) conditon implies that the poles will be bounded, i. e. that the coefficients of \((v+p)^{-n}\) will be zero for \(n \gg 0\). 
    Hence there are finitely many coefficients to account for, so \(L^{--,\leq \mu}\mathcal{G}\) is a scheme of finite type over \(\mathcal{O}\). 
\end{proof}

\begin{proposition}
    Let \(X \in L^{\leq \mu}\mathcal{G}(\mathcal{O})\). 
    The map 
    \begin{align*}
        L^{--, \leq \mu}\mathcal{G} & \hookrightarrow \operatorname{Gr}^{\leq \mu} \\
        A & \mapsto A X 
    \end{align*}
    is an open immersion. 
\end{proposition}
\begin{proof}
    It follows from \cref{res:formally-etale-monomorphism} that the map is a formally étale monomorphism on noetherian \(\mathcal{O}\)-algebras, and it is locally of finite type by \cref{res:negative-loop-group-is-of-finite-type}. 
    By \cite[\href{https://stacks.math.columbia.edu/tag/02HX}{Lemma 02HX}]{stacks-project} it follows that the map is smooth, hence an étale monomorphism, hence an open immersion by \cite[\href{https://stacks.math.columbia.edu/tag/025G}{Theorem 025G}]{stacks-project}. 
\end{proof}
Given \(z \in L^{\leq \mu}\mathcal{G}(\mathcal{O})\), we define \(U(z)\) as the image of the map \(L^{--,\leq \mu}\mathcal{G} \hookrightarrow \operatorname{Gr}^{\leq \mu}\), \(A \mapsto Az\). 
By the above proposition, \(U(z) \subset \operatorname{Gr}^{\leq \mu}\) is Zariski open. 

We now assume that \(p \nmid \pi_{1}(\widehat{G})\). 
Recall that \cref{res:special-fiber-of-pappas-zhu-local-model} identifies the special fiber of \(\operatorname{Gr}^{\leq \mu}\) in terms of the admissible set of \(\mu\), which is in particular 
a finite subset of the Iwahori--Weyl group of \(\widehat{G}\) over \(\overline{\mathbb{F}}\llp v \rrp\).
Suppose these elements are enumerated by \(z_{1}, \dots, z_{n}\), and we choose representatives in \(L^{\leq \mu}\mathcal{G}(\mathcal{O})\) representing these elements, denoted by the same symbols. 
Then \(U(z_{1})_{\mathbb{F}}, \dots, U(z_{n})_{\mathbb{F}}\) cover \(\operatorname{Gr}_{\mathbb{F}}^{\leq \mu}\), and by properness of \(\operatorname{Gr}^{\leq \mu}\) over \(\mathcal{O}\)
it follows that \(U(z_{1}), \dots, U(z_{n})\) cover \(\operatorname{Gr}^{\leq \mu}\).
\begin{corollary}
    The \(L^+\mathcal{G}\)-torsor \(L^{\leq \mu}\mathcal{G} \twoheadrightarrow \operatorname{Gr}^{\leq \mu}\) splits Zariski locally on \(\operatorname{Gr}^{\leq \mu}\). 
\end{corollary}
\begin{proof}
    The torsor splits over each \(U(z_{i})\) by construction of these charts. 
\end{proof}

\subsection{The stack \(\mathcal{Y}\)}
\label{sec:twisted-affine-grassmannian}
Consider the right action of \(L^+ \mathcal{G}^{\wedge \varpi}\) on \(L \mathcal{G}^{\wedge \varpi}\) given by 
\[
X \star A = A^{-1} X \varphi_{c}(A), \,\,\,\, \text{ for all }X \in L \mathcal{G}^{\wedge \varpi}(R), A \in L^+ \mathcal{G}^{\wedge \varpi}(R) . 
\]
(Recall the notation \(\varphi_{c}\) from \eqref{eq:varphi-c}.)
We call this the \(c\)-twisted \(\varphi\)-conjugation action. 

We define 
\[
\mathcal{Y} : = \left[ L \mathcal{G}^{\wedge \varpi} /^{\varphi,c} L^+ \mathcal{G}^{\wedge \varpi} \right] 
\]
as an fppf quotient stack over \(\operatorname{Spf}\mathcal{O}\).
The \((\varphi,c)\) over the slash is supposed to remind us that this quotient is with respect to the \(c\)-twisted \(\varphi\)-conjugation action. 

In contrast to \(\operatorname{Gr}_{\mathcal{G}}\), 
we note that the stack \(\mathcal{Y}\) is typically a genuine stack, and as such is not representable by an Ind-scheme. 
In fact, \(\mathcal{Y}\) is not even of finite type, although it will be Ind-algebraic as a consequence of our results concerning the ``bounded'' substacks \(\mathcal{Y}^{\leq \mu}\) defined below. 

\subsubsection{Bounded substacks of \(\mathcal{Y}\)}
Recall the subscheme \(L^{\leq \mu} \mathcal{G} \subset L \mathcal{G}\) from \cref{sec:bounded-loop-group}. 
Since \(L^{\leq \mu}\mathcal{G}\) is stable under both left and right translation by \(L^+\mathcal{G}\), we can define the fppf quotient stack
\begin{align*}
    \mathcal{Y}^{\leq \mu} & : = \left[ L^{\leq \mu} \mathcal{G}^{\wedge \varpi} /^{\varphi,c} L^+ \mathcal{G}^{\wedge \varpi} \right] . 
\end{align*}
We will show in \cref{res:algebraicity} that \(\mathcal{Y}^{\leq \mu}\) is a \(p\)-adic formal \emph{algebraic} stack over \(\operatorname{Spf}\mathcal{O}\), 
and in \cref{res:smooth-equivalence} we will show that it is smoothly equivalent to \(\operatorname{Gr}^{\leq \mu, \wedge \varpi}\). 

\begin{remark}
	Via the relation to Galois representations in \eqref{res:presentation-of-moduli-stack-of-BKL-parameters}, the substack \(\mathcal{Y}^{\leq \mu} \subset \mathcal{Y}\) classifies 
	Breuil--Kisin \((\Gamma,\widehat{G})\)-torsors with Hodge--Tate weights bounded by \(\mu\). 
\end{remark} 

\subsubsection{Covers of \(\mathcal{Y}\)}
For \(f \geq 0\) a concave function, we consider the fppf quotient stack over \(\mathcal{O}\)
\begin{align*}
    \mathcal{Y}_{f} & := \left[ L \mathcal{G}_{x}^{\wedge \varpi} /^{\varphi,c} L^+ \mathcal{G}_{x,f}^{\wedge \varpi} \right] . 
\end{align*}
Whenever \(L^+ \mathcal{G}_{x,f}\) is normal in \(L^+\mathcal{G}_{x}\), the natural map \(\mathcal{Y}_{f} \to \mathcal{Y} = \mathcal{Y}_{0}\) exhibits \(\mathcal{Y}_{f}\) as an \(L^+ \mathcal{G}_{x} / L^+\mathcal{G}_{x,f}\)-torsor over \(\mathcal{Y}\).
Again, we emphasize that \(L^+ \mathcal{G}_{x} / L^+\mathcal{G}_{x,f} = L^n \mathcal{G}_{x}\) when \(f = n \in \mathbb{Z}_{\geq 1}\) and \(L^+ \mathcal{G}_{x} / L^+\mathcal{G}_{x,f} = \overline{\mathcal{T}}\) when \(f = 0^+\) and \(x\) is \(0\)-generic, 
as explained in \cref{sec:quotient-loop-groups}.

\subsection{Moduli interpretation}
\label{sec:moduli-interpretation}
Throughout this section let \(f\geq 0\) denote a concave function. 
Recall from \cref{res:pappas-zhu-concave-function} that we have smooth affine group schemes \(\mathcal{G}_{x,f}\) over \(\mathbb{A}^1_{\mathcal{O}}\). 

\subsubsection{Moduli interpretation of the affine Grassmannian}
The stack \(\operatorname{Gr}_{f}\) has the following moduli interpretation. 
\begin{proposition}
    \label{res:moduli-affine-grassmannian}
	For any \(\mathcal{O}\)-algebra \(R\), a morphism \(\operatorname{Spec} R \to \operatorname{Gr}_f\) classifies a pair \((\mathcal{P},\beta)\), 
	where \(\mathcal{P}\) is a \(\mathcal{G}_{x,f}\)-torsor on \(R\llb v+p \rrb\), and \(\beta\) is as follows. 
	Let \(\mathcal{P}' = \mathcal{P} \times^{\mathcal{G}_{x,f}} \mathcal{G}_{x}\) denote the associated \(\mathcal{G}_{x}\)-torsor. 
	Then \(\beta : \mathcal{E}^0 \dashrightarrow \mathcal{P}'\) is an isomorphism of \(\mathcal{G}_{x}\)-torsors over \(R \llp v+p \rrp\). 
\end{proposition}
\begin{proof}
    Let \(\operatorname{Gr}'\) denote the moduli of pairs \((\mathcal{P},\beta)\) as above and let \(\widetilde{\operatorname{Gr}}\) denote the moduli of triples \((\mathcal{P},\beta, \alpha)\), where \(\mathcal{P},\beta\) are as above and \(\alpha : \mathcal{E}^0 \isoto \mathcal{P}\)
	is a trivialization of the \(\mathcal{G}_{x,f}\)-torsor over \(R \llb v + p \rrb\). 
	Consider the map \(L \mathcal{G}_{x} \isoto \widetilde{\operatorname{Gr}}\) with \(A \mapsto (\mathcal{E}^0, L_A, 1)\), where \(L_{A}\) denotes left translation by \(A\). 
	This is indeed an isomorphism, since an inverse is given by \((\mathcal{P},\beta, \alpha) \mapsto \alpha^{-1} \circ \beta\).
	The left translation action of \(L^+\mathcal{G}_{x,f}\) on \(L \mathcal{G}_{x}\) corresponds to the action \(B \cdot (\mathcal{P},\beta,\alpha) = (\mathcal{P},\beta, \alpha \circ L_{B}^{-1})\). 
	With respect to the obvious map \(\widetilde{\operatorname{Gr}} \to \operatorname{Gr}'\), \((\mathcal{P},\beta,\alpha)\mapsto (\mathcal{P},\beta)\), 
    the group \(L^+\mathcal{G}_{x,f}\) acts freely and transitively on fibers. 
    Using smoothness of \(\mathcal{G}_{x,f}\), there exists for any \(\mathcal{O}\)-algebra \(R\) and \(\mathcal{G}_{x,f}\)-torsor \(\mathcal{P}\) over \(\mathfrak{S}_{R}\)
    an étale covering \(R \to R'\) such that \(\mathcal{P}\) becomes trivializable after base change to \(\mathfrak{S}_{R'}\), which implies that \(\widetilde{\operatorname{Gr}} \to \operatorname{Gr}'\)
    is an epimorphism.
    Hence \(\left[ L^+ \mathcal{G}_{x,f} \backslash L \mathcal{G} \right] \cong \left[ L^+ \mathcal{G}_{x,f} \backslash \widetilde{\operatorname{Gr}} \right] \cong \operatorname{Gr}'\). 
\end{proof}

\subsubsection{Moduli interpretation of \(\mathcal{Y}\)}
\label{sec:moduli-interpretation-of-Y}
Let \(R\) be a \(p\)-adically complete \(\mathcal{O}\)-algebra, and recall that \(\mathfrak{S}_{L,R} = R^\mathcal{J}\llb v \rrb = R^\mathcal{J}\llb v + p \rrb\) and \(\mathfrak{S}_{R} = R\llb v \rrb = R\llb v + p \rrb\). 
Let us first define a map of groupoids 
\[
    \varphi_{c}^* : \lbrace (\Gamma,\widetilde{\mathcal{G}}_{x}) \text{-torsors over }\mathfrak{S}_{L,R} \rbrace \to \lbrace (\Gamma,\widetilde{\mathcal{G}}_{x}) \text{-torsors over }\mathfrak{S}_{L,R} \rbrace
\]
which is natural in \(R\). 
Namely, given a \((\Gamma,\widetilde{\mathcal{G}}_{x})\)-torsor \(\widetilde{\mathcal{P}}\) over \(\mathfrak{S}_{L,R}\), pulling back along the \(\Gamma\)-equivariant map \(\varphi : \operatorname{Spec}\mathfrak{S}_{L,R} \to \operatorname{Spec}\mathfrak{S}_{L,R}\) yields a \((\Gamma,\varphi^* \widetilde{\mathcal{G}}_{x})\)-torsor \(\varphi^* \widetilde{\mathcal{P}}_{x}\). 
We can identify \(\varphi^* \widetilde{\mathcal{G}}_{x} = \widetilde{\mathcal{G}}_{\varphi(x)}\) by \eqref{eq:frobenius-pullback-of-G-tilde}, and recalling the choice of \(c \in G^*(\mathcal{O}[v^{\pm 1}])\) with \(c \cdot \varphi(x) = x\), we then have a \(\Gamma\)-equivariant isomorphism \(\operatorname{Ad}_{c} : \varphi^* \widetilde{\mathcal{G}}_{x} = \widetilde{\mathcal{G}}_{\varphi(x)} \isoto \widetilde{\mathcal{G}}_{x}\) of group schemes over \(\mathfrak{S}_{L,R}\). 
We then define \(\varphi_{c}^* \widetilde{\mathcal{P}} = \varphi^* \widetilde{\mathcal{P}} \times^{\widetilde{\mathcal{G}}_{\varphi(x)}, \operatorname{Ad}_{c}} \widetilde{\mathcal{G}}_{x}\) as the associated \(\widetilde{\mathcal{G}}_{x}\)-torsor.

Note that if the type of a \((\Gamma,\widetilde{\mathcal{G}}_{x})\)-torsor \(\widetilde{\mathcal{P}}\) over \(\mathfrak{S}_{L,R}\) is trivial, then the type of \(\varphi_{c}^* \widetilde{\mathcal{P}}\) is also trivial. 
By \cref{res:fundamental-result-about-equivariant-torsors} it follows that we have a commutative diagram of groupoids 
\begin{equation}
    \label{eq:diagram-defining-varphi-c}
    \begin{tikzcd}
        \lbrace (\Gamma,\widetilde{\mathcal{G}}_{x}) \text{-torsors over }\mathfrak{S}_{L,R}\text{ of trivial type} \rbrace \arrow[r,"\varphi_{c}^*"] \arrow[d,"\pi_{*}^\Gamma","\sim"'] & \lbrace (\Gamma,\widetilde{\mathcal{G}}_{x}) \text{-torsors over }\mathfrak{S}_{L,R}\text{ of trivial type} \rbrace \arrow[d,"\pi_{*}^\Gamma","\sim"'] \\
        \lbrace \mathcal{G}_{x} \text{-torsors over }\mathfrak{S}_{R} \rbrace \arrow[r,dashed,"\varphi_{c}^*"] & \lbrace \mathcal{G}_{x} \text{-torsors over }\mathfrak{S}_{R} \rbrace ,
    \end{tikzcd}
\end{equation}
where the vertical maps are equivalences and the horizontal dashed map by definition is one making the diagram commute.
This is moreover natural in \(R\). 

\begin{proposition}
    \label{res:moduli-interpretation-of-Y}
    The following moduli problems in \(p\)-adically complete \(\mathcal{O}\)-algebras \(R\) are equivalent:
    \begin{enumerate}
        \item Pairs \((\widetilde{\mathcal{P}},\widetilde{\phi})\), where \(\widetilde{\mathcal{P}}\) is a \((\Gamma,\widetilde{\mathcal{G}}_{x,f})\)-torsor over \(\operatorname{Spf}\mathfrak{S}_{L,R}\) of trivial type, \(\widetilde{\mathcal{P}}' = \widetilde{\mathcal{P}} \times^{\widetilde{\mathcal{G}}_{x,f}} \widetilde{\mathcal{G}}_{x}\) is the associated \(\widetilde{\mathcal{G}}_{x}\)-torsor, and \(\widetilde{\phi} : \varphi^*_{c}\widetilde{\mathcal{P}}' \dashrightarrow \widetilde{\mathcal{P}}'\) is a morphism of \((\Gamma,\widetilde{\mathcal{G}}_{x})\)-torsors, where the dashed arrow indicates that this morphism is defined after base change to \(\operatorname{Spf} \mathfrak{S}_{L,R} \left[ (u^e + p)^{-1} \right]\). (Here, \(\operatorname{Spf}\) is formed with respect to the ideals generated by \(\varpi\).)
        \item Pairs \((\mathcal{P},\phi)\), where \(\mathcal{P}\) is a \(\mathcal{G}_{x,f}\)-torsor over \(\operatorname{Spf}\mathfrak{S}_{R}\), \(\mathcal{P}' = \mathcal{P} \times^{\mathcal{G}_{x,f}} \mathcal{G}_{x}\) is the associated \(\mathcal{G}_{x}\)-torsor, and \(\phi : \varphi^*_{c} \mathcal{P}' \dashrightarrow \mathcal{P}'\) is a morphism of \(\mathcal{G}_{x}\)-torsors, where the dashed arrow indicates that this morphism is only defined after base change to \(\operatorname{Spf}\mathfrak{S}_{R}\left[ (v+p)^{-1} \right]\). 
    \end{enumerate}
    Moreover, \(\mathcal{Y}_{f}\) naturally represents the moduli problem (2) (hence also (1)). 
\end{proposition}
\begin{proof}
    The equivalence of the moduli problems (1) and (2) is clear from the construction of \(\varphi_{c}^*\), i. e. commutativity of \eqref{eq:diagram-defining-varphi-c}.
    To see that the moduli problem (2) is represented by \(\mathcal{Y}\), one may argue as in the proof of \cref{res:moduli-affine-grassmannian}. 
    We mention that if \(L_{A} : \mathcal{E}^0 \isoto \mathcal{E}^0\) denotes an automorphism of the trivial \(\mathcal{G}_{x}\)-torsor over \(\operatorname{Spf}\mathfrak{S}_{R}\) given by left translation by \(A \in \mathcal{G}_{x}(\mathfrak{S}_{R}) = L^+ \mathcal{G}_{x}(R)\), 
    then \(\varphi_{c}^*(L_{A}) = L_{\varphi_{c}(A)} : \mathcal{E}^0 \to \mathcal{E}^0\) is given by left translation by the element \(\varphi_{c}(A) \in L^+_{0}\mathcal{G}_{x}(R) = L^+ \mathcal{G}_{x}(R)\) (which was defined in \cref{sec:frobenius-endomorphism-of-loop-groups}).
\end{proof}

When \(f = 0\), we also have the following result, which provides the interpretation of \(\mathcal{Y}\) as a moduli stack of Breuil--Kisin \((\Gamma,\widehat{G})\)-torsors of type \(\tau\). 
\begin{proposition}
    \label{res:moduli-interpretation-as-BK-modules}
    The fppf stack \(\mathcal{Y}\) equivalent to the moduli of pairs \((\widetilde{\mathcal{P}},\widetilde{\phi})\), where \(\widetilde{\mathcal{P}}\) is a \((\Gamma,\widehat{G})\)-torsor over \(\operatorname{Spf}\mathfrak{S}_{L,R}\) of type \(\tau\), and \(\phi : \varphi^* \widetilde{\mathcal{P}} \dashrightarrow \widetilde{\mathcal{P}}\) is a morphism of \((\Gamma,\widehat{G})\)-torsors, where the dashed arrow indicates that this morphism is only defined after base change to \(\operatorname{Spf}\mathfrak{S}_{L,R}\left[ (u^e + p)^{-1} \right]\).
    (Consequently, this moduli problem is equivalent to those of \cref{res:moduli-interpretation-of-Y} in the case \(f = 0\).)
\end{proposition}
Since this result may be of special interest, we give two proofs. 
The first proof is more conceptual, giving an equivalence between the moduli problem and the moduli problem (1) of \cref{res:moduli-interpretation-of-Y}. 
The second proof is direct, and shows explicitly how to pass between the moduli of Breuil--Kisin \((\Gamma,\widehat{G})\)-torsors and the quotient stack \(\mathcal{Y}\).
\begin{proof}[First proof]
    Consider first \(n \mathcal{E}^0 \cong {_{\tau}\mathcal{E}}^0\), the trivial \(\widehat{G}\)-torsor of type \(\tau\).
    We can identify \(\widetilde{\mathcal{G}}_{x} = n \widehat{G} n^{-1}\) as the automorphism group scheme of \(n \mathcal{E}^0\), so by twisting (see \cref{sec:twisting})
    we have an equivalence 
    \begin{align*}
        \left\lbrace (\Gamma,\widehat{G})\text{-torsors of type }\tau \right\rbrace & \isoto \left\lbrace (\Gamma,\widetilde{\mathcal{G}}_{x})\text{-torsors of trivial type} \right\rbrace \\
        \widetilde{\mathcal{P}} & \mapsto \widetilde{\mathcal{P}} \times^{\widehat{G}} \mathcal{E}^0 n^{-1} , 
    \end{align*}
    where \(\mathcal{E}^0 n^{-1} \cong (_{\tau}\mathcal{E}^0)^{-1}\).
    Via this equivalence, a morphism of \((\Gamma,\widehat{G})\)-torsors \(\widetilde{\phi} : \varphi^* \widetilde{\mathcal{P}} \dashrightarrow \widetilde{\mathcal{P}}\) corresponds to a morphism of \((\Gamma,\widetilde{\mathcal{G}}_{x})\)-torsors 
    \begin{equation}
        \label{eq:associated-bundle-frobenius}
        \left( \varphi^* \widetilde{\mathcal{P}} \right) \times^{\widehat{G}} \mathcal{E}^0 n^{-1} \dashrightarrow \widetilde{\mathcal{P}} \times^{\widehat{G}} \mathcal{E}^0 n^{-1} . 
    \end{equation}
    The point is now that \(\varphi^* \left( \mathcal{E}^0 n^{-1} \right) = \mathcal{E}^{0} \varphi(n)^{-1}\) is a \(\varphi^* \widetilde{\mathcal{G}}_{x} = \widetilde{\mathcal{G}}_{\varphi(x)}\)-torsor by \cref{eq:frobenius-pullback-of-G-tilde},
    but using the isomorphism \(\operatorname{Ad}_{c} : \widetilde{\mathcal{G}}_{\varphi(x)} \isoto \widetilde{\mathcal{G}}_{x}\) we find that 
    \[
        \varphi^* \left( \mathcal{E}^0 n^{-1} \right) \times^{\operatorname{Ad}_{c}} \widetilde{\mathcal{G}}_{x} \cong \mathcal{E}^0\varphi(n)^{-1}c^{-1} = \mathcal{E}^0 n^{-1} , 
    \]
    where we have used the assumption that \(\tau\) is strictly Frobenius invariant and \(c = n \varphi(n)^{-1}\).
    Therefore,
    \[
        \left( \varphi^* \widetilde{\mathcal{P}} \right) \times^{\widehat{G}} \mathcal{E}n^{-1} \cong 
        \left( \varphi^* \widetilde{\mathcal{P}} \right) \times^{\widehat{G}} \left( \varphi^* \mathcal{E} n^{-1} \right) \times^{\operatorname{Ad}_{c}} \widetilde{\mathcal{G}}_{x} \cong 
        \varphi^* \left( \widetilde{\mathcal{P}} \times^{\widehat{G}} \mathcal{E}n^{-1} \right) \times^{\operatorname{Ad}_{c}} \widetilde{\mathcal{G}}_{x} .
    \]
    So the morphism \(\widetilde{\phi}\), which corresponds to \eqref{eq:associated-bundle-frobenius}, in turn corresponds to a morphism 
    \[
        \varphi^* \left( \widetilde{\mathcal{P}} \times^{\widehat{G}} \mathcal{E}n^{-1} \right) \times^{\operatorname{Ad}_{c}} \widetilde{\mathcal{G}}_{x}
        \dashrightarrow \widetilde{\mathcal{P}} \times^{\widehat{G}} \mathcal{E}n^{-1} . 
    \]
    This implies that the moduli problem in question is equivalent to the moduli problem (1) in \cref{res:moduli-interpretation-of-Y} with \(f = 0\). 
\end{proof}
\begin{proof}[Second proof]
    We start by making some preliminary observations. 
    The assumption that \(\tau\) is strictly Frobenius invariant, i. e. \(\varphi \tau = \tau\), implies that \(\varphi^* {_{\tau}\mathcal{E}^0} = {_{\varphi \tau}\mathcal{E}^0} = {_{\tau}\mathcal{E}^0}\). 
    By \cref{res:equivariant-torsors-and-group-cohomology}, any morphism of \((\Gamma,\widehat{G})\)-torsors \(L_{C} : {_{\tau}\mathcal{E}^0} \dashrightarrow {_{\tau}\mathcal{E}^0}\) over \(\operatorname{Spf}\mathfrak{S}_{L,R}\)
    is given by left translation by some element \(C \in L \widehat{G}^{\wedge \varpi}(R)\) which gives a \(1\)-cocycle \(C : \tau \isoto \tau\), i. e. 
    \begin{equation}
        \label{eq:C-cocycle-condition}
        C n^{-1} (^\theta n) (^\theta C)^{-1} = n^{-1} (^\theta n)
    \end{equation}
    for all \(\theta \in \Gamma\) (recall that \(\tau(\theta) = n^{-1}(^\theta n)\)).
    If we let \(A = n C n^{-1}\), we see that \eqref{eq:C-cocycle-condition} is equivalent to the condition that \({^\theta A} = A\) for all \(\theta\), i. e. \(A\) is fixed by \(\Gamma\).

    Let \(\widetilde{\mathcal{Y}}\) be the moduli of triples \((\widetilde{\mathcal{P}}, \widetilde{\phi}, \alpha)\), where \((\widetilde{\mathcal{P}},\widetilde{\phi}) \in \mathcal{Y}(R)\) and \(\alpha : {_{\tau}\mathcal{E}^0} \isoto \widetilde{\mathcal{P}}\) is a morphism of \((\Gamma,\widehat{G})\)-torsors over \(\operatorname{Spf}\mathfrak{S}_{L,R}\).
    We then have an isomorphism \(\widetilde{\mathcal{Y}} \isoto L \mathcal{G}_{x}^{\wedge \varpi}\) as follows. 
    Given a triple \((\widetilde{\mathcal{P}}, \widetilde{\phi},\alpha)\), the morphism \(\widetilde{\phi}\) corresponds via \(\alpha\)
    to a morphism \(L_{C} : {_{\tau}\mathcal{E}^0} \dashrightarrow {_{\tau}\mathcal{E}^0}\), namely \(L_{C} = \alpha^{-1} \circ \widetilde{\phi} \circ \varphi^* \alpha\).
    We then associate to the triple \((\widetilde{\mathcal{P}}, \widetilde{\phi},\alpha) \in \widetilde{\mathcal{Y}}(R)\) the element \(A = n C n^{-1} \in \left( L \widehat{G}^{\wedge \varpi}(R) \right)^\Gamma = L \mathcal{G}_{x}^{\wedge \varpi}(R)\). 

    Next, we consider how the choice of \(\alpha\) affects \(A\). 
    That is, suppose we change \(\alpha\) into \(\alpha ' = \alpha \circ L_{Y}\), where \(L_{Y} : {_{\tau}\mathcal{E}^0} \isoto {_{\tau}\mathcal{E}^0}\) is a morphism of \((\Gamma,\widehat{G})\)-torsors over \(\operatorname{Spf}\mathfrak{S}_{L,R}\) given by left translation by some element \(C \in L^+ \widehat{G}^{\wedge \varpi}(R)\) representing a \(1\)-cocycle \(Y : \tau \isoto \tau\) (satisfying \eqref{eq:C-cocycle-condition} with \(Y\) in place of \(C\)).
    Then \(\widetilde{\phi}\) corresponds via \(\alpha'\) to left translation by some element \(C'\).
    We have a commutative diagram 
    \[
    \begin{tikzcd}
    \varphi^* \widetilde{\mathcal{P}} \arrow[rr,dashed,"\widetilde{\phi}"] & & \widetilde{\mathcal{P}} \\
    \varphi^* \left(_{\tau}\mathcal{E}^0\right) \arrow[r,equals] \arrow[u,"\varphi^*\alpha"] & {_{\tau}\mathcal{E}^0} \arrow[r,dashed,"L_{C}"] & {_{\tau}\mathcal{E}^0} \arrow[u,"\alpha"] \\ 
    \varphi^* \left(_{\tau}\mathcal{E}^0\right) \arrow[r,equals] \arrow[u,"\varphi^*L_{Y}"] & {_{\tau}\mathcal{E}^0} \arrow[u,"L_{\varphi(Y)}"] \arrow[r,dashed,"L_{C'}"] & {_{\tau}\mathcal{E}^0} \arrow[u,"L_{Y}"] . 
    \end{tikzcd}
    \]
    From the diagram we see that \(C' = Y^{-1} C \varphi(Y)\). 
    Therefore, if we let \(X := n Y n^{-1} \in \left( L^+ \widetilde{\mathcal{G}}_{x}^{\wedge \varpi}(R) \right)^\Gamma = L^+ \mathcal{G}_{x}^{\wedge \varpi}(R)\) and \(A' := n C n^{-1}\), then 
    \[
        A' =  n C' n^{-1} = (n Y n^{-1})^{-1} (n C n^{-1}) (n \varphi(n)^{-1} (\varphi(n Y n^{-1}))\varphi(n)n^{-1}) = X^{-1} A (c \varphi(X)c^{-1}) = A \star X 
    \]
    is the \(c\)-twisted \(\varphi\)-conjugate of \(A\). 

    The moduli of Breuil--Kisin \((\Gamma,\widehat{G})\)-torsors is the quotient of \(\widetilde{\mathcal{Y}}\) where we forget \(\alpha\) (where we use that a choice of \(\alpha\) exists étale locally by the correspondence between \((\Gamma,\widehat{G})\)-torsors of type \(\tau\) and torsors for the smooth group \(\mathcal{G}_{x}\), and \cref{res:equivalent-definitions-of-type}). 
    Since forgetting \(\alpha\) corresponds to taking the quotient with respect to all possible modifications, we see from the above that this is precisely the quotient stack \(\mathcal{Y} = \left[ L \mathcal{G}_{x}^{\wedge \varpi} /^{c,\varphi} L^+ \mathcal{G}_{x}^{\wedge \varpi} \right]\). 
\end{proof}
\begin{remark}
    In both proofs of \cref{res:moduli-interpretation-as-BK-modules}, we used the assumption that \(\tau\) is strictly Frobenius invariant and \(c = n \varphi(n)^{-1}\). 
    This is in fact the only place where we use this assumption; for all other results it suffices that \(c \in G^* (\mathcal{O}[v^{\pm 1}])\) and \(c \cdot \varphi(x) = x\).
\end{remark}

\subsubsection{Applications to the existence of sections}
From the proof of \cref{res:moduli-affine-grassmannian} we can make the following observation.
Let \(R\) be a \(\varpi\)-adically complete \(\mathcal{O}\)-algebra and let \(\operatorname{Spf}R \to \operatorname{Gr}_f\) be a morphism, classifying \((\mathcal{P},\beta)\).
Then a lift as in the diagram 
% https://q.uiver.app/#q=WzAsMyxbMCwxLCJcXG9wZXJhdG9ybmFtZXtTcGVjfVIiXSxbMSwxLCJcXG9wZXJhdG9ybmFtZXtHcn1fe1xcbWF0aGNhbHtHfX1eZiJdLFsxLDAsIkwgXFxtYXRoY2Fse0d9Il0sWzAsMV0sWzIsMV0sWzAsMiwiIiwyLHsic3R5bGUiOnsiYm9keSI6eyJuYW1lIjoiZGFzaGVkIn19fV1d
\[\begin{tikzcd}
    & {L \mathcal{G}} \\
    {\operatorname{Spf}R} & {\operatorname{Gr}_f}
    \arrow[from=2-1, to=2-2]
    \arrow[from=1-2, to=2-2]
    \arrow[dashed, from=2-1, to=1-2]
\end{tikzcd}\]
exists if and only if \(\mathcal{P}\) is a trivializable \(\mathcal{G}_{x,f}\)-torsor.
Similarly, if \(\operatorname{Spf}R \to \mathcal{Y}_{f}\) is a morphism classifying a pair \((\mathcal{P},\phi)\), then a lift as in the diagram 
\[\begin{tikzcd}
    & {L \mathcal{G}}^{\wedge \varpi} \\
    {\operatorname{Spf}R} & \mathcal{Y}_f
    \arrow[from=2-1, to=2-2]
    \arrow[from=1-2, to=2-2]
    \arrow[dashed, from=2-1, to=1-2]
\end{tikzcd}\]
exists if and only if \(\mathcal{P}\) is a trivializable \(\mathcal{G}_{x,f}\)-torsor.

From the above we conclude that the obstruction to surjectivity of \(L \mathcal{G}(R) \to \operatorname{Gr}_{f}(R)\) and \(L \mathcal{G} \to \mathcal{Y}_{f}(R)\) 
is the existence of non-trivial \(\mathcal{G}_{x,f}\)-torsors over \(R / \varpi R\).
(Here we use smoothness of \(\mathcal{G}_{x,f}\) to reduce the question about existence of non-trivial torsors over \(R \llb v + p \rrb\) to the same question over \(R / \varpi R\).)

\begin{proposition}
    \label{res:presheaf-surjectivity-of-affine-grassmannian}
    Assume that  \(f>0\). 
    Then for any \(p\)-adically complete \(\mathcal{O}\)-algebra \(R\), the maps \(L \mathcal{G}(R) \to \operatorname{Gr}_{f}(R)\), \(L^{\leq \mu}\mathcal{G}(R) \to \operatorname{Gr}_{f}^{\leq \mu}(R)\), \(L \mathcal{G}(R) \to \mathcal{Y}_{f}(R)\), and \(L^{\leq \mu}\mathcal{G}(R) \to \mathcal{Y}_{f}^{\leq \mu}(R)\) are surjective.
    In particular, the natural maps \(L^+\mathcal{G}_{x,f}(R) \backslash L \mathcal{G}(R) \isoto \operatorname{Gr}_{f}(R)\) and \(L^+\mathcal{G}_{x,f}(R) \backslash L^{\leq \mu}\mathcal{G}(R) \isoto \operatorname{Gr}_{f}^{\leq \mu}(R)\) are bijections, and the maps 
    \(\left[ L \mathcal{G}(R) /^{\varphi,c} L^+ \mathcal{G}_{x,f}(R) \right] \isoto \mathcal{Y}_{f}(R)\) and \(\left[ L^{\leq \mu}\mathcal{G}(R) /^{\varphi,c} L^+ \mathcal{G}_{x,f}(R) \right] \isoto \mathcal{Y}_{f}^{\leq \mu}(R)\) are equivalences of groupoids. 
\end{proposition}
\begin{proof}
    This is immediate from the above and \cref{res:unipotent-group-has-trivial-torsors}. 
\end{proof}
\begin{corollary}
    \label{res:zariski-local-splittings}
    Assume that  \(f >0\). 
    Then the \(L^+\mathcal{G}_{f}\)-torsor \(L^{\leq \mu}\mathcal{G} \to \operatorname{Gr}_{f}^{\leq \mu}\) trivializes Zariski locally on \(\operatorname{Gr}_{f}^{\leq \mu}\). 
\end{corollary}
\begin{proof}
    Let \(U = \operatorname{Spec}R \hookrightarrow \operatorname{Gr}_{f}^{\leq \mu}\) be the inclusion of an affine open chart. 
    There exists a lift \(U \to L^{\leq \mu} \mathcal{G}\) by \cref{res:presheaf-surjectivity-of-affine-grassmannian}, which provides a section over \(U\).  
\end{proof}
\section{Contraction and straightening}
\label{sec:contraction-and-straightening}
The Frobenius endomorphism \(\varphi_{c}\) from \cref{sec:frobenius-on-loop-group} has a contracting property, sending congruence subgroups into deeper congruence subgroups. 
On the other hand, conjugating by an element \(X \in L^{\leq \mu}\mathcal{G}\) generally does not preserve congruence subgroups.  
Quantifying and balancing these effects allow us to determine precisely when the operator \(A \mapsto X \varphi_{c}(A) X^{-1}\) is contracting. 
The answer is given in \cref{res:contraction}, and when the answer is affirmative we can ``straighten'' the \(c\)-twisted \(\varphi\)-conjugation action into a left translation action, as in \cref{res:straightening}.
This is the key to relating the affine Grassmannian and the stack \(\mathcal{Y}\). 

\subsection{Contraction properties of the Frobenius endomorphism}
We now discuss various contraction properties of the Frobenius endomorphism of \cref{sec:frobenius-on-loop-group}. 
We first consider the general case, and then we will consider the unramified case with extra genericity assumptions on \(x\). 

Recall that \(x \in \widetilde{\mathcal{A}}\left( T^*, \mathbb{F}\llp v \rrp \right) \subset \widetilde{\mathcal{A}}\left( \widehat{T}, \mathbb{F}^\mathcal{J}\llp u \rrp \right)\) 
is actually fixed by \(\Gamma\). Hence 
\[
    \varphi(x) - o = p (x-o) . 
\]
We will use this fact repeatedly.
For arguments that take place over \(\widetilde{\mathbb{A}}^1_{\mathcal{O}}\) or \(\breve{\mathbb{A}}^1_{\mathcal{O}}\) we will also often argue in each component labeled by \(j \in \mathcal{J}\) without comment, as we have done in previous sections. 
\begin{lemma}
    \label{res:frobenius-pullback-of-root-groups-with-congruence}
    Let \(m \in \frac{1}{e} \mathbb{Z}_{\geq 0}\). Then we can identify 
    \begin{align*}
        \varphi^* \widetilde{\mathcal{U}}_{a,x,m} & = \widetilde{\mathcal{U}}_{a,\varphi(x),pm}, &  \varphi^* \widetilde{\mathcal{T}}_{m} & = \widetilde{\mathcal{T}}_{pm} , \\
        \varphi^* \widetilde{\mathcal{U}}_{a,x,m^+} & = \widetilde{\mathcal{U}}_{a,\varphi(x),p\left(m + \frac{1}{e} \right)} , & \varphi^* \widetilde{\mathcal{T}}_{m^+} & = \widetilde{\mathcal{T}}_{p\left(m + \frac{1}{e} \right)} .
    \end{align*}
\end{lemma}
\begin{proof}
    We begin by proving that \(\varphi^* \widetilde{\mathcal{U}}_{a,x,m^+} = \widetilde{\mathcal{U}}_{a,x,p\left(m + \frac{1}{e}\right)}\). 
    Recall from \cref{sec:root-groups-attached-to-concave-functions} that \(\widetilde{\mathcal{U}}_{a,x,m^+} = \operatorname{Spec}\mathcal{O}^\mathcal{J}[u, u^{-n}t]\), where \(n = \lceil - e \langle a, x - o \rangle + e m^+ \rceil = - e \langle a , x - o \rangle + em + 1\), and we use that \(\langle a, x - o \rangle = - \frac{1}{e}\langle w a, \lambda \rangle \in \frac{1}{e}\mathbb{Z}\) by our construction of \(x\).
    Consider image of the map 
    \begin{align*}
        \mathcal{O}^\mathcal{J}[u] \otimes_{\varphi, \mathcal{O}^\mathcal{J}[u]} \mathcal{O}^\mathcal{J}[u, u^{-n} t] & \to \operatorname{Frac}\left( \mathcal{O}^\mathcal{J}[u,t] \right) \\ 
        a \otimes b & \mapsto a \varphi(b) ,
    \end{align*}
    is \(\mathcal{O}^\mathcal{J}[u,u^{-pn}t]\). 
    Since 
    \[
        p n = - p e \langle a, x - o \rangle + p e m + p = - e \langle a, \varphi(x) - o \rangle + e p \left( m + \frac{1}{e} \right) , 
    \]
    we see that \(\widetilde{\mathcal{U}}_{a,\varphi(x), p \left(m + \frac{1}{e} \right)} = \operatorname{Spec}\mathcal{O}^\mathcal{J}[u, u^{-pn}t]\). 
    This shows that \(\varphi^* \widetilde{\mathcal{U}}_{a,x,m^+} = \widetilde{\mathcal{U}}_{a,x,p \left(m + \frac{1}{e} \right)}\). 

    Next, we show that \(\varphi^* \widetilde{\mathcal{T}}_{m^+} = \widetilde{\mathcal{T}}_{p(m+1)}\).
    We may assume that \(\widehat{T} = \mathbb{G}_{m}\).
    Then we know from \cref{sec:root-groups-attached-to-concave-functions}
    that \(\widetilde{\mathcal{T}}_{m^+} = \operatorname{Spec}\mathcal{O}^\mathcal{J}[u, t^{-1}, u^{-n}(t-1)]\), where \(n = \lceil e m^+ \rceil = em + 1\). 
    The image of 
    \begin{align*}
        \mathcal{O}^\mathcal{J} \otimes_{\varphi, \mathcal{O}^\mathcal{J}[u]} \mathcal{O}^\mathcal{J}[u,t^{-1},u^{-(m+1)}(t-1)] & \to \operatorname{Frac}\left(\mathcal{O}^\mathcal{J}[u,t] \right) \\
        a \otimes b & \mapsto a \varphi(b) 
    \end{align*}
    is precisely \(\mathcal{O}^\mathcal{J}[u,t^{-1}, u^{-p(em+1)}(t-1)]\), and \(\widetilde{\mathcal{T}}_{p\left(m + \frac{1}{e}\right)} = \operatorname{Spec}\mathcal{O}^\mathcal{J}[u,t^{-1}, u^{-p(em+1)}(t-1)]\) since \(p(em + 1) = e p \left(m + \frac{1}{e}\right)\). 

    We omit the proofs of the remaining two statements, as the proofs are similar to those we have already produced. 
\end{proof}

\begin{lemma}
    \label{res:varphi-contraction-nongeneric}
    Let \(n \in \frac{1}{e}\mathbb{Z}_{\geq 0}\). Then
    \begin{enumerate}
        \item \( \varphi \left( L^{+}_{0}\mathcal{G}_{x,n^{+}} \right) \subset L^{+}_{0}\mathcal{G}_{\varphi(x),p \left( n + \frac{1}{e} \right) } .\)
        \item If \(n \geq \frac{1}{e}\), then \(\varphi\left( L^{+}_{0}\mathcal{G}_{x,n} \right) \subset L^{+}_{0}\mathcal{G}_{\varphi(x),pn } .\)
    \end{enumerate}
\end{lemma}
\begin{proof}
    It follows from \cref{res:frobenius-pullback-of-root-groups-with-congruence} that \(\varphi^* \widetilde{\mathcal{V}}_{x,n^+} = \widetilde{\mathcal{V}}_{x,p\left(n + \frac{1}{e}\right)}\). 
    By \cref{res:iwahori-decomposition} (Iwahori decomposition) we therefore have 
    \[
        \varphi\left( L^+_{0}\widetilde{\mathcal{G}}_{x,n^+} \right) = \varphi \left( L^+_{0} \widetilde{\mathcal{V}}_{x,n^+} \right) \subset L^+_{0} \widetilde{\mathcal{V}}_{x,p \left(n + \frac{1}{e} \right)} = L^+_{0} \widetilde{\mathcal{G}}_{x,p \left(n + \frac{1}{e}\right)} . 
    \]
    Taking \(\Gamma\)-fixed points, we then obtain (1) by \cref{res:iwahori-decomposition-corollary}. 
    The proof of (2) is similar. 
\end{proof}

\begin{lemma}
    \label{res:varphi-contraction-nongeneric-corollary}
    Let \(n \in \frac{1}{e} \mathbb{Z}_{\geq 0}\). Then 
    \begin{enumerate}
        \item \( \varphi_{c} \left( L^{+}_{0}\mathcal{G}_{x,n^{+}} \right) \subset L^{+}_{0}\mathcal{G}_{x,p \left( n + \frac{1}{e} \right) } .\)
        \item If \(n \geq \frac{1}{e}\), then \(\varphi_{c}\left( L^{+}_{0}\mathcal{G}_{x,n} \right) \subset L^{+}_{0}\mathcal{G}_{x,pn } .\)
    \end{enumerate}
\end{lemma}
\begin{proof}
    This is immediate from \cref{res:varphi-contraction-nongeneric} and the fact that conjugation by \(c\) preserves the congruence level.
\end{proof}

\subsubsection{The unramified case with genericity assumptions}
\label{sec:varphi-contraction-unramified-generic}
In the unramified case (i. e. \(I\) acts trivially on \(\widehat{G}\)) we will improve the bounds of \cref{res:varphi-contraction-nongeneric} with additional genericity conditions on \(x\).

\begin{lemma}
    \label{res:varphi-contraction-on-root-groups-generic}
    Assume that \(I\) acts trivially on \(\widehat{G}\). 
    Let \(d \in \mathbb{R}\) and assume that \(x\) is \(d\)-generic according to \cref{def:generic-type}. 
    For any \(n \in \mathbb{Z}_{\geq 0}\), we have 
    \begin{enumerate}
        \item \(\varphi\left( L_{0}^+ \breve{\mathcal{U}}_{a,x,n} \right) \subset L_{0}^+ \breve{\mathcal{U}}_{a,\varphi(x),pn + d}\). 
        \item \(\varphi\left( L_{0}^+ \breve{\mathcal{T}}_{n^+} \right) \subset L_{0}^+ \breve{\mathcal{T}}_{p(n+1)}\). (This is actually true without the genericity assumption.)
    \end{enumerate}
\end{lemma}
\begin{proof}
    We first show (1). 
    Let \(R\) be an \(\mathcal{O}\)-algebra.
    Using the pinning of \(\widehat{G}\) we identify \(L_{0}^+ \breve{\mathcal{U}}_{a,x,n}(R) = v^{\lceil - \langle a, x - o \rangle + n \rceil} R^\mathcal{J} \llb v \rrb\) and \(L_{0}^+ \breve{\mathcal{U}}_{a,\varphi(x),pn+d}(R) = v^{\lceil - \langle a, \varphi(x) - o \rangle + pn + d \rceil} R^\mathcal{J}\llb v \rrb\). 
    Since \(\varphi\left( v^{\lceil - \langle a, x-o \rangle + n \rceil} R^\mathcal{J}\llb v \rrb \right) \subset v^{p \lceil - \langle a, x - o \rangle + n \rceil} R^\mathcal{J}\llb v \rrb\) it suffices to show that
    \begin{equation}
        \label{eq:contraction-inequality-in-generic-case}
        p \lceil - \langle a, x - o \rangle + n \rceil \geq \lceil - \langle a, \varphi(x) - o \rangle + pn + d \rceil . 
    \end{equation}
    The assumption that \(x\) is \(d\)-generic implies that 
    \[
        - n_{a} - 1 + \frac{d}{p} + n < - \langle a, x - o \rangle + n < - n_{a} - \frac{d}{p} + n , 
    \] 
    where \(n_{a} \in \mathbb{Z}\) is as in \cref{def:generic-type}. 
    Multiplying by \(p\), we obtain \[
    - p n_{a} - p + d + n < - \langle a, \varphi(x) - o \rangle + pn < -p n_{a} - d + pn .
    \] Therefore, \[
    p \lceil - \langle a, x - o \rangle + n \rceil = - pn_{a} + pn > - \langle a, \varphi(x) - o \rangle + pn + d ,
    \] which implies \eqref{eq:contraction-inequality-in-generic-case}.

    Next, we show (2). 
    We can identify \(L_{0}^+ \breve{\mathcal{T}}_{n^+}(R) = 1 + v^{\lceil n^+ \rceil}R^\mathcal{J}\llb v \rrb = 1 + v^{n+1} R^\mathcal{J}\llb v \rrb\), 
    and then \(\varphi \left( 1 + v^{n+1}R^\mathcal{J}\llb v \rrb \right) \subset 1 + v^{p(n+1)}R\llb v \rrb = L_{0}^+ \breve{\mathcal{T}}_{p(n+1)}(R)\). 
\end{proof}

\begin{lemma}
    \label{res:varphi-contraction-generic} 
    Assume that \(I\) acts trivially on \(\widehat{G}\). 
    Let \(d \in \mathbb{R}\), and assume that \(x\) is \(d\)-generic. 
    Then for any \(n \in \mathbb{Z}_{\geq 0}\), we have 
    \begin{enumerate}
        \item \(\varphi \left( L^{+}_{0}\mathcal{G}_{x,n^{+}} \right) \subset L^{+}_{0}\mathcal{G}_{\varphi(x), pn + d}\). 
        \item \(\varphi_{c} \left( L^{+}_{0}\mathcal{G}_{x,n^{+}} \right) \subset L^{+}_{0}\mathcal{G}_{x,pn+d}\).
    \end{enumerate}
\end{lemma}
\begin{proof}
    It follows from \cref{res:varphi-contraction-on-root-groups-generic} that \(\varphi\left( L^+_{0} \breve{\mathcal{V}}_{x,n^+} \right) \subset L^+_{0} \breve{\mathcal{V}}_{x,pn+d}\) (where we use that \(d < p\) because \(x\) is \(d\)-generic).
    So by \cref{res:iwahori-decomposition} (Iwahori decomposition) we have \(\varphi \left( L^+_{0} \breve{\mathcal{G}}_{x,n^+} \right) \subset L^+_{0} \breve{\mathcal{G}}_{\varphi(x),pn+d}\).
    Taking \(\Gamma_{0}\)-fixed points, we obtain (1).
    As in the proof of \cref{res:varphi-contraction-nongeneric-corollary}, (2) is an
    immediate consequence of (1).
\end{proof}

\subsection{Congruence bounds for conjugation}
\label{sec:congruence-bounds-for-conjugation}
Define \[
h_{\mu} = \max_{a \in \Phi(G^*, T^*)} \langle a, \mu \rangle ,
\] the ``height'' of \(\mu\). 
\begin{lemma}
    \label{res:conjugation-congruence-bound}
    Let \(n \geq h_{\mu}\) be an integer, and let \(R\) be
    an \(\mathcal{O}\)-algebra. For all
    \(X \in L^{\leq \mu}\mathcal{G}(R)\) and all
    \(A \in L^{+n}\mathcal{G}(R)\) we have
    \(X A X^{-1} \in L^{+(n-h_{\mu})}\mathcal{G}(R)\).
\end{lemma}

\begin{proof} 
    Define \(\widetilde{\mathcal{X}}\) as the pullback \[\begin{CD}
    \widetilde{\mathcal{X}} @>>> L^{+(n-h_{\mu})}\mathcal{G} \\
    @VVV @VVV \\
    L \mathcal{G} \times_{\mathcal{O}} L^{+n}\mathcal{G} @>>> L \mathcal{G} ,
    \end{CD}\] where the bottom horizontal map is given by
    \((X,A)\mapsto X A X^{-1}\). We have to show that \[
    L^{\leq \mu} \mathcal{G} \times_{\mathcal{O}} L^{+n}\mathcal{G} \subset \widetilde{\mathcal{X}}
    \] as closed ind-subschemes of
    \(L\mathcal{G} \times_{\mathcal{O}} L^{+n}\mathcal{G}\). Since both
    \(\widetilde{\mathcal{X}}\) and
    \(L^{\leq \mu}\mathcal{G} \times_{\mathcal{O}}L^{+n}\mathcal{G}\) are
    stable under left translations by \(L^{+}\mathcal{G}\) (in the first
    factor), it suffices to prove that \[
    \operatorname{Gr}^{\leq \mu} \times_{\mathcal{O}} L^{+n}\mathcal{G} \subset \mathcal{X} := \left[ L^{+}\mathcal{G} \backslash \widetilde{\mathcal{X}} \right]
    \] as closed subschemes of
    \(\operatorname{Gr} \times_{\mathcal{O}} L^{+n}\mathcal{G}\).
    Note that
    \(\operatorname{Gr}^{\leq \mu} \times_{\mathcal{O}} L^{+n}\mathcal{G}\) is the
    reduced closure of its generic fiber
    \(\operatorname{Gr}_{E}^{\leq \mu} \times_{E} L^{+n}G^{*}_{E}\) by \cref{res:reduced-closure-lemma}. 
    Therefore it suffices to prove that we have the inclusion \[
    \operatorname{Gr}_{\mathcal{G},E}^{\leq \mu} \times_{E} L^{+n}G^{*}_{E} \subset \mathcal{X} ,
    \] which can be checked on \(\overline{E}\)-points because the left hand side is reduced by \cref{res:reduced-times-positive-loop-group}. 
    By definition, this amounts to the statement that \[
    \forall X \in L^{\leq \mu}G_{E}^{*}(\overline{E}) , \,\, \forall A \in L^{+n}G_{E}^{*}(\overline{E}) \, : \, X A X^{-1} \in L^{+(n-h_{\mu})}G^{*}_{E}(\overline{E}) .
    \] By \eqref{eq:cartan-decomposition} (Cartan decomposition),
    it suffices to show this for \(X\) of the form \(X = (v+p)^{\nu}\), where \(\nu \in X_{*}(T^*_{E})\) is dominant and \(\nu \leq \mu\), 
    since the set of \(X\) that satisfy the condition for all \(A\) are
    stable under left and right translation by elements of
    \(L^{+}G_{E}^{*}(\overline{E})\). 
    By Iwahori decomposition for
    \(L^{+n}G_{E}^{*}(\overline{E})\) (here we use \(n \geq 1\)) it suffices to show that
    \(\operatorname{Ad}_{(v+p)^{\nu}} L^{+n} U_{a} \subset L^{+n - h_{\mu}}U_{a}\)
    for all
    \(a \in \Phi(G^{*}_{E},T^{*}_{E})\), where \(U_{a} \subset G_{E}^*\) is the root subgroup corresponding to \(a\).
    But
    indeed, we have 
    \[
    \operatorname{Ad}_{(v+p)^{\nu}} L^{+n}U_{a} = L^{+n + \langle a, \nu \rangle} U_{a} \subset L^{+n - h_{\mu}} U_{a}
    \] because \(n+ \langle a,\nu \rangle \geq n - h_{\mu}\).
\end{proof}
\begin{remark}
    Let \(a \in \Phi(G^*_{E}, T^*_{E})\) be a root for which \(h_{\mu} = \langle a, \mu \rangle\). 
    Then \(n + \langle -a, \mu \rangle = n - h_{\mu}\), showing that the bound above is sharp. 
\end{remark}
\begin{remark}
    Note that the proof is purely group theoretic, and does not rely on choosing an embedding into \(\GL_{n}\). 
\end{remark}

\subsubsection{Bounds in terms of \(v\)}
Recall that \(L \mathcal{G} = L_{0} \mathcal{G}\) and \(L^+ \mathcal{G} = L^+_{0} \mathcal{G}\) when \(p\) is nilpotent in \(R\).
But as we pointed out in \cref{sec:congruence-loop-groups}, it is typically still the case that \(L^{+n}_{0}\mathcal{G} \neq L^{+n}\mathcal{G}\).
We can combine \cref{res:comparing-congruence-subgroups-at-v-and-v-plus-p} and \cref{res:conjugation-congruence-bound} to obtain bounds in terms of \(L^{+n}_{0}\mathcal{G}\). 
\begin{lemma}
    \label{res:conjugation-congruence-bound-at-v}
    Let
    \(a,n \in \mathbb{Z}_{\geq 1}\) be such that
    \[ n - h_{\mu} - 2a + 2 \geq 0 .\] 
    Then for any
    \(\mathcal{O}\)-algebra \(R\) with \(p^a R = 0\), all
    \(X \in L^{\leq \mu}\mathcal{G}(R)\) and all
    \(A \in L^{+n}_{0}\mathcal{G}(R)\) we have
    \(X A X^{-1} \in L^{+n - h_{\mu} - 2a + 2}_{0}\mathcal{G}(R)\).
\end{lemma}
\begin{proof}
    We have
    \(A \in L^{+n}_{0}\mathcal{G}(R) \subset L^{+n - a +1}\mathcal{G}(R)\) by
    \cref{res:comparing-congruence-subgroups-at-v-and-v-plus-p}, so by \cref{res:conjugation-congruence-bound} we have
    \(XAX^{-1} \in L^{+n - a +1 - h_{\mu}}\mathcal{G}(R)\), and by \cref{res:comparing-congruence-subgroups-at-v-and-v-plus-p} again we have
    \(L^{+n - a + 1 - h_{\mu}}\mathcal{G}(R) \subset L^{+n - h_{\mu} - 2a + 2}_{0}\mathcal{G}(R)\).
\end{proof}

\subsection{Contraction and straightening}
We will now combine the earlier results of this chapter to discuss when the \(c\)-twisted \(\varphi\)-conjugation can be ``straightened'' into a left translation action, meaning that the orbits for both actions coincide.  
The consequences will be explored in \cref{sec:comparing-affine-grassmannians}.
We will first focus on the unramified case with additional genericity assumptions, and then the general case. 

Recall that the group \(L^+ \mathcal{G}_{x,f}(R)\) comes equipped with a complete metric, namely the \emph{\(v\)-adic metric} (which is described in detail in \cref{sec:t-adic-metric}). 
\begin{lemma}[Contraction]
    \label{res:contraction}
    Assume that \(I\) acts trivially on \(\widehat{G}\). 
    Let \(d \in \mathbb{Z}_{\geq -1}\), \(f \in \mathbb{Z}_{\geq 1 }\cup \lbrace 0^+ \rbrace\), and \(a \in \mathbb{Z}_{\geq 1}\). 
    Assume that \(x\) is \(d\)-generic, and
    \begin{equation}
        \label{eq:contraction-condition}
        (p-1)\lfloor f \rfloor + d - h_{\mu} - 2a + 2 > 0 .
    \end{equation}
    Let \(R\) be an \(\mathcal{O}\)-algebra such that \(p^a R = 0\) and let \(X \in L^{\leq \mu }\mathcal{G}(R)\).
    Then the operator
    \begin{align*}
        \Psi : L^+\mathcal{G}_{x,f}(R) & \to L^+\mathcal{G}_{x,f}(R) \\ A & \mapsto X \varphi_{c}(A) X^{-1} 
    \end{align*}
    is well-defined, and it is a contraction for the \(v\)-adic metric. 
\end{lemma}
\begin{proof}
    Since \(p^a R = 0\) we can identify \(L^+\mathcal{G}_{x,f}(R) = L^+_{0}\mathcal{G}_{x,f}(R)\).  
    Note that \eqref{eq:contraction-condition} implies that \(p\lfloor f \rfloor + d > \lfloor f \rfloor + h_{\mu} \geq 1\) for all possible parameters. 
    
    Let \(A \in L^+ \mathcal{G}_{x,f}(R)\). 
    We have \(\varphi_{c}(A) \in L^+ \mathcal{G}_{x,p\lfloor f \rfloor +d}(R)\) by \cref{res:varphi-contraction-generic} (2). 
    Then \(X \varphi_{c}(A) X^{-1} \in L^+\mathcal{G}_{x, p \lfloor f \rfloor + d - h_{\mu} - 2 a + 2}(R)\) by \cref{res:conjugation-congruence-bound-at-v}. 
    By \eqref{eq:contraction-condition} again, \(\Psi(A) = X \varphi_{c}(A) X^{-1} \in L^+ \mathcal{G}_{x, \lfloor f \rfloor + 1}(R)\). 
    This shows that \(\Psi\) is well-defined since \(\lfloor f \rfloor +1 > f \). 
    
    Since \(f\) was arbitrary (the only restraint being satisfying \eqref{eq:contraction-condition}), the argument in the previous paragraph shows that \(\Psi\) maps congruence subgroups into deeper ones, and it is therefore a contraction with respect to the \(v\)-adic metric. 
\end{proof}
\begin{lemma}[Straightening]
    \label{res:straightening}
    Assume that \(I\) acts trivially on \(\widehat{G}\). 
    Let \(d \in \mathbb{Z}_{\geq -1}\), \(f \in \mathbb{Z}_{\geq 1 }\cup \lbrace 0^+ \rbrace\), and \(a \in \mathbb{Z}_{\geq 1}\). 
    Assume that \(x\) is \(d\)-generic, and
    \[
        (p-1)\lfloor f \rfloor + d - h_{\mu} - 2a + 2 > 0 .
    \]
    Let \(R\) be an \(\mathcal{O}\)-algebra such that \(p^a R = 0\) and let \(X \in L^{\leq \mu }\mathcal{G}(R)\). 
    Then we have the following:
    \begin{enumerate}
        \item For all \(A \in L^{+}\mathcal{G}_{x,f}(R)\) there exists a unique \(B \in L^{+}\mathcal{G}_{x,f}(R)\) such that \(X \star A = BX\).
        \item For all \(B \in L^{+}\mathcal{G}_{x,f}(R)\) there exists a unique \(A \in L^{+}\mathcal{G}_{x,f}(R)\) such that \(X \star A = BX\).
    \end{enumerate} 
\end{lemma}
\begin{proof}
    Let us first show (1). Given \(A\) we are forced to let
    \[ B = A^{-1} X \varphi_{c}(A)X^{-1} ,\]
    and we have to show that \(B \in L^{+}\mathcal{G}_{x,f}\). This is
    immediate from \cref{res:contraction}.

    It remains to show (2). Consider the operator
    \begin{align*}
    \Psi_{B} : L^{+}\mathcal{G}_{x,f}(R) & \to L^{+}\mathcal{G}_{x,f}(R) \\
    A & \mapsto X \varphi_{c}(A) X^{-1} B^{-1} .
    \end{align*}
    This is well-defined and in fact a contraction in the \(v\)-adic
    topology by \cref{res:contraction}. By Banach's fixed point
    theorem it follows that \(\Psi_{B}\) has a unique fixed point and this
    is our \(A\).
\end{proof}
We also have the following variant, which works without the unramified assumption at the cost of giving a weaker bound. 
\begin{lemma}[Straightening]
    \label{res:straightening-nongeneric}
    Let \(f, a \in \mathbb{Z}_{\geq 1}\), where 
    \[
        (p-1)f - h_{\mu} - 2a + 2 > 0 . 
    \]
    Let \(R\) be an \(\mathcal{O}\)-algebra such that \(p^a R = 0\) and let \(X \in L^{\leq \mu }\mathcal{G}(R)\). 
    Then we have the following:
    \begin{enumerate}
        \item For all \(A \in L^{+}\mathcal{G}_{x,f}(R)\) there exists a unique \(B \in L^{+}\mathcal{G}_{x,f}(R)\) such that \(X \star A = BX\).
        \item For all \(B \in L^{+}\mathcal{G}_{x,f}(R)\) there exists a unique \(A \in L^{+}\mathcal{G}_{x,f}(R)\) such that \(X \star A = BX\).
    \end{enumerate} 
\end{lemma}
\begin{proof}
    As with the proof of \cref{res:straightening}, the key is to show that the map
    \begin{align*}
        \Psi_{B} : L^{+}\mathcal{G}_{x,f}(R) & \to L^{+}\mathcal{G}_{x,f}(R) \\
        A & \mapsto X \varphi_{c}(A) X^{-1} B^{-1} 
    \end{align*}
    is a contraction. This can be shown as in the proof of \cref{res:contraction}, using \cref{res:varphi-contraction-nongeneric-corollary} in place of \cref{res:varphi-contraction-generic}. 
\end{proof}
\begin{remark}
    Note that \cref{res:straightening} with \(f = 0^+\) and \(a = 1\) is closely related to \cite[Lemma 5.2.2]{local-models}, 
    and \cref{res:straightening-nongeneric} is closely related to \cite[Proposition 2.2]{pappasPhiModulesCoefficient2009}.
\end{remark}

\subsection{Infinitesimal contraction and straightening}
We will also need an infinitesimal version of the straightening lemma,
for the two different actions of \(\operatorname{Lie}L^{+}\mathcal{G}(J)\)
on \(L \mathcal{G}(R)\), given a square zero extension \(J \to R \to \overline{R}\). 
This in turn will be based on an infinitesimal
version of the contraction lemma. A key feature of working with the Lie algebra 
is that we get the bounds for \cref{res:straightening} and \cref{res:straightening-nongeneric} 
with \(a = 1\) as long as \(pJ = 0\).

\subsubsection{Some remarks about the Lie algebra action}\label{remarks-about-the-lie-algebra-action}
First, we will say some words about the Lie algebra action. 
Let \(J \to R \twoheadrightarrow \overline{R}\) be a square zero
extension of \(\varpi\)-adically complete \(\mathcal{O}\)-algebras. By definition, we have an exact sequence
of groups \[
1 \to \operatorname{Lie}L^{+}\mathcal{G}(J) \to L^{+} \mathcal{G}(R) \to L^{+} G(\overline{R}) \to 1 . 
\]
The Lie algebra \(\operatorname{Lie}L^+ \mathcal{G}(J)\) has an \emph{abelian} group structure and a compatible \(\overline{R}\)-module structure \cite[Theorem II.4.3.5]{demazureIntroductionAlgebraicGeometry1980}. 
Note that \(\varphi\) as well as conjugation
\(\operatorname{Ad}_{B}\) by any element \(B \in L^{+}\mathcal{G}(R)\)
are \emph{natural} group homomorphisms and therefore restrict to an action on the Lie
algebra \(\operatorname{Lie}L^{+}\mathcal{G}(J)\), which we also denote
by \(\varphi\) and \(\operatorname{Ad}_{B}\), respectively. 

Suppose now that \(\mathcal{X}\) is any \(\mathcal{O}\)-scheme
with an \(L^{+}\mathcal{G}\)-action. Let
\(\overline{X} \in \mathcal{X}(\overline{R})\) and let
\(\mathcal{X}(R)_{\overline{X}}\) denote the fiber of \(X\) with respect
to the reduction map \(\mathcal{X}(R) \to \mathcal{X}(\overline{R})\).
Let \(L^{+}\mathcal{G}(R)_{\overline{X}}\) be the preimage in
\(L^{+}\mathcal{G}(R)\) of the stabilizer of \(\overline{X}\) in
\(L^{+}\mathcal{G}(\overline{R})\). Then
\(L^{+}\mathcal{G}(R)_{\overline{X}}\) is precisely the stabilizer of
\(\mathcal{X}(R)_{\overline{X}}\) in \(L^{+}\mathcal{G}(R)\) (but not
the pointwise stabilizer in general). In particular,
\(\operatorname{Lie}L^{+}\mathcal{G}(J)\) acts on
\(\mathcal{X}(R)_{\overline{X}}\) and in fact
\(\operatorname{Lie}L^{+}\mathcal{G}(J)\) is the stabilizer of
\(\mathcal{X}(R)_{\overline{X}}\) if and only if the stabilizer of
\(\overline{X}\) in \(L^{+}\mathcal{G}(\overline{R})\) is trivial.

We now specialize to the case \(\mathcal{X} = L \mathcal{G}\). Then for
any \(\overline{X} \in L \mathcal{G}(\overline{R})\) the left
translation action and \((\varphi,b)\)-twisted conjugation action give
two induced actions of \(\operatorname{Lie}L^{+}\mathcal{G}(J)\) on
\(L \mathcal{G}_{\overline{X}}\): 
\begin{align*}
B \cdot X & = B X , \\
X \star A & = A^{-1} X \varphi_{c}(A).
\end{align*}

\subsubsection{Infinitesimal contraction and straightening}
If \(J \to R \to \overline{R}\) is a square zero extension of
\(\mathcal{O}\)-algebras, we can equip
\(\operatorname{Lie}L^{+}\mathcal{G}(J)\) with the unique metric so that
\(\operatorname{Lie}L^{+}\mathcal{G}(J) \subset L^{+}\mathcal{G}(R)\) is
a metric subspace with respect to the \(v\)-adic metric. This is the
\(v\)-adic metric on \(\operatorname{Lie}L^{+}\mathcal{G}(J)\).

\begin{lemma}[Infinitesimal contraction]
    \label{res:infinitesimal-contraction-lemma}
    Assume that \(I\) acts trivially on \(\widehat{G}\). 
    Let \(d \in \mathbb{Z}_{\geq -1}\), and \(f \in \mathbb{Z}_{\geq 1} \cup \lbrace 0^+ \rbrace\). 
    Assume that \(x\) is \(d\)-generic, and 
    \[ (p-1)\lfloor f \rfloor + d - h_{\mu} > 0 . \]
    Let \(J \to R \twoheadrightarrow \overline{R}\) be a square
    zero extension of \(p\)-adically complete \(\mathcal{O}\)-algebras, where \(pJ = 0\).
    Let \(X \in L^{\leq \mu}\mathcal{G}(R)\), and consider the operator
    \begin{align*} 
        \Psi : L\mathcal{G}(R) & \to L\mathcal{G}(R) \\ A & \mapsto X \varphi_{c}(A) X^{-1} . 
    \end{align*}
    Then \(\Psi\) restricts to
    \(\operatorname{Lie}L^{+}\mathcal{G}_{f}(J) \to \operatorname{Lie}L^{+}\mathcal{G}_{f}(J)\),
    and this is a contraction with respect to the \(v\)-adic metric on
    \(\operatorname{Lie}L^{+}\mathcal{G}_{f}(J)\).
\end{lemma}
\begin{proof}
    The proof follows the outline of the proof in the
    non-infinitesimal case. We give the proof for \(f = 0^+\) and note that the proof for \(f=n \in \mathbb{Z}_{\geq 1}\) is similar. 

    We first show that \(\operatorname{Lie}L^{+}\mathcal{G}_{0^{+}}(J)\) is
    stable under \(\Psi\). Let
    \(A \in \operatorname{Lie}L^{+}\mathcal{G}_{0^{+}}(J) \subset L^{+}\mathcal{G}_{0^{+}}(R)\).
    By \cref{res:varphi-contraction-generic} we have a morphism of
    group schemes \(\varphi_{c} : L^+ \mathcal{G}_{0^+} \to L^+ \mathcal{G}_{d}\) which therefore induces 
    \(\varphi_{c}:\operatorname{Lie}L^{+}\mathcal{G}_{0^{+}}(J) \to \operatorname{Lie}L^{+}\mathcal{G}_{d}(J)\).
    Since \(pJ = 0\) we have
    \(\operatorname{Lie}L^{+}\mathcal{G}_{d}(J) = \operatorname{Lie}L^{+d}\mathcal{G}(J)\)
    by \cref{res:congruence-subgroups-of-lie-algebras}. The morphism
    \(\operatorname{Ad}_{X} : L^{+d}\mathcal{G} \to L^{+d - h_\mu}\mathcal{G}\)
    from \cref{res:conjugation-congruence-bound}
    induces 
    \(\operatorname{Lie}L^{+d}\mathcal{G}(J) \to \operatorname{Lie}L^{+d-h_{\mu}}\mathcal{G}(J)\),
    and we have
    \[ \operatorname{Lie}L^{+d-h_{\mu}}\mathcal{G}(J) = \operatorname{Lie}L^{+}\mathcal{G}_{d-h_{\mu}}(J) \subset \operatorname{Lie}L^{+}\mathcal{G}_{0^{+}}(J)\]
    since \(pJ = 0\) and \(d - h_{\mu} > 0\).

    It remains to show that \(\Psi\) is a contraction. Towards this end it
    suffices to show that \(\Psi\) maps any Lie congruence subgroup
    \(\operatorname{Lie}L^{+}\mathcal{G}_{n}\), where
    \(n \in \mathbb{Z}_{\geq 1}\), into a deeper Lie congruence subgroup. This follows from the \(f = n\) version of 
    the argument above.
\end{proof}
\begin{lemma}[Infinitesimal straightening]
    \label{res:infinitesimal-straightening}
    Assume that \(I\) acts trivially on \(\widehat{G}\). 
    Let \(d \in \mathbb{Z}_{\geq -1}\), and \(f \in \mathbb{Z}_{\geq 1} \cup \lbrace 0^+ \rbrace\). 
    Assume that \(x\) is \(d\)-generic, and 
    \[ (p-1)\lfloor f \rfloor + d - h_{\mu} > 0 . \]
    Let \(J \to R \twoheadrightarrow \overline{R}\) be a square
    zero extension of \(p\)-adically complete \(\mathcal{O}\)-algebras, where \(pJ = 0\). 
    Let \(X \in L^{\leq \mu} \mathcal{G}(R)_{\overline{X}}\). Then:
    \begin{enumerate}
        \item For any \(A \in \operatorname{Lie}L^{+}\mathcal{G}_{x,f}(J)\) there exists a
            unique \(B \in \operatorname{Lie}L^{+}\mathcal{G}_{x,f}(J)\) such
            that \(B \cdot X = X \star A\). 
        \item For any \(B \in \operatorname{Lie}L^{+}\mathcal{G}_{x,f}(J)\) there exists a
            unique \(A \in \operatorname{Lie}L^{+}\mathcal{G}_{x,f}(J)\) such
            that \(B \cdot X = X \star A\).
    \end{enumerate}
\end{lemma}

\begin{proof}
The proof is similar to the proof of \cref{res:straightening}.
Let us first prove (1). Given \(A\) we are forced to let \[
B = A^{-1} X \varphi_{c}(A)X^{-1} ,
\] and we have to show that
\(B \in\operatorname{Lie}L^{+}\mathcal{G}_{f}(J)\), which follows
from \cref{res:infinitesimal-contraction-lemma}.

It remains to show (2). Given \(B\), consider the operator 
\begin{align*}
\Psi_{B} : \operatorname{Lie} L^{+}\mathcal{G}_{x,f}(J) & \to \operatorname{Lie}L^{+}\mathcal{G}_{x,f}(J) \\
A & \mapsto X\varphi_{c}(A)X^{-1} B^{-1} .
\end{align*}
This is a contraction by \cref{res:infinitesimal-contraction-lemma}, so
\(\Psi_{B}\) has a unique fixed point by Banach's fixed point theorem.
\end{proof}

We also have the following version of infinitesimal straightening which holds without the unramified and genericity assumptions. 
\begin{lemma}[Infinitesimal straightening]
    \label{res:infinitesimal-straightening-nongeneric}
    Let \(f,a \in \mathbb{Z}_{\geq 1}\), where 
    \[ (p-1)f  - h_{\mu} > 0 . \]
    Let \(J \to R \twoheadrightarrow \overline{R}\) be a square
    zero extension of \(p\)-adically complete \(\mathcal{O}\)-algebras, where \(pJ = 0\). 
    Let \(X \in L^{\leq \mu} \mathcal{G}(R)_{\overline{X}}\). Then:
    \begin{enumerate}
        \item For any \(A \in \operatorname{Lie}L^{+}\mathcal{G}_{x,f}(J)\) there exists a
            unique \(B \in \operatorname{Lie}L^{+}\mathcal{G}_{x,f}(J)\) such
            that \(B \cdot X = X \star A\). 
        \item For any \(B \in \operatorname{Lie}L^{+}\mathcal{G}_{x,f}(J)\) there exists a
            unique \(A \in \operatorname{Lie}L^{+}\mathcal{G}_{x,f}(J)\) such
            that \(B \cdot X = X \star A\).
    \end{enumerate}
\end{lemma}
\begin{proof}
    Just like the proof of \cref{res:infinitesimal-straightening} (including \cref{res:infinitesimal-contraction-lemma}), using \cref{res:varphi-contraction-nongeneric-corollary} in place of \cref{res:varphi-contraction-generic}. 
\end{proof}
\section{The local model theorem}
\label{sec:comparing-affine-grassmannians}
In this section we apply the calculations of \cref{sec:contraction-and-straightening} towards comparing the stacks \(\mathcal{Y}^{\leq \mu}\) and \(\operatorname{Gr}^{\leq \mu}\).
In particular, we prove our main results \cref{res:smooth-equivalence} and \cref{res:explicit-smooth-equivalence}.

\subsection{Consequences of straightening}
\label{sec:consequences-of-straightening}
The upshot of straightening is that it enables us to relate the stacks \(\mathcal{Y}^{\leq \mu}\) and \(\operatorname{Gr}^{\leq \mu}\). 
We now point out some of the most immediate consequences.
In the situation of \cref{res:straightening} or \cref{res:straightening-nongeneric}, we have an isomorphism of stacks 
\begin{equation}
    \label{eq:general-equality}
    \mathcal{Y}_{f,\mathcal{O} / \varpi^a}^{\leq \mu} \cong \operatorname{Gr}_{f,\mathcal{O} / \varpi^a}^{\leq \mu} ,
\end{equation}
where the subscript \(\mathcal{O} / \varpi^a\) denotes base change to \(\operatorname{Spec} \mathcal{O} / \varpi^a\). 
We point out two special cases of \eqref{eq:general-equality}: 
\begin{itemize}
    \item If \(I\) acts trivially on \(\widehat{G}\) and \(x\) is \(d\)-generic, where \(d > h_{\mu}\) is an integer, then
        \begin{equation}
            \label{eq:generic-equality}
            \mathcal{Y}_{0^+,\mathbb{F}}^{\leq \mu} \cong \operatorname{Gr}_{0^+, \mathbb{F}}^{\leq \mu} . 
        \end{equation}
    \item As long as \(n > h_{\mu} / (p-1)\), we have 
        \begin{equation}
            \label{eq:nongeneric-equality}
            \mathcal{Y}_{n,\mathbb{F}}^{\leq \mu} \cong \operatorname{Gr}_{n,\mathbb{F}}^{\leq \mu} .
        \end{equation}
\end{itemize}
We have the following consequence of \eqref{eq:general-equality}. 
\begin{proposition}
    \label{res:algebraicity}
    Let \(f \in \mathbb{Z}_{\geq 0} \cup \lbrace 0^+ \rbrace\). 
    Then the stack \(\mathcal{Y}_{f}^{\leq \mu}\) is a \(p\)-adic formal algebraic stack over \(\operatorname{Spf}\mathcal{O}\).
\end{proposition}
\begin{proof}
    Note that in order to prove that \(\mathcal{Y}^{\leq \mu}\) is a \(p\)-adic formal algebraic stack, it suffices by \cite[Proposition A.13]{emertonModuliStacksEtale2023}
    to show that \(\mathcal{Y}^{\leq \mu}_{\mathcal{O} / \varpi^a}\) is an algebraic stack over \(\operatorname{Spec} \mathcal{O} / \varpi^a\) for every \(a \in \mathbb{Z}_{\geq 1}\). 
    For \(n \gg 0\) we have by \cref{res:straightening-nongeneric} a diagram 
    \[
    \begin{tikzcd}
    & \mathcal{Y}_{n,\mathcal{O} / \varpi^a}^{\leq \mu} \cong \operatorname{Gr}_{n,\mathcal{O} / \varpi^a}^{\leq \mu} \arrow[dl] \arrow[dr] & \\
    \mathcal{Y}^{\leq \mu}_{\mathcal{O} / \varpi^a} & & \operatorname{Gr}_{\mathcal{O} / \varpi^a}^{\leq \mu} ,
    \end{tikzcd}
    \]
    where both arrows are torsors for the smooth affine group scheme \(L^n \mathcal{G}_{x}\) from \cref{sec:congruence-loop-groups}. 
    Since \(\operatorname{Gr}_{\mathcal{O} / \varpi^a}^{\leq \mu}\) is a projective scheme over \(\mathcal{O}\),
    the \(L^n \mathcal{G}\)-torsor \(\operatorname{Gr}_{n,\mathcal{O} / \varpi^a}\) is also a scheme. 
    The quotient stack \(\mathcal{Y}^{\leq \mu}_{\mathcal{O} / \varpi^a} = \left[ \mathcal{Y}_{n,\mathcal{O} / \varpi^a}^{\leq \mu}  / L^n \mathcal{G} \right]\)
    is then algebraic by \cite[\href{https://stacks.math.columbia.edu/tag/06DC}{Theorem 06DC}]{stacks-project}.

    The above completes the proof in the case \(f = 0\), so assume now that \(f \neq 0\). 
    Algebraicity of \(\mathcal{Y}_{f,\mathcal{O} / \varpi^a}^{\leq \mu}\) follows from algebraicity of \(\mathcal{Y}_{\mathcal{O} / \varpi^a}^{\leq \mu}\). 
    Indeed, note that \(\widetilde{Y} := \mathcal{Y}_{n,\mathcal{O} / \varpi^a}^{\leq \mu} \times_{\mathcal{Y}_{\mathcal{O} / \varpi^a}^{\leq \mu}} \mathcal{Y}_{f,\mathcal{O} / \varpi^a}^{\leq \mu} \to \mathcal{Y}_{n,\mathcal{O} / \varpi^a}^{\leq \mu}\)
    is an \(L^f \mathcal{G}\)-torsor if \(f \in \mathbb{Z}_{\geq 1}\), a \(\mathcal{T}_{\mathcal{O}/ \varpi^a}\)-torsor if \(f = 0^+\). 
    Either way it is a torsor for a smooth affine group scheme, so \(\widetilde{Y}\) is representable by a scheme since \(\mathcal{Y}_{n,\mathcal{O} / \varpi^a}^{\leq \mu} \cong \operatorname{Gr}_{n,\mathcal{O} / \varpi^a}^{\leq \mu}\) is. 
    And the map \(\widetilde{Y} \to \mathcal{Y}_{f,\mathcal{O} / \varpi^a}^{\leq \mu}\) is smooth surjective since \(\mathcal{Y}_{n,\mathcal{O} / \varpi^a}^{\leq \mu} \to \mathcal{Y}_{\mathcal{O} / \varpi^a}^{\leq \mu}\) is. 
    
    We have to show that the diagonal \(\Delta : \mathcal{Y}_{f}^{\leq \mu} \to \mathcal{Y}_{f}^{\leq \mu} \times \mathcal{Y}_{f}^{\leq \mu}\) is representable by by algebraic spaces, in fact it is representable by affine schemes. 
    This follows from the statement that if \(R\) is an \(\mathcal{O}\)-algebra and \(X,Y \in L^{\leq \mu}\mathcal{G}_{x,f}(R)\), then the sheaf over \(R\) 
    defined by \(R' \mapsto \lbrace A \in L^+\mathcal{G}_{x,f}(R') : A^{-1}X \varphi_{c}(A) = Y \rbrace\) is affine. 
    But this is clearly a closed subscheme of \(L^+\mathcal{G}_{x,f}\) which is known to be an affine scheme. 
\end{proof}
\begin{remark}
    \label{rem:smooth-local-properties}
    It follows from the proof that the morphism \(\mathcal{Y}_{f,\mathcal{O} / \varpi^a}^{\leq \mu} \to \operatorname{Spec}\mathcal{O}/\varpi^a\) satisfies any property satisfied by \(\operatorname{Gr}_{\mathcal{O} / \varpi^a}^{\leq \mu} \to \operatorname{Spec}\mathcal{O} / \varpi^a\) which 
    is smooth local on the source. In particular, it is locally of finite presentation \cite[\href{https://stacks.math.columbia.edu/tag/06FC}{Remark 06FC}]{stacks-project}. 
    Moreover \(\mathcal{Y}^{\leq \mu}_{f,\mathcal{O} / \varpi^a}\) also satisifes any property satisfied by \(\operatorname{Gr}_{\mathcal{O} / \varpi^a}^{\leq \mu}\) which is smooth local, some of which are listed in \cite[\href{https://stacks.math.columbia.edu/tag/04YH}{Remark 04YH}]{stacks-project}. 
    In particular, \(\mathcal{Y}^{\leq \mu}_{f,\mathcal{O} / \varpi^a}\) is locally noetherian, and \(\mathcal{Y}^{\leq \mu}_{f,\mathbb{F}}\) is reduced under the assumption that \(p \nmid \pi_{1}(G^*_{\text{der}})\) by \cite[Theorem 9.1]{pappas-zhu}.
\end{remark}
\begin{remark}
    \label{rem:Y-satisfies-fpqc-descent}
    Since \(\mathcal{Y}_{f}^{\leq \mu}\) has affine diagonal as shown in the proof, it satisfies descent for the fpqc topology by \cite[\href{https://stacks.math.columbia.edu/tag/0GRH}{Proposition 0GRH}]{stacks-project}.
\end{remark}

\subsection{Smooth equivalence}
We remind the reader that the current setup, including standing assumptions, are summarized in \cref{sec:setup}.
The following is one of our main results. 
\begin{theorem}
    \label{res:smooth-equivalence}
    There exists a \(p\)-adic formal scheme \(Z^{\wedge \varpi}\) and a diagram 
    \[
    \begin{tikzcd}
        & Z^{\wedge \varpi} \arrow[dl, twoheadrightarrow] \arrow[dr, twoheadrightarrow] & \\ 
        \mathcal{Y}^{\leq \mu} & & \operatorname{Gr}^{\leq \mu, \wedge \varpi}
    \end{tikzcd}
    \]
    where both morphisms are smooth covers. 
\end{theorem}
The starting point for the proof is straightening over the special fiber, so we make the following assumption. 
\begin{AssumptionStraight}\refstepcounter{AssumptionStraight}
    \label{ass:straight}
    We fix \(f \in \mathbb{Z}_{\geq 1}\cup \lbrace 0^+ \rbrace\) and assume that one of the following holds:
    \begin{enumerate}[(I)]
        \item \(I\) acts trivially on \(\widehat{G}\), and \(x\) is \(d\)-generic, where \(d > h_{\mu}\) is an integer. 
        \item \(f > h_{\mu} / (p-1)\). 
    \end{enumerate}
\end{AssumptionStraight}
Then, as pointed out in \cref{sec:consequences-of-straightening}, the quotient stacks 
\begin{equation}
    \label{eq:equality-of-reduced-stacks}
    \mathcal{Y}_{f, \mathbb{F}}^{\leq \mu} \cong \operatorname{Gr}_{f, \mathbb{F}}^{\leq \mu}
\end{equation}
are isomorphic.
On the other hand, there is no obvious map between \(\mathcal{Y}_{f, \mathcal{O} / \varpi^a}^{\leq \mu}\) and \(\operatorname{Gr}_{f,\mathcal{O} / \varpi^a}^{\leq \mu}\)
for some arbitrary \(a \in \mathbb{Z}_{\geq 1}\).\footnote{Although straightening dictates that \(\mathcal{Y}_{n, \mathcal{O} / \varpi^a}^{\leq \mu} \cong \operatorname{Gr}_{n,\mathcal{O} / \varpi^a}^{\leq \mu}\) for \(n \gg 0\), we are now in the situation that \(f\) is fixed.}
But in some sense, the lack of an obvious map is the only issue. 
The idea is that if we are able to define a map \(\operatorname{Gr}_{f}^{\leq \mu, \wedge \varpi} \dashrightarrow \mathcal{Y}_{f}^{\leq \mu}\) Zariski locally,
then this map will automatically be an isomorphism between open neighborhoods, as we will show in \cref{res:neighborhood-lemma}. 

\subsubsection{Proof of \cref{res:smooth-equivalence}}
\label{sec:proof-of-smooth-equivalence}
Throughout the proof we keep \cref{ass:straight}.
Let \(U^{\wedge \varpi}\) be a formal scheme over \(\operatorname{Spf}\mathcal{O}\) 
and \(g : U^{\wedge \varpi} \to L^{\leq \mu} \mathcal{G}^{\wedge \varpi}\) a morphism over \(\operatorname{Spf}\mathcal{O}\). 
We have a commutative diagram 
\[
    \begin{tikzcd}
        & L^{\leq \mu} \mathcal{G}^{\wedge \varpi} \arrow[dl,"p_{Y}"] \arrow[dr,"p_{G}"'] & \\ 
        \mathcal{Y}_{f}^{\leq \mu} & & \operatorname{Gr}_{f}^{\leq \mu, \wedge \varpi} \\ 
        & U^{\wedge \varpi} \arrow[ur,"g_{G}"'] \arrow[uu, "g"] \arrow[ul, "g_{Y}"] & 
    \end{tikzcd}
\]
where by definition \(g_{G}\) and \(g_{Y}\) are the maps making the diagram commutative. 
\begin{lemma}
    \label{res:neighborhood-lemma}
    We have the following: 
    \begin{enumerate}
        \item \(g_{G}\) is smooth if and only if  \(g_{Y}\) is smooth. 
        \item \(g_{G}\) is étale if and only if \(g_{Y}\) is étale. 
        \item \(g_{G}\) is an open immersion if and only if \(g_{Y}\) is an open immersion. 
    \end{enumerate}
\end{lemma}
\begin{proof}
    We prove the implications \(g_{G}\) is smooth/étale/an open immersion implies that so is \(g_{Y}\). 
    The arguements will be sufficiently symmetric that they can also be used prove the converse, but we leave it to the reader to check the details. 

    By definition it suffices to prove the statements after
    base change to \(\mathcal{O} / \varpi^{a}\), \(a \in \mathbb{Z}_{\geq 1}\), so
    from now on we assume that we are over \(\mathcal{O} / \varpi^a\) but we
    suppress the subcripts normally used to indicate this fact.
    Part (1) will take up the majority of the proof. 
    Throughout we make the assumption that \(g_{G}\) is smooth. 

    Let \(\widetilde{Y} = \mathcal{Y}_{n, \mathcal{O} / \varpi^a}^{\leq \mu} \cong \operatorname{Gr}_{n,\mathcal{O}/ \varpi^a}^{\leq \mu}\)
    be as in the proof of \cref{res:algebraicity}.
    Consider the pullback diagram 
    \[
    \begin{tikzcd}
        U' \arrow[r,"g_{Y}'"] \arrow[d] & \widetilde{Y} \arrow[d] \\
        U \arrow[r,"g_{Y}"] & \mathcal{Y}_{f}^{\leq \mu} . 
    \end{tikzcd}
    \]
    Then \(g_{Y}\) is smooth if and only if \(g_{Y}'\) is smooth. 
    It will be useful to observe that for \(R\) an \(\mathcal{O} / \varpi^a\)-algebra we have 
    \[
        \widetilde{Y}(R) \cong L^{\leq \mu}\mathcal{G}(R) /^{c,\varphi} L^+ \mathcal{G}_{n}(R) \cong L^+ \mathcal{G}_{n}(R) \backslash L^{\leq \mu} \mathcal{G}(R)
    \]
    by \cref{res:presheaf-surjectivity-of-affine-grassmannian}. In particular, we can identify 
    \[
    U'(R) = \left\lbrace \left( [X], u \right) \in \widetilde{Y}(R) \times U(R) : \exists A \in L^+ \mathcal{G}_{f}(R) \text{ such that } X \star A = g(u) \right\rbrace . 
    \]

    Note that \(\widetilde{Y}\) is a locally noetherian scheme by \cref{rem:smooth-local-properties}, and \(U' \to \widetilde{Y}\) is locally of finite type since \(U\) is locally of finite type over \(\mathcal{O} / \varpi^a\) by the smoothness assumption on \(g_{G}\).
    So in order to show that \(g_{Y}'\) is smooth, it suffices by \cite[\href{https://stacks.math.columbia.edu/tag/02HY}{Lemma 02HY}]{stacks-project}
    to show that for any small\footnote{The extension being \emph{small} implies that \(pJ = 0\).} extension
    \(J \to R \twoheadrightarrow \overline{R}\) of Artinian local
    \(\mathcal{O} / \varpi^a \)-algebras of finite type over
    \(\mathcal{O} / \varpi^a\), 
    the lifting problem 
    % https://q.uiver.app/#q=WzAsNCxbMCwwLCJcXG9wZXJhdG9ybmFtZXtTcGVjfVxcb3ZlcmxpbmV7Un0iXSxbMiwwLCJVJyJdLFsyLDIsIkxee1xcbGVxIFxcbXV9IFxcbWF0aGNhbHtHfSJdLFswLDIsIlxcb3BlcmF0b3JuYW1le1NwZWN9UiJdLFszLDIsIltYXSJdLFswLDEsIihbXFxvdmVybGluZXtYfV0sXFxvdmVybGluZXt1fSkiXSxbMSwyLCJnX1knIl0sWzMsMSwiKFtYXSx1KSIsMSx7InN0eWxlIjp7ImJvZHkiOnsibmFtZSI6ImRhc2hlZCJ9fX1dLFswLDMsIiIsMSx7InN0eWxlIjp7InRhaWwiOnsibmFtZSI6Imhvb2siLCJzaWRlIjoidG9wIn19fV1d
    \[\begin{tikzcd}
        {\operatorname{Spec}\overline{R}} && {U'} \\
        \\
        {\operatorname{Spec}R} && \widetilde{Y}
        \arrow["{[X]}", from=3-1, to=3-3]
        \arrow["{([\overline{X}],\overline{u})}", from=1-1, to=1-3]
        \arrow["{g_Y'}", from=1-3, to=3-3]
        \arrow["{([X],u)}"{description}, dashed, from=3-1, to=1-3]
        \arrow[hook, from=1-1, to=3-1]
    \end{tikzcd}\]
    has a solution. 
    A solution to the lifting problem will be provided by a \(u \in U(R)\) and a \(C \in L^+\mathcal{G}_{f}(R)\) such that \(X \star C = g(u)\).

    Note that commutativity of the diagram means that there exists \(\overline{A} \in L^+\mathcal{G}_{f}(\overline{R})\) 
    such that \(\overline{X} \star \overline{A} = g(\overline{u})\). Choose \(A \in L^+\mathcal{G}_{f}(R)\) lifting \(\overline{A}\) (which exists by formal smoothness of \(L^+\mathcal{G}_{f}\)), 
    and let \(Y = X \star A\). 
    Then \(\overline{Y} = g(\overline{u})\), so by smoothness of \(g_{G}\) there exists \(u \in U(R)\) and \(B \in L^+\mathcal{G}_{f}(R)\) such that \(BY = g(u)\). 
    Since \(Y\) and \(g(u)\) reduce to the same point \(\overline{Y} = g(\overline{u})\) in \(L^{\leq \mu}\mathcal{G}(\overline{R})\), 
    in fact \(B \in \operatorname{Lie} L^+ \mathcal{G}_{f}(J)\).
    By \cref{ass:straight} and infinitesimal straightening\footnote{More precisely, by \cref{res:infinitesimal-straightening} if we are in situation (I) of \cref{ass:straight}, by \cref{res:infinitesimal-straightening-nongeneric} if we are in situation (II) of \cref{ass:straight}.} 
    there exists \(B' \in \operatorname{Lie}L^+\mathcal{G}_{f}(J)\) such that \(g(u) = BY = Y \star B' = X \star (AB')\), 
    so \(u\) and \(C = AB'\) provides a solution to the lifting problem.

    Now let us show (2), so assume that \(g_{G}\) is étale, or equivalently
    smooth of relative dimension 0 \cite[\href{https://stacks.math.columbia.edu/tag/039P}{Section 039P}]{stacks-project}. 
    By (1) it follows that \(g_{Y}\) is smooth so we only need to show that \(g_{Y}\) is of relative dimension
    0. This can be checked after base change to \(\mathbb{F}\), where it is a consequence of \eqref{eq:equality-of-reduced-stacks}
    and the assumption that \(g_{G}\) is étale.

    Finally, let us show (3). Being an open immersion is equivalent to being
    universally injective and étale \cite[\href{https://stacks.math.columbia.edu/tag/025F}{Section 025F}]{stacks-project},
    so by part (2) it suffices to show that \(g_{Y}\) is universally
    injective if \(g_{G}\) is an open immersion. Being universally injective can be checked after base change to \(\mathbb{F}\), 
    so \(g_{Y}\) is universally injective because of \eqref{eq:equality-of-reduced-stacks} and the assumption that \(g_{G}\) is an open immersion.
\end{proof}

Zariski local triviality of the \(L^+\mathcal{G}_{f}\)-torsor \(L^{\leq \mu}\mathcal{G} \to \operatorname{Gr}_{f}^{\leq \mu}\) follows from \cref{res:zariski-local-splittings}.
So let \(\lbrace Z_{i} \hookrightarrow \operatorname{Gr}_{f}^{\leq \mu} \rbrace_{i \in \mathcal{I}}\) be a Zariski open covering for which there are lifts \(g_{i} : Z_{i} \to L^{\leq \mu}\mathcal{G}\). 
By \cref{res:neighborhood-lemma} the associated maps \(Z_{i}^{\wedge \varpi} \hookrightarrow \mathcal{Y}_{f}^{\leq \mu}\) are also open immersions, and the collection 
\(\lbrace Z_{i}^{\wedge \varpi} \hookrightarrow \mathcal{Y}_{f}^{\leq \mu} \rbrace_{i \in \mathcal{I}}\) is an open covering by \eqref{eq:equality-of-reduced-stacks}.
The compositions \(Z_{i}^{\wedge \varpi} \hookrightarrow \mathcal{Y}_{f}^{\leq \mu} \twoheadrightarrow \mathcal{Y}^{\leq \mu}\) 
and \(Z_{i}^{\wedge \varpi} \hookrightarrow \operatorname{Gr}_{f}^{\leq \mu, \wedge \varpi} \twoheadrightarrow \operatorname{Gr}^{\leq \mu, \wedge \varpi}\)
are both smooth, being a composition of smooth maps (where we use that \(L^f \mathcal{G}\) is a smooth affine group scheme). 
So taking \(Z = \bigsqcup_{i \in \mathcal{I}} Z_{i}\), we obtain a diagram as in \cref{res:smooth-equivalence}.
This completes the proof of \cref{res:smooth-equivalence}.

\subsection{Explicit smooth equivalence}
\label{sec:explicit-smooth-equivalence}
Make \cref{ass:straight}.  
Let \(g : U^{\wedge \varpi} \to L^{\leq \mu} \mathcal{G}^{\wedge \varpi}\) be a morphism for which \(g_{G} : U^{\wedge \varpi} \hookrightarrow \operatorname{Gr}_{f}^{\leq \mu, \wedge \varpi}\)
and \(g_{Y} : U^{\wedge \varpi} \hookrightarrow \mathcal{Y}_{f}^{\leq \mu}\) are open immersions (or equivalently by \cref{res:neighborhood-lemma}, one of them is an open immersion). 
The composites \(U^{\wedge \varpi} \hookrightarrow \operatorname{Gr}_{f}^{\leq \mu, \wedge \varpi} \twoheadrightarrow \operatorname{Gr}^{\leq \mu, \wedge \varpi}\) and \(U^{\wedge \varpi} \hookrightarrow \mathcal{Y}_{f}^{\leq \mu} \twoheadrightarrow \mathcal{Y}^{\leq \mu}\)
are smooth, being the composite of smooth maps, hence universally open by \cite[\href{https://stacks.math.columbia.edu/tag/056G}{Lemma 056G}]{stacks-project}. 
So the images of these smooth maps are open substacks \(U_{G}^{\wedge \varpi} \hookrightarrow \operatorname{Gr}^{\leq \mu, \wedge \varpi}\) and \(U_{Y}^{\wedge \varpi} \hookrightarrow \mathcal{Y}^{\leq \mu}\). 
The situation is summarized by the diagram 
% https://q.uiver.app/#q=WzAsNSxbMSwxLCJVX1lee1xcd2VkZ2UgcH0iXSxbMCwxLCJcXG1hdGhjYWx7WX1ee1xcbGVxIFxcbXUsIFxcd2VkZ2UgcH0iXSxbMiwwLCJVXntcXHdlZGdlIHB9Il0sWzMsMSwiVV9HXntcXHdlZGdlIHB9Il0sWzQsMSwiXFxvcGVyYXRvcm5hbWV7R3J9XntcXGxlcSBcXG11LCBcXHdlZGdlIHB9Il0sWzAsMSwiIiwwLHsic3R5bGUiOnsidGFpbCI6eyJuYW1lIjoiaG9vayIsInNpZGUiOiJib3R0b20ifSwiYm9keSI6eyJuYW1lIjoiYmFycmVkIn19fV0sWzIsMCwiIiwwLHsic3R5bGUiOnsiaGVhZCI6eyJuYW1lIjoiZXBpIn19fV0sWzIsMywiIiwyLHsic3R5bGUiOnsiaGVhZCI6eyJuYW1lIjoiZXBpIn19fV0sWzMsNCwiIiwyLHsic3R5bGUiOnsidGFpbCI6eyJuYW1lIjoiaG9vayIsInNpZGUiOiJ0b3AifX19XV0=
\[\begin{tikzcd}
	&& {U^{\wedge \varpi}} \\
	{\mathcal{Y}^{\leq \mu}} & {U_Y^{\wedge \varpi}} && {U_G^{\wedge \varpi}} & {\operatorname{Gr}^{\leq \mu, \wedge \varpi}}
	\arrow[hook', circled, from=2-2, to=2-1]
	\arrow[two heads, from=1-3, to=2-2]
	\arrow[two heads, from=1-3, to=2-4]
	\arrow[hook, circled, from=2-4, to=2-5]
\end{tikzcd}\]
Both diagonal arrows are smooth coverings. Note that:
\begin{enumerate}
    \item If the open subscheme \(g_{G} : U^{\wedge \varpi} \hookrightarrow \operatorname{Gr}_{f}^{\leq \mu, \wedge \varpi}\) is stable under action of \(L^f \mathcal{G}\) by left translations, then \(U^{\wedge \varpi } \twoheadrightarrow U^{\wedge \varpi}_{G}\) is an \(L^f \mathcal{G}\)-torsor for this action. Here, we use the convention \(L^{0^+}\mathcal{G} := \overline{\mathcal{T}}\).
    \item If the open subscheme \(g_{Y} : U^{\wedge \varpi} \hookrightarrow \mathcal{Y}_{f}^{\leq \mu}\) is stable under the action of \(L^f \mathcal{G}\) by \(c\)-twisted \(\varphi \)-conjugation, then \(U^{\wedge \varpi} \twoheadrightarrow U_{Y}^{\wedge \varpi} \) is an \(L^f \mathcal{G}\)-torsor for this action. 
\end{enumerate}
We deal with one special case where (1) and (2) above is satisfied, corresponding to situation (I) of \cref{ass:straight}. 
\begin{theorem}
    \label{res:explicit-smooth-equivalence}
    Assume that \(I\) acts trivially on \(\widehat{G}\), invoke \cref{ass:pappas-zhu-groups-as-dilations}, assume that \(p \nmid \pi_{1}(\widehat{G})\), and assume that \(x\) is \(d\)-generic, where \(d > h_{\mu}\) is an integer. 
    Then there are Zariski open covers \(U(z_{1})^{\wedge \varpi} \cup \cdots \cup U(z_{n})^{\wedge \varpi} \twoheadrightarrow \operatorname{Gr}^{\leq \mu, \wedge \varpi}\)
    and \(\mathcal{Y}^{\leq \mu}(z_{1}) \cup \cdots \cup \mathcal{Y}^{\leq \mu}(z_{n}) \twoheadrightarrow \mathcal{Y}^{\leq \mu}\)
    such that for each \(i \in \lbrace 1, \dots, n\rbrace\) there is a diagram 
    \[
    \begin{tikzcd}
        & \widetilde{U}(z_{i})^{\wedge \varpi} \arrow[dl] \arrow[dr] & \\ 
        \mathcal{Y}^{\leq \mu}(z_{i}) & & U(z_{i})^{\wedge \varpi} ,
    \end{tikzcd}
    \]
    where both maps are \(\widehat{T}\)-torsors.
    Moreover, the open charts \(U(z_{i}) \hookrightarrow \operatorname{Gr}^{\leq \mu}\) are the explicit charts from \cref{sec:negative-loop-group-charts} labeled by elements of the admissible set \(\operatorname{Adm}(\mu)\). 
\end{theorem}
\begin{proof}
    The assumptions ensure that \(\mathcal{Y}_{0^+,\mathbb{F}}^{\leq \mu} \cong \operatorname{Gr}_{0^+,\mathbb{F}}^{\leq \mu}\) by \eqref{eq:generic-equality}, as a consequence of \cref{res:straightening} (straightening).  
    The group scheme \(\mathcal{G}_{x}\) is the dilation of \(\widehat{G}\) in the Borel \(\widehat{B}\) at \(v=0\) by \cref{ass:pappas-zhu-groups-as-dilations}, 
    and \(\mathcal{G}_{x,0^+}\) is then the dilation of \(\widehat{G}\) in the unipotent radical \(\widehat{U} \subset \widehat{B}\) at \(v=0\) by \cref{ex:dilation-in-unipotent-radical}.
    Since \(p \nmid \pi_{1}(\widehat{G})\), we can take \(U(z_{1}), \dots, U(z_{n})\) as in \cref{sec:negative-loop-group-charts}, giving an open covering of \(\operatorname{Gr}^{\leq \mu}\).
    Let \(\widetilde{U}(z_{1}), \dots, \widetilde{U}(z_{n})\) denote the preimages of \(U(z_{1}), \dots, U(z_{n})\) in \(\operatorname{Gr}^{\leq \mu}_{0^+}\), which since 
    \(L^+ \mathcal{G}_{x} / L^+ \mathcal{G}_{x,0^+} = \overline{\mathcal{T}} = \widehat{T}\) by \cref{sec:quotient-loop-groups} is given by the orbits under the action of \(\widehat{T}\cdot L^{--} \mathcal{G}_{x}\). 
    From \cref{res:explicit-negative-loop-group} we see that 
    \begin{align*}
        \widehat{T}\cdot L^{--} \mathcal{G}_{x}(R) = \left\lbrace A \in \mathcal{G}\left( R \left[ \frac{1}{v+p} \right] \right) : A \mod \frac{1}{v+p} \in \widehat{B}^\text{op}(R)
        \text{ and } A \mod \frac{v}{v+p} \in \widehat{B}\left( R \left[ \frac{1}{p} \right] \right) \right\rbrace . 
    \end{align*}
    In particular, note that \(\widehat{T}\cdot L^{--} \mathcal{G}_{x}\) is stable under both left and right translation by \(\widehat{T}\), and consequently the open charts \(\widetilde{U}(z_{1}), \dots, \widetilde{U}(z_{n})\)
    are stable under both actions of \(\widehat{T}\). 
    If we define \(\mathcal{Y}^{\leq \mu}(z_{1}), \dots, \mathcal{Y}^{\leq \mu}(z_{n}) \subset \mathcal{Y}^{\leq \mu}\) as the images of \(\widetilde{U}(z_{1})^{\wedge \varpi}, \dots, \widetilde{U}(z_{n})^{\wedge \varpi}\) under the \(\widehat{T}\)-torsor \(\mathcal{Y}_{0^+}^{\leq \mu} \to \mathcal{Y}^{\leq \mu}\), 
    the conclusion of the theorem is satisfied. 
\end{proof}

\subsection{Odds and ends: The geometry of \(\mathcal{Y}^{\leq \mu}\)}
\label{sec:geometry-of-Y}
Make \cref{ass:straight}. 
Then we have a local model diagram over \(\mathbb{F}\)
\begin{equation}
    \label{diagram:special-fiber-local-model-diagram}
    \begin{tikzcd}
    & \mathcal{Y}_{f,\mathbb{F}}^{\leq \mu} \cong \operatorname{Gr}_{f,\mathbb{F}}^{\leq \mu} \arrow[dl] \arrow[dr] & \\
    \mathcal{Y}^{\leq \mu}_{\mathbb{F}} & & \operatorname{Gr}_{\mathbb{F}}^{\leq \mu} ,
    \end{tikzcd}
\end{equation}
where both arrows are \(L^f \mathcal{G}\)-torsors (and we use the convention \(L^{0^+}\mathcal{G} = \overline{\mathcal{T}}\)).
Let us now assume that we are in this situation. 

Assume that \(p \nmid \pi_{1}(G^*)\). 
Then the geometry of \(\operatorname{Gr}_{\mathbb{F}}^{\leq \mu}\) is well understood due to \cref{res:special-fiber-of-pappas-zhu-local-model}.
Indeed, for each \(\widetilde{w} \in \operatorname{Adm}^x (\mu)\) (we abuse notation and write \(\widetilde{w}\) instead of \(W_{x} \widetilde{w} W_{x}\)) we have an affine Schubert cell \(S^\circ (\widetilde{w}) = L^+\mathcal{G}_{\mathbb{F}} \widetilde{w} L^+\mathcal{G}_{\mathbb{F}} \subset \operatorname{Gr}_{\mathbb{F}}^{\leq \mu}\) whose closure is the affine Schubert variety \(S(\widetilde{w}) = \bigsqcup_{\widetilde{w}' \in \operatorname{Adm}^x(\mu), \widetilde{w}' \leq \widetilde{w}} S^{\circ}(\widetilde{w}')\) (see \cref{sec:admissible-set} for more details). 
For each \(\widetilde{w} \in \operatorname{Adm}^x(\mu)\) we can then define an ``affine Schubert cell''
\[
    S_{f}^\circ (\widetilde{w}) := L^+ \mathcal{G}_{\mathbb{F}} \widetilde{w} L^+ \mathcal{G}_{\mathbb{F}} \subset \operatorname{Gr}_{f,\mathbb{F}}^{\leq \mu} , 
\]
the preimage of \(S^{\circ}(\widetilde{w})\) in \(\operatorname{Gr}_{f,\mathbb{F}}^{\leq \mu}\). 
Then the natural map \(S_{f}^\circ(\widetilde{w}) \to S^\circ (\widetilde{w})\) is an \(L^f \mathcal{G}\)-torsor.
Since \(S_{f}^{\circ}(\widetilde{w})\) is stable under the \(c\)-twisted \(\varphi\)-conjugation action, it descends to a locally closed substack 
\[
    \mathcal{Y}^{\leq \mu}_{\mathbb{F},\widetilde{w}} := \left[ S_{f}^\circ(\widetilde{w}) /^{c,\varphi} L^+ \mathcal{G}_{\mathbb{F}} \right] \subset \mathcal{Y}^{\leq \mu}_{\mathbb{F}} . 
\]
The analogue of these locally closed substacks in \cite[Section 5.1]{caraiani-levin} are referred to as \emph{Kottwitz-Rapoport strata}. 
We have a stratification 
\[
    \mathcal{Y}^{\leq \mu}_{\mathbb{F}} = \bigsqcup_{\widetilde{w} \in \operatorname{Adm}^x(\mu)} \mathcal{Y}^{\leq \mu}_{\mathbb{F},\widetilde{w}} ,
\]
and the closure of \(\mathcal{Y}^{\leq \mu}_{\mathbb{F},\widetilde{w}}\) in \(\mathcal{Y}^{\leq \mu}_{\mathbb{F}}\) is precisely 
\[
    \mathcal{Y}^{\leq\mu}_{\mathbb{F},\leq \widetilde{w}} := \bigsqcup_{\widetilde{w}' \in \operatorname{Adm}^x(\mu), \widetilde{w}' \leq \widetilde{w}} \mathcal{Y}^{\leq \mu}_{\mathbb{F}, \widetilde{w}'} . 
\]
Moreover, the irreducible components of \(\mathcal{Y}^{\leq \mu}_{\mathbb{F}}\) are \(\lbrace \mathcal{Y}^{\leq \mu}_{\mathbb{F},\leq v^{w \mu}} \rbrace_{w \in W}\), where \(W\) denotes the finite Weyl group of \(G^*\). 

In the literature on Breuil--Kisin modules with descent data, the analogue of an element in \(\mathcal{Y}^{\leq \mu}(\mathbb{F})\) is said to have \emph{shape} (or \emph{genre}) \(\widetilde{w}\) if it belongs to \(\mathcal{Y}^{\leq \mu}_{\widetilde{w}}(\mathbb{F})\).
A defect with this notion of shape is that it is not an open condition for \(\widetilde{w} \in \operatorname{Adm}^x(\mu)\) which are not of maximal length (i. e. \(\widetilde{w} \neq v^{w \mu}\) for some \(w \in W\)), hence it is not stable under deformations. 
For this reason the open substacks \(\mathcal{Y}^{\leq \mu}(\widetilde{z}) \subset \mathcal{Y}^{\leq \mu}\) from \cref{res:explicit-smooth-equivalence} (when the conclusion of the theorem holds) are better suited for some purposes.
In \cite[Section 5.2]{local-models} an element of \(\mathcal{Y}^{\leq \mu}(\widetilde{z})(R)\) is said to have a \emph{\(\widetilde{z}\)-gauge basis}.
\section{Examples}
\label{sec:examples}
We have already given the following detailed examples:
\begin{enumerate}
    \item In \cref{ex:SL2} we worked out some Galois types for \(\widehat{G} = \operatorname{SL}_{2}\), and in particular which ones are Frobenius invariant and which ones are not. 
    \item In \cref{ex:PGL3-fibers} we considered the question of when \(\mathcal{G}_{x}\) has connected fibers in the case \(\widehat{G} = \operatorname{PGL}_{3}\). 
\end{enumerate}
In this section, we will give a few additional examples. 

In \cref{sec:weil-restriction-of-GL3} we study the case of \(\widehat{G}\) being the dual group of \(G = \operatorname{Res}^{\mathbb{Q}_{p^2}}_{\mathbb{Q}_{p}} \GL_{3}\).
This may be helpful for readers familiar with \cite{local-models}, since it corresponds to the case of \(\GL_{3}\) over \(K = \mathbb{Q}_{p^2}\) in loc. cit. 
In particular, we show how to compare inertial types as they appear in loc. cit. with our notion of Galois types and the corresponding points in the enlarged building of \(\GL_{3}\).
We also give a very explicit description of the admissible set in one example. 

In \cref{sec:tamely-ramified-unitary-group} we consider the case of a tamely ramified unitary group. 
This is the first example of a genuinely ramified group, where we have \(G^* \neq \widehat{G}\).
(The reader will recall from \cref{rem:G-star-is-G-hat} that we can undo any twisting in the unramified direction by taking a large coefficient field.)

In all examples, we assume that \(\mathbb{F}\) is large, so that we can identify \(\mathcal{J}\) with the set of embeddings \(\mathbb{F}_{q} \hookrightarrow \mathbb{F}\) as in \cref{ex:large-coefficient-field}. 
We fix some embedding \(\iota_{0} : \mathbb{F}_{q} \hookrightarrow \mathbb{F}\) and let \(\iota_{j} = \iota_{0} \varphi^{-j}\) for \(j = 0, 1, \dots, r-1\). 
Then \(\iota_{0},\iota_{1}, \dots, \iota_{r-1}\) is a complete enumeration of the embeddings \(\mathbb{F}_{q} \hookrightarrow \mathbb{F}\). 

\subsection{Weil restriction of \(\GL_{3}\) over \(\mathbb{Q}_{p^2}\)}
\label{sec:weil-restriction-of-GL3}
Let \(G = \operatorname{Res}_{\mathbb{Q}_{p}}^{\mathbb{Q}_{p^2}} \GL_{3}\) and \(L = \mathbb{Q}_{p^4}((-p)^{1/(p^4 - 1)})\), and assume that \(E\) contains the Galois closure of \(L\). 
Then \(\widehat{G} = \GL_{3} \times \GL_{3}\). 
The group \(\Gamma' = \operatorname{Gal}(\mathbb{Q}_{p^2} / \mathbb{Q}_{p})\) acts on \(\widehat{G}\) by permuting the factors cyclically, and \(\Gamma\) acts via the quotient map \(\Gamma \twoheadrightarrow \Gamma'\).
In particular, \(\sigma\) acts by the ``swap'' automorphism of \(\widehat{G}\) which we denote by \(\psi\). 
As a \(\Gamma\)-group we can identify \(\widehat{G} = \operatorname{Ind}_{\Gamma_{\mathbb{Q}_{p^2}}}^\Gamma \GL_{n}\), where \(\Gamma_{\mathbb{Q}_{p^2}} := \operatorname{Gal}(L / \mathbb{Q}_{p^2})\) is viewed as acting trivially.

\subsubsection{Galois type}
We will make a concrete choice of a Galois type \(\tau : \Gamma \to \widehat{G}(\mathbb{F}^\mathcal{J}\llb u \rrb)\). 
First, we will describe \(\tau\) using the parametrization of \cite[Section 9.2]{GHS}, as is done in \cite{local-models}.
This is primarily done to orient the reader and serves as a helpful example in comparing the notation of \cite{local-models} to ours. 
We then describe how to view this same \(\tau\) as a point in the building \(\widetilde{\mathcal{B}}(\widehat{G}, \mathbb{F}\llp v \rrp)\), following the classification from \cref{res:building-classification-of-types}. 

\subsubsection{Galois type in the manner of \cite{local-models}} 
\label{sec:galois-type-in-the-manner-of-lm}
Let \(s = \left( (123), (12) \right) \in W\) and \(\mu = \left( (16, 11, 7), (4, 2, 1) \right) \in X_{*}(\widehat{T})\).
Recall that \(\eta = \left( (2,1,0), (2,1,0) \right)\) in this case, so \(\mu + \eta = \left( (18, 12, 7), (6, 3, 1) \right)\). 
We assume that \(p \geq 19\) so that all entries of \(\mu+ \eta\) are in the range \([0, p-1]\). 
According to the parametrization of \cite[Section 9.2]{GHS} we obtain a 1-cocycle \(\tau(s,\mu+\eta) : I \to \widehat{G}(\mathbb{F})\) given by 
\begin{align*}
    \tau(s,\mu+\eta) =
     \bigg( &
        {\begin{pmatrix}
            \iota_{0}(\omega)^{18 + 3p + 7p^2 + p^3} & & \\ 
            & \iota_{0}(\omega)^{12 + 6p + 12p^2 + 6 p^3} & \\ 
            & & \iota_{0}(\omega)^{7 + p + 18p^2 + 3p^3} 
        \end{pmatrix}} , \\ 
        & {\begin{pmatrix}
            \iota_{0}(\omega)^{6 + 12p + 6p^2 + 12p^3} & & \\ 
            & \iota_{0}(\omega)^{3 + 7p + p^2 + 18p^3} & \\
            & & \iota_{0}(\omega)^{1 + 18p + 3p^2 + 7p^3} 
        \end{pmatrix}}
        \bigg) . 
\end{align*}
Note that we have
\begin{equation}
    \label{eq:cocycle-condition-in-explicit-example}
    \tau(s,\mu+\eta)^p = s \psi(\tau(s,\mu+\eta))s^{-1},
\end{equation}
so \(\tau(s,\mu+\eta)\) is an inertial type. 

From \(\tau(s,\mu+\eta)\) we obtain a 1-cocycle \(\tau : \Gamma \to \widehat{G}(\mathbb{F} \otimes_{\mathbb{F}_{p}} \mathbb{F}_{p^4}) = \widehat{G}(\mathbb{F}_{0}) \times \widehat{G}(\mathbb{F}_{1}) \times \widehat{G}(\mathbb{F}_{2}) \times \widehat{G}(\mathbb{F}_{3})\). 
This is given by \(\tau_{j}(\sigma) = s\) for every \(j \in \mathcal{J} = \lbrace 0,1,2,3 \rbrace\), and \(\tau|_{I} = \omega^\lambda\),\footnote{
    The expression \(\tau|_{I} = \omega^\lambda\) means that the projection onto factor \(\widehat{G}(\mathbb{F}_{j})\) is given by \(\iota_{j}(\omega)^{\lambda_{j}}\). 
}
where 
\begin{align*} 
\lambda_{0} = \big(
    & (18 + 3p + 7p^2 + p^3, 12 + 6p + 12 p^2 + 6 p^3, 7 + p + 18p^2 + 3p^3), \\
    & (6 + 12 p + 6p^2 + 12p^3, 3 + 7p + p^2 + 18p^3, 1 + 18p + 3p^2 + 7p^3), \big) = s \psi(\lambda_{3}),  \\
\lambda_{1} = \big(
    & (1 + 18p + 3p^2 + 7p^3, 6 + 12p + 6p^2 + 12p^3, 3 + 7p + p^2 + 18p^3), \\
    & (12 + 6p + 12p^2 + 6p^3, 18 + 3p + 7p^2 + p^3, 7 + p + 18p^2 + 3p^3) \big) = s \psi(\lambda_{0}), \\
\lambda_{2} = \big(
    & (7 + p + 18p^2 + 3p^3, 12 + 6p + 12p^2 + 6p^3, 18 + 3p + 7p^2 + p^3) , \\
    & (6 + 12p + 6p^2 + 12p^3, 1 + 18p + 3p^2 + 7p^3, 3 + 7p + p^2 + 18p^3) \big) = s \psi(\lambda_{1}), \\
\lambda_{3} = \big(
    & (3 + 7p + p^2 + 18p^3, 6 + 12p + 6p^2 + 12p^3, 1 + 18p + 3p^2 + 7p^3) , \\
    & (12 + 6p + 12p^2 + 6p^3, 7 + p + 18p^2 + 3p^3, 18 + 3p + 7p^2 + p^3) \big) = s \psi(\lambda_{2}). 
\end{align*}
Note that for every \(\iota_{j} \in \mathcal{J}\), we have \(\tau_{0}|_{I} = \tau_{j}|_{I} : I \to \widehat{T}(\mathbb{F})\) because \(\tau_{j}|_{I} = \iota_{j}(\omega)^{\lambda_{j}} = \iota_{0}(\varphi^{-j}(\omega))^{p^j \lambda_{0}} = \iota_{0}(\omega)^{\lambda_{0}}\). 
Note that \(\varphi \tau = \tau\), so \(\tau\) is strictly Frobenius invariant.
Since \(\sigma\) acts on \(\widehat{G}(\mathbb{F} \otimes_{\mathbb{F}_{p}} \mathbb{F}_{p^4})\) as \(\psi \circ \varphi = \varphi \circ \psi\), we see that \eqref{eq:cocycle-condition-in-explicit-example} implies
the cocycle condition \(\tau(\gamma)^p = \tau(\sigma) (^\sigma \tau(\gamma)) \tau(\sigma)^{-1}\). 
Note that we also have \(\tau(\gamma)^e = 1\) since \(\omega^e = 1\) and that
\[
    \left( \tau(\sigma) (^\sigma \tau(\sigma)) (^{\sigma^2} \tau(\sigma)) (^{\sigma^3}\tau(\sigma)) \right)_{0} = s' s'' s' s'' = (13)^2 = 1 ,
\] 
where \(s = (s',s'')\), and similarly the entry corresponding to another \(j \in \mathcal{J}\) is equal to 1.

Define \(w \in W^\mathcal{J}\) by 
\begin{align*}
    w_{0} & = \left( (123), (12) \right) = s  , \\
    w_{1} & = \left( (13), (23) \right) = s \psi(w_{0}) , \\
    w_{2} & = \left( (12), (321) \right) = s \psi(w_{1}) , \\ 
    w_{3} & = \left( 1, 1 \right) = s \psi(w_{2}) . 
\end{align*}
In fact, \(w\) is characterized by the property that \(w^{-1}\lambda\) is dominant, or even just making the leading \(p^3\)-coefficients dominant. 
Let \(n = w^{-1} u^\lambda . \)
Then it is straightforward to check that \(\tau(\theta) = n^{-1}(^\theta n)\):
\begin{align*}
    n^{-1}(^\gamma n) & = u^{-\lambda}(^\gamma u^\lambda) = \omega^{\lambda} = \tau(\gamma) , \\
    (n^{-1}(^\sigma n))_{j} & = u^{-\lambda_{j}} w_{j} \psi(w_{j-1})^{-1} u^{\psi(\lambda_{j-1})} = u^{-\lambda_{j} + s \psi(\lambda_{j-1})}s = s .  
\end{align*}

We have the point \(x = o - \frac{1}{p^4 - 1}w^{-1}\lambda \in \mathcal{A}\left( \widehat{G}, \mathbb{F}^\mathcal{J}\llp u \rrp \right) ,\) and for each \(j \in \mathcal{J}\), 
\(\varphi(x)_{j} = o_{j} - \frac{p}{p^4 -1} w^{-1}_{j-1}\lambda_{j-1} \in \mathcal{A}\left(\widehat{G}, \mathbb{F}_{j}\llp u \rrp \right) \). 
Now consider \(c \in \widehat{N}(\mathbb{F}^\mathcal{J}\llp u \rrp)^{\Gamma}\) with \(c_{j} = \psi^j \left( s^{-1} v^{-\mu - \eta} \right)\).
Explicitly, 
\begin{align*}
    c_{0} = c_{2} & = ((321),(12))v^{- ((18,12,7),(6,3,1))} , \\ 
    c_{1} = c_{3} & = ((12),(321))v^{- ((6,3,1),(18,12,7))} . 
\end{align*}
Then we have \(c \cdot \varphi(x) = x\). Indeed, for any \(j \in \mathcal{J}\) we have 
\[
    c_{j} = w_{j}^{-1}w_{j-1} u^{w_{j-1}^{-1}(p\lambda_{j-1} - \lambda_{j})} = w_{j}^{-1}u^{p\lambda_{j-1} - \lambda_{j}} w_{j-1} . 
\]
Hence 
\begin{align*}
    c_{j} \cdot \varphi(x)_{j} & = o_{j} - w_{j}^{-1} \left( w_{j-1}\left(\frac{p}{p^4-1} w_{j-1}^{-1} \lambda_{j-1}\right) - \frac{p}{p^4-1} \lambda_{j-1} + \frac{1}{p^4-1}\lambda_{j} \right) = x_{j} . 
\end{align*}

\begin{remark}
    \label{rem:local-model-notation-in-weil-restriction-of-GL3}
    Let us compare how our notation relates to that of \cite{local-models}, in particular \cite[Section 5.1]{local-models}. 
    Note that loc. cit. works with the group \(\GL_{3}\) but only descent data with respect to \(\Gamma_{\mathbb{Q}_{p^2}}\). 
    By Shapiro's Lemma we have \(H^1 \left( \Gamma, \widehat{G}\left( \mathbb{F} \otimes_{\mathbb{F}_{p}} \mathbb{F}_{q} \right) \right) \cong H^1 \left( \Gamma_{\mathbb{Q}_{p^2}}, \GL_{3}\left( \mathbb{F} \otimes_{\mathbb{F}_{p}} \mathbb{F}_{q} \right) \right)\),
    so we can specialize to their situation by projecting to a chosen factor of \(\widehat{G} = \GL_{3} \times \GL_{3}\) and restricting the Galois action to \(\Gamma_{\mathbb{Q}_{p^2}}\). 
    What we call \(\tau(s,\mu+\eta)\) is precisely what is denoted by the same symbol in \cite[Section 2.4]{local-models}. 
    What is called \(\mathbf{a}'^{(j)}\) and \(s_{\text{or},j}'\) in \cite[Section 5.1]{local-models} corresponds to the chosen \(\GL_{3}\)-coordinate of our \(\lambda_{j}\) and \(w_{j}\), respectively. 
    The chosen \(\GL_{3}\)-coordinate of \(c_{j}\) makes an appearance in \cite[Proposition 5.1.8]{local-models}. 
\end{remark}

\subsubsection{Galois type as a point in the building}
We keep the notation from the previous subsection. 
With \(\lambda\) and \(w\) as there, recall that we have \(n = w^{-1}u^\lambda\) and \(x = n \cdot o = o - \frac{1}{e} w^{-1}\lambda \in \widetilde{\mathcal{A}}\left( \widehat{T}, \mathbb{F}^\mathcal{J}\llp u \rrp \right)\), where \(e = p^{4}-1\).
Recall from \cref{sec:buildings} that \(\widetilde{\mathcal{A}}\left( \widehat{T}, \mathbb{F}^\mathcal{J} \llp u \rrp \right) = \prod_{j=0}^3 \widetilde{\mathcal{A}}\left( \widehat{T}, \mathbb{F}_{j}\llp u \rrp \right)\). 
So \(x = (x_{0}, x_{1}, x_{2}, x_{3})\), where \(x_{j} = o_{j} - \frac{1}{e}w_{j}^{-1}\lambda_{j} \in \widetilde{\mathcal{A}}\left( \widehat{T}, \mathbb{F}_{j}\llp u \rrp \right)\). 
We have
\begin{align*}
    w_{2}^{-1} \lambda_{2} = w_{0}^{-1}\lambda_{0} = \big(
        & ( 12 + 6p + 12p^2 + 6p^3, 7 + p + 18p^2 + 3p^3, 18 + 3p + 7p^2 + p^3 ) , \\
        & ( 3 + 7p + p^2 + 18p^3 , 6 + 12p + 6p^2 + 12p^3 , 1 + 18p + 3p^2 + 7p^3) \big) , \\
    w_{3}^{-1} \lambda_{3} = w_{1}^{-1}\lambda_{1} = \big( 
        & ( 3 + 7p + p^2 + 18p^3, 6 + 12p + 6p^2 + 12p^3, 1 + 18p + 3p^2 + 7p^3) , \\ 
        & ( 12 + 6p + 12p^2 + 6p^3, 7 + p + 18p^2 + 3p^3, 18 + 3p + 7p^2 + p^3) \big) . 
\end{align*}
Note that \(x\) is fixed by \(\Gamma\) since \(\psi(x_{j-1}) = x_{j}\) for every \(j \in \mathcal{J}\).
We will identify \(x \in \widetilde{\mathcal{A}}\left( \widehat{T}, \mathbb{F}\llp v \rrp \right) = \widetilde{\mathcal{A}}\left( \widehat{T}, \mathbb{F}^\mathcal{J}\llp u \rrp \right)^\Gamma\) with \(x_{0} \in \widetilde{\mathcal{A}}\left( \widehat{T}, \mathbb{F}_{0} \llp u \rrp \right)\), 
via the projection 
\(\widetilde{\mathcal{A}}\left( \widehat{T}, \mathbb{F}^\mathcal{J}\llp u \rrp \right) \to \widetilde{\mathcal{A}}\left( \widehat{T}, \mathbb{F}_{0}\llp u \rrp \right)\),
which yields an isomorphism 
\[\widetilde{\mathcal{A}}\left( \widehat{T}, \mathbb{F}\llp v \rrp \right) = \widetilde{\mathcal{A}}\left( \widehat{T}, \mathbb{F}^\mathcal{J}\llp u \rrp \right)^\Gamma \isoto \widetilde{\mathcal{A}}\left( \widehat{T}, \mathbb{F}_{0}\llp u \rrp \right)^I . \]
We may moreover identify \(\widetilde{\mathcal{A}}\left( \widehat{T}, \mathbb{F}\llp v \rrp \right) = \widetilde{\mathcal{A}}\left( S , \mathbb{F}\llp v \rrp \right) \times \widetilde{\mathcal{A}}\left( S, \mathbb{F}\llp v \rrp \right)\), where \(S \subset \GL_{3}\) is the diagonal torus. 
Then \(x = (x^{(0)}, x^{(1)})\), where each \(x^{(i)} \in \widetilde{\mathcal{A}}\left( S , \mathbb{F}\llp v \rrp \right)\), and in fact 
\begin{align*}
    x^{(0)} & = o - \frac{1}{e} \left( 12 + 6p + 12p^2 + 6p^3, 7 + p + 18p^2 + 3p^3, 18 + 3p + 7p^2 + p^3 \right) , \\
    x^{(1)} & = o - \frac{1}{e} \left( 3 + 7p + p^2 + 18p^3 , 6 + 12p + 6p^2 + 12p^3 , 1 + 18p + 3p^2 + 7p^3 \right) . 
\end{align*}
See \cref{fig:GL3-alcove} on page \pageref{fig:GL3-alcove} for an overview of the situation in the case \(p = 19\).
From \cref{fig:GL3-alcove} we see that \(x\) is \(2\)-generic when \(p = 19\), something that can also easily be checked by hand.

\begin{figure}[h!]
    \centering
    \resizebox{\textwidth}{!}{
    \begin{tikzpicture}[scale = 12]
        % Set the number of segments and the prime
        \pgfmathsetmacro{\numSegments}{36}
        \pgfmathsetmacro{\p}{19}
        \pgfmathsetmacro{\kbound}{floor(\p/3)}

        % Function to draw the divided triangle
        \newcommand{\dividedTriangle}{
            % Define the coordinates of the large triangle
            \coordinate (A) at (0, 0);
            \coordinate (B) at (1, 0);
            \coordinate (C) at (0.5, 0.866); % height of equilateral triangle with side length 1

            % Draw the red triangles representing genericity zones
            \foreach \k in {0,...,\kbound} {
                \pgfmathparse{1.5*\k/\p}
                \let\xi\pgfmathresult
                \pgfmathparse{0.866*\k/\p}
                \let\yi\pgfmathresult
                \pgfmathparse{1 - (1.5*\k/\p)}
                \let\xii\pgfmathresult
                \pgfmathparse{0.866*(1 - (2*\k/\p))}
                \let\yiii\pgfmathresult
                % \draw[line width=0.05pt,color=red] (\xi,\yi) -- (\xii,\yi) -- (0.5,\yiii) -- cycle;
                \fill[color=red,fill opacity=0.15] (\xi,\yi) -- (\xii,\yi) -- (0.5,\yiii) -- cycle;
            }

            % Divide each side into equal segments and draw the smaller triangles
            \foreach \i in {0,...,\numSegments} {
                \coordinate (P) at ($(A)!\i/\numSegments!(B)$);
                \coordinate (Q) at ($(A)!\i/\numSegments!(C)$);
                \coordinate (R) at ($(B)!\i/\numSegments!(C)$);
                \coordinate (S) at ($(C)!\i/\numSegments!(B)$);
                
                % Draw lines parallel to the sides of the large triangle
                \draw[line width=0.05pt,color=gray] (P) -- (Q);
                \draw[line width=0.05pt,color=gray] (P) -- (S);
                \draw[line width=0.05pt,color=gray] (Q) -- (R);
            }

            % Draw the large triangle
            \draw[thick] (A) -- (B) -- (C) -- cycle;
        }

        % Define the clipping region to create a square "photograph"
        \clip (-0.2,-0.173) rectangle (1.2,1.039);

        % Draw triangles which are shifts of divided triangle
        \foreach \x / \y in {-0.5/-0.866, 0.5/-0.866, -1/0, 0/0, 1/0, -0.5/0.866, 0.5/0.866} {
            \begin{scope}[shift={(\x,\y)}]
                \dividedTriangle
            \end{scope}
        }

        % Draw traingles which are shifts of upside down divided triangle
        \foreach \x / \y in {-0.5/-0.866, 0.5/-0.866, -1/0, 0/0, 1/0, 0/-1.732} {
            \begin{scope}[yscale=-1,shift={(\x,\y)}]
                \dividedTriangle
            \end{scope}
        }

        % Place a dot labeled "x" at (0.5, 0.3)
        \node[circle, fill=black, inner sep=1pt, label={[fill=white, fill opacity=0.5, text opacity=1, rounded corners] above:{$x^{(0)}$}}] at (0.497, 0.626) {};
        \node[circle, fill=black, inner sep=1pt, label={[fill=white, fill opacity=0.5, text opacity=1, rounded corners] above:{$x^{(1)}$}}] at (0.485, 0.371) {};

        % Label corners of triangle
        \node[circle, fill=black, inner sep=1.5pt, label={[fill=white, fill opacity=0.7, text opacity=1, rounded corners] above:{$o$}}] at (0.5,0.866) {};
        \node[circle, fill=black, inner sep=1.5pt, label={[fill=white, fill opacity=0.7, text opacity=1, rounded corners] above:{$o - (1,0,0)$}}] at (0,0) {};
        \node[circle, fill=black, inner sep=1.5pt, label={[fill=white, fill opacity=0.7, text opacity=1, rounded corners] above:{$o - (1,1,0)$}}] at (1,0) {};
    \end{tikzpicture}}
    \caption{
        An alcove in the apartment \(\mathcal{A}\left( S, \mathbb{F}\llp v \rrp \right)\), where \(S \subset \GL_{3}\) is the diagonal torus, when \(p = 19\). 
        The large triangles bounded by the thick black lines are the walls of the apartment \(\mathcal{A}\left( S, \mathbb{F}\llp v \rrp \right)\). 
        The smaller triangles bounded by gray lines represent the walls of the apartment \(\mathcal{A}\left( S,\mathbb{F}\llp u \rrp \right)\) (recall that \(\mathcal{A}\left( S, \mathbb{F}\llp v \rrp \right) = \mathcal{A}\left( S,\mathbb{F}\llp u \rrp \right)\) as a set, but with a different simplicial structure). 
        But in fact, we have not drawn all the small triangles since this is not feasible. 
        Because \(e = p^4 - 1 = 130320\), each large triangle should contain \(130320^2\) small triangles, of which we have drawn \(36^2\), 
        meaning that each of the shown small triangles actually bound \(3620^2\) alcoves of \(\mathcal{A}\left( S,\mathbb{F}\llp u \rrp \right)\). 
        The shade of red represents the genericity, where darker red means more generic.
        We have labeled the points \(x^{(0)}\) and \(x^{(1)}\), or more precisely the projection of these points to the reduced apartment. 
        From the picture we can see that \(x\) is \(2\)-generic, since \(x^{(0)}\) is \(2\)-generic and \(x^{(1)}\) is more than \(2\)-generic. 
    }
    \label{fig:GL3-alcove}
\end{figure}
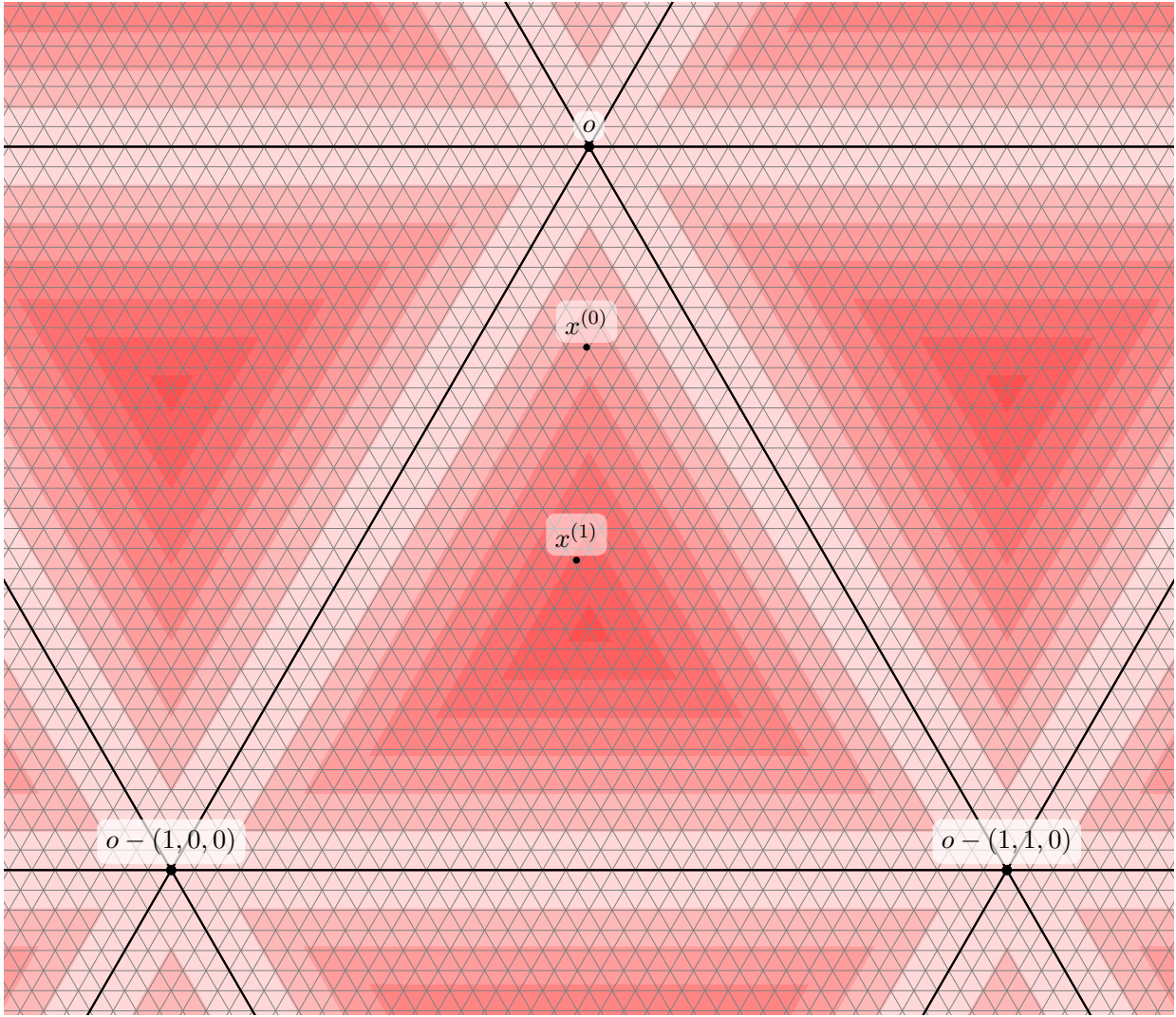

\subsubsection{Bruhat--Tits group scheme}
Let \(\widehat{B} \subset \widehat{G} = \GL_{3} \times \GL_{3}\) be the Borel which is \textit{lower} triangular in each \(\GL_{3}\)-factor. 
Then the Bruhat--Tits group scheme \(\mathcal{G}_{x}\) is the dilation of \(\widehat{G}\), viewed as a group scheme over \(\mathbb{A}^1_{\mathcal{O}} = \operatorname{Spec}\mathcal{O}[v]\), in the Borel \(\widehat{B}\) along the origin \(v = 0\).
This is a consequence of \cref{rem:pappas-zhu-dilation} (since \(x\) is lowest alcove according to \cref{def:lowest-alcove} for the choice of positive roots corresponding to \(\widehat{B}\)). 
It follows that for any \(\mathcal{O}[v]\)-algebra \(R\) in which \(v\) is a non-zerodivisor, we can identify 
\[
    \mathcal{G}_{x}(R) = \lbrace A \in \widehat{G}(R) | A \mod v \in \widehat{B}\left( R / v R \right) \rbrace . 
\]
In particular, we have 
\begin{align*}
    L_{0}^+ \mathcal{G}_{x}(R) & = \lbrace A \in \widehat{G}(R\llb v \rrb) | A \mod v \in \widehat{B}(R) \rbrace , \\
    L_{0} \mathcal{G}_{x}(R) & = \widehat{G}(R\llp v \rrp) . 
\end{align*}
Similarly, we have for a noetherian \(\mathcal{O}\)-algebra \(R\)
\begin{align*}
    L^+ \mathcal{G}_{x}(R) & = \lbrace A \in \widehat{G}(R\llb v + p \rrb) | A \mod v \in \widehat{B}(R\llb v + p \rrb / v R \llb v + p \rrb) \rbrace , \\ 
    L \mathcal{G}_{x}(R) & = \lbrace A \in \widehat{G}(R \llp v + p \rrp) | A \mod v \in \widehat{B}(R\llp v + p \rrp / v R \llp v + p \rrp) \rbrace . 
\end{align*}
The reason we impose the noetherian condition here is to ensure that \(v\) is a non-zerodivisor in \(R\llb v + p \rrb\) and \(R\llp v + p \rrp\). 

\begin{remark} 
    In \cite{local-models}, \(\mathcal{G}_{x}\) is the dilation in the \textit{upper} triangular Borel for the same data as in \cref{sec:galois-type-in-the-manner-of-lm}. 
    The reason is the appearance of the dual in \cite[Definition 5.1.3]{local-models}, which is adapted to the \emph{contravariant} anti-equivalence between étale \(\varphi\)-modules and representations of \(\operatorname{Gal}(\mathbb{Q}_{p,\infty})\). 
    Our notion of \emph{type} from \cref{sec:specialized-equivariant-torsors}, which is better suited to the \emph{covariant} functor, is therefore not compatible with \cite[Definition 5.1.3]{local-models}.
\end{remark}

\subsubsection{The admissible set and the moduli stack \(\mathcal{Y}^{\leq \mu}\)}
\label{sec:weil-res-example-admissible-set}
Let us take
\[
    \mu = \left( (1,0,0),(1,0,0) \right) \in X_{*}(\widehat{T}) . 
\]
Then \(h_{\mu} = 1\). Since \(x\) is \(2\)-generic, \cref{res:explicit-smooth-equivalence} applies and we have a local model diagrams
\[
\begin{tikzcd}
    & \widetilde{U}(\widetilde{z})^{\wedge \varpi} \arrow[dl] \arrow[dr] & \\ 
    \mathcal{Y}^{\leq \mu}(\widetilde{z}) & & U(\widetilde{z})^{\wedge \varpi} , 
\end{tikzcd}
\]
where the elements \(\widetilde{z} \in L^{\leq \mu} \mathcal{G}_{x}(\mathcal{O})\) are representatives of elements in the \emph{admissible set} \(\operatorname{Adm}(\mu)\) that we described in \cref{sec:admissible-set}.

As outlined in \cref{sec:geometry-of-Y} we can understand the underlying geometry of \(\mathcal{Y}^{\leq \mu}\) in terms of Kottwitz-Rapoport strata parametrized by elements of \(\operatorname{Adm}(\mu)\). 
To make this more concrete, we give an explicit description of \(\operatorname{Adm}(\mu)\). 

In the case of \(\widehat{G}\), the Iwahori--Weyl group \(\widetilde{W}^* = \widehat{N}(\mathbb{F}\llp v \rrp) / \widehat{T}(\mathbb{F}\llb v \rrb) \cong X_{*}(\widehat{T}) \rtimes W\), where \(W = \widehat{N}(\mathbb{F}\llp v \rrp) / \widehat{T}(\mathbb{F}\llp v \rrp)\) is the finite Weyl group, coincides with the extended affine Weyl group \(W^\text{ext}\), because the same is true for \(\GL_{3}\).
In fact, \(\widetilde{W}^* = \widetilde{W}' \times \widetilde{W}'\), where \(\widetilde{W}'\) is the Iwahori--Weyl group for \(\GL_{3}\). 
Moreover, the admissible set \(\operatorname{Adm}(\mu) = \operatorname{Adm}(\mu') \times \operatorname{Adm}(\mu') \subset \widetilde{W}' \times \widetilde{W}'\), where \(\mu' = (1,0,0)\). 
So let us just focus on the admissible set for \(\mu'\) in \(\widetilde{W}'\). 

Let \(S \subset \GL_{3}\) be the diagonal torus and let \(\mathcal{C} \subset \mathcal{A}\left( S, \mathbb{F}\llp v \rrp \right)\) be the base alcove containing \(x^{(0)}\) and \(x^{(1)}\). 
Simple reflections across the walls of \(\mathcal{C}\) are given by \(\widetilde{s}_{1} = s_{(12)}\), \(\widetilde{s}_{2} = s_{(23)}\), \(\widetilde{s}_{3} = v^{(1,0,-1)}s_{(13)} = s_{(13)}v^{(-1,0,1)}\), where \(s_{(ij)}\) denotes the permutation that switches the \(i\)'th and \(j\)'th coordinates.
Note that \(\widetilde{s}_{1}, \widetilde{s}_{2}, \widetilde{s}_{3}\) are elements of the affine Weyl group \(W^{\text{aff}}\) (which corresponds to the Iwahori--Weyl group of \(\operatorname{SL}_{3}\)). 
The admissible set of \(\mu' = (1,0,0)\) is by definition the elements of \(\widetilde{W}'\) that are \(\leq\) the elements \(v^{\mu'} = v^{(1,0,0)}, v^{s_{(13)}\mu'} = v^{(0,0,1)}, v^{s_{(12)} \mu'} = v^{(0,1,0)}\) in the Bruhat order (whose definition is given in \cref{sec:admissible-set}). 
Note that \(\widetilde{t} = v^{(1,0,0)}s_{(123)}\) is an element of \(\widetilde{W}'_{\mathcal{C}}\), the stabilizer of \(\mathcal{C}\) in \(\widetilde{W}'\), acting as a rotation centered at the barycenter of \(\mathcal{C}\). 
We have 
\begin{align*}
    v^{\mu'} & = \widetilde{s}_{3} \widetilde{s}_{2} \widetilde{t} , \\ 
    v^{s_{(13)}\mu'} & = \widetilde{s}_{2} \widetilde{s}_{1} \widetilde{t} , \\
    v^{s_{(12)}\mu'} & = \widetilde{s}_{1} \widetilde{s}_{3} \widetilde{t} . 
\end{align*}
From this, we see that the admissible set of \(\mu'\) is precisely 
\begin{align*}
    \operatorname{Adm}(\mu') = \lbrace v^{\mu'}, \,\, v^{s_{(13)}\mu'}, \,\, v^{s_{(12)}\mu'}, \,\, \widetilde{s}_{1}\widetilde{t}, \,\, \widetilde{s}_{2}\widetilde{t}, \,\,\widetilde{s}_{3}\widetilde{t}, \,\, \widetilde{t} \rbrace .
\end{align*}
Since the Bruhat ordering on \(\widetilde{W}'\) is induced by the Bruhat ordering on the affine Weyl group for \(\GL_{3}\), the elements of \(\operatorname{Adm}(\mu')\) correspond to alcoves. 
This is illustrated in \cref{fig:admissible-set} on page \pageref{fig:admissible-set}. 
In fact, an efficient way to find the factorization \(v^{\mu'} = \widetilde{s}_{3} \widetilde{s}_{2}\widetilde{t}\) by hand is to 
first draw the picture of the alcoves as shown in \cref{fig:admissible-set}. 
By doing a ``gallery walk'' we can then easily see that \(v^{\mu'}\mathcal{C} = \widetilde{s}_{3}\widetilde{s}_{2}\mathcal{C}\), and then \(\widetilde{t} = (\widetilde{s}_{3}\widetilde{s}_{2})^{-1}v^{\mu'}\).

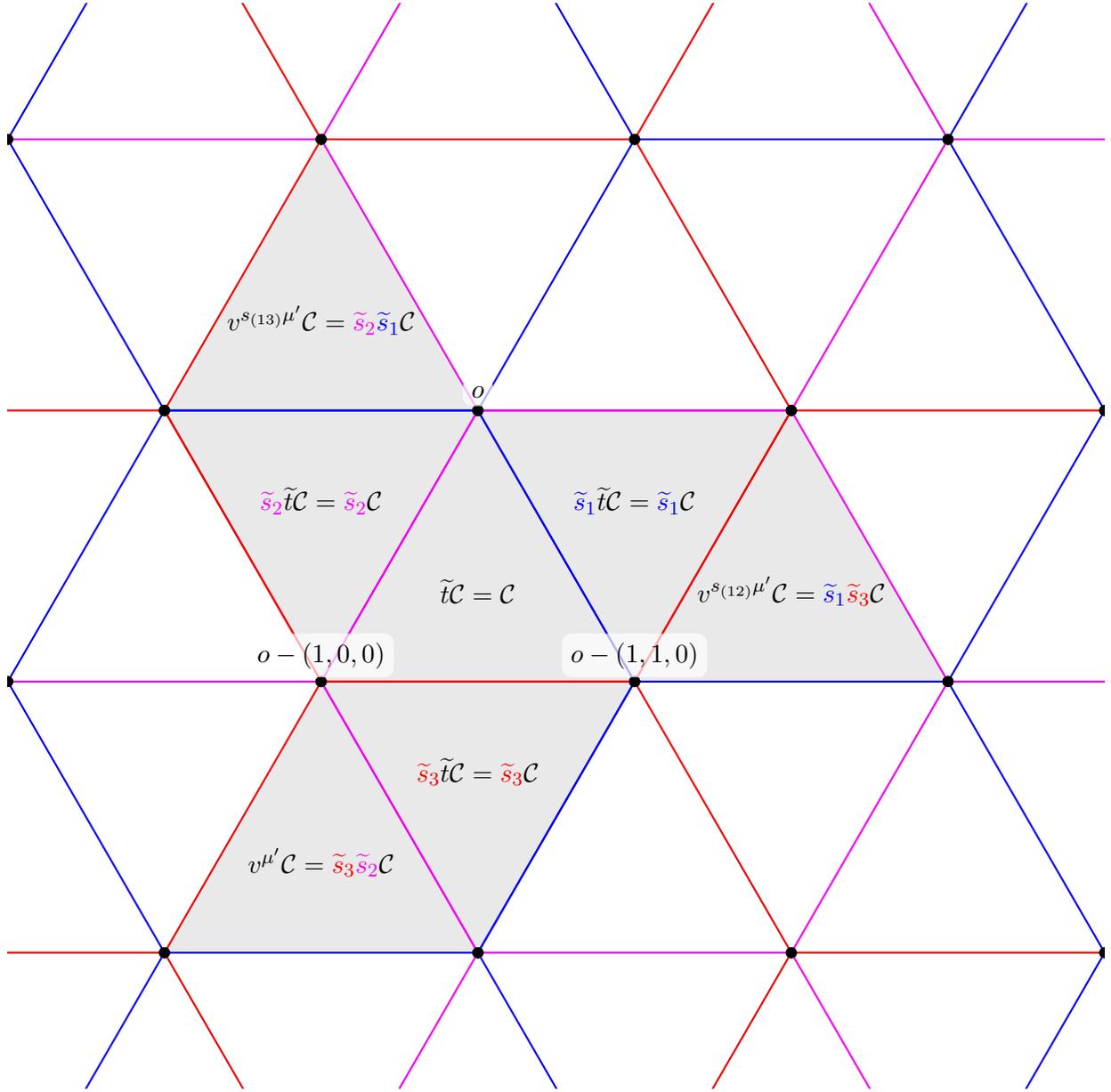
\begin{figure}[h!]
    \centering
    \resizebox{\textwidth}{!}{
    \begin{tikzpicture}[scale=2.3]

        % Define colors
        \definecolor{s1}{RGB}{0,0,240}      % Dark blue
        \definecolor{s2}{RGB}{240,0,240}  
        \definecolor{s3}{RGB}{240,0,0}       % Dark red

        % Height of equilateral triangle with sides 2
        \pgfmathsetmacro{\h}{1.732}

        % Define the shaded alcove
        \newcommand{\shadedAlcove}[6]{
            \coordinate (A) at (#1, #2);
            \coordinate (B) at (#1 + 1, #2 - \h);
            \coordinate (C) at (#1 - 1, #2 - \h);

            \fill[opacity=0.09] (A) -- (B) -- (C) -- cycle;

            \draw[thick,color=#3] (A) -- (B);
            \draw[thick,color=#4] (B) -- (C);
            \draw[thick,color=#5] (C) -- (A);

            \fill (A) circle (1pt);
            \fill (B) circle (1pt);
            \fill (C) circle (1pt);

            \node at (barycentric cs:A=1,B=1,C=1) {#6};
        }

        % Upside down shaded alcove
        \newcommand{\upsideShadedAlcove}[6]{
            \coordinate (A) at (#1, #2);
            \coordinate (B) at (#1 - 1, #2 + \h);
            \coordinate (C) at (#1 + 1, #2 + \h);

            \fill[opacity=0.09] (A) -- (B) -- (C) -- cycle;

            \draw[thick,color=#3] (A) -- (B);
            \draw[thick,color=#4] (B) -- (C);
            \draw[thick,color=#5] (C) -- (A);

            \fill (A) circle (1pt);
            \fill (B) circle (1pt);
            \fill (C) circle (1pt);

            \node at (barycentric cs:A=1,B=1,C=1) {#6};
        }

        % Define the unshaded alcove
        \newcommand{\alcove}[5]{
            \coordinate (A) at (#1, #2);
            \coordinate (B) at (#1 + 1, #2 - \h);
            \coordinate (C) at (#1 - 1, #2 - \h);

            \draw[thick,color=#3] (A) -- (B);
            \draw[thick,color=#4] (B) -- (C);
            \draw[thick,color=#5] (C) -- (A);

            \fill (A) circle (1pt);
            \fill (B) circle (1pt);
            \fill (C) circle (1pt);
        }

        % Define relevant region
        \clip (-3,-2.5*\h) rectangle (4,1.5*\h);

        % First row of alcoves
        \alcove{-2}{2*\h}{s3}{s2}{s1}; 
        \alcove{0}{2*\h}{s1}{s3}{s2};
        \alcove{2}{2*\h}{s2}{s1}{s3};
        \alcove{4}{2*\h}{s3}{s2}{s1};

        % Second row of alcoves
        \alcove{-3}{\h}{s1}{s3}{s2};
        \shadedAlcove{-1}{\h}{s2}{s1}{s3}{$v^{s_{(13)}\mu'}\mathcal{C} = \textcolor{s2}{\widetilde{s}_{2}}\textcolor{s1}{\widetilde{s}_{1}}\mathcal{C}$};
        \alcove{1}{\h}{s3}{s2}{s1};
        \alcove{3}{\h}{s1}{s3}{s2};

        % Third row of alcoves
        \alcove{-2}{0}{s3}{s2}{s1};
        \upsideShadedAlcove{-1}{-\h}{s3}{s1}{s2}{$\textcolor{s2}{\widetilde{s}_{2}}\widetilde{t}\mathcal{C}=\textcolor{s2}{\widetilde{s}_{2}}\mathcal{C}$};
        \shadedAlcove{0}{0}{s1}{s3}{s2}{$\widetilde{t}\mathcal{C} = \mathcal{C}$};
        \upsideShadedAlcove{1}{-\h}{s1}{s2}{s3}{$\textcolor{s1}{\widetilde{s}_{1}}\widetilde{t}\mathcal{C} = \textcolor{s1}{\widetilde{s}_{1}} \mathcal{C}$};
        \shadedAlcove{2}{0}{s2}{s1}{s3}{$v^{s_{(12)}\mu'}\mathcal{C} = \textcolor{s1}{\widetilde{s}_{1}}\textcolor{s3}{\widetilde{s}_{3}}\mathcal{C}$};
        \alcove{4}{0}{s3}{s2}{s1};

        % Fourth row of alcoves
        \alcove{-3}{-\h}{s1}{s3}{s2};
        \shadedAlcove{-1}{-\h}{s2}{s1}{s3}{$v^{\mu'}\mathcal{C} = \textcolor{s3}{\widetilde{s}_{3}}\textcolor{s2}{\widetilde{s}_{2}}\mathcal{C}$};
        \upsideShadedAlcove{0}{-2*\h}{s2}{s3}{s1}{$\textcolor{s3}{\widetilde{s}_{3}}\widetilde{t}\mathcal{C} = \textcolor{s3}{\widetilde{s}_{3}}\mathcal{C}$};
        \alcove{1}{-\h}{s3}{s2}{s1};
        \alcove{3}{-\h}{s1}{s3}{s2};

        % Fifth row of alcoves
        \alcove{-2}{-2*\h}{s3}{s2}{s1};
        \alcove{0}{-2*\h}{s1}{s3}{s2};
        \alcove{2}{-2*\h}{s2}{s1}{s3};
        \alcove{4}{-2*\h}{s3}{s2}{s1};

        % Label corners of base alcove
        \node[circle, fill=black, inner sep=1pt, label={[fill=white, fill opacity=0.7, text opacity=1, rounded corners] above:{$o$}}] at (0,0) {};
        \node[circle, fill=black, inner sep=1pt, label={[fill=white, fill opacity=0.7, text opacity=1, rounded corners] above:{$o - (1,0,0)$}}] at (-1,-\h) {};
        \node[circle, fill=black, inner sep=1pt, label={[fill=white, fill opacity=0.7, text opacity=1, rounded corners] above:{$o - (1,1,0)$}}] at (1,-\h) {};
    \end{tikzpicture}}
    \caption{
        A snapshot of \(\mathcal{A}\left( S, \mathbb{F}\llp v \rrp \right)\), where \(S \subset \GL_{3}\) is the diagonal torus. 
        Shaded alcoves correspond to elements of \(\operatorname{Adm}(\mu')\).
        The simple reflections \(\textcolor[RGB]{0,0,240}{\widetilde{s}_{1}}\), \(\textcolor[RGB]{240,0,240}{\widetilde{s}_{2}}\) and \(\textcolor[RGB]{240,0,0}{\widetilde{s}_{3}}\) correspond to the reflections across the walls of \(\mathcal{C}\) of the same colors, 
        and we can factorize any \(\widetilde{w}\) in the affine Weyl group by doing a ``gallery walk''.
    }
    \label{fig:admissible-set}
\end{figure}

\subsection{Tamely ramified unitary group}
\label{sec:tamely-ramified-unitary-group}
Let \(L = \mathbb{Q}_{p}\left( (-p)^{1/e} \right)\), where \(2 | e\) and \(e | p-1\), in which case \(\Gamma = I = \langle \gamma \rangle = \operatorname{Gal}(L / \mathbb{Q}_{p})\). 
Let \(G = U_{3}\) be the (quasi-split) tamely ramified unitary group over \(\mathbb{Q}_{p}\) with \(\mathbb{Q}_{p}\)-points 
\[
    G(\mathbb{Q}_{p}) = \left\lbrace g \in \GL_{3}\left( \mathbb{Q}_{p}\left( (-p)^{1/2} \right) \right) | J {^\gamma g}^{- t} J^{-1} = g \right\rbrace , 
\]
where \(J = \begin{pmatrix} 0 & 0 & 1 \\ 0 & 1 & 0 \\ 1 & 0 & 0 \end{pmatrix}\) and \(\gamma\) acts via its image in \(\operatorname{Gal}\left( \mathbb{Q}_{p}\left( (-p)^{1/2} \right) / \mathbb{Q}_{p} \right)\).
Let \(\widehat{G} = \GL_{3}\) be the dual group of \(G\), on which \(\gamma\) acts as the pinned automorphism \(g \mapsto J g^{-t}J^{-1}\). 

In this case, the group \(G^*\) over \(\mathcal{O}[v^{\pm 1}]\) is not \(\widehat{G}\), and it is a genuinely non-split group.
For example,
\[
    G^* (E\llp v \rrp) = \left\lbrace g \in \GL_{3}\left( E \llp v^{1/2} \rrp \right) | J {^\gamma g}^{- t} J^{-1} = g \right\rbrace . 
\]

\subsubsection{Galois type}
The action of \(\gamma\) on \(X_{*}(\widehat{T}) \cong \mathbb{Z}^3\), where \(\widehat{T} \subset \widehat{G}\) is the diagonal torus, is precisely 
\begin{equation}
    \label{eq:gamma-action-on-cocharacters-of-unitary-group}
    \gamma(a,b,c) = (-c, -b , -a) . 
\end{equation}
Therefore we can identify \(X_{*}(\widehat{T})^I = \lbrace (a, 0, -a) \in X_{*}(\widehat{T}) | a \in \mathbb{Z} \rbrace \cong \mathbb{Z}\). 

By \cref{res:types-via-tate-cohomology} we have a surjective map \(\mathbb{Z} \twoheadrightarrow H^1 \left( I, \widehat{G}(\mathcal{O}\llb u \rrb) \right)\)
which maps \(a \in \mathbb{Z}\) to the \(1\)-cocycle \(\tau : I \to \widehat{G}(\mathcal{O})\) with 
\[
    \tau(\gamma) = u^{- (a, 0, -a)} (^\gamma u^{(a,0,-a)}) = u^{- (a,0,-a) + \gamma(a,0,-a)} \omega(\gamma)^{\gamma(a,0,-a)} = \omega(\gamma)^{(a,0,-a)} . 
\]
In fact, \cref{res:types-via-tate-cohomology} implies that this surjective map factors through \(\widehat{H}^0(I,X_{*}(\widehat{T})) \cong \mathbb{Z} / (e/2)\mathbb{Z}\).

As a concrete example, suppose \(e = 6\) and \(a = - \frac{1}{6}\). Then \(\tau(\gamma) = \omega(\gamma)^{(-1/6,0,1/6)}\). The corresponding point in the apartment \(\widetilde{\mathcal{A}}\left( U_{3}, \mathbb{F}\llp v \rrp \right)\)
is \(x = u^{(-1/6,0,1/6)} \cdot o = o + (1/6,0,-1/6)\). This is illustrated in \cref{fig:unitary-group-apartment}. 

\begin{figure}[h!]
    \centering
    \begin{tikzpicture}[scale = 5]
        % Set the number of segments and the prime
        \pgfmathsetmacro{\numSegments}{6}
        \pgfmathsetmacro{\p}{13}
        \pgfmathsetmacro{\kbound}{floor(\p/3)}

        % Function to draw the divided triangle
        \newcommand{\dividedTriangle}{
            % Define the coordinates of the large triangle
            \coordinate (A) at (0, 0);
            \coordinate (B) at (1, 0);
            \coordinate (C) at (0.5, 0.866); % height of equilateral triangle with side length 1

            % Divide each side into equal segments and draw the smaller triangles
            \foreach \i in {0,...,\numSegments} {
                \coordinate (P) at ($(A)!\i/\numSegments!(B)$);
                \coordinate (Q) at ($(A)!\i/\numSegments!(C)$);
                \coordinate (R) at ($(B)!\i/\numSegments!(C)$);
                \coordinate (S) at ($(C)!\i/\numSegments!(B)$);
                
                % Draw lines parallel to the sides of the large triangle
                \draw[line width=0.05pt,color=blue] (P) -- (Q);
                \draw[line width=0.05pt,color=blue] (P) -- (S);
                \draw[line width=0.05pt,color=blue] (Q) -- (R);
            }
        }

        % Define the clipping region to create a square "photograph"
        \clip (-0.2,-1.039) rectangle (1.2,0.346);

        % Draw triangles which are shifts of divided triangle
        \foreach \x / \y in {-0.5/0.866,0.5/0.866,-1/0,0/0,1/0,-0.5/-0.866,0.5/-0.866,0/-1.732} {
            \begin{scope}[shift={(\x,\y)}]
                \dividedTriangle
            \end{scope}
        }

        % Draw traingles which are shifts of upside down divided triangle
        \foreach \x / \y in {0/-1.732,-0.5/-0.866,0.5/-0.866,-1/0,0/0,1/0,-0.5/0.866,0.5/0.866} {
            \begin{scope}[yscale=-1,shift={(\x,\y)}]
                \dividedTriangle
            \end{scope}
        }

        \fill[fill=white,fill opacity=0.8] (-2,-2) rectangle (2,2);

        \draw[thick] (0.5,-2) -- (0.5,2);

        % Label vertices
        \node[circle, fill=black, inner sep=1.5pt, label={[fill=white, fill opacity=0.7, text opacity=1, rounded corners] right:{$o$}}] at (0.5,-0.866) {};
        \node[circle, fill=black, inner sep=1.5pt, label={[fill=white, fill opacity=0.7, text opacity=1, rounded corners] right:{$o + (1,0,-1)$}}] at (0.5,0.866) {};
        \node[draw, circle, fill=black, inner sep=1.5pt, label={[fill=white, fill opacity=0.7, text opacity=1, rounded corners] right:{$o + (\frac{1}{2},0,-\frac{1}{2})$}}] at (0.5,0) {};
        \node[draw, circle, fill=white, inner sep=1.5pt, label={[fill=white, fill opacity=0.7, text opacity=1, rounded corners] right:{$o + (\frac{1}{4},0,-\frac{1}{4})$}}] at (0.5,-0.433) {};

        % Label small vertices
        \foreach \y in {-0.577,-0.289,0.289,0.577} {
            \node[circle,fill=black,inner sep=1pt] at (0.5,\y) {};
        }
        \node[circle, fill=black, inner sep=1pt, label={[fill=white, fill opacity=0.7, text opacity=1, rounded corners] right:{$x$}}] at (0.5,-0.577) {};
    \end{tikzpicture}
    \caption{
        A snapshot of the apartment \(\widetilde{\mathcal{A}}\left( U_{3}, \mathbb{F}\llp v \rrp \right) = \widetilde{\mathcal{A}}\left( \GL_{3}, \mathbb{F}\llp u \rrp \right)^I\) (the solid black line) viewed inside \(\widetilde{\mathcal{A}}\left( \GL_{3}, \mathbb{F}\llp u \rrp \right)\) (the faint blue lines), for \(e = 6\). 
        A point classifying a Galois type corresponds to one of the black dots, for example \(x = o + (1/6,0,-1/6)\).
        The three large dots \(o\), \(o + \left( 1/4, 0 -1/4 \right)\) and \(o + \left(1/2,0,-1/2\right)\) are the vertices of \(\mathcal{A}\left( U_{3}, \mathbb{F}\llp v \rrp \right)\) that can be seen. 
        There are two full alcoves visible.
    }
    \label{fig:unitary-group-apartment}
\end{figure}
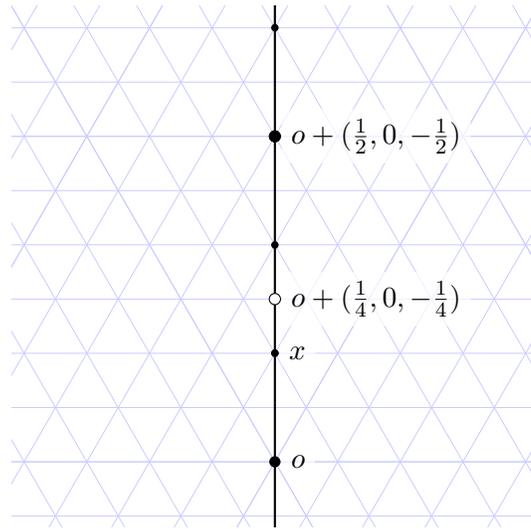

\subsubsection{Bruhat--Tits group scheme} 
Via \(X_{*}(\widehat{T}) \cong \mathbb{Z}^3\), we can identify \(\Phi^\vee(\widehat{G},\widehat{T})\mathbb{Z} \cong \lbrace (a,b,c) \in \mathbb{Z}^3 | a + b + c = 0 \rbrace\). 
It follows that 
\begin{align*}
    \pi_{1}(G^*) = \pi_{1}(\widehat{G}) = X_{*}(\widehat{T}) / \Phi^\vee(\widehat{G},\widehat{T})\mathbb{Z}
    & \isoto \mathbb{Z} \\ 
    (a,b,c) & \mapsto a + b + c . 
\end{align*}
Since \(\gamma(a,b,c) = (-c, -b, -a)\) by \eqref{eq:gamma-action-on-cocharacters-of-unitary-group}, the corresponding action of \(\gamma\) on \(\mathbb{Z}\) is multiplication by \(-1\). 
It follows that 
\[
    \pi_{1}(G^*)_{I} \cong \mathbb{Z} / 2 . 
\]
As such, \(\mathcal{G}_{x}\) does not automatically have connected fibers, since \cref{res:connected-fibers} doesn't apply directly.
In fact, \(\mathcal{G}_{x}\) does \emph{not} have connected fibers for any \(0\)-generic \(x\). 
This is a consequence of the short exact sequence \eqref{eq:connected-fibers-SES} and the fact that the stabilizer in \(G^*(\mathbb{F}\llp v \rrp)\)
of an alcove \(\mathcal{C} \subset \mathcal{A}(T^*,\mathbb{F}\llp v \rrp)\) must fix every point of \(\mathcal{C}\) (which we can see by noting that \(\mathcal{C}\) has two bounding vertices of different specialities, see \cite[Example 7.11.6, Figure 7.11.1]{kaletha-prasad}).
\begin{remark}
    If we had considered a slightly different example in which \(G^*_{\mathbb{F}\llp v \rrp}\) was a \emph{special} unitary group instead, then the fibers of \(\mathcal{G}_{x}\) would be connected for any \(x\).
\end{remark}

\printbibliography

\appendix 

\section{Invariant pushforward and \((\Gamma,H)\)-torsors}
\label{sec:equivariant-torsors}
In this section, \(\Gamma\) denotes any finite group.
And \(H\) denotes a group object with a \(\Gamma\)-action (where we will consider different categories). 
Roughly speaking, a \((\Gamma,H)\) is an \(H\)-torsor with a \(\Gamma\)-action which is semilinear over the base and makes the action map equivariant. 

The notion of \((\Gamma,H)\)-torsors goes back to Grothendieck's Bourbaki talk \cite{grothendieckMemoireWeilGeneralisation1956} in the context of analytic complex spaces, 
where he considers the problem of descending \((\Gamma,H)\)-torsors along a quotient map \(X \to \Gamma \backslash X\).
If \(\Gamma\) acts freely on \(X\), then \((\Gamma,H)\)-torsors over \(X\) correspond to \(H\)-torsors on \(X\), but this is not true in general. 
As observed in \cite{balajiModuliParahoricMathcal2014}, if \(X \to Y = \Gamma \backslash X\) is a ramified covering of Riemann surfaces, then a \((\Gamma,H)\)-torsor over \(X\)
corresponds to a \(\mathcal{H}\)-torsor over \(Y\) where \(\mathcal{H}\) is a Bruhat--Tits group over \(Y\) which is equal to \(H\) over the part where \(X \to Y\) is unramified, 
and admits parahoric structure over the branching points of \(Y\). It is important to note that which parahoric structure on \(\mathcal{H}\) we take in this correspondence 
is dictated by local invariants of \((\Gamma,H)\)-torsors called the type.

These ideas were developed further in \cite{damioliniLocalTypesGamma2023} and \cite{pappas-rapoport-tamely-ramified-bundles} for tamely ramified covers of curves over algebraically 
closed fields (of characteristic zero in the former case).
In particular, \cite{pappas-rapoport-tamely-ramified-bundles} shows that if \(X\) is a curve over an algebraically closed field and \(\mathcal{G}\) is a Bruhat--Tits group scheme over \(X\), 
then (under suitable additional assumptions) there exists a tamely ramified covering \(X' \to X\) such that \(\mathcal{G}' = \mathcal{G} \times_{X} X'\) is a connected reductive 
group scheme. 

In this section, we will explain a generalization of \cite[Theorem 4.1.6]{balajiModuliParahoricMathcal2014} which holds in any topos (and in particular in the big étale or fppf topos over a scheme), 
which we specialize to the algebro geometric setting in \cref{sec:equivariant-torsors-in-algebraic-geometry}. 
We will apply this in \cref{sec:specialized-equivariant-torsors} to describe \((\Gamma,\widehat{G})\)-torsors 
over \(\mathfrak{S}_{L,R}\) in terms of torsors for a certain group \(\mathcal{G}\) over \(\mathfrak{S}_{R}\), 
and in \cref{sec:bruhat-tits-group-schemes} we will identify the group scheme \(\mathcal{G}\) with a Bruhat--Tits group scheme in the sense of Pappas--Zhu \cite[Theorem 4.1]{pappas-zhu}.

\begin{remark}
    \label{rem:balaji-seshadri-error}
    There is an error in the statment of \cite[Theorem 4.1.6]{balajiModuliParahoricMathcal2014}, since it does not account for the local type, as pointed out in Damiolini's PhD thesis \cite[Appendix A]{damioliniConformalBlocksAttached2017}.\footnote{We thank Vikraman Balaji for pointing us towards Damiolini's thesis.}
    However, note that this does not affect other results of \cite{balajiModuliParahoricMathcal2014}, which implicitly use the corrected version of the theorem. 
\end{remark}

\subsection{\((\Gamma,H)\)-torsors in general}
\label{sec:equivariant-torsors-general-case}
Let \(\mathcal{C}\) be a topos and \(\Gamma\) be a finite group.
If \(X \in \mathcal{C}\) then a \(\Gamma\)-action on \(X\) is the same as a homomorphism \(\Gamma \to \operatorname{Aut}_{\mathcal{C}}(X)\).
Assume that \(\widetilde{A} \in \mathcal{C}\) is an object with \(\Gamma\)-action and that the action of \(\theta \in \Gamma\) on \(\widetilde{A}\) is given by \(a_{\theta}\). 

A \(\Gamma\)-equivariant object over \(\widetilde{A}\) is an object \(X \to \widetilde{A}\) which is equipped with a \(\Gamma\)-action making this morphism \(\Gamma\)-equivariant.
In this situation, we will also say the the \(\Gamma\)-action on \(X\) is \emph{semilinear} over \(\widetilde{A}\). 
The data of a \(\Gamma\)-equivariant structure on \(X\) is equivalent to any of the following: 
\begin{enumerate}
    \item A collection of isomorphisms \(X \cong (a_{\theta})^* X\) over \(\widetilde{A}\) satisfying a cocycle condition (by the adjunction \((a_{\theta})_{!} \dashv (a_{\theta})^*\)).
    \item A collection of isomorphisms \((a_{\theta})_{*} X \cong X\) over \(\widetilde{A}\) satisfying a cocycle condition (by the adjunction \((a_{\theta})^* \dashv (a_{\theta})_{*} \) applied to the inverses of the isomorphisms from (1)).
    \item A collection of isomorphisms \((a_{\theta}^{-1})^* X \cong X\) over \(\widetilde{A}\) satisfying a cocycle condition (since \((a_{\theta})_{*} \cong (a_{\theta}^{-1})^*\)).
\end{enumerate}
In algebraic geometry it is customary to use (3), taking into account the fact that if \(\theta \in \Gamma\) acts on a ring \(A\) by \(f_{\theta} : A \to A\), 
then the corresponding left action on \(\operatorname{Spec}A\) is given by \(a_{\theta} = \operatorname{Spec}f_{\theta}^{-1} : \operatorname{Spec}A \to \operatorname{Spec}A\). 

Note that if \(Y \to \widetilde{A}\) is another \(\Gamma\)-equivariant object over \(\widetilde{A}\), then \(X \times_{\widetilde{A}}Y \to Y\) is a \(\Gamma\)-equivariant object over \(Y\). 

Let \(H \to \widetilde{A}\) be a \(\Gamma\)-equivariant group, that is, a \(\Gamma\)-equivariant object over \(\widetilde{A}\) which is also a group object over \(\widetilde{A}\), and for which the multiplication, identity section, and inversion maps are \(\Gamma\)-equivariant. 
A \((\Gamma,H)\)-torsor over \(\widetilde{A}\) is an \(H\)-torsor \(P \to \widetilde{A}\) which is equipped with a \(\Gamma\)-action which is semilinear over \(\widetilde{A}\) and for which the action map \(P \times_{\widetilde{A}} H \to P\) is \(\Gamma\)-equivariant. 
A morphism of \((\Gamma,H)\) torsors \(P \to Q\) is a morphism of objects over \(\widetilde{A}\) which is both \(\Gamma\)-equivariant and a morphism of \(H\)-torsors (i. e. compatiable with the \(H\)-action).
It will also be important to consider \((\Gamma,H)\)-torsors over \(Y\), where \(Y \to \widetilde{A}\) is a \(\Gamma\)-equivariant object, and they are defined in the obvious way. 

\begin{remark}
    By default, we consider torsors as being \emph{right} torsors, with the group acting on the right.
    We will also consider left torsors as well as bitorsors, but will specify this whenever we do. 
\end{remark}

\subsubsection{Group cohomology and \((\Gamma,H)\)-torsors}
\label{sec:group-cohomology-and-equivariant-torsors}
\((\Gamma,H)\)-torsors whose underlying \(H\)-torsor is trivial are classified by group cohomology \(H^1\left( \Gamma, H(\widetilde{A}) \right)\), as we now explain.
If \(H(\widetilde{A})\) denotes the group of sections \(\widetilde{A} \to H\) over \(\widetilde{A}\), then the group \(\Gamma\) acts on \(H(\widetilde{A})\) by ``conjugation''. 
That is, given such a section \(h \in H(\widetilde{A})\) and \(\theta \in \Gamma\), we define \({^\theta h} \in H(\widetilde{A})\) as the map fitting into the commutative diagram 
\[
\begin{tikzcd}
    \widetilde{A} \arrow[r,"h"] \arrow[d,"\theta"] & H \arrow[d,"\theta"] \\
    \widetilde{A} \arrow[r,"{^\theta h}"] & H , 
\end{tikzcd}
\]
where both maps labeled \(\theta\) are given by the action map describing \(\theta\) (called \(a_{\theta}\) above).
Given a \(1\)-cocycle \(\tau : \Gamma \to H(\widetilde{A})\), we define a \((\Gamma,H)\)-torsor \({_{\tau}\mathcal{E}^0 }\) as the \((\Gamma,H)\)-torsor whose underlying \(H\)-torsor is the trivial \(H\)-torsor \(\mathcal{E}^0\), and on which \(\theta \in \Gamma\) acts as the composition 
\[
    \mathcal{E}^0 \stackrel{\theta}{\longrightarrow} \mathcal{E}^0 \stackrel{L_{\tau(\theta)}}{\longrightarrow} \mathcal{E}^0 , 
\]
where the map labeled \(\theta\) gives the action of \(\theta\) on \(H = \mathcal{E}^0\), and \(L_{\tau(\theta)}\) denotes left translation by \(\tau(\theta)\). 
\begin{lemma}
    \label{res:equivariant-torsors-and-group-cohomology}
    The map 
    \begin{align*}
        Z^1 \left( \Gamma, H(\widetilde{A}) \right) & \isoto \lbrace (\Gamma,H)\text{-torsors over }\widetilde{A}\text{ whose underlying }H\text{-torsor is }\mathcal{E}^0 \rbrace \\ 
        \tau & \mapsto {_{\tau}\mathcal{E}^0}
    \end{align*}
    is an equivalence of groupoids. Consequently, \(H^1 \left( \Gamma, H(\widetilde{A}) \right)\) is in bijection with isomorphism classes of 
    \((\Gamma,H)\)-torsors on \(\widetilde{A}\) whose underlying \(H\)-torsor is trivializable.
\end{lemma}
\begin{proof}
    The proof is straightforward. The key point is that associativity of the \(\Gamma\)-action on \({_{\tau}\mathcal{E}^0}\) corresponds to the 
    \(1\)-cocycle condition for \(\tau\). 
\end{proof}

\subsubsection{Invariant pushforward and \((\Gamma,H)\)-torsor of trivial type}
Let \(p : \widetilde{A} \to A\) be a \(\Gamma\)-equivariant morphism where \(\Gamma\) acts trivially on \(A\). 
Given a \(\Gamma\)-equivariant object \(X \to \widetilde{A}\) we define the \emph{invariant pushforward} \(p_{*}^\Gamma(X)\) as follows. 
First, \(p_{*}X \to A\) is defined in the usual way.
For any object \(V \to A\) we have
\[
 \operatorname{Hom}_{A}\left( V, p_{*}(X) \right) \cong \operatorname{Hom}_{\widetilde{A}}\left( V \times_{A} \widetilde{A}, X \right) , 
\]
and \(\Gamma\) acts on the right hand side by conjugation (using the action on \(V \times_{A} \widetilde{A}\) where \(\theta \in \Gamma\) acts as \(1 \times a_{\theta}\)).
By Yoneda's lemma, this gives rise to a \(\Gamma\)-action on \(p_{*} X\) over \(A\), and we define \(p_{*}^\Gamma X = (p_{*} X)^\Gamma\) as the obvious limit.

Note that \(p_{*}^\Gamma H\) is a group since \(p_{*}^\Gamma\) preserves limits and in particular products. 
Similarly, if \(P\) is a \((\Gamma,H)\)-torsor over \(\widetilde{A}\), then the action map \(P \times_{\widetilde{A}} H \to P\) gives rise to an action map 
\(p_{*}^\Gamma P \times_{A} p_{*}^\Gamma H \to p_{*}^\Gamma P\). 
Therefore, \(p_{*}^\Gamma P\) is a \(p_{*}^\Gamma H\)-pseudotorsor. 
But it is not always a \(p_{*}^\Gamma H\)-torsor, because the map \(p_{*}^\Gamma P \to A\) is not always surjective.
This failure is accounted for by the type of the torsor. 

\begin{definition}
    \label{def:torsor-type}
    Two \((\Gamma,H)\)-torsors \(P, Q \) over \(\widetilde{A}\) are said to have the same \emph{type} if there exists a covering \(\lbrace U_{i} \to A \rbrace_{i \in \mathcal{I}}\)
    such that for every \(i \in \mathcal{I}\) the base change of \(P\) and \(Q\) to \(U_{i} \times_{A} \widetilde{A}\) are isomorphic \((\Gamma,H)\)-torsors. 
    A \((\Gamma,H)\)-torsor \(P \to \widetilde{A}\) is said to have \emph{trivial type} if there \(P\) has the same type as the trivial \((\Gamma,H)\)-torsor.
    If \([\tau] \in H^1 \left( \Gamma , H(\widetilde{A})\right)\), then a \((\Gamma,H)\)-torsor \(P \to \widetilde{A}\) is said to have type \([\tau]\) (or just \(\tau\)) if 
    it has the same type as the \((\Gamma,H)\)-torsor corresponding to \([\tau]\) via \cref{res:equivariant-torsors-and-group-cohomology}. 
\end{definition}

\begin{proposition}
    \label{res:fundamental-result-about-equivariant-torsors}
    We have an equivalence of groupoids 
    \begin{align*}
        p_{*}^\Gamma : \lbrace (\Gamma,H)\text{-torsors over }\widetilde{A}\text{ of trivial type} \rbrace & \isoto \lbrace p_{*}^\Gamma H \text{-torsors over }A \rbrace \\
        P & \mapsto p_{*}^\Gamma P . 
    \end{align*}
\end{proposition}
\begin{proof}
    It follows from unraveling the definitions that if \(P \to \widetilde{A}\) is a \((\Gamma,H)\)-torsor of trivial type, then \(p_{*}^\Gamma P \to A\) is an epimorphism. 
    Hence the functor is well-defined. 

    Define an inverse \(F\) to \(p_{*}^\Gamma\) as follows. Given a \(p_{*}^\Gamma H\)-torsor \(P \to A\), 
    we obtain a \(p^* p_{*}^\Gamma H\)-torsor \(p^* P\) over \(\widetilde{A}\). 
    Via the natural map \(p^* p_{*}^\Gamma H \to H\) we can let \(F(P) := p^* P \times^{p^* p_{*}^\Gamma H} H\) be the associated \(H\)-torsor.\footnote{
        If \(P'\) is a right \(H'\)-torsor, \(Q'\) is a \((H',H'')\)-bitorsor, then the associated bundle \(P' \times^{H'} Q' := P' \times_{\widetilde{A}} Q' / (ph,q) \sim (p,hq)\)
        is a right \(H''\)-torsor. More precisely, \(P' \times^{H'} Q'\) is defined as the reflexive coequalizer \(P' \times H' \times Q' \rightrightarrows P' \times Q' \to P' \times^{H'} Q'\). 
    }
    Then one checks as in \cite[A.0.3]{damioliniConformalBlocksAttached2017} that there are natural maps \(F p_{*}^\Gamma (P) \to P\) and \(P \to p_{*}^\Gamma F(P)\)
    (in fact \(F\) is left adjoint to \(p_{*}^\Gamma\) on the level of pseudotorsors). 
    Since all morphisms of torsors (and consequently, \((\Gamma,H)\)-torsors) are isomorphisms, it follows that \(F\) is inverse to \(p_{*}^\Gamma\). 
\end{proof}
 
\begin{remark}
    One can alternatively see that above equivalence as follows. 
    Let \(BH\) denote the classifying stack of \(H\)-torsors over \(\widetilde{A}\). 
    Then \(p_{*}^\Gamma BH\) is a stack over \(A\) which for an object \(V \to A\) classifies \((\Gamma,H)\)-torsors over \(V \times_{A} \widetilde{A}\). 
    The trivial \((\Gamma,H)\)-torsor viewed as an object in this stack has automorphisms \(p_{*}^\Gamma H\).
    If \(B(p_{*}^\Gamma H)\) denotes the classifying stack for \(p_{*}^\Gamma H\)-torsors, 
    we can therefore view \(B(p_{*}^\Gamma H) \subset p_{*}^\Gamma B(\Gamma,H)\) as the gerbe (i. e. 1-truncated substack) containing the trivial \((\Gamma,H)\)-torsor, 
    and one can check that being contained in this gerbe corresponds to having trivial type. 
    By taking \(A\)-points, one recovers \cref{res:fundamental-result-about-equivariant-torsors}. 
\end{remark}

\subsection{Twisting}
\label{sec:twisting}
We remain in the situation of \cref{sec:equivariant-torsors-general-case}, so \(H \to \widetilde{A}\) is a \(\Gamma\)-equivariant group. 
Let \(\mathcal{P} \to \widetilde{A}\) be a \((\Gamma,H)\)-torsor. 
Then we can define a \(\Gamma\)-equivariant group 
\[
    {_{\mathcal{P}}H} := \operatorname{Aut}_{H}(\mathcal{P})
\] 
over \(\widetilde{A}\), where the \(\Gamma\)-action is by conjugation.
\begin{remark}
    \label{rem:classical-twisting}
    Let \(\tau : \Gamma \to H(\widetilde{A})\) be a 1-cocycle, and let \(\mathcal{P} = _{\tau}\mathcal{E}^0\) be the \((\Gamma,H)\)-torsor whose underlying \(H\)-torsor is trivial, 
    and with the \(\Gamma\)-action twisted by \(\tau\) (as in \cref{sec:group-cohomology-and-equivariant-torsors}). Then we define  \({_{\tau}}H := {_{\mathcal{P}}H}\) as a \(\Gamma\)-group.
    This is compatible with \cite[Section I.5.3]{galois-cohomology}, in that we have the following formulas on \(\widetilde{A}\)-points: 
    \begin{align*}
        \theta \cdot p & = \tau(\theta)(^\theta p), & p \in {_{\tau}\mathcal{E}^0}(\widetilde{A}) , \\ 
        \theta \cdot h & = \tau(\theta)(^\theta h) \tau(\theta)^{-1} , & h \in {_{\tau}H}(\widetilde{A}) ,
    \end{align*}
    where \(\theta \cdot p\) and \(\theta \cdot h\) denotes the ``new'' \(\tau\)-twisted action and \({^\theta p}\) and \({^\theta h}\) denotes the ``old'' action. 
\end{remark}

Note that \({_{\mathcal{P}}H}\) acts on \(\mathcal{P}\) on the left, and this makes \(\mathcal{P}\) into a \((_{\mathcal{P}}H,H)\)-bitorsor (since locally, \(\mathcal{P} \simeq \mathcal{E}^0\) and \({_{\mathcal{P}}H} \simeq H\)). 
We have an equivalence of groupoids 
\begin{align}
    \label{eq:twisted-groupoid-of-equivariant-torsors}
    \left\lbrace (\Gamma,{_{\mathcal{P}}H})\text{-torsors over }\widetilde{A} \right\rbrace & \isoto \left\lbrace (\Gamma,H)\text{-torsors over }\widetilde{A} \right\rbrace \\
    \mathcal{Q} & \mapsto \mathcal{Q} \times^{_{\mathcal{P}}H} \mathcal{P} , \nonumber 
\end{align}
where \(\mathcal{Q} \times^{_{\mathcal{P}}H} \mathcal{P}\) is the associated bundle construction.
The inverse is given by \(\mathcal{R} \mapsto \mathcal{R} \times^H \mathcal{P}^{-1}\), where \(\mathcal{P}^{-1}\) is the \((H,_{\mathcal{P}}H)\)-bitorsor 
whose underlying object is \(\mathcal{P}\) on which \(H\) and \({_{\mathcal{P}}H}\) act through the inverse on the opposite side. 
\begin{remark}
    One can view \eqref{eq:twisted-groupoid-of-equivariant-torsors} as a consequence of the fact that the classifying stack of torsors is a ``delooping'' of the group, as explained for example in \cite{NSS}. 
\end{remark}

Observe that a \((\Gamma,_{\mathcal{P}}H)\)-torsor \(\mathcal{Q}\) has trivial type if and only if the associated \((\Gamma,H)\)-torsor \({\mathcal{Q} \times^{_\mathcal{P}H}\mathcal{P}}\) has the same type as \(\mathcal{P}\). 
Combining \cref{res:fundamental-result-about-equivariant-torsors} and \eqref{eq:twisted-groupoid-of-equivariant-torsors}, we obtain the following. 
\begin{corollary}
    \label{res:twisted-fundamental-result-about-equivariant-torsors}
    We have equivalences of groupoids 
    \begin{align*}
        & \lbrace (\Gamma,H)\text{-torsors over }\widetilde{A}\text{ of type } \mathcal{P} \rbrace \\
        \isoto & \lbrace (\Gamma,_{\mathcal{P}}H)\text{-torsors over }\widetilde{A}\text{ of trivial type}\rbrace \\ 
        \isoto & \lbrace p_{*}^\Gamma (_{\mathcal{P}}H) \text{-torsors over }A \rbrace . 
    \end{align*}
\end{corollary}

\subsection{\((\Gamma,H)\)-torsors in algebraic geometry}
\label{sec:equivariant-torsors-in-algebraic-geometry}
We now specialize the above results to the setting of algebraic geometry. 
This will recover some results of \cite{balajiModuliParahoricMathcal2014} and \cite{damioliniConformalBlocksAttached2017}, at least when over a base field.

By embedding schemes into the big étale or fppf topos via Yoneda's lemma, we define \(\Gamma\)-equivariant schemes etc. as the corresponding notion for sheaves.
Because our group schemes will be smooth affine, torsors for such a group are representable by schemes, hence sheaves in the fpqc topology, yet étale locally trivializable.
Therefore it makes little difference which big topology we choose to work with.

Let \(A\) be a scheme, \(\widetilde{A}\) a scheme with \(\Gamma\)-action and \(p : \widetilde{A} \to A\) a morphism of schemes which is \(\Gamma\)-equivariant for the 
trivial \(\Gamma\)-action on \(A\). 
\begin{proposition}
    \label{res:smooth-invariant-pushforward-in-general}
    Assume that \(A\) is locally noetherian and \(p : \widetilde{A} \to A\) is finite flat. 
    Let \(X \to \widetilde{A}\) be a \(\Gamma\)-equivariant and quasi-projective morphism of schemes. 
    \begin{enumerate}
        \item \(p_{*}^\Gamma X\) is representable by a scheme over \(A\).
        \item If \(X \to \widetilde{A}\) is affine, then so is \(p_{*}^\Gamma X \to A\). 
        \item If \(X \to \widetilde{A}\) is smooth and \(\# \Gamma\) is invertible over \(A\), then \(p_{*}^\Gamma X \to A\) is smooth.
    \end{enumerate}
\end{proposition}
\begin{proof}
    Since all statements are Zariski local on \(A\) and \(\pi\) is a finite (hence affine) morphism, we may assume that \(A\) and \(\widetilde{A}\) are the spectra of noetherian rings. 
    Representability of the Weil restriction \(p_{*}X\) is due to Grothendieck \cite[Exposé 221, Section 4.c]{grothendieckFondementsGeometrieAlgebrique1956}; see also \cite[Section 7.6]{boschNeronModels1990}. 
    As shown in \cite[Proposition 3.1]{edixhovenNeronModelsTame1992} or \cite[Proposition A.8.10]{CGP}, \(p_{*}^\Gamma X = (p_{*}X)^\Gamma\) is then represented by a closed subscheme of \(p_{*}X\). 
    If \(X \to \widetilde{A}\) is affine/smooth/étale/an open immersion, then so is \(p_{*}X \to A\) by \cite[Proposition A.5.2]{CGP}, implying (2) since closed subschemes of affine schemes are affine.
    The smoothness part of (3) is shown in \cite[Proposition 3.4]{edixhovenNeronModelsTame1992}.
\end{proof}
\begin{remark}
    In \cite[Proposition 3.4]{edixhovenNeronModelsTame1992} only the weaker assumption that \(\# \Gamma\) is invertible over \(p_{*}X\) is needed, so our assumption in (3) is stronger than necessary.
    From the point of view of \cite[Proposition A.8.10]{CGP}, the key is that the constant group determined by \(\Gamma\) is linearly reductive when \(\# \Gamma\) is invertible. 
\end{remark}
\begin{remark}
    Instead of requiring that \(A\) is locally noetherian and \(p : \widetilde{A} \to A\) is finite flat, we could have simply required that \(p : \widetilde{A} \to A\) is finite locally free. 
    For a comparison of these conditions, see \cite[\href{https://stacks.math.columbia.edu/tag/02KB}{Lemma 02KB}]{stacks-project}. 
    The latter condition is the one which is used in \cite[Section 7.6]{boschNeronModels1990}. 
\end{remark}
Let \(H \to \widetilde{A}\) be a \(\Gamma\)-equivariant smooth affine group scheme. 
If \(A\) is locally noetherian and \(p : \widetilde{A} \to A\) is finite flat, then in light of \cref{res:smooth-invariant-pushforward-in-general} the invariant pushforward 
\(p_{*}^\Gamma\) is well-behaved on scheme-theoretic \((\Gamma,H)\)-torsors.
In particular, \cref{res:fundamental-result-about-equivariant-torsors} and \cref{res:twisted-fundamental-result-about-equivariant-torsors} apply to this situation.
The assumptions on \(H\) ensure that it doesn't matter whether we use the étale, fppf, or fpqc topology in the definition of type, and it is even a pointwise local invariant, as the following result shows. 
\begin{proposition}
    \label{res:equivalent-definitions-of-type-in-general}
    Assume that \(A\) is locally noetherian and \(p : \widetilde{A} \to A\) is finite flat, and let \(H \to \widetilde{A}\) be a \(\Gamma\)-equivariant smooth affine group scheme.
    Let \(P\) and \(Q\) be \((\Gamma,H)\)-torsors on \(\widetilde{A}\). 
    Then the following are equivalent: 
    \begin{enumerate}
        \item There exists an fpqc covering \(\lbrace U_{i} \to \widetilde{A} \rbrace_{i \in \mathcal{I}}\) and an isomorphism of \((\Gamma,H)\)-torsors \(P_{U_{i} \times_{A} \widetilde{A}} \isoto Q_{U_{i} \times_{A} \widetilde{A}}\) for every \(i \in \mathcal{I}\), 
        \item There exists an étale covering \(\lbrace U_{i} \to \widetilde{A} \rbrace_{i \in \mathcal{I}}\) and an isomorphism of \((\Gamma,H)\)-torsors \(P_{U_{i} \times_{A} \widetilde{A}} \isoto Q_{U_{i} \times_{A} \widetilde{A}}\) for every \(i \in \mathcal{I}\), 
        \item For every geometric point \(\overline{x} : \operatorname{Spec} \overline{\kappa} \to A\) there exists an isomorphism of \((\Gamma,H)\)-torsors \(P_{\overline{\kappa} \times_{\overline{x}, A} \widetilde{A}} \isoto P_{\overline{\kappa} \times_{\overline{x}, A} \widetilde{A}}\).
        \item For every geometric point \(\overline{x} : \operatorname{Spec} \overline{\kappa} \to A\) and every point \(y \in \overline{\kappa} \times_{\overline{x}, A} \widetilde{A}\) there exists an isomorphism of \((\Gamma_{y}, H)\)-torsors \(P_{y} \isoto Q_{y}\), where \(\Gamma_{y} \leq \Gamma\) is the stabilizer of \(y\). 
    \end{enumerate}
\end{proposition}
\begin{proof}
    Consider the sheaf \(\underline{\operatorname{Iso}}_{H}\left( P,Q \right)\) of isomorphisms of \(H\)-torsors \(P \isoto Q\).
    By fpqc descent, this is representable by a smooth affine scheme over \(\widetilde{A}\) and it has a \(\Gamma\)-conjugation action.
    Hence \(p_{*}^\Gamma \underline{\operatorname{Iso}}_{H}\left( P,Q \right)\) is a smooth affine scheme over \(A\) by \cref{res:smooth-invariant-pushforward-in-general}. 
    Note that if \(V \to A\) is a scheme, then 
    \[
        p_{*}^\Gamma \underline{\operatorname{Iso}}_{H}\left( P,Q \right)(V) = \operatorname{Iso}_{H}\left( P_{V \times_{A} \widetilde{A}} , Q_{V \times_{A} \widetilde{A}} \right)^\Gamma = \operatorname{Iso}_{(\Gamma,H)} \left( P_{V \times_{A} \widetilde{A}} , Q_{V \times_{A} \widetilde{A}} \right)
    \]
    is precisely the collection of isomorphisms of \((\Gamma,H)\)-torsors from \(P\) to \(Q\) after base change to \(V \times_{A} \widetilde{A}\).
    The equivalence of (1), (2), and (3) now follows from the well-known fact that a smooth morphism of schemes admits sections fpqc locally if and only if it admits sections étale locally if and only if it is surjective.

    It is obvious that (3) implies (4) so it remains to see that (4) implies (3). This may be thought of as a geometric version of Shapiro's Lemma. 
    By partitioning \(\overline{\kappa} \times_{\overline{x}, A} \widetilde{A}\) into \(\Gamma\)-orbits we may assume that \(\Gamma\) acts transitively on \(\overline{\kappa} \times_{\overline{x}, A} \widetilde{A}\).
    Note that 
    \[\operatorname{Iso}_{H}\left( P_{\overline{\kappa} \times_{\overline{x}, A} \widetilde{A}}, Q_{\overline{\kappa} \times_{\overline{x}, A} \widetilde{A}} \right) = \prod_{y \in \overline{\kappa} \times_{\overline{x}, A} \widetilde{A}} \operatorname{Iso}_{H} \left( P_{y}, Q_{y} \right) .  \]
    If \(\theta \in \Gamma\), then multiplication by \(\theta\) yields for any \(y \in \overline{\kappa} \times_{\overline{x}, A} \widetilde{A}\) a bijection \(\operatorname{Iso}_{H}\left( P_{y}, Q_{y} \right)^{\Gamma_{y}} \isoto \operatorname{Iso}_{H}\left( P_{\theta y}, Q_{\theta y} \right)^{\Gamma_{\theta y}}\). 
    From this observation it's not hard to see that projection onto some factor \(y\) gives a bijection \(\operatorname{Iso}_{H}\left( P_{\overline{\kappa} \times_{\overline{x}, A} \widetilde{A}}, Q_{\overline{\kappa} \times_{\overline{x}, A} \widetilde{A}} \right)^{\Gamma} \cong \operatorname{Iso}_{H} \left( P_{y}, Q_{y} \right)^{\Gamma_{y}}\). 
\end{proof}
In the context of \((\Gamma,H)\)-torsors on curves, as in \cite{balajiModuliParahoricMathcal2014}, \cite{damioliniLocalTypesGamma2023}, or \cite{pappas-rapoport-tamely-ramified-bundles}, 
it is condition (4) above that is used to define the local type.
Since every \(H\)-torsor over an algebraically closed field is trivializable, it follows from \cref{res:equivariant-torsors-and-group-cohomology} that the types in this case are even determined by cohomology 
classes of \(H^1\left( \Gamma_{y}, H(\kappa(y)) \right)\), and consequently all \((\Gamma,H)\)-torsors are classified by such data. 
This classification is used extensively in the cited works. 

\section{Generalities on positive loop groups}
\label{sec:loop-groups-generalities}
In this section we state some general results about positive loop groups that are difficult to track down in the literature. 

Let \(A\) be a base ring, and let \(X\) be an fpqc sheaf over \(A\llb t \rrb\).
We define a presheaf on \(A\)-algebras
\[
L^{+}X(R) = X\left( R\llb t \rrb \right) ,
\] the positive loop ``group'' for \(X\). We also define the \(n\)-jet
\[
L^{n}X(R) = X \left( R\llb t \rrb/t^{n} \right)
\] for any integer \(n \geq 1\).
Note that \(L^n\) the pushforward from fpqc sheaves over \(A[t]/t^n\) to fpqc sheaves over \(A\). 

\begin{proposition}
    \label{res:loop-group-lemma}
    Assume that \(A\) is noetherian and let \(X\) be an affine scheme over \(A\llb t \rrb\). 
    \begin{enumerate}
        \item For each \(n\), the fpqc sheaf \(L^n X\) is represented by an affine scheme over \(A\), say \(B_{n}\). Then \(L^+ X\) is represented by \(B := \colim B_{n}\), corresponding to the isomorphism of fpqc sheaves \(L^+ X \isoto \clim L^n X\). 
        \item If \(X\) is formally smooth, then for each \(n \in \mathbb{Z}_{\geq 1}\) the inclusions \(B_{n} \hookrightarrow B_{n+1}\) and \(B_{n} \hookrightarrow B\) admit retractions.
            Therefore we can identify each \(B_{n}\) as a subring of \(B\), and \(B = \bigcup_{n = 1}^\infty B_{n}\). 
        \item If \(X\) is smooth, then so is \(L^n X\). In addition, \(L^+X\) is formally smooth, flat, and reduced. 
        \item If \(Z \hookrightarrow X\) is a closed immersion, then so is \(L^n Z \hookrightarrow L^n X\) and \(L^+ Z \hookrightarrow L^+ X\).
        \item If \(U \hookrightarrow X\) is an open immersion, then so is \(L^n U \hookrightarrow L^n X\) and \(L^+ U \hookrightarrow L^+ X\). 
    \end{enumerate}
\end{proposition}
\begin{proof}
    \begin{enumerate}
        \item Note that \(L^n X = \operatorname{Res}^{A[t]/t^n}_{A} X\). Since \(A[t]/t^n\) is finite flat over \(A\), the first statement follows from \cite[Proposition A.5.2]{CGP}.
            Since \(X = \operatorname{Spec} \Lambda\) is affine, we have 
            \[
            L^+ X (R) = \operatorname{Hom}_{A\llb t \rrb}\left( \Lambda, R\llb t \rrb\right) 
            \cong \operatorname{Hom}_{A\llb t \rrb}\left( \Lambda, \clim R[t]/t^n \right)
            \cong \clim \operatorname{Hom}_{A\llb t \rrb}\left( \Lambda, R[t] / t^n \right) 
            \cong \clim L^n X (R) . 
            \]
            On the other hand, if \(B = \colim B_{n}\), then \(\operatorname{Hom}_{A}\left(B, R\right) \cong \clim \operatorname{Hom}_{A}\left(B_{n}, R\right) = \clim L^n X (R)\), so we see that \(L^+X\)
            is representable by \(B\).
        \item Assuming \(X\) is formally smooth, the map \(L^{n+1} X(R) = X(R[t]/t^{n+1}) \to  X(R[t]/t^{n}) = L^{n}X(R)\) is surjective for any \(A\)-algebra \(R\)
            because \(R[t]/t^{n+1} \to R[t]/t^{n}\) is a square zero extension. 
            Taking \(R = B_{n}\) we see that the identity map \(B_{n} = B_{n}\) lifts to a retraction \(B_{n+1} \to B_{n}\) of the map \(B_{n} \to B_{n+1}\).
        \item Smoothness of \(L^n X\) follows from \cite[Proposition A.5.2]{CGP}. 
            Formal smoothness of \(L^+ X\) is immediate from \(X\) being smooth (and in
            particular formally smooth). Being flat follows from
            \(B = \colim_{n}B_{n}\) being a colimit of smooth (and in particular
            flat) \(A\)-algebras, the fact that \(M\otimes_{A}\) preserves all
            colimits for any \(A\)-module \(M\), and the fact that taking filtered
            colimits of \(A\)-modules is exact. 

            Note that each \(L^n X\) is reduced by \cite[\href{https://stacks.math.columbia.edu/tag/034E}{Lemma 034E}]{stacks-project}.
            If \(f \in B\) is a (nilpotent) element, then \(f \in B_{n}\) for some \(n\), so reducedness of \(B\) follows from reducedness of each \(B_{n}\). 
        \item Let \(f : Z \hookrightarrow X\) be a closed immersion of affine \(A\llb t \rrb\)-schemes. 
            For each \(n\) the map \(L^n Z \hookrightarrow L^n Z\) is a closed immersion by \cite[Proposition A.5.5]{CGP}. 
            So if \(C_{n}\) and \(C\) are to \(Z\) what \(B_{n}\) and \(B\) are to \(X\), then the corresponding map \(B_{n} \to C_{n}\) identifies \(C_{n} = B_{n} / I_{n}\),
            where \(I_{n} \subset B_{n}\) is an ideal. If we let \(I = \colim I_{n}\), then by exactness of filtered colimits it follows that the map \(B \to C\) identifies \(C = B / I\),
            meaning that \(L^+ Z \hookrightarrow L^+ X\) is a closed immersion. 
        \item The statement about \(L^n\) follows from \cite[Proposition A.5.2]{CGP}, so it remains to prove the statement about \(L^+\).
            
            As an intermediate step, let us first prove that \(L^+ \mathbb{G}_{m} \hookrightarrow L^+ \mathbb{G}_{a}\) is open. 
            If we write \(\mathbb{G}_{a} = \operatorname{Spec} A\llb t \rrb[x]\) 
            then we can identify \(L^+ \mathbb{G}_{a} = \operatorname{Spec}A[a_{0},a_{1}, \dots]\) where \(a_{i}\) corresponds to the \(i\)'th coefficient of a power series. 
            Similarly, if we write \(\mathbb{G}_{m} = \operatorname{Spec} A \llb t \rrb [x]\), then we can identify \(L^+ \mathbb{G}_{m} = \operatorname{Spec}A[a_{0}^{\pm 1}, a_{1}, a_{2}, \dots]\), 
            using the fact that a power series \(a_{0} + a_{1}t + \cdots \in R \llb t \rrb\) is invertible if and only if \(a_{0}\) is invertible.
            So \(L^+ \mathbb{G}_{m} = D(a_{0}) \subset L^+ \mathbb{G}_{a}\) is the affine open on which \(a_{0}\) doesn't vanish. 
            
            We now consider the general case, so let \(U \subset X \) be an open immersion of \(A\llb t \rrb\)-schemes. 
            Let \(\lbrace f_{i} \rbrace_{i \in \mathcal{I}}\) be a collection of functions on \(X\) such that
            \(U = \bigcup_{i\in \mathcal{I}} D(f_{i})\).
            Viewing each \(f_{i}\) as a function \(X \to \mathbb{G}_{a}\), there are induced maps \(L^+ f_{i} : L^+ X \to L^+ \mathbb{G}_{a}\), and we can identify 
            \begin{align*}
            L^+ U(R) & = U(R\llb t \rrb) =  \lbrace x \in X(R\llb t \rrb) : f_{i}(x) \in R\llb t \rrb^\times \text{ for some }i\in \mathcal{I} \rbrace \\
            & = \lbrace x \in L^+ X(R) : L^+f_{i}(x) \in R\llb t \rrb^\times \text{ for some }i \in \mathcal{I} \rbrace \\
            & = \bigcup_{i\in \mathcal{I}} (L^+ f_{i})^{-1} \left( L^+ \mathbb{G}_{m} \right)(R) . 
            \end{align*}
            Since \(L^+ \mathbb{G}_{m} \subset L^+ \mathbb{G}_{a}\) is open, it follows that \(L^+U \subset L^+X\) is open.
    \end{enumerate}
\end{proof}
\begin{remark}
    In the case that \(A\) is a field, and \(X\) is of finite type over \(A \llb t \rrb\),
    some parts of the above result are stated without proof in \cite[Proposition 1.3.2]{zhu-affine-grassmannians}
    and proved in \cite[Section 1.a]{pappasTwistedLoopGroups2008}.
    Some general results on various loop groups can also be found in \cite{cesnaviciusAffineGrassmannianPresheaf2024}.
\end{remark}

\subsection{Congruence loop groups}
Let \(H\) be a smooth affine group scheme over \(A \llb t \rrb\). 
Define
\(L^{+n}H\) as the kernel in the exact sequence of fpqc sheaves 
\[
1 \to L^{+n}H \to L^{+}H \to L^{n}H \to 1 .
\]
Note that right exactness of this sequence is a consequence of formal smoothness of \(H\). 

\begin{lemma}
    \label{res:congruence-loop-group-lemma}
    Let
    \(H_{n}(R) := \ker \left( H(R) \to H \left( R / t^{n}R \right) \right)\)
    for any \(A\llb t \rrb\)-algebra \(R\). Then \(H_{n}\) is a smooth
    affine group scheme over \(A\llb t \rrb\), and \(L^{+n}H = L^+ H_{n}\). 
\end{lemma}
\begin{proof}
    It is obvious that \(L^{+n}H = L^+ H_{n}\). 
    By \cite[Lemma 3.7]{mayeuxNeronBlowupsLowdegree2023} we can identify \(H_{n}\) with the dilation of \(H\) in 
    the identity section over \(j : \operatorname{Spec} A [t]/t^{n} \hookrightarrow \operatorname{Spec}A\llb t \rrb\) 
    along this closed subscheme. The result is then immediate from \cite[Theorem 3.2]{mayeuxNeronBlowupsLowdegree2023}. 
\end{proof}

\subsection{The \(t\)-adic metric}
\label{sec:t-adic-metric}
Let \(X\) be a smooth affine scheme over \(A \llb t \rrb\). 
By \cref{res:loop-group-lemma} we can write 
\[
    L^+ X(R) = \clim L^n X (R)
\]
for any \(A\)-algebra \(R\). 
A point \(x \in L^+ X(R)\) can therefore be identified with a compatible sequence of points \((x_{1},x_{2},\dots)\), where \(x_{n} \in L^n X(R)\). 
We can equip \(L^+ X (R)\) with a complete ultrametric \(d\) by setting 
\[
    d(x,y) = 2^{- \inf \lbrace n | x_{n} \neq y_{n} \rbrace} . 
\]
We call \(d\) the \emph{\(t\)-adic metric} on \(L^+ X(R)\).

It is clear that for any map \(R \to S\) of \(A\)-algebras, the induced map \(L^+ X(R) \to L^+ X(S)\) is continuous. 
If \(f : X \to Y\) is a morphism of smooth affine group schemes over \(A \llb t \rrb\), then so is \(L^+ f : L^+ X(R) \to L^+ Y(R)\) for any \(R\). 

Suppose now that \(X = H\) is a smooth affine group scheme over \(A \llb t \rrb\). 
Then we can identify 
\[
    L^{+n} H (R) = \lbrace g \in L^+ H(R) | d(1, g) < 2^{-n} \rbrace . 
\]
In other words, the principal congruence subgroups \(\lbrace L^{+n}H(R) \rbrace_{n \in \mathbb{Z}_{\geq 1}}\) form a neighborhood basis of the identity. 

\subsection{Questions about reducedness}
We take care of some reducedness questions that are needed in the proof of \cref{res:conjugation-congruence-bound}. 
\begin{lemma}
    \label{res:reduced-times-positive-loop-group}
    Let \(X\) be a smooth affine scheme over \(A\llb t \rrb\), 
    and let \(Y\) be a reduced scheme over \(A\). 
    The scheme \(Y \times_{A} L^+ X\) is reduced. 
\end{lemma}
\begin{proof}
    According to \cref{res:loop-group-lemma}, we can write \(L^n X = \operatorname{Spec}B_{n}\) and \(L^+ X = \operatorname{Spec}B\), where \(B_{n} \subset B_{n+1}\subset \cdots \subset \bigcup_{n=1}^\infty B_{n} = B\). 
    Since the claim is Zariski local we may assume that \(Y = \operatorname{Spec}C\). 
    We have 
    \(C \otimes_{A} B \cong \colim \left( C \otimes_{A} B_{i} \right)\)
    and each
    \(C \otimes_{A} B_{i} \hookrightarrow C \otimes_{A} B\)
    is injective since \(B_{i} \hookrightarrow B\) has a retraction.
    Since any (nilpotent) element
    \(f \in C \otimes_{A} B\) is contained in some
    \(C \otimes_{A} B_{i}\), the result follows from the fact that \(C \otimes_{A} B_{i}\) is reduced by 
    \cref{res:loop-group-lemma} (3) and \cite[\href{https://stacks.math.columbia.edu/tag/034E}{Lemma 034E}]{stacks-project}. 
\end{proof}

\begin{lemma}
    \label{res:reduced-closure-lemma}
    Let \(H\) be a smooth affine group scheme over \(A \llb t \rrb\), 
    amd let \(Y\) be a scheme over \(A\). 
    Let \(Z \subset Y\) be a subscheme, and let \(\overline{Z} \subset Y\) be the reduced closure of \(Z\) in \(Y\). 
    Then \(\overline{Z} \times_{A} L^+ H\) is the reduced closure of \(Z \times_{A} L^+ H\) in \(Y \times_{A} L^+ H\). 
\end{lemma}
\begin{proof}
    Let \(V \subset Y \times_{A} L^+ H\) denote the reduced closure of \(Z \times_{A} L^+ H\).
    The inclusion \(V \hookrightarrow Y \times_{A} L^+ H\) factors through \(\overline{Z} \times_{A} L^+ H\) by \cref{res:reduced-times-positive-loop-group}, 
    and we have to show that \(V \hookrightarrow \overline{Z} \times_{A} L^+ H\) is an isomorphism.
    By continuity of the right \(L^+ H\)-action via translations, \(V\) is stable under this action, and therefore descends to 
    a scheme \(V' = \left[ V / L^+ H \right]\) by fpqc descent \cite[\href{https://stacks.math.columbia.edu/tag/0244}{Section 0244}]{stacks-project}. 
    We have \(Z \hookrightarrow V' \hookrightarrow \overline{Z} \hookrightarrow Y\), and since \(V' \hookrightarrow Y\) is a closed immersion \cite[\href{https://stacks.math.columbia.edu/tag/0420}{Lemma 0420}]{stacks-project},
    it follows that \(V' = \overline{Z}\). So \(V= V' \times_{A} L^+H = \overline{Z} \times_{A} L^+ H\).
\end{proof}

\end{document}